\documentclass[11pt]{article}

\usepackage{amsmath,amsthm,amsfonts,amssymb,graphicx,amsbsy}
\usepackage{fullpage}
\usepackage{hyperref}
\usepackage{cite}
\usepackage{tocloft}

\usepackage{enumerate}
\makeatletter
\renewcommand*\env@matrix[1][*\c@MaxMatrixCols c]{%
  \hskip -\arraycolsep
  \let\@ifnextchar\new@ifnextchar
  \array{#1}}
\makeatother

\newcommand{\reals}{\mathbb{R}}
\newcommand{\complexes}{\mathbb{C}}
\newcommand{\ints}{\mathbb{Z}}
\newcommand{\posints}{{\ints^+}}
\newcommand{\rats}{\mathbb{Q}}
\newcommand{\union}{\cup}
\newcommand{\intersect}{\cap}
\newcommand{\map}[3]{{{#1}:{#2}\rightarrow{#3}}}
\newcommand{\floor}[1]{{\lfloor{#1}\rfloor}}
\newcommand{\bigfloor}[1]{{\left\lfloor{#1}\right\rfloor}}
\newcommand{\ceiling}[1]{{\lceil{#1}\rceil}}
\newcommand{\bigceiling}[1]{{\left\lceil{#1}\right\rceil}}
\newcommand{\bigabs}[1]{{\left|{#1}\right|}}
\newcommand{\convex}{\mathcal{C}}
\newcommand{\discrete}{\mathcal{D}}
\newcommand{\p}{\varphi}
\newcommand{\myand}{\;\mathrel{\&}\;}
\newcommand{\myor}{\;\mathrel{\vee}\;}
\newcommand{\eps}{\varepsilon}
\newcommand{\clcl}[1]{{\left[{#1}\right]}}
\newcommand{\clop}[1]{{\left[{#1}\right)}}
\newcommand{\opcl}[1]{{\left({#1}\right]}}
\newcommand{\opop}[1]{{\left({#1}\right)}}
\newcommand{\mat}[1]{{\left[\begin{matrix}#1\end{matrix}\right]}}

\newcommand{\tp}[1]{{{#1}^{\textsf T}}}
\newcommand{\tpp}[2]{{{#1}^{\textsf T}_{#2}}}
\newcommand{\one}{\mathbf{1}}
\newcommand{\bv}[1]{\mathbf{#1}}

\newcommand{\eqdf}{\mathrel{:=}}
\newcommand{\cmpl}{\setminus}
\newcommand{\loopcl}{\textit{lcl}}
\newcommand{\packrad}{{\mu_\textup{p}}}
\newcommand{\coverrad}{{\mu_\textup{c}}}

\newcommand{\mymax}{\textup{max}}
\newcommand{\mymin}{\textup{min}}
\newcommand{\cM}{\mathcal{M}}
\newcommand{\x}{\mathop{\star}}
\newcommand{\xs}[1]{\mathop{\star_{#1}}}
\newcommand{\xbx}{\xs{[x]}}
\newcommand{\xD}{\xs{D}}
\newcommand{\xL}{\xs{\Lambda}}
\newcommand{\xl}{\xs{\lambda}}
\newcommand{\xa}{\xs{\alpha}}
\newcommand{\xb}{\xs{\beta}}
\newcommand{\xsig}{\xs{\sigma}}
\newcommand{\xssig}{\xs{2\sigma}}

\newcommand{\xm}{\xs{\mu}}
\newcommand{\xmi}{\xs{\mu_i}}

\newcommand{\xx}{\xs{x}}
\newcommand{\xz}{\xs{z}}
\newcommand{\xq}{{\textstyle\xs{q}}}

\newcommand{\xainv}{\xs{1/\alpha}}
\newcommand{\xbinv}{\xs{1/\beta}}
\newcommand{\affemb}{\mathrel{\hookrightarrow}}
\newcommand{\affeq}{\mathrel{\leftrightharpoons}}
\newcommand{\gal}{\textup{Gal}}
\newcommand{\stab}{\textup{Stab}}
\newcommand{\bigoh}{\mathcal{O}}
\newcommand{\Pe}{{\sf P}}
\newcommand{\NP}{{\sf NP}}

\renewcommand{\Re}{\textup{Re}}
\renewcommand{\Im}{\textup{Im}}
\newcommand{\divides}{\mid}
\newcommand{\notdiv}{\nmid}
\newcommand{\tuple}[1]{{\langle{#1}\rangle}}
\newcommand{\diag}{\textup{diag}}
\newcommand{\intr}[1]{{{#1}^\textup{int}}}
\newcommand{\characteristic}{\textup{char}}
\newcommand{\ordst}[1]{{{#1}^{\textit{st}}}}

\newcommand{\ordth}[1]{{{#1}^{\textit{th}}}}

\newenvironment{algo}{\begin{tabbing}\hspace{0.25in}\=\hspace{0.25in}\=\hspace{0.25in}\=\hspace{0.25in}\=\hspace{0.25in}\=\kill}{\end{tabbing}}
\newcommand{\assn}{{\ensuremath{\;\mathrel{\leftarrow}}\;}}

\newcommand{\mywhile}{\textbf{while }}
\newcommand{\mydo}{\textbf{ do }}
\newcommand{\myfor}{\textbf{for }}
\newcommand{\myto}{\textbf{ to }}
\newcommand{\mydownto}{\textbf{ downto }}
\newcommand{\myend}{\textbf{end}}

\newenvironment{remark}{\begin{paragraph}{Remark.}}{\hfill$\Box$\end{paragraph}\medskip}

\theoremstyle{definition}
\newtheorem{definition}{Definition}[section]
\newtheorem{notation}[definition]{Notation}

\theoremstyle{plain}
\newtheorem{theorem}[definition]{Theorem}
\newtheorem{lemma}[definition]{Lemma}
\newtheorem{fact}[definition]{Fact}

\newtheorem{proposition}[definition]{Proposition}
\newtheorem{corollary}[definition]{Corollary}
\newtheorem{claim}[definition]{Claim}
\newtheorem{conjecture}[definition]{Conjecture}
\newtheorem{open}[definition]{Open Question}
\newtheorem{research}[definition]{Research Plan}

\title{Fixed-Parameter Extrapolation and Aperiodic Order}

\author{Stephen Fenner\thanks{Computer Science and Engineering Department, Columbia, SC 29208 USA\@.  Partially supported by NSF grant CCF-0915948.}\\
\small University of South Carolina\\[-0.5ex]
\small \texttt{fenner@cse.sc.edu}
\and
Frederic Green\thanks{Department of Mathematics and Computer Science, Clark University, Worcester, MA 01610}\\
\small Clark University\\[-0.5ex]
\small \texttt{fgreen@clarku.edu}
\and
Steven Homer\thanks{Computer Science Department, Boston University, Boston, MA 02215.  Partially supported by NSF grant CCF-1533663.}\\
\small Boston University\\[-0.5ex]
\small \texttt{homer@cs.bu.edu}
}

\begin{document}

\bibliographystyle{plain}

\maketitle

\begin{abstract}
Fix any $\lambda\in\complexes$.  We say that a set $S\subseteq\complexes$ is \emph{$\lambda$-convex} if, whenever $a$ and $b$ are in $S$, the point $(1-\lambda)a+\lambda b$ is also in $S$.  If $S$ is also (topologically) closed, then we say that $S$ is \emph{$\lambda$-clonvex}.  We investigate the properties of $\lambda$-convex and $\lambda$-clonvex sets and prove a number of facts about them.  Letting $R_\lambda\subseteq\complexes$ be the least $\lambda$-clonvex superset of $\{0,1\}$, we show that if $R_\lambda$ is convex in the usual sense, then $R_\lambda$ must be either $[0,1]$ or $\reals$ or $\complexes$, depending on $\lambda$.  We investigate which $\lambda$ make $R_\lambda$ convex, derive a number of conditions equivalent to $R_\lambda$ being convex, and give several conditions sufficient for $R_\lambda$ to be convex or not convex; in particular, we show that $R_\lambda$ is either convex or uniformly discrete.  Letting $\convex := \{\lambda\in\complexes\mid \mbox{$R_\lambda$ is convex}\}$, we show that $\complexes\cmpl\convex$ is closed, discrete and contains only algebraic integers.  We also give a sufficient condition on $\lambda$ for $R_\lambda$ and some other related $\lambda$-convex sets to be discrete by introducing the notion of a strong PV number.  These conditions give rise to a number of periodic and aperiodic Meyer sets (the latter sometimes known as ``quasicrystals'').

The paper is in four parts.  Part~I describes basic properties of $\lambda$-convex and $\lambda$-clonvex sets, including convexity versus uniform discreteness.  Part~II explores the connections between $\lambda$-convex sets and quasicrystals and displays a number of such sets, including several with dihedral symmetry. Part~III generalizes a result from Part~I about the $\lambda$-convex closure of a path, and Part~IV contains our conclusions and open problems.

Our work combines elementary concepts and techniques from algebra and plane geometry.\\


\end{abstract}

\pagebreak

\tableofcontents

\pagebreak

\begin{center}
\noindent{\Large\bf Part I: Introduction and Basic Properties}\addcontentsline{toc}{part}{Part I: Introduction and Basic Properties}
\end{center}

\section{Introduction}

\begin{definition}\label{def:lambda-convex}\rm
Fix a number $\lambda\in\complexes$.  For any $a,b\in\complexes$ define $a \xl b \eqdf (1-\lambda)a + \lambda b$.

Then for any set $S\subseteq\complexes$,
\begin{enumerate}
\item\label{clprop:alg}
we say that $S$ is \emph{$\lambda$-convex} iff for every $a,b\in S$, the point $a\xl b$ is in $S$, and
\item\label{clprop:top}
we say that $S$ is \emph{$\lambda$-convex closed} (or \emph{$\lambda$-clonvex} for short) iff $S$ is $\lambda$-convex and (topologically) closed.
\end{enumerate}
In either case, we say that $S$ is \emph{nontrivial} if $S$ contains at least two distinct elements.  We will informally say, ``$\lambda$-c[l]onvex'' when we want to assert analogous things about both notions, respectively.
\end{definition}

For fixed $\lambda$, we defined $\x \eqdf \xl$ as a two-place operation on $\complexes$.  We call $a\x b$ the \emph{$\lambda$-extrapolant} of $a$ and $b$, and we say that $a\x b$ is obtained from $a$ and $b$ by \emph{$\lambda$-extrapolation}.  Then the first property in Definition~\ref{def:lambda-convex} just says that $S$ is closed under $\lambda$-extrapolation.  Of course, if $0\le\lambda\le 1$, then this might more appropriately be called $\lambda$-interpolation, but as we will see, the case where $\lambda \notin \clcl{0,1}$ is much more interesting. When we are not explicit about $\lambda$, we refer to the operation $\xl$ as
\emph{fixed-parameter extrapolation} or \emph{fixed-parameter affine combination}.

We may drop the subscript and just say $a\x b$ if the value of $\lambda$ is clear from the context.  We may also drop parentheses in an expression involving $\x$ or $\xl$, assuming that this operator binds more tightly than $+$ or $-$ but less tightly than multiplication or division.

By the definition of convexity, a set $S\subseteq\complexes$ is convex if and only if, for all $\lambda \in \opop{0,1}$, $S$ is $\lambda$-convex.  Definition~\ref{def:lambda-convex} above is in part motivated by the following additional observation (Proposition~\ref{prop:convex-iff-lambda-closed}, below): If $S$ is a \emph{closed} set, then for \emph{any fixed} $\lambda\in\opop{0,1}$, we have that $S$ is convex if and only if $S$ is $\lambda$-convex.  We are generally interested in $\lambda$-c[l]onvexity for $\lambda\notin\clcl{0,1}$, and we are particularly interested in minimal nontrivial $\lambda$-c[l]onvex sets.

\begin{definition}\label{def:lambda-convex-closure}\rm
For any $\lambda\in\complexes$ and any set $S\subseteq\complexes$,
\begin{enumerate}
\item
We define the \emph{$\lambda$-convex closure} of $S$, denoted $Q_\lambda(S)$, to be the $\subseteq$-minimum $\lambda$-convex superset of $S$.  We let $Q_\lambda$ be shorthand for $Q_\lambda(\{0,1\})$, the $\lambda$-convex closure of $\{0,1\}$.
\item
We define the \emph{$\lambda$-clonvex closure} of $S$, denoted $R_\lambda(S)$, to be the $\subseteq$-minimum $\lambda$-clonvex superset of $S$.  We let $R_\lambda$ be shorthand for $R_\lambda(\{0,1\})$, the $\lambda$-clonvex closure of $\{0,1\}$.
\end{enumerate}
\end{definition}

$R_\lambda$ is a minimal nontrivial $\lambda$-clonvex set because it is generated by just two distinct points.  We choose the points $0$ and $1$ for convenience, but since $\lambda$-c[l]onvexity is invariant under orientation-preserving similarity transformations (i.e., $\complexes$-affine transformations, i.e., polynomials of degree $1$; see Definition~\ref{def:rho}, below), any two initial points would yield a set with the same essential properties.  One of our main goals, then, is to characterize $R_\lambda$ for as many $\lambda$ as we can.

\bigskip



We conclude this introduction with some historical background and motivation.
The notion of fixed-parameter extrapolation was first investigated for its
own intrinsic interest by Calvert~\cite{Calvert} and, some years later, by  Pinch~\cite{Pinch},
under the name of ``$a$-convexity." Berman \& Moody~\cite{BeMo} were the first
to notice its application to discrete $Q_\lambda$
in the context of {\it quasicrystals}, where they call
it ``quasicrystal addition,"
specifically in the case of $\lambda = 1+\p$ (where $\p$
 is the golden ratio), which we cover in this paper as well.
   This value of $\lambda$ is significant because it gives the simplest example where $Q_\lambda$ is discrete and aperiodic.
Other authors have investigated fixed-parameter extrapolation under
various other names, including quasiaddition and $\tau$-inflation \cite{MPP:selfsimilar}, \cite{MPP:sconvexjournal}, \cite{MPS}.

 The comparatively young field of {\it aperiodic order} \cite{BG:aperiodic-order} arose, in part, to explain
such mathematical phenomena as aperiodic tilings and natural phenomena such as quasicrystals. Traditionally, aperiodic order has been studied in terms
of paradigms such as local substitution, inflation tilings, and model sets
(i.e., cut-and-project sets). 
As Berman and Moody \cite{BeMo} have pointed out, and as investigated further
by other authors
\cite{MPP:selfsimilar}, \cite{MPP:sconvexjournal}, and \cite{MPS},
fixed-parameter extrapolation offers an alternative
approach. Indeed, as is hinted in \cite{BeMo} and stated explicitly
in \cite{MPP:sconvexjournal}, one may view the binary extrapolation operation
as a mathematical model of aperiodic crystal growth.
The discrete sets generated by fixed-parameter extrapolation
 share many properties
with {\it Meyer sets} \cite{Meyer:sets72}, widely regarded as the mathematical counterpart of quasicrystals.  Indeed, we can establish in many cases that these sets \emph{are} Meyer sets.

The goals of this paper are threefold. 
First, we work towards a systematic and unified theory of fixed-parameter
extrapolation and the point sets closed under that operation.
We characterize convex sets closed under the operation, 
although many open questions remain in this regard.
 Secondly, we extend aspects of the theory
  already accomplished in previous work (notably \cite{BeMo}, \cite{Calvert}, \cite{Pinch},
\cite{MPP:selfsimilar}, \cite{MPP:sconvexjournal}, and \cite{MPS})
from the case of \emph{real} $\lambda$ to \emph{complex} $\lambda$.
To our knowledge, our work is the first to examine 
 $\lambda \in \complexes \setminus \reals$. In this context, we
 study precise connections between
fixed-parameter extrapolation and other constructions such as inflation tilings and model sets.
For example, we give a partial characterization of parameters
that lead to discrete sets, in terms of a refinement of Pisot-Vijayaraghavan (PV) numbers, which we call
{\it strong PV} numbers.
These generalize
the subset of real irrationals considered in \cite{Pinch}.
 All of the sets we find to be discrete are in fact subsets of cut-and-project sets.
 It remains to be
established if they are also Meyers sets, although we have been able to establish
this in many cases.  Finally, and along
 somewhat different lines, although the sets we construct are manifestly hierarchical (with
the fixed parameter playing a role analogous to the inflation scale), the relationship
between our approach and inflation tilings is at present unclear. Ultimately, we would like to
understand what added insight our particular approach offers to the theory of quasicrystals.
 For that purpose, we strive in this paper to classify these sets according to a number of their properties, including but not limited to aperiodicity, uniform discreteness, relative density, finite local complexity, and repetivity.\\


The techniques of this paper
draw on diverse fields, including complex analysis, algebra, algebraic number theory,
topology, combinatorics, and computer science. In support of theoretical studies, 
 computation (including symbolic computation)
 and computer graphics have been used to guide our work and
 determine future problems and directions.

This paper is divided into four parts. Part~I treats the basic definitions and properties
of $\lambda$-convexity and fixed-parameter extrapolation. Here we include various
characterizations of $\lambda$-clonvex sets and criteria for determining when
$Q_\lambda$ is or is not discrete. Part~II investigates the connection between
fixed-parameter extrapolation and aperiodic order. The central result of Part~II
is a sufficient condition for discreteness of $Q_\lambda(S)$ for certain finite $S\subseteq\complexes$. 
We also determine the $\lambda$-convex closure of various regular shapes,
in both two and three dimensions, and explore
particular values of $\lambda$ that yield discrete sets that are also relatively dense.
Part~III generalizes one of the characterizations of Part~I, Section~\ref{sec:equivalences}, from differentiable
paths to what we call ``bent paths,'' which are not required to be differentiable.
Finally, in Part~IV we present concluding remarks and open problems.

\section{Basics}

We start with a few basic facts and definitions.  In this paper, we call a theorem a ``Fact'' when it is either immediately obvious or has a routine, straightforward proof.  We omit the proofs of Facts.

For $z\in\complexes$, we let $\Re(z)$ and $\Im(z)$ denote the real and imaginary parts of $z$, respectively, and we let $z^*$ denote the complex conjugate of $z$.

If $\map{f}{X}{Y}$ is some function with domain $X$, and $A$ is any subset of $X$, then we let $f|A$ denote the function $f$ restricted to domain $A$.

Any topological references assume the usual topology on $\complexes \cong \reals^2$.  For $A\subseteq\complexes$, we let $\overline A$ denote the topological closure of $A$.

We use the symbol $:=$ to mean, ``equals by definition.''  We set $\tau := 2\pi$ throughout.  For any $x\in\reals$, let $x \bmod \tau$ denote the unique $y\in\clop{0,\tau}$ such that $(x-y)/\tau$ is an integer.

We let $\posints$ denote the set of positive integers.

Whenever a ring is mentioned, it will be assumed to be unital, that is, possessing a multiplicative identity.

Operations on numbers lift to operations on sets of numbers in the usual way.  This includes subtraction, and so $S-T = \{x-y \mid x\in S \myand y\in T\}$, i.e., the Minkowski difference of $S$ and $T$.  We use $S \cmpl T$ to denote the relative complement of $T$ in $S$.\\

We note the following simple property of the fixed-parameter extrapolation operator.

\begin{fact}\label{fact:E-lambda}
For all $a,b,c,d,x,y,\lambda\in\complexes$,
\begin{align*}
a\xl a &= a\;, \\
(xa+yb)\xl(xc+yd) &= x(a\xl c) + y(b\xl d)\;.
\end{align*}
In particular, setting $x := 1-\mu$ and $y:=\mu$ gives $(a\xm b)\xl (c\xm d) = (a\xl c)\xm (b\xl d)$ for any $\mu\in\complexes$.  If $\mu = \lambda$, then we have the \emph{entropic law}
\[ (a\xl b)\xl (c\xl d) = (a\xl c)\xl (b\xl d)\;. \]
\end{fact}

\begin{remark}
A groupoid satisfying the identity $(ab)(cd) \approx (ac)(bd)$ is known as a \emph{medial} groupoid.  If the operation is also idempotent ($aa\approx a$), then the groupoid is sometimes called a \emph{medial band}, \emph{(groupoid) mode}, or \emph{idempotent medial groupoid}, as well as other names.  These structures have been studied extensively in the literature.  See, for example, \cite{JK:medial-groupoids,Dudek:medial1,Dudek:medial2,CD:medial3}.  These two identities are not the only ones universally satisfied by $\xl$.  For example, $(u \xl (v \xl w)) \xl ((x \xl y) \xl z) = (u \xl (x \xl w)) \xl ((v \xl y) \xl z)$ for all $u,v,w,x,y,z,\lambda\in\complexes$, and this identity does not follow from idempotence and the entropic law, above.
\end{remark}

\begin{notation}
For the rest of this section, we use $\x$ with no subscript to mean $\xl$.
\end{notation}

We can stratify the set $Q_\lambda(S)$ as follows:

\begin{definition}\label{def:stratify-Q-lambda}\rm
For any $\lambda\in\complexes$ and $S\subseteq\complexes$, we define $Q_\lambda^{(0)}(S) := S$, and for all integers $n\ge 0$ we inductively define $Q_\lambda^{(n+1)}(S) := \{ a\x b \mid a,b \in Q_\lambda^{(n)}(S) \}$.  We use $Q_\lambda^{(n)}$ to denote $Q_\lambda^{(n)}(\{0,1\})$.
\end{definition}

\begin{fact}\label{fact:stratify-Q-lambda}
For any $\lambda\in\complexes$ and $S\subseteq\complexes$,
\begin{itemize}
\item
$Q_\lambda^{(n)}(S) \subseteq Q_\lambda^{(n+1)}(S)$ for all integers $n\ge 0$ (noticing that $\x$ is idempotent), and
\item
$\bigcup_{n=0}^\infty Q_\lambda^{(n)}(S) = Q_\lambda(S)$.
\item
If $S$ is countable, then $Q_\lambda(S)$ is countable.
\end{itemize}
\end{fact}

\begin{definition}\label{def:lambda-S-rank}\rm
For  any $\lambda\in\complexes$, any $S\subseteq\complexes$, and any $z\in Q_\lambda(S)$, we define the \emph{$(\lambda,S)$-rank} of $z$ to be the least $n$ such that $z \in Q_\lambda^{(n)}(S)$.
\end{definition}

Some of our proofs will use induction on the $(\lambda,S)$-rank of a point.

In the expression $(1-\lambda)a+\lambda b$, it will sometimes be useful to treat $\lambda$ as the variable.

\begin{definition}\label{def:rho}\rm
For any $a,b\in\complexes$, define the function $\map{\rho_{a,b}}{\complexes}{\complexes}$ by
\[ \rho_{a,b}(z) := a\xz b = (1-z)a + zb \]
for all $z\in\complexes$.
\end{definition}

\begin{fact}\label{fact:rho-basic}
For all $a,b\in\complexes$,
\begin{enumerate}
\item
$\rho_{a,b}$ is the unique $\complexes$-affine map (polynomial of degree $\le 1$ or orientation-preserving similarity transformation) that maps $0\mapsto a$ and $1\mapsto b$.
\item
$\rho_{a,b}$ is continuous.
\item
If $a\ne b$, then $\rho_{a,b}$ is a bijection (a homeomorphism, in fact), and for all $z\in\complexes$,
\[ (\rho_{a,b})^{-1}(z) = \frac{z-a}{b-a}\;. \]
It follows that $(\rho_{a,b})^{-1} = \rho_{x,y}$, where
\begin{align*}
x &= \frac{a}{a-b}\;,  & y &= \frac{a-1}{a-b}\;.
\end{align*}
\item\label{item:rho-composition}
For all $x,y\in\complexes$,
\[ \rho_{a,b} \circ \rho_{x,y} = \rho_{\rho_{a,b}(x),\rho_{a,b}(y)}\;. \]
Equivalently, we have the following distributive law: for all $z\in\complexes$,
\[ \rho_{a,b}(x\xz y) = \rho_{a,b}(x)\xz \rho_{a,b}(y)\;. \]
\end{enumerate}
\end{fact}

\begin{lemma}\label{lem:rho-preserves-convexity}
For any $a,b,\lambda\in\complexes$ and $S\subseteq\complexes$, if $S$ is $\lambda$-convex (respectively, $\lambda$-clonvex), then $\rho_{a,b}(S)$ is $\lambda$-convex (respectively, $\lambda$-clonvex).
\end{lemma}

\begin{proof}
Suppose $S$ is $\lambda$-convex.  If $a=b$, then the statement is trivial, so we assume $a\ne b$.  Fix any $x,y\in\rho_{a,b}(S)$ and let $u,v\in S$ be such that $x=\rho_{a,b}(u)$ and $y=\rho_{a,b}(v)$.  Then
\[ x\x y = \rho_{a,b}(u)\x \rho_{a,b}(v) = \rho_{a,b}(u\x v) \]
by the distributive law above.  We have $u\x v \in S$ because $S$ is $\lambda$-convex; thus $x\x y \in \rho_{a,b}(S)$.  This proves that $\rho_{a,b}(S)$ is $\lambda$-convex.

If, in addition, $S$ is closed, then so is $\rho_{a,b}(S)$, because $\rho_{a,b}$ is a homeomorphism.  This proves that $\rho_{a,b}$ preserves $\lambda$-clonvexity as well.
\end{proof}

\begin{fact}
$Q_\lambda(S) \subseteq R_\lambda(S)$ for all $\lambda\in\complexes$ and $S\subseteq\complexes$.
\end{fact}

The next lemma gives a basic relationship between $Q_\lambda$ and $R_\lambda$.  Recall that $\overline A$ denotes the topological closure of set $A$.

\begin{lemma}\label{lem:topo-closure}
For any $\lambda\in\complexes$ and $S\subseteq\complexes$, \ $\overline{Q_\lambda(S)} = R_\lambda(S)$.
\end{lemma}

\begin{proof}
The $\subseteq$-containment is obvious because $R_\lambda(S)$ is closed and contains $Q_\lambda(S)$.  For the $\supseteq$-containment, we just need to show that $\overline{Q_\lambda(S)}$ is $\lambda$-convex.  This just follows from the continuity of $\map{\x}{\complexes\times\complexes}{\complexes}$: for all $A,B\subseteq\complexes$, we have $\overline A\x \overline B = \x\left(\overline A\times\overline B\right) = \x\left(\overline{A\times B}\right) \subseteq \overline{\x(A\times B)}= \overline{A\x B}$.  Setting $A \eqdf B \eqdf Q_\lambda(S)$ gives $\overline{Q_\lambda(S)}\x \overline{Q_\lambda(S)} \subseteq \overline{Q_\lambda(S)\x Q_\lambda(S)} \subseteq \overline{Q_\lambda(S)}$.
\end{proof}

Next we give a general lemma from which many of the results of this section follow easily.  This lemma will also be used in Part~II.

\begin{lemma}\label{lem:distribute}
For all $x,y,\lambda\in\complexes$ and all $S,T\subseteq\complexes$,
\begin{align*}
Q_\lambda(xS+yT) &= xQ_\lambda(S) + yQ_\lambda(T)\;, \\
R_\lambda(xS+yT) &= \overline{xR_\lambda(S) + yR_\lambda(T)}\;.
\end{align*}
\end{lemma}

\begin{proof}
For the first equation of the lemma, we get $\subseteq$ by noticing that the right-hand side includes $xS+yT$ and is $\lambda$-convex.  To show this latter fact, start with any $a,b\in xQ_\lambda(S) + yQ_\lambda(T)$, let $a_s\in Q_\lambda(S)$ and $a_t\in Q_\lambda(T)$ be such that $a = xa_s + ya_t$, and choose $b_s\in Q_\lambda(S)$ and $b_t\in Q_\lambda(T)$ similarly for $b$.  Then
\[ a\x b = (xa_s+ya_t)\x(xb_s+yb_t) = x(a_s\x b_s) + y(a_t\x b_t) \in xQ_\lambda(S) + yQ_\lambda(T)\;, \]
the second equation above using Fact~\ref{fact:E-lambda}.

To prove $\supseteq$ in the first equation of the lemma, let $a\in Q_\lambda(S)$ and $b\in Q_\lambda(T)$ be arbitrary.  Let $s$ be the $(\lambda,S)$-rank of $a$ and let $t$ be the $(\lambda,T)$-rank of $b$ (cf.\ Definition~\ref{def:lambda-S-rank}).  We show that $xa+yb\in Q_\lambda(xS+yT)$ by induction on $s+t$.  If $s+t=0$, then $a\in S$ and $b\in T$, so $xa+yb \in xS+yT \subseteq Q_\lambda(xS+yT)$.  Now suppose $s+t>0$ and the inclusion holds for all rank sums less than $s+t$.  We prove the case where $s>0$, the case where $t>0$ being similar.  Since $s>0$, we have $a = a_1\x a_2$ for some $a_1,a_2\in Q_\lambda(S)$, both with $(\lambda,S)$-rank less than $s$.  Then by Fact~\ref{fact:E-lambda} again,
\[ xa+yb = x(a_1\x a_2) + y(b\x b) = (xa_1+yb)\x(xa_2+yb)\;. \]
By the inductive hypothesis, $xa_1+yb$ and $xa_2+yb$ are both in $Q_\lambda(xS+yT)$, and so by $\lambda$-convexity, $xa+yb \in Q_\lambda(xS+yT)$.

The second equation follows from the first by taking the closure of both sides and using Lemma~\ref{lem:topo-closure} and the fact that $xC+yD \subseteq x\overline{C} + y\overline{D} \subseteq \overline{xC+yD}$ for any $C,D\subseteq\complexes$.
\end{proof}

The next lemma helps to justify our arbitrary choice of $0$ and $1$ in the definitions of $Q_\lambda$ and $R_\lambda$.

\begin{lemma}\label{lem:translate-Q-R}
For any $a,b,\lambda\in\complexes$ and any set $S\subseteq\complexes$,
\begin{align*}
\rho_{a,b}(Q_\lambda(S)) &= Q_\lambda(\rho_{a,b}(S))\;, \\
\rho_{a,b}(R_\lambda(S)) &= R_\lambda(\rho_{a,b}(S))\;.
\end{align*}
In particular, $\rho_{a,b}(R_\lambda)$ is the $\lambda$-clonvex closure of $\{a,b\}$.
\end{lemma}

\begin{proof}
Note that $\rho_{a,b}(z) = a + (b-a)z$ for any $z\in\complexes$.  So applying Lemma~\ref{lem:distribute}, we get
\begin{align*}
\rho_{a,b}(Q_\lambda(S)) &= a + (b-a)Q_\lambda(S) = a\{1\} + (b-a)Q_\lambda(S) = aQ_\lambda(\{1\}) + (b-a)Q_\lambda(S) \\
&= Q_\lambda(a\{1\} + (b-a)S) = Q_\lambda(a+(b-a)S) = Q_\lambda(\rho_{a,b}(S))\;,
\end{align*}
which establishes the first equation.  The second equation is obtained by taking the closure of both sides of the first, observing that $\overline{\rho_{a,b}(Q_\lambda(S))} = \rho_{a,b}(\overline{Q_\lambda(S)})$, and using Lemma~\ref{lem:topo-closure}.
\end{proof}

The next fact can be seen by noticing that $a\xl b = b\xm a$ for all $a,b,\lambda\in\complexes$, where $\mu = 1-\lambda$.

\begin{fact}\label{fact:inner-dual}
A set is $\lambda$-c[l]onvex if and only if it is $(1-\lambda)$-c[l]onvex.  Thus $Q_\lambda(S) = Q_{1-\lambda}(S)$ and $R_\lambda(S) = R_{1-\lambda}(S)$ for any $S\subseteq\complexes$ and $\lambda\in\complexes$.
\end{fact}

\begin{fact}\label{fact:conjugate}
For any $\lambda\in\complexes$,
\begin{itemize}
\item
$Q_\lambda(S)^* = Q_{\lambda^*}(S^*)$ and $R_\lambda(S)^* = R_{\lambda^*}(S^*)$ for any $S\subseteq\complexes$.
\item
In particular, $(Q_\lambda)^* = Q_{\lambda^*}$ and $(R_\lambda)^* = R_{\lambda^*}$.
\item
Thus $R_\lambda$ is convex if and only if $R_{\lambda^*}$ is convex.
\end{itemize}
\end{fact}

The following geometric picture of $a, b$ and $a\x b$ is especially useful for constructions involving $\lambda \in \complexes \cmpl \reals$. See Figure~\ref{fig:similar}.

\begin{figure}[htbp]
\begin{center}
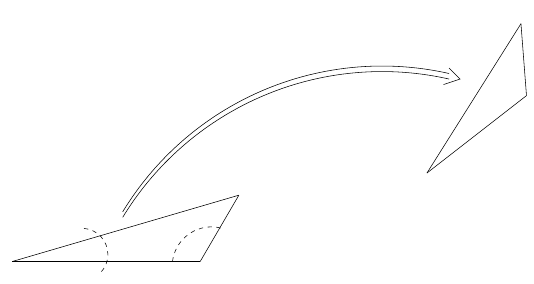
\caption{The $\complexes$-affine transformation that takes $(0,1,\lambda) \mapsto (a, b, a\xl b)$ preserves angles, so that the triangles $(0,1,\lambda)$ and $(a, b, a\xl b))$ are similar.}\label{fig:similar}
\end{center}
\end{figure}

\begin{fact}\label{fact:similar-triangles}
By Fact~\ref{fact:rho-basic}(1), for any $a, b, \lambda \in \complexes$, the points $a$, $b$, and $a\x b$ form a triangle that is similar to the one formed by $0$, $1$, and $\lambda$.
\end{fact}

\begin{definition}\label{def:char-angles}\rm
For $\lambda\in\complexes\cmpl\reals$, we call the angles $\theta$ (formed by $(1,0,\lambda)$) and $\p$ (formed by $(\lambda, 1, 0)$), as indicated in Figure~\ref{fig:similar}, the {\it characteristic angles} of $\lambda$.  We assume $0 < \theta,\p < \pi$.
\end{definition}


\begin{proposition}\label{prop:convex-iff-lambda-closed}
If $S\subseteq \complexes$ is closed, then for any fixed $\lambda\in\opop{0,1}$, we have that $S$ is convex if and only if $S$ is $\lambda$-convex. 
\end{proposition}

\begin{proof}
Clearly, if $S$ is convex, it is $\lambda$-convex for any fixed $\lambda\in\opop{0,1}$.

Now suppose $S$ is $\lambda$-convex for some fixed $\lambda\in\opop{0,1}$.  Let $I := \clcl{0,1}$.  Then consider two points $a, b \in S$, and the line segment $L := \rho_{a,b}(I) = \{x \mid x = (1-\ell)a + \ell b$ for some $\ell \in I\}$, which connects $a$ and $b$. We now show that $S$ is dense in the set $L$.  This is sufficient for the proposition: Since $S$ is closed by hypothesis, this implies that in fact $L \subseteq S$, from which it follows that $S$ is convex.

Suppose, then, that $S$ is not dense in $L$.  That is, there is some nonempty open\footnote{with respect to the induced topology on $L$} subset $O\subseteq L$ that does not intersect with $S$.  $O$ is the unique union of disjoint nonempty open intervals.  Let $J := \opop{x,y}\subseteq O$ be one of these intervals, for some $x<y$.  Then $x\in S$ and $y\in S$, and by $\lambda$-convexity, $x\x y\in S$.  But $x<x\x y<y$, and so $x\x y\in S\cap J$, a contradiction.
%
\end{proof}


\begin{corollary}\label{cor:convex-hull}
If $T\subseteq\complexes$ and $0<\lambda<1$, then $R_\lambda(T)$ is the closure of the convex hull of $T$.
\end{corollary}

\begin{proof}
Let $S := R_\lambda(T)$, and let $S'$ be the (topological) closure of the convex hull of $T$.  We have $T\subseteq S$ and $S$ is closed and $\lambda$-convex, whence it follows that $S$ is convex by Proposition~\ref{prop:convex-iff-lambda-closed}, and thus $S' \subseteq S$.  Conversely, the closure of the convex hull of any set is also convex.  Thus $S'$ is $\lambda$-convex by the same proposition, and this together with the inclusion $T\subseteq S'$ imply $S\subseteq S'$.
\end{proof}

Now we consider the minimal nontrivial $\lambda$-clonvex set $R_\lambda = R_\lambda(\{0,1\})$.  If $R_\lambda$ happens to be convex, then characterizing $R_\lambda$ is easy.

\begin{theorem}\label{thm:convex-is-easy}
Suppose $R_\lambda$ is convex.
\begin{enumerate}
\item
If $\lambda \in \clcl{0,1}$, then $R_\lambda = \clcl{0,1}$.
\item
If $\lambda \in \reals \cmpl \clcl{0,1}$, then $R_\lambda = \reals$.
\item
If $\lambda \in \complexes \cmpl \reals$, then $R_\lambda = \complexes$.
\end{enumerate}
\end{theorem}

It will be convenient later to define the following:

\begin{definition}\label{def:F-lambda}
For any $\lambda\in\complexes$, define
\[ F_\lambda = \left\{ \begin{array}{ll}
\clcl{0,1} & \mbox{if $\lambda\in\clcl{0,1}$,} \\
\reals & \mbox{if $\lambda \in \reals \cmpl \clcl{0,1}$,} \\
\complexes & \mbox{if $\lambda \in \complexes \cmpl \reals$.}
\end{array} \right. \]
\end{definition}

Then Theorem~\ref{thm:convex-is-easy} states simply that if $R_\lambda$ is convex, then $R_\lambda = F_\lambda$.

\begin{proof}
For (1), we have $\clcl{0,1}\subseteq R_\lambda$ by convexity, and if $\lambda\in\clcl{0,1}$, then it is obvious that $\clcl{0,1}$ is $\lambda$-clonvex, since $(1-\lambda)a + \lambda b$ always lies on the line segment connecting $a$ and $b$.  Thus $R_\lambda \subseteq \clcl{0,1}$ by the minimality of $R_\lambda$.

For (2), we can assume WLOG that $\lambda > 1$ (otherwise consider $1-\lambda$ and use Fact~\ref{fact:inner-dual}).  Certainly, $\lambda^0 = 1 \in R_\lambda$, and if $\lambda^n \in R_\lambda$ for some integer $n\ge 0$, then $\lambda^{n+1} = 0\x\lambda^n \in R_\lambda$ as well.  Thus by induction, $\lambda^n \in R_\lambda$ for all integers $n\ge 0$.  Since the sequence $1,\lambda,\lambda^2,\lambda^3,\ldots$ increases without bound, we have $\clop{1,\infty} \subseteq R_\lambda$ by convexity.  Similarly, the sequence $1, 1-\lambda, 1-\lambda^2,1-\lambda^3,\ldots$ lies entirely within $R_\lambda$ (by induction, if $1-\lambda^n$ is in $R_\lambda$, then so is $1-\lambda^{n+1} = 1\x(1-\lambda^n)$).  This latter sequence decreases without bound, and thus $\opcl{-\infty,1}\subseteq R_\lambda$ by convexity.

For (3), we use a trick suggested by George McNulty: we show that $R_\lambda$ is open, and thus, since $R_\lambda$ is nonempty and also closed, we must have $R_\lambda = \complexes$.  Since $\lambda \notin \reals$, we can represent $\lambda$ in polar form as $\lambda = r e^{i\theta}$, where $r = |\lambda| > 0$, and $\theta = \arg\lambda\in\reals$ is not a multiple of $\pi$.  The value of $\theta$ is determined modulo $\tau$, and so we take $\theta$ to have the least possible absolute value, giving $0 < |\theta| < \pi$.  Now consider any point $a\in R_\lambda$.  Since $R_\lambda$ has at least two points, there is some other point $b\in R_\lambda \cmpl \{a\}$.  Now define the following sequence of points, all of which are in $R_\lambda$:
\begin{align*}
b_0 &:= b\;, \\
b_1 &:= a\x b_0\;, \\
&\;\;\vdots \\
b_{i+1} &:= a\x b_i\;, \\
&\;\;\vdots
\end{align*}
Set
\[ k := \bigfloor{\frac{\pi}{|\theta|}} + 1\;, \]
the least integer such that $k|\theta| > \pi$.  Then $a$ lies in the interior of the convex hull of $\{b_0,b_1,\ldots,b_k\}$, as illustrated in Figure~\ref{fig:convex-hull}.
\begin{figure}[htbp]
\begin{center}
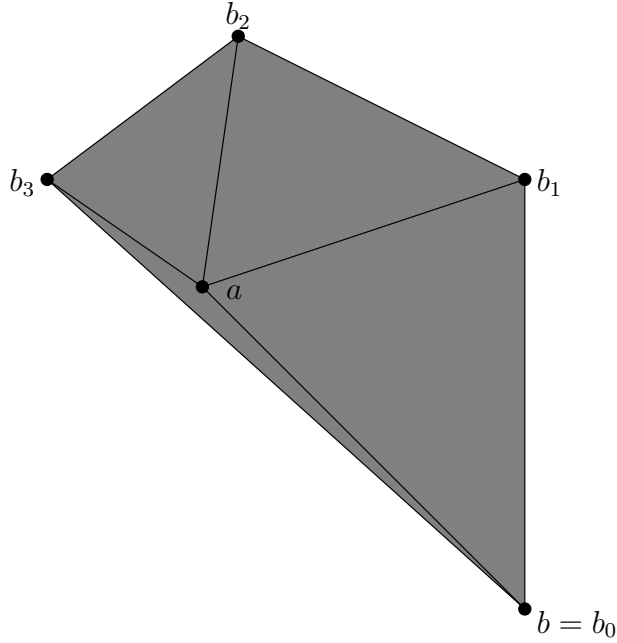
\caption{In this example where $\lambda = (1+2i)/3$ and $k = 3$, the point $a$ lies in the interior of the convex hull of $\{b_1,b_1,b_2,b_3\}$.}\label{fig:convex-hull}
\end{center}
\end{figure}

Since $R_\lambda$ is convex, it contains this convex hull, whence $a$ lies in the interior of $R_\lambda$.  Since $a\in R_\lambda$ was chosen arbitrarily, it follows that $R_\lambda$ is open.
\end{proof}

In light of Theorem~\ref{thm:convex-is-easy}, most of the rest of the paper concentrates on determining, for various $\lambda\in\complexes$, whether or not $R_\lambda$ is convex, and if not, characterizing $R_\lambda$.  We start with a basic definition followed by a trivial observation.

\begin{definition}\label{def:convex}\rm
Let $\convex := \{ \lambda\in\complexes : \mbox{$R_\lambda$ is convex} \}$.  Let $\discrete := \complexes \cmpl \convex$.
\end{definition}

\begin{fact}\label{fact:trivial}\ 
\begin{enumerate}
\item
$R_0 = R_1 = Q_0 = Q_1 = \{0,1\}$, hence $0\notin\convex$ and $1\notin\convex$.
\item
If $0<\lambda<1$, then $R_\lambda = \clcl{0,1}$ by Proposition~\ref{prop:convex-iff-lambda-closed} and Theorem~\ref{thm:convex-is-easy}, hence $\opop{0,1} \subseteq \convex$.
\end{enumerate}
\end{fact}

\begin{lemma}\label{lem:basic-translate}
For any $a,b,\lambda\in\complexes$ and any $S\subseteq\complexes$,
\begin{enumerate}
\item
if $\rho_{a,b}(S) \subseteq Q_\lambda(S)$, then $\rho_{a,b}(Q_\lambda(S)) \subseteq Q_\lambda(S)$;
\item
if $\rho_{a,b}(S) \subseteq R_\lambda(S)$, then $\rho_{a,b}(R_\lambda(S)) \subseteq R_\lambda(S)$.
\end{enumerate}
\end{lemma}

\begin{proof}
Using the first assumption and Lemma~\ref{lem:translate-Q-R}, we get
\[ \rho_{a,b}(Q_\lambda(S)) = Q_\lambda(\rho_{a,b}(S)) \subseteq Q_\lambda(Q_\lambda(S)) = Q_\lambda(S)\;. \]
Using the second assumption and Lemma~\ref{lem:translate-Q-R}, we get
\[ \rho_{a,b}(R_\lambda(S)) = R_\lambda(\rho_{a,b}(S)) \subseteq R_\lambda(R_\lambda(S)) = R_\lambda(S)\;. \]
\end{proof}

Lemma~\ref{lem:basic-translate} and Lemma~\ref{lem:extend} (below) have some useful corollaries.

\begin{corollary}\label{cor:translate}
For any $a,b,\lambda\in\complexes$,
\begin{enumerate}
\item
if $a\in Q_\lambda$ and $b\in Q_\lambda$, then $\rho_{a,b}(Q_\lambda) \subseteq Q_\lambda$;
\item
if $a\in R_\lambda$ and $b\in R_\lambda$, then $\rho_{a,b}(R_\lambda) \subseteq R_\lambda$.
\end{enumerate}
\end{corollary}

\begin{proof}
Set $S = \{0,1\}$ and use Lemma~\ref{lem:basic-translate}.
\end{proof}

Part~(1.) of the next lemma will be used in Section~\ref{sec:bent-path}.

\begin{lemma}\label{lem:extend}
For any $\lambda,\mu\in\complexes$ and $S\subseteq\complexes$,
\begin{enumerate}
\item
if $\mu\in Q_\lambda$, then $Q_\lambda(S)$ is $\mu$-convex, and consequently, $Q_\mu(S) \subseteq Q_\lambda(S)$;
\item
if $\mu\in R_\lambda$, then $R_\lambda(S)$ is $\mu$-clonvex, and consequently, $R_\mu(S) \subseteq R_\lambda(S)$.
\end{enumerate}
\end{lemma}

\begin{proof}
For part~(1.), suppose $\mu\in Q_\lambda$.  Then for any $a,b\in Q_\lambda(S)$,
\[ a\xm b = \rho_{a,b}(\mu) \in \rho_{a,b}(Q_\lambda) = Q_\lambda(\{a,b\}) \subseteq Q_\lambda(Q_\lambda(S)) = Q_\lambda(S)\;, \]
where the second equation follows from Lemma~\ref{lem:translate-Q-R} with $S = \{0,1\}$.  This shows that $Q_\lambda(S)$ is $\mu$-convex.  A similar argument holds for part~(2.).
\end{proof}

\begin{corollary}\label{cor:extend}
For any $\lambda,\mu\in\complexes$,
\begin{enumerate}
\item
if $\mu\in Q_\lambda$, then $Q_\lambda$ is $\mu$-convex, and consequently, $Q_\mu \subseteq Q_\lambda$;
\item
if $\mu\in R_\lambda$, then $R_\lambda$ is $\mu$-clonvex, and consequently, $R_\mu \subseteq R_\lambda$.
\end{enumerate}
\end{corollary}

\begin{corollary}
For any $\lambda\in\complexes$, the sets $Q_\lambda$ and $R_\lambda$ are both closed under the ternary operation $(\mu,a,b)\mapsto a\xm b$.  In particular, $Q_\lambda$ and $R_\lambda$ are both closed under multiplication.
\end{corollary}

\begin{proof}
Given any $\mu,a,b\in Q_\lambda$, we have $a\xm b \in Q_\mu(\{a,b\}) \subseteq Q_\lambda(\{a,b\}) \subseteq Q_\lambda(Q_\lambda) = Q_\lambda$, the first inclusion using part~(1.)\ of Lemma~\ref{lem:extend}.  Similarly for $R_\lambda$, using part~(2.)\ of Lemma~\ref{lem:extend}.  For closure under multiplication, we notice that $\mu\nu = 0\xm\nu$ for any $\mu,\nu\in\complexes$.
\end{proof}

\begin{definition}
For any $S\subseteq\complexes$, define $1-S := \{ 1-x \mid x\in S\}$ as usual.
\end{definition}

Note that $1-S = \rho_{1,0}(S)$, for any $S\subseteq\complexes$.

\begin{corollary}\label{cor:outer-dual}
$R_\lambda = 1 - R_\lambda$ for any $\lambda\in\complexes$.
\end{corollary}

\begin{proof}
We have
\[ 1-R_\lambda = \rho_{1,0}(R_\lambda) \subseteq R_\lambda = \rho_{1,0}(\rho_{1,0}(R_\lambda)) \subseteq \rho_{1,0}(R_\lambda) = 1 - R_\lambda\;. \]
Both $\subseteq$-steps follow from Corollary~\ref{cor:translate}.
\end{proof}

\begin{corollary}\label{cor:convex-extend}
For any $\lambda,\mu\in\complexes$, if $\mu\in R_\lambda$ and $R_\mu$ is convex, then $R_\lambda$ is convex.
\end{corollary}

\begin{proof}
Assume $\mu\in R_\lambda$ and $R_\mu$ is convex.  To show that $R_\lambda$ is convex, it suffices to show that for any $a,b\in R_\lambda$ and $x\in \clcl{0,1}$, the point $a\xx b$ is in $R_\lambda$.  We have $0,1\in R_\mu$, and so by Corollary~\ref{cor:extend} and the convexity of $R_\mu$, we have
\[ \clcl{0,1} \subseteq R_\mu \subseteq R_\lambda\;. \]
Thus, for any $x\in \clcl{0,1}$, we have
\[ a\xx b = \rho_{a,b}(x) \in \rho_{a,b}(\clcl{0,1}) \subseteq \rho_{a,b}(R_\lambda) \subseteq R_\lambda\;. \]
\end{proof}

\begin{corollary}
For any $\lambda\in\complexes$, \ $\lambda\in \convex$ if and only if $R_\lambda \intersect \convex \ne \emptyset$.
\end{corollary}

\begin{proposition}\label{prop:convex-implies-convex}
If $R_\lambda$ is convex, then all $\lambda$-clonvex sets are convex.
\end{proposition}

\begin{proof}
Suppose $R_\lambda$ is convex, and let $A$ be any $\lambda$-clonvex set.  For any $a,b\in A$, the line segment connecting $a$ and $b$ is $\rho_{a,b}(\clcl{0,1})$.  Since $R_\lambda$ is convex, we have $\clcl{0,1}\subseteq R_\lambda$, and thus
\[ \rho_{a,b}(\clcl{0,1}) \subseteq \rho_{a,b}(R_\lambda) = R_\lambda(\rho_{a,b}(\{0,1\})) = R_\lambda(\{a,b\}) \subseteq R_\lambda(A) = A\;. \]
The first equality follows from Lemma~\ref{lem:translate-Q-R}; the last equality holds because $A$ is $\lambda$-clonvex.
\end{proof}

\section{Equivalent characterizations of convexity for $\lambda$-clonvex sets}
\label{sec:equivalences}

Throughout this section, we continue to use $\x$ without a subscript to denote $\xl$.

For any $a, b \in \complexes$, a \emph{path} from $a$ to $b$ is a continuous function $\map{\sigma}{\clcl{0,1}}{\complexes}$ such that $\sigma(0) = a$ and $\sigma(1) = b$.  If $a=b$, then $\sigma$ is a \emph{loop}.  A set $S \subseteq \complexes$ is said to be \emph{path-connected} if it contains a path between any two of its points.\footnote{Strictly speaking, as we identify the path with the function $\sigma$, it is more accurate to say that $S$ contains all the points in the image of some path connecting the two points. However, we will assume that the meaning will be clear from the context.}

In this section we consider five possible properties of a $\lambda$-clonvex set and the implications between them.  Throughout this section, we will adopt the convention that $\lambda$ denotes an arbitrary complex number and that $A$ denotes an arbitrary $\lambda$-clonvex set containing at least two distinct points.  Here are the five properties we will consider:
\begin{enumerate}
\item\label{item:convex} $A$ is convex.
\item\label{item:connected} $A$ is path-connected.
\item\label{item:path} $A$ contains a nontrivial (i.e., nonconstant) path.
\item\label{item:accumulation} $A$ has an accumulation point.
\item\label{item:exists-a-and-b} There exist $a, b\in\complexes$ such that $0 < |a - b| < 1$ and $\rho_{a,b}(A) \subseteq A$ (i.e., $A$ is self-similar).
\end{enumerate}
In particular, we show (Corollary~\ref{cor:general-equivalences}, below) that these five properties are all equivalent when $A = R_\lambda$, while some implications do not hold for all $\lambda$-clonvex sets.  Results similar to some of these below were shown in the case of $\lambda\in\reals$ by Pinch~\cite{Pinch}.

We refer to the above properties by their numbers in parentheses.

\begin{fact}\label{fact:1-to-4}
For all $\lambda$ and $A$ subject to this section's convention, (\ref{item:convex}) $\Rightarrow$ (\ref{item:connected}) $\Rightarrow$ (\ref{item:path}) $\Rightarrow$ (\ref{item:accumulation}).
\end{fact}

\begin{theorem}
For all $\lambda$ and $A$ subject to this section's convention, (\ref{item:convex}) $\Rightarrow$ (\ref{item:exists-a-and-b}).
\end{theorem}

\begin{proof}
Choose any point $x\in A$, and consider the map
\[ \psi := \rho_{x,x+1}\circ \rho_{0,1/2}\circ (\rho_{x,x+1})^{-1}\;. \]
It is easy to check that for any $z\in\complexes$, \ $\psi(z)$ is the midpoint $(x+z)/2$ of $x$ and $z$.  Thus $\psi(A)\subseteq A$, because $A$ is assumed to be convex.  Using Fact~\ref{fact:rho-basic}, we get $\psi = \rho_{a,b}$, where
\begin{align*}
a &= \psi(0) = \frac{x}{2}\;, & b &= \psi(1) = \frac{x}{2} + \frac{1}{2}\;.
\end{align*}
We have $|a-b| = 1/2$, which implies (\ref{item:exists-a-and-b}).
\end{proof}

\begin{theorem}\label{thm:exists-a-b-accumulation}
For all $\lambda$ and $A$ subject to this section's convention, (\ref{item:exists-a-and-b}) $\Rightarrow$ (\ref{item:accumulation}).
\end{theorem}

\begin{proof}
Let $a,b$ be such that $0<|a-b|<1$ and $\rho_{a,b}(A) \subseteq A$.  Letting $\delta := |a-b|$, we see that for any $x,y\in\complexes$,
\[ \rho_{a,b}(x) - \rho_{a,b}(y) = (1-x)a + xb - (1-y)a - yb = (b-a)(x-y)\;, \]
and thus
\begin{equation}\label{eqn:contraction}
\left|\rho_{a,b}(x) - \rho_{a,b}(y)\right| = \delta |x-y|\;.
\end{equation}
It is easy to check that the map $\rho_{a,b}$ on $\complexes$ has the unique fixed point
\[ z := \frac{a}{1+a-b}\;, \]
and a routine induction on $n$ using Equation~(\ref{eqn:contraction}) shows that for any $w\in\complexes$, \ $\left|z - \rho_{a,b}^{(n)}(w)\right| = \delta^n|z-w|$ for $n = 0,1,2,\ldots\,$.  Thus if $z\ne w$, then $z$ is an accumulation point of the sequence
\[ w,\; \rho_{a,b}(w), \;\rho_{a,b}(\rho_{a,b}(w)), \;\ldots, \;\rho_{a,b}^{(n)}(w),\;\ldots. \]
If, in addition, $w\in A$ (and there must exist such a $w$, because $A$ contains at least two points by convention), then all the elements of this sequence are in $A$ by the self-similarity assumption.  We then get $z\in A$ by the fact that $A$ is closed.
\end{proof}

Recall the definition of $F_\lambda$ in Definition~\ref{def:F-lambda}.

\begin{theorem}\label{thm:accum-implies-convex}
For all $\lambda$ and $A$ subject to this section's convention and such that $\lambda \notin \{0,1\}$ and $A\subseteq F_\lambda$, (\ref{item:accumulation}) $\Rightarrow$ (\ref{item:convex}).
\end{theorem}

\begin{proof}
We consider the case where $\lambda \in \reals\cmpl\clcl{0,1}$ first, which was essentially proved by Pinch~\cite[Proposition~7]{Pinch}.  This case is required, but it also gives a simpler version of the proof for when $\lambda$ is complex.  By Proposition~\ref{prop:convex-iff-lambda-closed}, if $0 < \lambda < 1$, then $A$ is convex, regardless of whether or not $A\subseteq F_\lambda$ (which equals $\clcl{0,1}$ in this case).  Therefore---since $\lambda\notin\{0,1\}$ by assumption---we may assume that $\lambda \not \in \clcl{0,1}$, in which case, $A\subseteq F_\lambda = \reals$.  Since $\lambda$-clonvexity is the same as $(1-\lambda)$-clonvexity by Fact~\ref{fact:inner-dual}, we may further assume that $\lambda > 1$.

Now let $C$ be the set of all accumulation points of $A$, and suppose that $C\ne\emptyset$.  Note that $C\subseteq A$, because $A$ is closed.  We show for any $x \in \reals$ that $x\in C$, from which it follows that $A = \reals$.  Let $a\in C$ be closest to $x$ among all the elements of $C$.  Such a point $a$ exists, because $C$ is closed and nonempty.  If $x=a$, then we are done, so suppose that $x < a$ (there is no essential difference with the case where $a < x$).  Then for some sequence $\{a_n\}\subseteq A \cmpl \{a\}$, \ $a = \lim_{n\rightarrow \infty} a_n$.  Define the sequence $\{b_n\}$ as follows:
\begin{eqnarray*}
 b_n
     =   \left\{ \begin{array}{ll}
                     a\x a_n   & \mbox{if $a_n < a$,} \\
                     a_n\x a   & \mbox{if $a_n > a$.}
                 \end{array}
         \right.
\end{eqnarray*}
Evidently, $b_n < a$ for all $n$, and $\lim_{n\rightarrow\infty} b_n = a$.  Furthermore, for each $n$, we have $b_n \in C$, because
\begin{eqnarray*}
 b_n
     =   \left\{ \begin{array}{ll}
                     \lim_{m\rightarrow\infty} a_m\x a_n   & \mbox{if $a_n < a$,} \\
                     \lim_{m\rightarrow\infty} a_n\x a_m   & \mbox{if $a_n > a$,}
                 \end{array}
         \right.
\end{eqnarray*}
and moreover, for all $m,n$, \ $a_m\x a_n \in A$, and $a_m\x a_n \ne b_n$ (respectively, $a_n\x a_m \ne b_n$) if $a_n < a$ (respectively, $a_n > a$).  Since $b_n$ converges to $a$, we find for sufficiently large $n$ that $x < b_n < a$, so $b_n$ is in $C$ and is closer to $x$ than $a$ is, contradicting the hypothesis that $a$ was the closest.

Now suppose that $\lambda \not \in \reals$, and that $A$ is arbitrary (subject to this section's convention) but contains an accumulation point $a$.  We show that in such a case, $A = \complexes$.  We do this by showing that any point $x \in \complexes$ is an accumulation point of $A$, from which the result follows by the fact that $A$ is closed.  The proof is in the same in spirit as the case where $\lambda \in \reals$, but here there is no division into cases as to whether $x$ is to the ``left'' or ``right'' of $a$.  For ease of illustration, we will assume that $\Im(\lambda) > 0$, as depicted in Fig.~\ref{fig:similar}.  The case where $\Im(\lambda) < 0$ is entirely similar.

The proof is by contradiction.  Again, let $C\subseteq A$ be the set of all accumulation points of $A$, and let $b \in C$ be the closest to $x$ of any point in $C$.  Such a point $b$ exists, because the set $\{z\in C : |z-x| \le |a-x|\}$, as well as being nonempty, is closed and bounded, and hence compact.  Now assume for the sake of contradiction that $x\notin A$.  Then $x\ne b$.  Draw a circle with $x$ as the center and $b$ on the circle.  Let $D$ denote the open disk bounded by the circle.  We will show now that $C \intersect D \ne \emptyset$, which contradicts the hypothesis that $b$ is the closest point in $C$ to $x$.  For ease in visualization, suppose $b$ is at the top of the circle (see Fig.~\ref{fig:accpoint}).

\begin{figure}[htbp]
\begin{center}
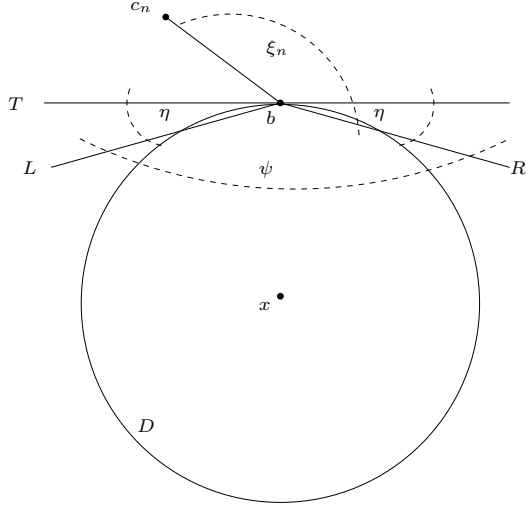
\caption{Theorem~\ref{thm:accum-implies-convex}: Construction of accumulation point $x$.}\label{fig:accpoint}
\end{center}
\end{figure}

Suppose the sequence $\{b_n\}\subseteq A \cmpl \{b\}$ converges to $b$, that is, $\lim_{n\rightarrow \infty} b_n = b$.  Let $\{c_n\}$ denote the sequence defined by $c_n = b_n\x b$ for all $n$.  Note that for each $n$, \ $c_n\in A\cmpl \{b\}$, and moreover, $c_n$ is itself in $C$, because $c_n \ne b_n\x b_m\in A$ for all $m$, and
\begin{eqnarray*}
   \lim_{m\rightarrow \infty} b_n\x b_m = \lim_{m\rightarrow\infty} ((1-\lambda)b_n + \lambda b_m) = (1-\lambda)b_n + \lambda b = c_n\;.
\end{eqnarray*}
Thus if any $c_n \in D$, then we are done, but it is possible that $c_n \not \in D$ for all $n$.  We therefore show how to ``rotate'' the sequence $\{c_n\}$ so that it is contained in $D$ for all sufficiently large $n$.  Let $T$ denote the tangent to the circle at $b$.  With $\theta$ and $\p$ denoting the characteristic angles of $\lambda$ (see Definition~\ref{def:char-angles} and Fig.~\ref{fig:similar}), let $\eta$ be any angle obeying $0 < \eta < \min(\theta,\p,(\pi-\theta-\p)/2)$, and observe that $\eta < \pi/2$.  Form the two rays $L$ and $R$ intersecting at $b$ and making an angle $\eta$ with $T$ as shown in the figure.  Let $\xi_n$ denote the counterclockwise angle formed by $R$ and the line segment $(b, c_n)$ connecting $c_n$ with $b$, also as shown in the figure.  Since $c_n$ could be anywhere except $b$, we have $0\le \xi_n < \tau$, where $\xi_n = 0$ corresponds to the ray $R$.  Let $\psi$ denote the angle subtended by $L$ and $R$; thus $\psi = \pi - 2\eta$.  Note that by the choice of $\eta$, we have $\p < \theta + \p < \pi - 2\eta = \psi$.  Our goal now is to find an accumulation point below the lines $L$ and $R$, and inside $D$.

To do this, let $q_n$ denote the least integer such that $q_n\p > \xi_n$. Then $q_n\p = (q_n - 1)\p + \p \le \xi_n + \p < \xi_n + \psi$.  Thus $\xi_n < q_n\p < \xi_n  + \psi$, so that the angle $q_n\p$ takes us from the line segment $(b,c_n)$ \emph{clockwise} to a ray through $b$, below $T$, and strictly between $L$ and $R$.  Note that since $\xi_n < \tau$, \ $0\le q_n-1 \le \lfloor \tau/\p \rfloor$, so for all $n$, \ $q_n$ can only take on a finite number of values, independent of $n$.

Next, for each $n$, form the finite sequence $c_n^{(0)}, c_n^{(1)},\dots,c_n^{(q_n)}$, where
\begin{align*}
 c_n^{(0)} &= c_n\\
 c_n^{(1)} &= c_n^{(0)}\x b\\
&\;\;\vdots \\
 c_n^{(i+1)} &= c_n^{(i)}\x b \\
&\;\;\vdots
\end{align*}
(This is essentially the same construction as in Theorem~\ref{thm:convex-is-easy}.  Also note Fig.~\ref{fig:convex-hull}, although it is not necessary here to form a convex hull.)
Each $c_n^{(i)}$ is in $A$, similarly to $c_n$. 
By the definition of $\p$, for each $i$, the clockwise angle going from the line segment $(b, c_n^{(i)})$ to the line segment $(b, c_n^{(i+1)})$ is $\p$.  Thus the clockwise angle from $(b, c_n)$ to $(b, c_n^{(q_n)})$ is $q_n\p$.  Hence, the point $c_n^{(q_n)}$ is in the desired wedge-shaped region beneath $L$ and $R$.  However, it may be too far from $b$ to be in the interior of $D$.  Now observe that, by virtue of the fact that $c_n^{(i+1)}$ is always constructed from $b$ and $c_n^{(i)}$ via similar triangles, there is a constant $k$ such that $|b - c_n^{(i+1)}| \le k|b - c_n^{(i)}|$.  Thus $|b - c_n^{(q_n)}| \le k^{q_n}|b - c_n|$.  But since $\{c_n\}$ converges to $b$, for any $\epsilon > 0$, there exists an $m$ such that $|b - c_m| \le \epsilon/k^{\lfloor 2\pi/\p\rfloor + 1} \le \epsilon/k^{q_m}$.  In that case, $|b -  c_m^{(q_m)}| \le \epsilon$, so $\epsilon$ may be chosen sufficiently small that $c_m^{(q_m)}$ is contained in $D$.  (And indeed, the sequence $\{c_n^{(q_n)}\}$ converges to $b$.)
\end{proof}

Some kind of constraint on $\lambda$ and $A$ in Theorem~\ref{thm:accum-implies-convex}---beyond this section's convention---is necessary to obtain the implication (\ref{item:accumulation}) $\Rightarrow$ (\ref{item:convex}).  For example, if $\lambda\in\{0,1\}$, then any closed subset of $\complexes$ is $\lambda$-clonvex, and so we may take $A := \{0\} \union \{ 1/n : n\in\posints \}$, which has $0$ as an accumulation point but is not convex.  If $\lambda\in\reals \cmpl \clcl{0,1}$ but $A\not\subseteq\reals$, then the implication still holds provided either $A$ lies entirely on a single line or $A$ contains a nonempty open set (cf.\ Proposition~\ref{prop:open-set}, below).  Otherwise, the implication may not hold: let $\lambda := 2$ and consider the set $A := R_2(\{0,1,\sqrt 2,i\})$.  Then it is a short exercise to show that
\[ A = \{x+yi \mid x\in\reals \myand y\in\ints \}\;, \]
which has accumulation points (paths, in fact) but is not convex.

Property (\ref{item:exists-a-and-b}) of the next corollary provides a useful shortcut for proving that $R_\lambda$ is convex.  Pinch essentially proved for $\lambda\in\reals$ that $(4)\Rightarrow (1)$~\cite[Propositions~5,7]{Pinch} and that $(5)\Rightarrow (4)$~\cite[Proposition~10]{Pinch}.

\begin{corollary}\label{cor:general-equivalences}
For any $\lambda\in\complexes$, the following are equivalent:
\begin{enumerate}
\item $R_\lambda$ is convex.
\item $R_\lambda$ is path-connected.
\item $R_\lambda$ contains a path.
\item $R_\lambda$ has an accumulation point.
\item There exist $a, b\in R_\lambda$ such that $0 < |a - b| < 1$.
\end{enumerate}
\end{corollary}

\begin{proof}
If $A = R_\lambda$, then we merely note that the property (\ref{item:exists-a-and-b}) of Corollary~\ref{cor:general-equivalences} is equivalent to the property (\ref{item:exists-a-and-b}) given earlier in this section: if $a,b\in R_\lambda$, then $\rho_{a,b}(R_\lambda) \subseteq R_\lambda$ by Corollary~\ref{cor:translate}.  Conversely, if $\rho_{a,b}(R_\lambda) \subseteq R_\lambda$, then $\{a,b\} = \rho_{a,b}(\{0,1\}) \subseteq \rho_{a,b}(R_\lambda) \subseteq R_\lambda$.
\end{proof}

Corollary~\ref{cor:general-equivalences} presents a nice dichotomy: $R_\lambda$ is either convex (hence either $\clcl{0,1}$, $\reals$, or $\complexes$) or uniformly discrete (with no two points less than unit distance apart).  The former holds when $\lambda\in\convex$; the latter when $\lambda\in\discrete$.

We end this section with some basic facts about $Q_\lambda(S)$ for certain $\lambda$ and $S$.  First we show that, given any disk $D \subseteq \complexes$ and any $\lambda\in\complexes \cmpl \clcl{0,1}$, we can construct a larger concentric disk $D' \subseteq Q_\lambda(D)$ (analogous to the construction of successively larger intervals in the case $\lambda \in \reals$).  From this it follows immediately that $Q_{\lambda}(D) = \complexes$ (Corollary~\ref{cor:disk-enlarge}, below).

\begin{lemma}\label{lem:disk-enlarge}
Fix $\lambda\in\complexes$ and let $\delta := |\lambda| + |1-\lambda|$.  For any $x\in\complexes$ and $r \ge 0$, let $D_{x,r} := \{z\in\complexes : |z-x|\le r\}$ be the closed disk of radius $r$ centered at $x$.  Then $Q_\lambda^{(1)}(D_{x,r}) = D_{x,\delta r}$.
\end{lemma}

\begin{proof}
For any $a,b\in D_{x,r}$, we have $|a\x b - x| = |(a-x)\x(b-x)| \le |1-\lambda||a-x| + |\lambda||b-x| \le \delta r$, and thus $Q_\lambda^{(1)}(D_{x,r}) \subseteq D_{x,\delta r}$.  For the reverse inclusion, pick any $y\in D_{x,\delta r}$ (and so $|y-x|\le \delta r$).  We can assume $\lambda\notin\{0,1\}$, for otherwise the result is trivial.  Define
\begin{align*}
a &:= x + (y-x)\frac{|1-\lambda|}{\delta(1-\lambda)}\;, & b &:= x + (y-x)\frac{|\lambda|}{\delta\lambda}\;.
\end{align*}
One readily checks that $|a-x|\le r$ and $|b-x|\le r$ (so $a,b\in D_{x,r}$) and that $y = a\x b$.
\end{proof}

\begin{corollary}\label{cor:disk-enlarge}
Let $D_{x,r}$ be defined as in Lemma~\ref{lem:disk-enlarge}.  If $r>0$ and $\lambda\notin\clcl{0,1}$, then $Q_\lambda(D_{x,r}) = \complexes$.
\end{corollary}

\begin{proof}
We have $\delta = |1-\lambda| + |\lambda| > 1$ and $Q_\lambda^{(n)} = D_{x,\delta^nr}$ for all $n$.
\end{proof}

Thus we have the following:


\begin{proposition}\label{prop:open-set}
Fix any $\lambda\in\complexes \cmpl \clcl{0,1}$ and $B\subseteq\complexes$.
\begin{enumerate}
\item
If $B$ includes a nonempty open subset of $\complexes$, then $Q_\lambda(B) = R_\lambda(B) = \complexes$.
\item
If $B$ includes a nonempty open subset of $\reals$, then $F_\lambda \subseteq Q_\lambda(B)$.
\end{enumerate}
\end{proposition}

\begin{proof}
Part (1.)\ follows from Corollary~\ref{cor:disk-enlarge}.  For Part (2.), we have two cases: (i) $\lambda\in\reals$ and $F_\lambda = \reals$; and (ii) $\lambda \notin\reals$ and $F_\lambda = \complexes$.  In case~(i), we apply a one-dimensional version of the disk expansion construction above to expand any interval $\clcl{a,b}\subseteq Q_\lambda(B)$ to a larger interval $\clcl{c,d}\subseteq Q_\lambda(B)$, where (assuming $\lambda >1$ without loss of generality) $c = b\x a$ and $d = a\x b$.  Note that for every point $z\in \clcl{c,d}$ there exist $x,y\in\clcl{a,b}$ such that $z=x\x y$.  The expansion is by a factor of $2\lambda - 1 > 1$.  Applying the expansion repeatedly puts all of $\reals$ into $Q_\lambda(B)$.  For case~(ii), if we start with some interval $\clcl{a,b} \subseteq B$, then the entire triangle formed by $a$, $b$, and $a\x b$ and its interior lies in $Q_\lambda(B)$.  Indeed, for any point $z$ inside this triangle, there exist $x,y\in\clcl{a,b}$ such that $z = x\x y$, as shown in Fig.~\ref{fig:realpath}.  Then applying Part~(1.)\ to $Q_\lambda(B)$ gives $Q_\lambda(Q_\lambda(B)) = Q_\lambda(B) = \complexes$.
\begin{figure}[htbp]
\begin{center}
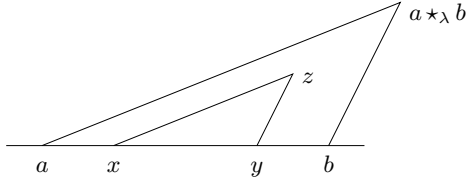
\caption{From the interval $\clcl{a,b}\subseteq B$ we get a triangle in $Q_\lambda(B)$ with a nonempty interior.}\label{fig:realpath}
\end{center}
\end{figure}
\end{proof}

The proof above can easily be generalized to show that if $B$ includes a differentiable path in $\complexes$, then $Q_\lambda(B) = \complexes$ for all $\lambda \in\complexes \cmpl \reals$.  In fact, we can prove something much stronger:

\begin{theorem}\label{thm:bent-path-first}
If $d$ is any path in $\complexes$ that is not contained in a straight line, then $Q_\lambda(d) = R_\lambda(d) = \complexes$ for all $\lambda \in \complexes \cmpl \clcl{0,1}$.
\end{theorem}

In Theorem~\ref{thm:bent-path-first}, we do not require $d$ to be differentiable, or even simple; we only require that $d$ be continuous.  We defer the proof of this theorem until Section~\ref{sec:bent-path}.

\section{Finding $\lambda$ such that $R_\lambda$ is convex}

As in previous sections, we continue to use $\x$ without subscript to denote $\xl$.

Corollary~\ref{cor:general-equivalences} itself has two useful corollaries:

\begin{corollary}\label{cor:unit-disk}
$R_\lambda$ is convex for any $\lambda\in\complexes$ such that either $0 < |\lambda| < 1$ or $0 < |1-\lambda| < 1$.
\end{corollary}

\begin{corollary}\label{cor:basic-convex}
For any $\lambda\in\complexes$, \ $R_\lambda$ is convex if and only if there exists $\mu \in R_\lambda$ such that either $0 < |\mu| < 1$ or $0 < |1-\mu| < 1$.
\end{corollary}

\begin{proof}
The forward implication is obvious, since $\clcl{0,1}\subseteq R_\lambda$ if $R_\lambda$ is convex.  For the converse, we have $R_\mu$ is convex by Corollary~\ref{cor:unit-disk}, whence $R_\lambda$ is convex by Corollary~\ref{cor:convex-extend}.
\end{proof}

The proof of Corollary~\ref{cor:general-equivalences} is the last place where we explicitly use the fact that $R_\lambda$ is closed.  In fact, all convexity arguments for $R_\lambda$ from now on can follow directly or indirectly from Corollaries~\ref{cor:convex-extend} and \ref{cor:basic-convex}, or alternatively from Corollary~\ref{cor:general-equivalences}.  They now allow us to expand our set of $\lambda$ such that $R_\lambda$ is known to be convex.

\begin{proposition}\label{prop:unit-circle}
If $|\lambda| = 1$ and $\lambda$ is neither a fourth nor a sixth root of unity, then $R_\lambda$ is convex.
\end{proposition}

\begin{proof}
Let $\lambda\in\complexes$ be any point on the unit circle.  Write $\lambda = x+iy$ for real $x,y$ such that $x^2+y^2=1$.  The following point is evidently in $R_\lambda$:
\[ \mu := 1\x \lambda = \lambda^2 - \lambda + 1 = (x^2 - y^2 -x+1) + (2xy-y)i = (2x^2-x) + y(2x-1)i = (2x-1)\lambda\;. \]
Thus $|\mu| = |2x-1|$.  If $0<x<1$, then $|\mu| < 1$, and moreover, $0<|\mu|$ if $x\ne 1/2$.  Corollary~\ref{cor:unit-disk} then implies that $R_\mu$ is convex (and thus $R_\lambda$ is convex) for all $x \in \opop{0,1} \cmpl \{1/2\}$, which proves the current proposition when $x\ge 0$.  (The cases where $x\in\{0,1/2,1\}$ correspond to $\lambda$ being a fourth or a sixth root of unity.)

Now assume $x<0$.  It is easy to check (on geometric grounds alone) that either $\lambda^2$ or $\lambda^3$ has positive real part, and so, provided $\lambda^2$ (respectively $\lambda^3$) is not a sixth root of unity, we have $R_{\lambda^2}$ (respectively $R_{\lambda^3}$) is convex by the argument in the previous paragraph, and hence $R_\lambda$ is convex.  Thus the only cases left to show are where: (1) $\lambda$ is neither a fourth nor a sixth root of unity; but (2) $\lambda^2$ has nonpositive real part or is a sixth root of unity, and (3) similarly for $\lambda^3$.  There are only four such cases: $\lambda = e^{\pm i\tau(5/12)}$ and $\lambda = e^{\pm i\tau(5/18)}$.  If $\lambda = e^{\pm i\tau(5/12)}$, then $\lambda^5 = e^{\pm i\tau/12}$, which has positive real part and is not a sixth root of unity.  If $\lambda = e^{\pm i\tau(5/18)}$, then the same can be said for $\lambda^4 = e^{\pm i\tau/9}$.  So we can apply the first paragraph argument to $R_{\lambda^5}$ and $R_{\lambda^4}$, respectively.
\end{proof}

The converse of Proposition~\ref{prop:unit-circle} for $\lambda$ on the unit circle follows from the following fact:

\begin{fact}\label{fact:discrete-ring}
If $D$ is any subring of $\complexes$ that is (topologically) closed, and $\lambda\in D$, then $R_\lambda \subseteq D$.
\end{fact}

So in particular, if $\lambda\in\ints$, then $R_\lambda\subseteq\ints$; if $\lambda$ is a Gaussian integer, then $R_\lambda$ consists only of Gaussian integers; if $\lambda$ is an Eisenstein integer,\footnote{i.e., a number of the form $a+be^{i\tau/3}$ for some $a,b\in\ints$} then $R_\lambda$ consists only of Eisenstein integers.  The fourth roots of unity are all Gaussian integers, and the sixth roots of unity are all Eisenstein integers.  None of these choices of $\lambda$ makes $R_\lambda$ convex.

$R_\lambda$ is usually a proper subset of $D$ for the choices of ring $D$ mentioned above.  More on this in Section~\ref{sec:discrete-ring}.

\begin{corollary}\label{cor:unbounded}
If $\lambda \in \complexes \cmpl \clcl{0,1}$, then $R_\lambda$ is unbounded (and thus $Q_\lambda$ is unbounded).
\end{corollary}

\begin{proof}
Suppose $\lambda \notin \clcl{0,1}$.  If $|\lambda|>1$, then $R_\lambda$ is unbounded, since $\lambda^n\in R_\lambda$ for all integers $n>0$.  Similarly, if $|1-\lambda|>1$, then $R_\lambda$ is unbounded, since $(1-\lambda)^n\in R_{1-\lambda} = R_\lambda$ for all integers $n>0$.  If either $|\lambda| < 1$ or $|1-\lambda| < 1$, then $R_\lambda$ is convex by Corollary~\ref{cor:unit-disk}, and thus $\reals \subseteq R_\lambda$ by Theorem~\ref{thm:convex-is-easy}.  The only case left is when $|\lambda| = |1-\lambda| = 1$.  In this case, $\lambda = (1\pm i\sqrt 3)/2$.  Letting
\[ \mu := \lambda\x 1 = 2\lambda - \lambda^2 = \frac{3 \pm i\sqrt 3}{2}\;, \]
we have $|\mu|>1$, and thus $R_\mu$ is unbounded.  But since $\mu\in R_\lambda$, we have $R_\mu \subseteq R_\lambda$, which makes $R_\lambda$ unbounded.

If $R_\lambda$ is unbounded, then so is $Q_\lambda$, because $R_\lambda = \overline{Q_\lambda}$ by Lemma~\ref{lem:topo-closure}.
\end{proof}

\begin{lemma}\label{lem:wedge}
$R_\lambda$ is convex for all $\lambda = x + iy$ where $0 < x \le 1/2$ and $\sqrt{1-x^2} < y \le 1$.
\end{lemma}

\begin{proof}
We know that $\mu \in R_\lambda$, where $\mu := 1\x\lambda = 1-\lambda + \lambda^2 = 1-x+x^2-y^2 + y(2x-1)i$.  Letting $\alpha := \Re(\mu) = 1-x+x^2-y^2$ and $\beta := \Im(\mu) = y(2x-1)$, we have, for the values of $x$ and $y$ in question,
\begin{align*}
-x < x(x-1) \le \alpha &< 2x^2-x = x(2x-1) \le 0\;, \\
2x-1 \le \beta &\le 0\;.
\end{align*}
Then $0 < \alpha^2 + \beta^2 < (-x)^2 + (2x-1)^2 = 5x^2-4x+1 < 1$, giving $0<|\mu|<1$.  It follows from Corollary~\ref{cor:basic-convex} that $R_\lambda$ is convex.
\end{proof}

\begin{proposition}\label{prop:wedges}
$R_\lambda$ is convex for all $\lambda = x+iy$ where $0<x<1$ and $-1 \le y \le 1$, except for the two points $e^{i\tau/6}$ and $e^{-i\tau/6}$.
\end{proposition}

\begin{proof}
We just need to notice that the rectangular region given in the proposition is included in the union of a handful of other known subregions of $\convex$ (see Definition~\ref{def:convex}).  Let $\lambda$ be as in the proposition.  If $|\lambda|<1$ or $|1-\lambda|<1$, then $R_\lambda$ is convex by Corollary~\ref{cor:unit-disk}.  If $|\lambda|=1$ or $|1-\lambda| = 1$, then $R_\lambda$ is convex by Proposition~\ref{prop:unit-circle} and the fact that $R_\lambda = R_{1-\lambda}$.  Let $W$ be the wedge-shaped region of Lemma~\ref{lem:wedge}.  Then the rest of the possible values of $\lambda$ are covered by either $W$, $1-W$, $W^*$, or $1-W^*$, which all yield convex $R_\lambda$ by Lemma~\ref{lem:wedge} and Facts~\ref{fact:inner-dual} and \ref{fact:conjugate}.
\end{proof}

Figure~\ref{fig:sausage} shows in part what points are in $\convex$ and in $\discrete$, based on the results of this and the next section.
\begin{figure}
\begin{center}
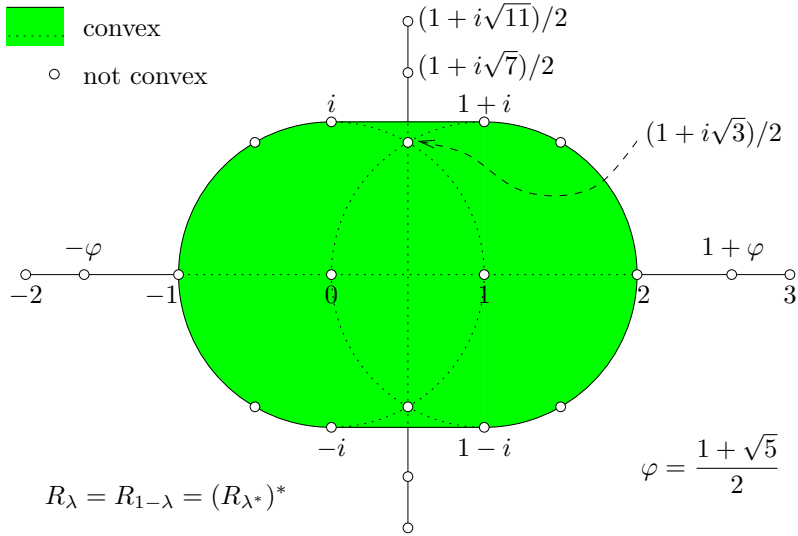
\caption{A portion of the complex plane showing points $\lambda$ such that $R_\lambda$ is convex by our results (green region and lines) and points $\lambda$ where we know that $R_\lambda$ is discrete (white points).}\label{fig:sausage}
\end{center}
\end{figure}


\section{Some $\lambda$ such that $R_\lambda = \reals$}

In this section we establish that $R_\lambda$ is convex (and thus $R_\lambda = \reals$) for various $\lambda \in \reals \cmpl \clcl{0,1}$.  We can assume without loss of generality that $\lambda > 1$, since $R_\lambda = R_{1-\lambda}$.  If $1 < \lambda < 2$, then we already know that $R_\lambda$ is convex by Corollary~\ref{cor:unit-disk}, and if $\lambda\in\ints$, then $R_\lambda \subseteq\ints$ and thus is not convex.  So we investigate the case where $\lambda > 2$ and $\lambda\notin\ints$.  At one point in time, we conjectured that $R_\lambda$ is convex for all $\lambda$ strictly between $2$ and $3$, but this turns out not to be the case, and the \emph{unique} counterexample---where $\lambda = 1+\p \approx 2.618\ldots$ where $\p := (1+\sqrt 5)/2$ is the Golden Ratio---gives a discrete set $R_\lambda$ that is aperiodic.  In fact, $R_{1+\p}$ is an example of an aperiodic Meyer set (see Part~II); although unbounded, it has no infinite arithmetic progressions.

\begin{proposition}\label{prop:between-2-and-3}
If $2 < \lambda < 3$ and $\lambda \ne 1+\p$, then $R_\lambda$ is convex.
\end{proposition}

\begin{proof}
Set $\beta := (\lambda - 1)^2$.  Then one checks that $\beta = 1 - \lambda\x 1 \in 1 - R_\lambda = R_\lambda$.  Further, if $2 < \lambda < 3$, then $-1 < \beta - \lambda < 1$.  One easy way to see this is to note that the function $f(x) := (x-1)^2 - x$ satisfies $f(2) = -1$ and $f(3) = 1$, and $f'(x) = 2x - 3 > 0$ for all $x\in\clcl{2,3}$.  Thus $f$ is strictly monotone increasing on $\clcl{2,3}$, and $f(x) = 0$ only when $x = 1+\p$.  Thus for all the $\lambda$ in question, we have $0 < |\lambda - \beta| < 1$, and so $R_\lambda$ is convex by Corollary~\ref{cor:general-equivalences}, since $\lambda$ and $\beta$ are both in $R_\lambda$.
\end{proof}

\section{$R_{1+\p}$ is not convex}
\label{sec:1-plus-phi}

The next proposition shows that $R_{1+\p}$ is not convex.  It was originally shown by Berman \& Moody~\cite{BeMo}.  This is a special case of a more general theorem (Theorem~\ref{thm:main}) in Part~II.

\begin{proposition}[Berman \& Moody~\cite{BeMo}]\label{prop:1-plus-phi}
\[ R_{1+\p} = \left\{ a + b\p \mathrel{:} a,b\in\ints\myand\frac{b}{\p} \le a \le \frac{b}{\p}+1\right\} = \{1\}\union \left\{ \bigceiling{\frac{b}{\p}} + b\p \mathrel{:} b\in\ints\right\}\;. \]
In particular, $R_{1+\p}$ is discrete, and except for $0$ and $1$, any two adjacent points of $R_{1+\p}$ differ either by $\p$ or by $1+\p$.
\end{proposition}

The set of pairs $(b,a)$ such that $a+b\p \in R_{1+\p}$ is illustrated in Figure~\ref{fig:1-plus-phi}.  Although we give a complete, self-contained proof here, the first inclusion we show below---that $R_{1+\p}$ is a subset of the right-hand side---is actually a special case of a more general result (Theorem~\ref{thm:main}) we prove in Part~II, Section~\ref{sec:alg-int}.  We prove the inclusion here both to make Part~I self-contained and to give a foreshadowing of the more general proof in Part~II.
\begin{figure}[htbp]
\begin{center}
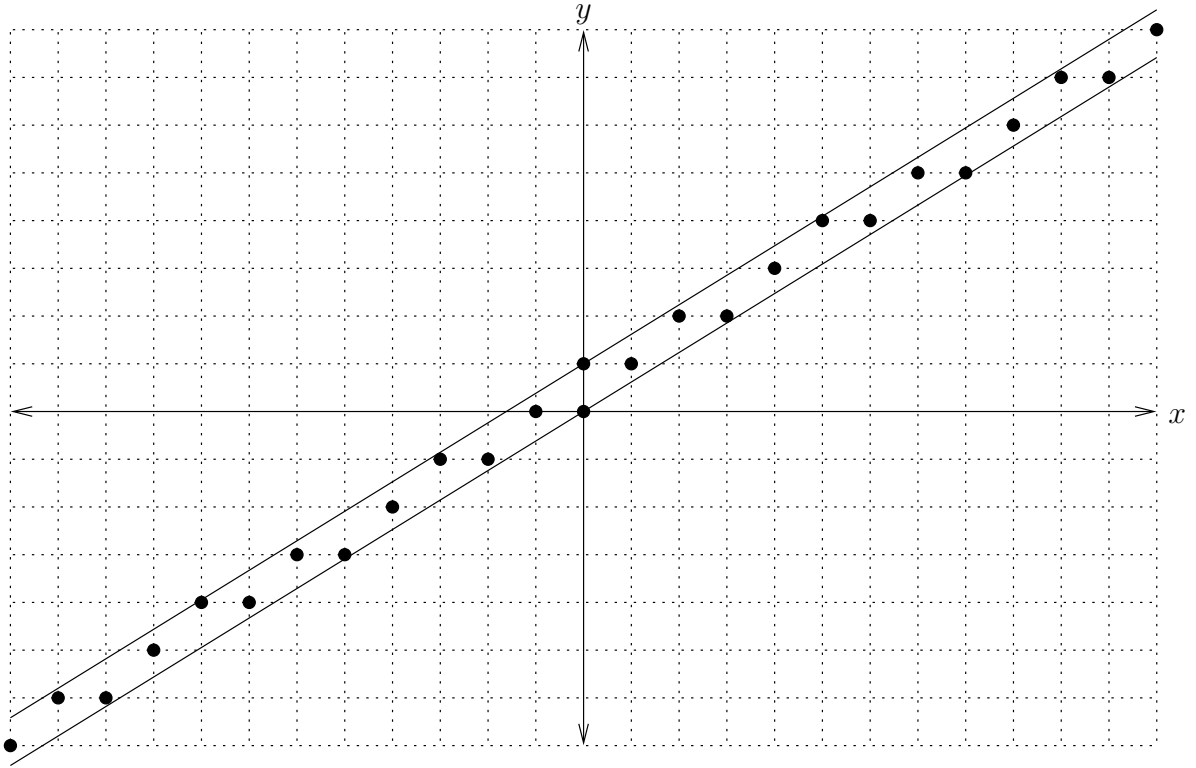
\caption{The points $(b,a)\in\ints\times\ints$ such that $a+b\p\in R_{1+\p}$ are shown.  They are all the lattice points lying in the closed strip bounded by the lines $y = x/\p$ and $y = x/\p + 1$ (also shown).  The figure illustrates the fact that $R_{1+\p}$ contains no infinite arithmetic progressions, because any two points are connected either by the $y$-axis or by a line with rational slope, and this line eventually leaves the strip.  $R_{1+\p}$ is a typical example of an aperiodic model set obtained by a cut-and-project scheme (see Part~II).}\label{fig:1-plus-phi}
\end{center}
\end{figure}

\begin{proof}[Proof of Proposition~\ref{prop:1-plus-phi}]
For this proof, set $\lambda \eqdf 1+\p$.  The second equality is obvious, because $\p$ is irrational.  For the first equality, let
\[ S := \left\{ a + b\p \mathrel{:} a,b\in\ints\myand b/\p \le a \le b/\p+1\right\}\;. \]
We show that $R_\lambda = S$ via two containments.

$R_\lambda \subseteq S$:  It suffices to show that $S$ is $\lambda$-convex, since $\{0,1\}\subseteq S$ and $S$ is closed.  For any $x = a+b\p \in \ints[\p] = \ints + \p\ints$, define $\delta(x) := a - b/\p$.  Then $S = \left\{ x\in\ints[\p] \mathrel{:} \delta(x) \in\clcl{0,1}\right\}$.  For all $x,y\in\ints[\p]$, \ $x\xl y$ is also in $\ints[\p]$, and using the fact that $1/\p = \p-1$, a routine calculation shows that
\[ \delta(x\xl y) = \delta(y)\xm\delta(x)\;, \]
where $\mu \eqdf 1/\p$.  Since $0<\mu<1$, we have $\delta(y)\xm\delta(x) \in \clcl{0,1}$ provided $\delta(x),\delta(y)\in\clcl{0,1}$.  This just means that $x\xl y\in S$ provided $x,y\in S$.  Thus $S$ is $\lambda$-convex, and so $R_\lambda\subseteq S$.

\medskip

$S\subseteq R_\lambda$: It is enough to show that $\ceiling{b/\p} + b\p \in R_\lambda$ for all $b\in\ints$.  We show this by induction on $|b|$.  For $b \in \{-1,0,1\}$ this is easily checked; in particular, $\lambda = 0\xl 1$ and $-\p = 1\xl 0$.  Thus we can start the induction with $|b|\ge 2$.

Notice that for all $x\in\ints\cmpl\{0\}$,
\[ \bigceiling{\frac{-x}{\p}} + (-x)\p = 1 - \left(\bigceiling{\frac{x}{\p}} + x\p\right)\;, \]
which implies that the left-hand side is in $R_\lambda$ if and only if the right-hand side is in $R_\lambda$, which in turn is true if and only if $\ceiling{x/\p} + x\p \in R_\lambda$.  From this fact, we can assume WLOG that $b\ge 2$, the result for $-b$ following immediately.

Assume $b\in\ints$ and $b\ge 2$.  Set $a := \floor{(b+1)/\p}$.  We have $1 \le a < b$, and so by the inductive hypothesis, both $\ceiling{a/\p}+a\p$ and $\ceiling{-a/\p}-a\p$ are in $R_\lambda$.  Then letting $y := \ceiling{-a/\p}-a\p$, the following two values are both elements of $R_\lambda$:
\begin{align*}
y\xl 0 &= -\p y = -\p\left(\bigceiling{\frac{-a}{\p}}-a\p\right) = -\p(\ceiling{-a(\p-1)}-a\p) = -\p(\ceiling{-a\p} + a - a\p) \\
&= -\p(-\floor{a\p}+a-a\p) = \p\floor{a\p}-a\p +a\p^2 = \p\floor{a\p} + a = a + (\ceiling{a\p} - 1)\p\;, \\
y\xl 1 &= -\p y + 1+\p =a + (\ceiling{a\p}-1)\p + 1 + \p = a+1 + \ceiling{a\p}\p\;.
\end{align*}
By the definition of $a$,  we have $b-1 < b+1-\p < a\p < b+1$, and so the following two cases are exhaustive:
\begin{description}
\item[Case 1:] $\ceiling{a\p} = b+1$.  Then $y\xl 0 = a + b\p \in R_\lambda$.  Furthermore, in this case, we have $b < a\p < b+1$, and thus
\[ \frac{b}{\p} < a < \frac{b+1}{\p} < \frac{b}{\p} + 1\;, \]
and so $a = \ceiling{b/\p}$ as desired.
\item[Case 2:] $\ceiling{a\p} = b$.  Then $y\xl 1 = a + 1 + b\p \in R_\lambda$.  Furthermore, in this case, we have $b-1 < a\p < b$, and thus
\[ \frac{b}{\p} - 1 < \frac{b-1}{\p} < a < \frac{b}{\p} \;. \]
Adding $1$ to both sides gives
\[ \frac{b}{\p} < a+1 < \frac{b}{\p} + 1\;, \]
and so $a+1 = \ceiling{b/\p}$ as desired.
\end{description}
The case for $-b$ follows immediately as described above.  This finishes the induction.
\end{proof}

The following corollary implies that $R_{1+\p}$ is aperiodic, that is, it possesses no translational symmetry, and neither does any nonempty subset of $R_{1+\p}$.

\begin{corollary}
$R_{1+\p}$ contains no infinite arithmetic progressions.
\end{corollary}

\begin{proof}
Suppose $x,x+d,x+2d,x+3d,\ldots \in R_{1+\p}$ for some $x\in\reals$ and $d\in\reals\cmpl\{0\}$.  Then since $R_{1+\p}\subseteq\ints[\p]$, we must have $x,d\in\ints[\p]$ as well.  Defining the function $\delta$ as in the proof of Proposition~\ref{prop:1-plus-phi}, one can easily check that $\delta(x+jd) = \delta(x) + j\delta(d)$ for all $j\in\ints$.  Since $\delta(d) \ne 0$, we have $\delta(x+jd) \notin \clcl{0,1}$---and hence $x+jd\notin R_{1+\p}$---for all sufficiently large $j$, contradicting our assumption.
\end{proof}

\section{$\convex$ is a big set}
\label{sec:c-is-open}

In this section, we show that $\convex$ is open and contains all transcendental numbers, which implies that its complement is countable.  We also show that every element of $\discrete = \complexes \cmpl \convex$ has a deleted neighborhood contained in $\convex$.  From these two facts it follows immediately that $\discrete$ is discrete, with no accumulation points in $\complexes$, and contains only algebraic numbers.  In Section~\ref{sec:algebraic-integer}, below, we prove the stronger result that $\discrete$ contains only algebraic \emph{integers} (Theorem~\ref{thm:general-case}), via a much more difficult proof.  Beforehand, we introduce some new facts and concepts that will also be useful elsewhere, including the set $Q_{[x]}\subseteq \ints[x]$ (Definition~\ref{def:Q-x}) and a characterization of it due to Pinch~\cite{Pinch} (Lemma~\ref{lem:level-characterization}).  (We give another useful characterization of $Q_{[x]}$ in Section~\ref{sec:Q-x}.)

Recall the definition of $Q_\lambda$ in Definition~\ref{def:lambda-convex-closure}.


\begin{definition}\label{def:Q-x}
For any polynomials $S,T\in\ints[x]$, define $S\xbx T := (1-x)S + xT$ (which is clearly also in $\ints[x]$).  Let $Q_{[x]}$ denote the least set of polynomials such that
\begin{enumerate}
\item
The constant polynomials $0$ and $1$ are both in $Q_{[x]}$, and
\item
For every $S,T\in Q_{[x]}$, \ $S\xbx T \in Q_{[x]}$.
\end{enumerate}
\end{definition}

Note that $Q_{[x]} \subseteq \ints[x]$.  The operation $\xbx$ and set $Q_{[x]}$ are completely analogous to the various $\xl$ and $Q_\lambda$, respectively, for $\lambda\in\complexes$.  For example, the analogue of Fact~\ref{fact:E-lambda} holds for $\xbx$, and, similarly to Definitions~\ref{def:stratify-Q-lambda} and \ref{def:lambda-S-rank}, we can define $Q_{[x]}^{(0)} := \{0,1\} \subseteq\ints[x]$ and $Q_{[x]}^{(n+1)} := \{S\xbx T \mid S,T\in Q_{[x]}^{(n)}\}$ for all integers $n\ge 0$.  Then the analogue of Fact~\ref{fact:stratify-Q-lambda} holds for $Q_{[x]}$, which allows us to define the \emph{rank} of a polynomial $P\in Q_{[x]}$ as the least $n$ such that $P\in Q_{[x]}^{(n)}$.

In Section~\ref{sec:Q-x}, we will obtain some further constraints on the elements of $Q_{[x]}$, including upper bounds on the number of elements of $Q_{[x]}$ of degree $\le n$, for $n=0,1,2,\ldots\,$.

\begin{fact}\label{fact:evaluation-map}
For any $\lambda\in\complexes$, the evaluation map $P\mapsto P(\lambda)$ is a ring homomorphism from $\ints[x]$ into $\complexes$, and $(S\xbx T)(\lambda) = S(\lambda)\xl T(\lambda)$ for all $S,T\in\ints[x]$).
\end{fact}

The next lemma says that this map maps $Q_{[x]}$ onto $Q_\lambda$.

\begin{lemma}\label{lem:poly-eval}
For any $\lambda\in\complexes$, \ $Q_\lambda = \{ P(\lambda) \mid P\in Q_{[x]} \}$.  In fact,  $Q_\lambda^{(n)} = \{ P(\lambda) \mid P\in Q_{[x]}^{(n)} \}$ for any integer $n\ge 0$.
\end{lemma}

\begin{proof}
The first statement follows immediately from the second, which is proved by a routine induction on $n$: We clearly have $Q_\lambda^{(0)} = \{0,1\} = \{ P(\lambda) \mid P\in Q_{[x]}^{(0)} \}$.  For any $n\ge 0$, if $Q_\lambda^{(n)} = \{ P(\lambda) \mid P\in Q_{[x]}^{(n)} \}$, then
\begin{align*}
Q_\lambda^{(n+1)} &= \left\{ a\xl b \mid a,b\in Q_\lambda^{(n)} \right\} = \left\{ S(\lambda)\xl T(\lambda) \mid S,T\in Q_{[x]}^{(n)} \right\} \\
&= \left\{ (S\xbx T)(\lambda) \mid S,T\in Q_{[x]}^{(n)} \right\} = \{ P(\lambda) \mid P\in Q_{[x]}^{(n+1)} \}\;.
\end{align*}
\end{proof}


Now we can prove the first of the two main theorems of this section.  Theorem~\ref{thm:c-is-open} was proved in the real case by Pinch~\cite[Proposition~9]{Pinch}.  The complex case is also straightforward.

\begin{theorem}\label{thm:c-is-open}
$\convex$ is open.
\end{theorem}

\begin{proof}
Let $D := \{x\in\complexes \mathrel{:} 0 < |x| < 1 \myor 0<|1-x|<1\}$.  Note that $D$ is open.  For any $\lambda\in\complexes$, we have
\begin{align*}
\mbox{$R_\lambda$ is convex} &\iff  R_\lambda \intersect D \ne \emptyset & \mbox{(Corollary~\ref{cor:basic-convex})} \\
&\iff Q_\lambda \intersect D \ne \emptyset & \mbox{(Lemma~\ref{lem:topo-closure})} \\
&\iff (\exists P\in Q_{[x]}) [ P(\lambda) \in D ] &\mbox{(Lemma~\ref{lem:poly-eval})} \\
&\iff \lambda \in \bigcup_{P\in Q_{[x]}} P^{-1}(D)\;.
\end{align*}
Thus $\convex = \bigcup_{P\in Q_{[x]}} P^{-1}(D)$, which is the union of open sets, because each $P\in Q_{[x]}$ corresponds to a continuous map $\complexes \rightarrow \complexes$.  Thus $\convex$ is open.
\end{proof}

\begin{corollary}\label{cor:D-is-closed}
$\discrete$ is closed.
\end{corollary}

To prove the second main theorem of this section, that $R_\lambda$ is convex for all transcendental $\lambda$, we first give two lemmas, the first is routine, and the second is a key observation made by Stuart Kurtz.

\begin{lemma}\label{lem:lin-indep}
For any natural number $n$, the set $\{ x^k(1-x)^{n-k} \mid k\in \ints \myand 0 \le k \le n \}$ is linearly independent over $\complexes$, and is thus a basis for the space of all polynomials in $\complexes[x]$ of degree $\le n$.
\end{lemma}

\begin{proof}
Let $r_0,\ldots,r_n$ be any complex numbers, not all zero.  Let $j$ be least such that $r_j \ne 0$.  Then letting $P(x) := \sum_{k=0}^n r_k x^k(1-x)^{n-k}$, we have
\[ P(x) = \sum_{k=j}^n r_k x^k(1-x)^{n-k} = x^j\left(r_j (1-x)^{n-j} + x\sum_{k=j+1}^n r_k x^{k-j-1}(1-x)^{n-k}\right)\;. \]
Evaluating the expression in the big parentheses at $x=0$ shows that it is not the zero polynomial, whence $P$ is not the zero polynomial, either.
\end{proof}

\begin{lemma}[Kurtz]\label{lem:kurtz}
For all $\lambda\in\complexes$ and integers $n\ge 0$, let $c_\lambda(n)$ be the cardinality of $Q_\lambda^{(n)}$.  Let $c(n)$ be the cardinality of $Q_{[x]}^{(n)}$.
\begin{enumerate}
\item
For any $\lambda\in\complexes$, if $c_\lambda(n) \notin e^{\bigoh(n)}$ as $n\rightarrow\infty$, then $R_\lambda$ is convex.
\item
If $c(n) \notin e^{\bigoh(n)}$ as $n\rightarrow\infty$, then $R_\lambda$ is convex for all transcendental $\lambda\in\complexes$.
\end{enumerate}
\end{lemma}

\begin{proof}
We first prove Part~(1.).  For any $r\ge 0$, define $D_r\subseteq\complexes$ to be the closed disk of radius $r$ centered at the origin.  Set $\delta := |1-\lambda| + |\lambda|$.  Notice that $Q_\lambda^{(0)} = \{0,1\} \subseteq D_1$.  By Lemma~\ref{lem:disk-enlarge} (and induction on $n$), we have $Q_\lambda^{(n)} \subseteq D_{\delta^n}$ for all $n\ge 0$.  If $R_\lambda$ is not convex, then by Corollary~\ref{cor:general-equivalences}, any two distinct elements of $R_\lambda$ are at least unit distance apart, and so we can draw an open disk around each element of $Q_\lambda^{(n)}$ of radius $1/2$, and these disks are pairwise disjoint, for a total area of $c_\lambda(n)\pi/4$.  These disks must in turn all be included in $D_{\delta^n + 1/2}$, which has area $\pi(\delta^n+1/2)^2$.  Thus we get $c_\lambda(n) \le 4(\delta^n+1/2)^2 \in e^{\bigoh(n)}$ if $R_\lambda$ is not convex.

For Part~(2.), notice that if $\lambda$ is transcendental, then the evaluation map $\ints[x]\rightarrow\complexes$ sending $P$ to $P(\lambda)$ is one-to-one.  By Lemma~\ref{lem:poly-eval}, this means that $c(n) = c_\lambda(n)$ for all $n\ge 0$.  So we get that $R_\lambda$ is convex by Part~(1.) if $\lambda$ is transcendental.
\end{proof}

By the second item of Lemma~\ref{lem:kurtz}, we are done if we can get a good lower bound on $c(n)$.  To this end, we next characterize the level sets $Q_{[x]}^{(n)}$ so as to determine their cardinalities exactly.  The following was proved by Pinch using a straightforward induction~\cite[Proposition~4, Corollary~4.1]{Pinch}.  Here we include an alternate, holistic proof.

\begin{lemma}[Pinch]\label{lem:level-characterization}
Fix any integer $n\ge 0$.  For any polynomial $P\in\ints[x]$, \ $P$ is in $Q_{[x]}^{(n)}$ if and only if there exist integers $b_0,\ldots,b_n$ such that $0 \le b_k \le \binom{n}{k}$ for all $0\le k\le n$ and
\[ P(x) = \sum_{k=0}^n b_k x^k (1-x)^{n-k}\;. \]
\end{lemma}

\begin{proof}
One could prove this formally by induction on $n$, but it is more illustrative to consider the general case all at once.  For convenience, set $y := 1-x$.  Then a typical polynomial $P\in Q_{[x]}^{(n)}$ for $n\ge 2$ is of the form $P_0\xbx P_1 = yP_0 + xP_1$, for some $P_0,P_1\in Q_{[x]}^{(n-1)}$.  Then $P_0$ is of the form $P_{00}\xbx P_{01} = yP_{00} + xP_{01}$ and similarly $P_1$ is of the form $P_{10}\xbx P_{11} = yP_{10} + xP_{11}$ for some $P_{00},P_{01},P_{10},P_{11} \in Q_{[x]}^{(n-2)}$, making $P = y^2P_{00} + yx(P_{01} + P_{10}) + x^2P_{11}$.  Similarly, if $n\ge 3$, then there are eight polynomials $P_{000},\ldots,P_{111}\in Q_{[x]}^{(n-3)}$ such that
\[ P = y^3P_{000} + y^2x(P_{001} + P_{010} + P_{100}) + yx^2(P_{011} + P_{101} + P_{110}) + x^3P_{111}\;. \]
This continues until we get polynomials in $Q_{[x]}^{(0)}$, i.e., $0$ or $1$.  Then the completely expanded expression for $P$ resembles a full binary tree with leaves either $0$ or $1$.  Such a tree is shown below for $n=4$ with leaves chosen arbitrarily:
\begin{center}
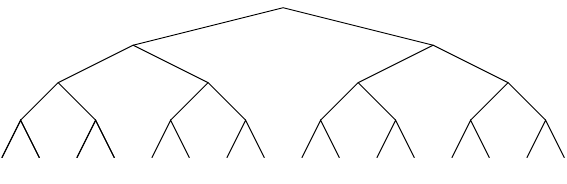
\end{center}
In the expression tree above, each edge represents multiplication by either $y$ or $x$, and each internal node is the sum of its children, weighted by $y$ and $x$, respectively.  The root of the tree yields $P$.  Note that each path in the tree from the root to a leaf contributes one term to $P$ of the form $b x^ky^{n-k}$, where $k$ is the number of right jogs in the path and $b$ is the value at the leaf (either $0$ or $1$).  For each possible $k$, there are exactly $\binom{n}{k}$ many paths with $k$ right jogs, and each of these contributes either $0$ or $x^ky^{n-k}$ to $P$.  The lemma follows.  (The polynomial given by the tree above is $y^4 + y^3x + 3y^2x^2 + 0yx^3 + x^4$.)
\end{proof}

\begin{lemma}[Pinch~{\cite[Corollary~4.2]{Pinch}}]\label{lem:level-cardinality}
For every $n\ge 0$, the cardinality $c(n)$ of $Q_{[x]}^{(n)}$ is exactly $\prod_{k=0}^n\left(1+\binom{n}{k}\right)$.
\end{lemma}

\begin{proof}
Lemma~\ref{lem:level-characterization} immediately gives $\prod_{k=0}^n\left(1+\binom{n}{k}\right)$ as an upper bound on $c(n)$.  For the lower bound, we have that the set of monomials $\{x^k(1-x)^{n-k} \mid 0\le k\le n\}$ is linearly independent by Lemma~\ref{lem:lin-indep}, and thus any two distinct choices of $b_0,\ldots,b_n$ in Lemma~\ref{lem:level-characterization} must give different polynomials.
\end{proof}

\begin{lemma}\label{lem:big-level}
Let $c(n)$ be as in Lemmas~\ref{lem:kurtz} and \ref{lem:level-cardinality}.  Then $c(n) \in e^{\Omega(n\log n)}$ as $n\rightarrow\infty$.
\end{lemma}

\begin{proof}
For $n\ge 2$, we have $c(n) \ge \prod_{k=1}^{n-1} \binom{n}{k} \ge n^{n-1} = e^{(n-1)\log n}$.
\end{proof}

Lemmas~\ref{lem:kurtz} and \ref{lem:big-level} together prove the second main theorem of this section:

\begin{theorem}\label{thm:transcendental-lambda}
$R_\lambda$ is convex for all transcendental $\lambda\in\complexes$.
\end{theorem}

In the next section, we strengthen this result by showing (by a very different proof) that if $Q_\lambda$ is discrete, then $\lambda$ must be an algebraic \emph{integer}.  This fact was proved for real $\lambda$ by Pinch~\cite{Pinch}.  The generalization to all complex $\lambda$ is \emph{not} straightforward.

For now, we prove next that $\convex$ ``surrounds'' all elements of $\discrete$.  The restriction of this theorem to $\reals$ was shown by Pinch~\cite[Proposition~11]{Pinch}.  We give an independent proof for the general case.

\begin{theorem}\label{thm:deleted-neighborhood}
The set $\discrete$ has no accumulation points in $\complexes$.  That is, for any $\lambda\in\complexes$, there exists an open neighborhood $N$ of $\lambda$ such that $N \cmpl \{\lambda\} \subseteq \convex$.
\end{theorem}

\begin{proof}
If $\lambda\in\convex$, then the result is immediate by Theorem~\ref{thm:c-is-open}, so suppose $\lambda\notin\convex$.

We can find distinct polynomials $p,q\in Q_{[x]}$ such that $p(\lambda) = q(\lambda)$.  This can be seen as follows: Let $c(n)$ and $c_\lambda(n)$ be the functions defined in Lemma~\ref{lem:kurtz}.  By part~(1.)\ of that lemma, we have $c_\lambda(n) \in e^{\bigoh(n)}$ as $n\rightarrow\infty$ (because $R_\lambda$ is not convex), but by Lemma~\ref{lem:big-level}, we have $c(n) \in e^{\Omega(n\log n)}$ as $n\rightarrow\infty$.  Therefore, we can choose $n$ such that $c_\lambda(n) < c(n)$, and it follows by the pigeonhole principle that there exist distinct polynomials $p,q\in Q_{[x]}^{(n)}$ such that $p(\lambda) = q(\lambda)$, that is, $\lambda$ is a root of the nonzero polynomial $r \eqdf p-q$.

By the continuity of $r$, there exists a neighborhood $N'$ of $\lambda$ such that $|r(z)| < 1$ for all $z\in N'$.  Since $r$ has only finitely many roots, there exists an $\eps > 0$ such that $0<|r(z)|<1$ for all $z$ such that $0<|z-\lambda|<\eps$.  Letting $N := \{ z\in\complexes : |z-\lambda|<\eps \}$, we have $0 < |r(z)| = |p(z) - q(z)| < 1$ for all $z\in N \cmpl \{\lambda\}$.  For these $z$, since $p(z)$ and $q(z)$ are distinct members of $R_z$ (Lemma~\ref{lem:poly-eval}) that are less than unit distance apart, we know that $R_z$ is convex by Corollary~\ref{cor:general-equivalences}.  Thus $z\in\convex$ for all $z\in N \cmpl \{\lambda\}$.
\end{proof}

%

\subsection{Threshold polynomials}
\label{sec:thresholds}

The next lemma does not apply to $\convex$, but it is an easy consequence of Lemma~\ref{lem:level-characterization} and it will be used in Part~II, so we include it in this section.  First, a definition.

\begin{definition}
For every $n\in\posints$ and $\gamma\in\opop{0,1}$, define the polynomial
\[ t_\gamma^{(n)}(x) := \sum_{i=0}^{\floor{\gamma n}} \binom{n}{i}x^i(1-x)^{n-i}\;. \]
\end{definition}

For large $n$, the polynomial $t_\gamma^{(n)}$ approximates a ``threshold'' function on $\clcl{0,1}$.

\begin{lemma}\label{lem:threshold}
For any $0<\gamma<1$ and $\eps > 0$, there exists a polynomial $T_{\gamma,\eps}\in Q_{[x]}$ such that $1-\eps \le T_{\gamma,\eps}(x) \le 1$ for all $x\in\clcl{0,\gamma-\eps}$ and $0\le T_{\gamma,\eps}(x) \le \eps$ for all $x\in\clcl{\gamma+\eps,1}$.
\end{lemma}

\begin{proof}
All the $t_\gamma^{(n)}$ are in $Q_{[x]}$ by Lemma~\ref{lem:level-characterization}.  Also, for all $x\in\clcl{0,1}$,
\[ 0 \le t_\gamma^{(n)}(x) \le \sum_{i=0}^n \binom{n}{i}x^i(1-x)^{n-i} = 1\;. \]
Taking $T_{\gamma,\eps} := t_\gamma^{(n)}$ for sufficiently large $n$ will satisfy the lemma.  This follows from Hoeffding's inequality \cite{Hoeffding:inequality}, which in the current context states that for all $x$ such that $\gamma + \eps \le x \le 1$,
\[ t_\gamma^{(n)}(x) \le \exp\left(-2n(x-\gamma)^2\right) \le\exp\left(-2n\eps^2\right)\;. \]
The right-hand side is $\le\eps$ provided $n \ge - (\log\eps)/(2\eps^2)$.

By symmetry, we have for all $0\le x \le \gamma - \eps$,
\begin{align*}
1-t_\gamma^{(n)}(x) &= \sum_{i=\floor{\gamma n}+1}^n \binom{n}{i}x^i(1-x)^{n-i} = \sum_{j=0}^{n-\floor{\gamma n}-1}\binom{n}{n-j}x^{n-j}(1-x)^j \\
&= \sum_{j=0}^{n-\floor{\gamma n}-1}\binom{n}{j}x^{n-j}(1-x)^j \le \sum_{j=0}^{\floor{(1-\gamma) n}}\binom{n}{j}x^{n-j}(1-x)^j = t_{1-\gamma}^{(n)}(1-x)\;.
\end{align*}
Since $(1-\gamma)+\eps \le 1-x \le 1$, we apply Hoeffding's inequality again to get
\[ t_\gamma^{(n)}(x) \ge 1 - t_{1-\gamma}^{(n)}(1-x) \ge 1 - \exp(-2n\eps^2) \ge 1-\eps \]
provided $n \ge - (\log\eps)/(2\eps^2)$ as above.

Therefore we can choose $T_{\gamma,\eps} := t_\gamma^{(n)}$, where $n := \ceiling{- (\log\eps)/(2\eps^2)}$.  (We can assume $\eps < 1$ without loss of generality.)
\end{proof}

\section{If $Q_\lambda$ is Discrete, Then $\lambda$ is an Algebraic Integer}
\label{sec:algebraic-integer}

Pinch proved that for $\lambda\in\reals$, if $Q_\lambda$ is discrete, then $\lambda$ is an algebraic integer~\cite{Pinch}.

\begin{theorem}[Pinch~{\cite[Theorem~8]{Pinch}}]\label{thm:real-case}
For any $\lambda\in\reals$, if $Q_\lambda$ is discrete, then $\lambda$ is an algebraic integer.
\end{theorem}

We have the same result for arbitrary complex $\lambda$.

\begin{theorem}\label{thm:general-case}
For any $\lambda\in\complexes$, if $Q_\lambda$ is discrete, then $\lambda$ is an algebraic integer.
\end{theorem}

The rest of this section is devoted to the proof of this theorem.  It adapts Pinch's overall technique to the complex case but is considerably more intricate.  Along the way, we prove a weak relative density result for $Q_\lambda$.  We do this in stages, obtaining stronger and stronger density results for $Q_\lambda$.


\begin{notation}
For any nonzero $z\in\complexes$, we define $\arg z$ to be the unique $\theta$ such that $-\pi \le \theta < \pi$ and $z = |z|e^{i\theta}$.
\end{notation}

The following technical lemma will make our later proofs easier.  We defer the proof until the end of this section.

\begin{lemma}\label{lem:technical}
For all $\lambda\in\complexes\cmpl\reals$, there exists $\nu \in Q_\lambda$ such that $|\nu| > 1$ and $0<\arg\nu<\pi/6$.
\end{lemma}

Now fix $\lambda\in\complexes$ such that $Q_\lambda$ is discrete.  If $\lambda\in\reals$, then $\lambda$ is an algebraic integer by Theorem~\ref{thm:real-case}, so we assume $\lambda\notin\reals$.  We then fix some $\nu\in Q_\lambda$ satisfying Lemma~\ref{lem:technical}, above.


\begin{notation}
We define $\ell$ to be the least positive integer such that, setting $\mu := \nu^\ell$:
\begin{itemize}
\item
$|\mu| > (1+\sqrt 3)/2$,
\item
$0 < \arg(\mu) < \pi/6$, and
\item
for all integers $m$ such that $-6\le m\le 5$, there exists an integer $0 < j <\ell$ such that $\pi m/6 < \arg(\nu^j) < \pi(m+1)/6$.
\end{itemize}
Set $B := |\mu|$ and $q := \mu/(\mu - 1)$.  Define
\[ P := \{ z\in\complexes : 1<|z|<B \myand 0 < \arg z < \pi/6 \}\;. \]
\end{notation}

The first two items guarantee that $\Re(\mu) > 1/2$, and thus $|\mu-1| < |\mu|$, which in turn implies $|q| > 1$.
The idea of the third item is that we have a power of $\nu$ (and thus an element of $Q_\lambda$) in each of the twelve $30^\circ$ ``pie slices'' of $\complexes$ centered at the origin $0$, that is, $Q_\lambda \intersect e^{i\pi m/6} P \ne \emptyset$ for all $m\in\ints$.  Obviously, $\ell > 12$, and $1,\nu,\nu^2,\ldots,\nu^\ell$ are all contained in the closed ball of radius $B$ centered at the origin.

Note that, since $\mu\in Q_\lambda$, \ $Q_\lambda$ is closed under $\xm$.  Also, for all $x,y,z\in\complexes$, we have $x \xm y = z$ if and only if $z \xq y = x$.

\begin{definition}
Define the open region
\[ W := \bigcup_{0\le\theta\le\pi/6} e^{i\theta}P = \{ w\in\complexes : 1<|w|<B \myand 0<\arg w < \pi/3\}\;. \]
For $z\in\complexes$ where $|z|\ge 1$, we will call regions of the form $zW$ \emph{wedges}.
\end{definition}

The four ``corners'' of a wedge $zW$ are $z$, $e^{i\pi/3} z$, $Bz$, and $Be^{i\pi/3}z$.

\begin{lemma}\label{lem:wedges}
Every wedge intersects $Q_\lambda$.
\end{lemma}

\begin{proof}
Given $z$ such that $|z|\ge 1$, let $p\in\ints$ be largest such that $|\nu|^p \le |z|$.  We must have $p\ge 0$ by our condition on $|z|$.  Letting $T :=\{\nu,\nu^2,\ldots,\nu^{\ell-1}\}$, by our choice of $\ell$, we have that $T$ intersects $e^{i\pi m/6}P$ for all $m\in\ints$ (this is established explicitly for $-6 \le m \le 5$ and extends to all $m\in\ints$ by periodicity).  It follows from our choice of $p$ that $|z| < |w| < B|z|$ for all $w\in\nu^p T$, and thus $\nu^p T$ intersects $|z| e^{i(\arg(\nu^p) + \pi m/6)}P$ for all $m\in\ints$.  Choose $m\in\ints$ such that $0 \le \theta_m < \pi/6$, where $\theta_m := \arg(\nu^p) + \pi m/6 - \arg z$.  Then
\[ |z|e^{i(\arg(\nu^p) + \pi m/6)}P = ze^{i\theta_m}P \subseteq zW\;. \]
Thus $zW$ intersects $\nu^p T$, the latter being a subset of $Q_\lambda$, and we are done.
\end{proof}

\begin{notation}
For $z\in\complexes$ and real $k>0$, we define $D(z;k) := \{w\in\complexes : |w-z| < k|z|\}$, that is, the open disk centered at $z$ with radius $k|z|$.
\end{notation}

\begin{definition}
Let $k>0$ be given.  For $x,y\in\complexes$, we say that \emph{$y$ is $k$-close to $x$} iff $y \in D(x;k)$.  If $S\subseteq\complexes$ is some point set and $R\subseteq\complexes$ is some open region, we say that \emph{$S$ is $k$-dense in $R$} if every point in $R$ is $k$-close to a point in $S\intersect R$.
\end{definition}

Notice that $k$-closeness is not a symmetric relation.

By definition, $\{x\}$ is $k$-dense in $D(x;k)$ for all $k>0$ and nonzero $x\in\complexes$.  The next lemma says that we can increase the radius a bit for certain elements of $Q_\lambda$.

\begin{lemma}\label{lem:increase-diameter1}
For every $k$ such that $0<k<1-B^{-1}\sqrt{B^2-\sqrt 3 B + 1}$, there exists $u>k$ such that, for all $x\in Q_\lambda$ with $|x| \ge \frac{B}{k|q|}$, \ $Q_\lambda$ is $k$-dense in $D(\mu x;u)$.
\end{lemma}

\begin{proof}
Given $k$, let
\[ u := \frac{k}{2B}\left(\sqrt 3 + \sqrt{4B^2(1-k)^2-1}\right)\;, \]
that is, the larger of the two solutions to the quadratic equation
\[ u^2 - \frac{k\sqrt 3}{B}\,u + \frac{k^2}{B^2} - k^2(1-k)^2 = 0\;. \]
The upper bound on $k$ guarantees that $u>k$, which can be seen as follows: The inequality $k < u$ is clearly equivalent to
\[ 2B - \sqrt 3 < \sqrt{4B^2(1-k)^2-1}\;. \]
Since $B = |\mu| > (1+\sqrt 3)/2$, both sides are nonnegative, so squaring both sides yields an equivalent inequality:
\[ 4B^2 - 4B\sqrt 3 + 3 < 4B^2(1-k)^2-1\;, \]
or equivalently,
\[ (1-k)^2 > 1 - B^{-1}\sqrt 3 + B^{-2}\;. \]
The lower bound on $B$ above makes both sides nonnegative, so we can take the square root of both sides to get the equivalent statement,
\[ |1-k| > B^{-1}\sqrt{B^2 - \sqrt 3 B + 1}\;, \]
which is implied by our constraint on $k$.

Now let $z$ be any point in $D(\mu x;u)$.  We show that $z$ is $k$-close to an element of $Q_\lambda\intersect D(\mu x;u)$.  In fact, we show that $z$ is $k$-close to an element of $Q_\lambda\intersect D(\mu x;k)$, from which $k$-density follows, because $D(\mu x;k)\subseteq D(\mu x;u)$.  If $z\in D(\mu x;k)$, we are done, so assume otherwise.  Let $z'$ be the point on the line segment connecting $z$ with $\mu x$ that is distance $k|\mu x|/B = k|x|$ away from $\mu x$, as in Figure~\ref{fig:wedges}, and let $y := z' \xq x$.  (Note then that $z' = y \xm x$.)  Using the lower bound on $|x|$ and noting that $|q| = |1-\mu|^{-1}B$, we have
\[ |y| = |z' \xq x| = |(1-q)z'+qx| = |1-\mu|^{-1}|z' - \mu x| = |1-\mu|^{-1}k|x| = |q|k|x|/B \ge 1\;. \]
Set $Y := y\,e^{-i\pi/6}W$.  Note that $Y$ is a wedge, because $|y|\ge 1$, and thus $Y\intersect Q_\lambda \ne \emptyset$ by Lemma~\ref{lem:wedges}.  Letting $Z := Y \xm x = (1-\mu)Y + \mu x$, we have $Z\intersect Q_\lambda \ne \emptyset$ as well.  (Note that $Y$ and $Z$ are similar.)
To finish the proof, we will show that $z$ is $k$-close to every point in $Z$ and that $Z\subseteq D(\mu x;k)$.

Let $w\in Z$ be arbitrary.  Then $Y$ contains the point $w' := w\x_q x = (1-\mu)^{-1}(w-\mu x)$.  We can thus write $w' = r\,e^{i(\theta + \arg y)}$, where $|y| < r < B|y|$ and $-\pi/6 < \theta < \pi/6$.  Translating back, we have
\[ w - \mu x = (1-\mu)\,r\,e^{i(\theta + \arg y)} = (1-\mu)\,\frac{r}{|y|}\,y\,e^{i\theta} = \frac{r}{|y|}\,e^{i\theta}(z'-\mu x) = s\,e^{i\theta}(z'-\mu x)\;, \]
where we set $s := r/|y|$ and thus $1 < s < B$.  It follows that
\[ |w-\mu x| = s|z'-\mu x| = sk|x| < Bk|x| = k|\mu x|\;, \]
which shows that $Z\subseteq D(\mu x;k)$, as $w\in Z$ was chosen arbitrarily.  Let $z''$ be the point on the line connecting $\mu x$ with $z$ such that $z$ is between $\mu x$ and $z''$ and $|z''-\mu x| = u|\mu x| = uB|x|$ ($z''$ is on the boundary of $D(\mu x;u)$).  Figure~\ref{fig:wedges} shows what is going on.
\begin{figure}
\begin{center}
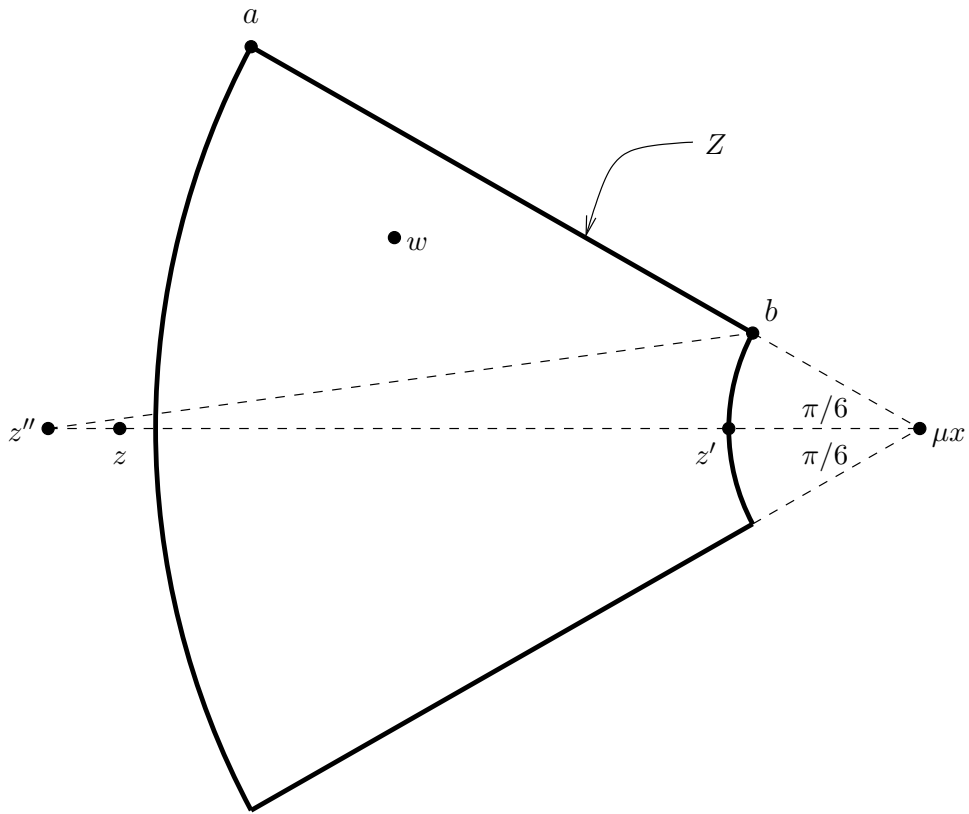
\caption{The $Z$ region is bounded by the thick lines and arcs.  The arcs are concentric about the point $\mu x$ through an angle of $\pi/3$.  The dotted line connecting $\mu x$ with $z$ bisects $Z$ and intersects the inner arc at the point $z'$.  The point $z$ lies outside the outer arc, and $z''$ lies to its left.  Two corners $a$ and $b$ and an arbitrary element $w$ of $Z$ are also labeled.}\label{fig:wedges}
\end{center}
\end{figure}
We also identify two corners of $Z$, namely, $b := \mu x + e^{-i\pi/6}(z'-\mu x)$ and $a := \mu x + B(b-\mu x)$.  By definition, $|b-\mu x| = |z'-\mu x| = k|x|$ and $|a-\mu x| = B|b-\mu x| = Bk|x|$.  Evidently, $|w - \mu x| < |a - \mu x|$, whence by the triangle inequality, $|w| > |\mu x| - Bk|x| = B(1-k)|x|$.  It is also evident from the diagram that the point $b$ is farther away from $z''$ than any point in $Z$ is from $z$, and so $|z-w| < |z''-b|$.  (The point $a$ is closer to $z$ than $b$ is to $z$; this follows from the fact that $B > (1+\sqrt 3)/2$).  Using, say, the Law of Cosines with the triangle $(z'',b,\mu x)$, we can find $|z''-b|$:
\begin{align*}
|z''-b|^2 &= |z''-\mu x|^2 + |b-\mu x|^2 - 2|z''-\mu x|\,|b-\mu x|\cos(\pi/6) \\
&= B^2|x|^2\left(u^2 - \frac{k\sqrt 3}{B}\,u + \frac{k^2}{B^2}\right) = B^2|x|^2k^2(1-k)^2
\end{align*}
by our choice of $u$.  Thus $|z-w| < |z''-b| = kB(1-k)|x| < k|w|$, making $z$ $k$-close to $w$.
\end{proof}


Suppose $x\in\complexes\cmpl\{0\}$ and $r>0$ are such that $Q_\lambda$ is $k$-dense in $D(x;r)$ (for some $k>0$).  It immediately follows, just by multiplying everything by $\mu$, that $\mu Q_\lambda$ is $k$-dense in $D(\mu x;r)$.  As with Lemma~\ref{lem:increase-diameter1}, the next lemma increases the diameter a little bit going from $x$ to $\mu x$.  It is actually a generalization of Lemma~\ref{lem:increase-diameter1}

\begin{lemma}\label{lem:increase-diameter2}
Let $k$ and $u$ be as in Lemma~\ref{lem:increase-diameter1}.  Suppose $0 < r < 1$, and let $v := r + (1-r)(u-k)$.  Then $Q_\lambda$ is $k$-dense in $D(\mu x;v)$ for all $x\in\complexes$ such that $Q_\lambda$ is $k$-dense in $D(x;r)$ and $|x| \ge \frac{B}{(1-r)k|q|}$.
\end{lemma}

\begin{proof}
Let $z\in D(\mu x;v)$ be arbitrary.  Then $z$ is at most $(v-r)|\mu x|$ distance away from some element $y\in D(\mu x;r)$.  Since $Q_\lambda$ is $k$-dense in $D(x;r)$ by assumption, $\mu Q_\lambda$ is $k$-dense in $D(\mu x;r)$, and thus $y$ is $k$-close to $\mu x'$ for some $x' \in Q_\lambda\intersect D(x;r)$.  By the triangle inequality, $|x'| > (1-r)|x| \ge \frac{B}{k|q|}$, and so we can apply Lemma~\ref{lem:increase-diameter1} to $x'$ to get that $Q_\lambda$ is $k$-dense in $D(\mu x';u)$.  It remains to show that $z\in D(\mu x';u)$, thus making $z$ $k$-close to some element of $Q_\lambda$.  By the triangle inequality using the triangle $(\mu x',y,z)$, and noting that $|\mu x'| = B|x'| > B(1-r)|x| = (1-r)|\mu x|$, we have
\begin{align*}
|z-\mu x'| &\le |y-\mu x'| + |z-y| < k|\mu x'| + (v-r)|\mu x| \\
&< (k + (v-r)/(1-r))|\mu x'| = (k + (u-k))|\mu x'| = u|\mu x'|\;,
\end{align*}
and thus $z\in D(\mu x';u)$ as required.
\end{proof}

\begin{remark}
In the last lemma, we needed $|x| \ge \frac{B}{(1-r)k|q|}$ so that every point $x'\in D(x;r)$ satisfies $|x'| \ge \frac{B}{k|q|}$, allowing us to apply Lemma~\ref{lem:increase-diameter1} to $x'$.  The same is true for every point $y \in D(\mu x;v)$, that is, $|y| \ge \frac{B}{k|q|}$.  The latter condition is equivalent to the inequality $1-u+k \ge 1/B$, which can be verified via a rather tedious calculation.\footnote{Using the fact that $B>1$, this inequality can be converted into the equivalent form $p(B) \ge 0$, where $p$ is a real quadratic polynomial with leading term $(1+1/k)^2-(1-k)^2 > 0$ and discriminant $4(1-k)^2-(1+1/k)^2 < 0$.}  We also have $0<v<1$, because $1-v = (1-r)(1-u+k) > 0$.  These facts will be important for the proof of Theorem~\ref{thm:k-dense}, because they allow us to iterate the passage from $D(x;r)$ to $D(\mu x;v)$ while maintaining $k$-density of $Q_\lambda$ throughout.
\end{remark}

Theorem~\ref{thm:k-dense} below, the first main result of this section, asserts that $Q_\lambda$ is close to being relatively dense, at least asympototically.  Before giving it, we present a few technical lemmas.

Recall that $\mu$ was chosen such that $0<\arg\mu<\pi/6$.

\begin{definition}\label{def:n-V-A}
Define $n := \ceiling{2\pi/\arg\mu}$, noting that $n$ is the least positive integer such that $0 \le \arg(\mu^n) < \arg\mu$.  Define the closed region $V := \{ r\,e^{i\p} \mid 1 \le r \le B^n \myand 0 \le \p \le \pi/6 \}$.

Define the closed annulus $A := \{ z\in\complexes : B^{n-1} \le |z| \le B^n \}$.  \end{definition}

Note that $n$ is chosen so that every closed pie slice $S_\theta := \{ r\,e^{i\p} \mid r\ge 0 \myand \theta-\pi/6 \le \p \le \theta\}$ for $\theta\in\reals$ contains $\mu^j$ for some integer $0 \le j < n$.  $V$ resembles the region $P$, but extends out much farther away from the origin and is closed.  The next lemma is routine and stated without proof.

\begin{lemma}\label{lem:V-in-disk}
$V$ is included within the open disk $C := D(B^n;1-B^{-n}/2)$ centered at $B^n$ with radius $B^n-1/2$.
\end{lemma}

\begin{lemma}\label{lem:annulus}
$A \subseteq U$, where $U := \bigcup_{j=0}^{n-1} \mu^j V$.
\end{lemma}

\begin{proof}
Given $z\in A$, let $\theta := \arg z$.  Evidently, $z\in S_\theta$, the pie slice defined above.  Let $0\le j < n$ be such that $S_\theta$ contains $\mu^j$.  Then
\[ \mu^j V = \{ r\,e^{i\p} \mid B^j \le r \le B^{j+n} \myand \arg(\mu^j) \le \p \le \arg(\mu^j)+\pi/6 \}\;, \]
which contains $z$.
\end{proof}

\begin{lemma}\label{lem:V-cover}
$\{z\in\complexes : |z|\ge B^{n-1}\} \subseteq T$, where $T := \bigcup_{j=0}^\infty \mu^j C$ and $C$ is as in Lemma~\ref{lem:V-in-disk}.
\end{lemma}

\begin{proof}
Every $z$ such that $|z|\ge B^{n-1}$ is contained in $\mu^p A$ for some integer $p\ge 0$, and by Lemmas~\ref{lem:annulus} and \ref{lem:V-in-disk} the latter region is included in $\mu^p U = \bigcup_{j=p}^{p+n-1} \mu^j V \subseteq \bigcup_{j=p}^{p+n-1} \mu^j C \subseteq T$.
\end{proof}

\begin{theorem}\label{thm:k-dense}
Given any $k>0$, there exists $R>0$ such that $Q_\lambda$ is $k$-dense in $\{z\in\complexes : |z| > R\}$.
\end{theorem}

\begin{proof}
The idea is that we can increase the sizes of disks in which $Q_\lambda$ is $k$-dense until one of them includes $zC$ for some $z$.  Without loss of generality, we can take $k$ to be as small as we want, so we assume it satisfies the conditions in Lemma~\ref{lem:increase-diameter1}, and we also define $u$ as in that lemma.  Let $n$ be as in Definition~\ref{def:n-V-A} and $C$ be as in Lemma~\ref{lem:V-in-disk}.  Fix some $x\in Q_\lambda$ such that $|x| \ge \frac{B}{(1-k)k|q|}$.  For example, we can take $x$ to be the lowest power of $\mu$ satisfying this norm bound, whence $|x| < \frac{B^2}{(1-k)k|q|}$.  Set $r_0 := k$, and for all integers $j\ge 0$, inductively define $r_{j+1} := r_j + (1-r_j)(u-k)$.  Then by induction, for all $j\ge 0$, we have that $Q_\lambda$ is $k$-dense in $D(\mu^j x;r_j)$.  Also by induction we have $r_j = 1 - (1-u+k)^j(1-k)$ for all $j$.  We know that $0 < 1-u+k < 1$ (see the Remark following Lemma~\ref{lem:increase-diameter2}), so we can choose an $m\ge n$ large enough so that $1 - B^{-n}/2 < r_m < r_{m+1} < \cdots < 1$.  Then for all $p \ge m$, \ $D(\mu^p x;r_p)$ is big enough to include $zC$ for some $z$.  In fact,
\[ \frac{\mu^p x}{B^n}\,C = D(\mu^p x;1-B^{-n}/2) \subseteq D(\mu^p x;r_m) \subseteq D(\mu^p x;r_p) \]
by our choice of $m$.  Finally, letting $E := \{ z\in\complexes : |z| \ge B^{n-1} \}$ and using Lemma~\ref{lem:V-cover},
\[ \{z\in\complexes : |z| \ge B^{m-1}|x| \} = \frac{\mu^m x}{B^n}\, E \subseteq \frac{\mu^m x}{B^n}\, T = \frac{x}{B^n}\bigcup_{p=m}^\infty \mu^p C \subseteq \bigcup_{p\ge m} D(\mu^p x;r_p)\;. \]
$Q_\lambda$ is $k$-dense in the right-hand side, so we can take $R := \frac{B^{m+1}}{(1-k)k|q|} > B^{m-1}|x|$.
\end{proof}

\begin{remark}
Although we were assuming all along that $Q_\lambda$ is discrete, Theorem~\ref{thm:k-dense} actually holds for all $\lambda\in\complexes\cmpl\reals$, for if $Q_\lambda$ is not discrete, then it is \emph{dense} in $\complexes$ and hence trivially $k$-dense in $\complexes$ for all $k>0$.  For real $\lambda$, we have the following situation: if $Q_\lambda$ is not discrete, then we know that $Q_\lambda$ is dense in either $\clcl{0,1}$ or $\reals$ (depending on $\lambda$ being in $\opop{0,1}$ or $\reals\cmpl\clcl{0,1}$, respectively).  In this case, $Q_\lambda$ is again obviously $k$-dense in these respective sets, for all $k>0$.  If $Q_\lambda$ is discrete and $\lambda\notin\{0,1\}$, then given $k>0$, Pinch implicitly proves $k$-density of $Q_\lambda$ in $\reals\cmpl\clcl{-A,A}$ for some $A>0$ (depending on $k$)~\cite{Pinch}.
\end{remark}

%

We now turn to the second main result of this section, showing that $\lambda$ is an algebraic integer.  Pinch's proof for real $\lambda > 1$ works by showing that every sufficiently large $x\in Q_\lambda$ is a $\ints$-linear combination of elements of $Q_\lambda \intersect \opop{0,R}$ for some fixed $R>0$.  In this case, given $x$, he finds elements $u,v\in Q_\lambda$ such that $u < x < v$ and that are ``close enough'' to $x$ so that the three points $r := x\xl u$, $s := v\xl u$, and $t := v\xl x$ are all strictly between $0$ and $x$.  One has $x = r-s+t$, and then he can argue by induction using the discreteness of $Q_\lambda$.


Here, given $x\in Q_\lambda$ (for nonreal $\lambda$) such that $|x|$ is sufficiently large, we follow roughly the same outline as Pinch, using the $k$-density of $Q_\lambda$ to find $u,v\in Q_\lambda$ such that the points $r := x\xl u$, $s := v\xl u$, and $t := v\xl x$ are all smaller than $x$ in norm, allowing a similar inductive argument.  Our situation is complicated by the fact that, not only must $u$ and $v$ be close enough to $x$, they must also be oriented in suitable directions relative to $x$ and to each other.

\begin{lemma}\label{lem:reduce-norm}
For all $x\in Q_\lambda$ with $|x|$ sufficiently large, there exist $u,v\in Q_\lambda$ such that the three points $r := x \xm u$, \ $s := v \xm u$, and $t := v \xm x$ all have norms strictly smaller than $|x|$.
\end{lemma}

\begin{proof}
Recall that $q = \mu/(\mu-1)$, and it follows that $\mu = q/(q-1)$.  Also, $|q|>1$.  Let $a$ be the square root of $q$ with positive real part, i.e, $a^2 = q$ and $\Re(a) > 0$ (we know that $q \not< 0$).  Let $c := |a| = |q|^{1/2}$ and let $d := |a+1|$, noting that $1 < c < d$.  Choose $k$ such that
\begin{equation}\label{eqn:k}
0 < k < \min\left(\frac{d-c}{d-c+Bdc}\,,\;\frac{1}{B(c+c^{-1})+1}\right)\;,
\end{equation}
observing that $k<1$.  From (\ref{eqn:k}) it follows that
\begin{equation}\label{eqn:k-again}
0 < \frac{k}{1-k} < \frac{1}{B}\min\left(c^{-1}-d^{-1}\,,\;(c+c^{-1})^{-1}\right)\;.
\end{equation}

Given $k$ as above, let $R$ be as in Theorem~\ref{thm:k-dense}, and let $x$ be any element of $Q_\lambda$ such that $|x| > cR$.  Let $y := x/a$ and $z := ax$.  We have $R < |y| < |z|$, so by $k$-density we can choose $u,v\in Q_\lambda$ such that $y$ is $k$-close to $u$ and $z$ is $k$-close to $v$.  We have
\begin{align*}
|u| &\le |y| + |u-y| < |x|/c + k|u| &&\implies |u| < \frac{c^{-1}|x|}{1-k}\;, \\
|v| &\le |z| + |v-z| < c|x| + k|v| &&\implies |v| < \frac{c|x|}{1-k}\;.
\end{align*}
Define $r$, $s$, and $t$ as in the lemma.  Observe that $z\xm y = (1-\mu)ax + \mu x/a = (x/a)((1-\mu)q + \mu) = 0$, and thus, using (\ref{eqn:k-again}) for the last step,
\begin{align*}
|s| &= |v\xm u| \le |z\xm u| + |v\xm u - z\xm u| \le |z\xm y| + |z\xm u - z\xm y| + |v\xm u - z\xm u| \\
&= |z\xm u - z\xm y| + |v\xm u - z\xm u| = |\mu(u-y)| + |(1-\mu)(v-z)| < Bk(|u|+|v|) \\
&< B\left(\frac{k}{1-k}\right)(c+c^{-1})|x| < |x|\;.
\end{align*}
Using the fact that $\mu = q/(q-1) = a^2/(a^2-1)$, we get
\[ a\xm 1 = \left(1-\frac{a^2}{a^2-1}\right)a + \frac{a^2}{a^2-1} = \frac{a^2-a}{a^2-1} = \frac{a}{a+1}\;, \]
and we plug this into the following calculation:
\begin{align*}
|t| &= |v\xm x| \le |z\xm x| + |v\xm x - z\xm x| = |ax\xm x| +|(1-\mu)(v-z)| \le |x||a\xm 1| + B|v-z| \\
&< |x||a\xm 1| + Bk|v| < |x|\left(|a\xm 1| + B\frac{k}{1-k}c\right) = |x|\left(\left|\frac{a}{a+1}\right| + B\frac{k}{k-1}c\right) \\
&< |x|\left(\frac{c}{d} + (c^{-1} - d^{-1})c\right) = |x|\;. \\
|r| &= |x\xm u| \le |x\xm y| + |x\xm u - x\xm y| = |x\xm a^{-1}x| + |\mu(u-y)| = c^{-1}|x||a\xm 1| + B|u-y| \\
&< c^{-1}|x||a\xm 1| + Bk|u| < c^{-1}|x|\left(|a\xm 1| + B\frac{k}{1-k}\right) < |x|\left(|a\xm 1| + B\frac{k}{1-k}c\right) < |x|\;.
\end{align*}
(We reused some of the calculation for $|t|$ for the bound on $|r|$.)
\end{proof}


\begin{proof}[Proof of Theorem~\ref{thm:general-case}]
The case where $\lambda\in\reals$ was proved by Pinch~\cite{Pinch}, so we assume (as we have throughout this section) that $\lambda\notin\reals$.  Let $R$ and $c$ be as in the proof of Lemma~\ref{lem:reduce-norm}, and let $D := \{z\in Q_\lambda : |z| \le cR\}$.  We show first that every $x\in Q_\lambda$ is a $\ints$-linear combination of elements of $D$.  This is done by induction on $|x|$, which is possible because $Q_\lambda$ is discrete:  If $|x| \le cR$, then already $x\in D$ and we are done.  Otherwise, by Lemma~\ref{lem:reduce-norm} we have $u,v\in Q_\lambda$ such that $r$, $s$, and $t$ all have norm less than $|x|$, where $r$, $s$, and $t$ are as in Lemma~\ref{lem:reduce-norm}.  Obviously, $r,s,t\in Q_\lambda$, so applying the inductive hypothesis to $r$, $s$, and $t$, each is a $\ints$-linear combination of elements of $D$.  It is straightforward to check that $x = r-s+t$, and thus $x$ is a $\ints$-linear combination of elements of $D$ as well.  This ends the inductive argument.

Note that $D$ is finite, because $Q_\lambda$ is discrete.  Every element of $Q_\lambda$ can be expressed as $p(\lambda)$, where $p\in Q_{[x]}$ is a polynomial with integer coefficients.  Choose some positive integer $N$ large enough so that every $z\in D$ can be written as $p(\lambda)$ where $p\in \ints[x]$ and $\deg(p) < N$.  We have $\lambda^N \in Q_\lambda$.  By our inductive argument, $\lambda^N$ is a $\ints$-linear combination of elements of $D$, each of which is a $\ints$-linear combination of lower powers of $\lambda$.  Thus $\lambda$ is the root of an integer polynomial, and this polynomial is monic, having leading term $\lambda^N$.
\end{proof}

\subsection{Proof of Lemma~\ref{lem:technical}}

\begin{proof}[Proof of Lemma~\ref{lem:technical}]
If $Q_\lambda$ is not discrete, then it is dense in $\complexes$ (by Theorem~\ref{thm:convex-is-easy} and Corollary~\ref{cor:general-equivalences}) and we are done, so we can assume that $Q_\lambda$ is discrete.  We may also assume that $\Im(\lambda) > 0$, for otherwise, we can argue the following with $1-\lambda$ in place of $\lambda$ (recall that $Q_\lambda = Q_{1-\lambda}$).  If $|\lambda| < 1$, then $Q_\lambda$ is not discrete, so we can assume $|\lambda| \ge 1$.  We know from Proposition~\ref{prop:unit-circle} that there are exactly three values of $\lambda$ on the unit circle where $\Im(\lambda) > 0$ and $Q_\lambda$ is discrete:
\begin{align*}
e^{\pi i/3} &= \frac{1+i\sqrt 3}{2}\;, & e^{\pi i/2} &= i\;, & e^{2\pi i/3} &= \frac{-1+i\sqrt 3}{2}\;.
\end{align*}
For the first value, $\lambda = (1 + i\sqrt 3)/2$, one can see that $Q_\lambda = \ints[\lambda]$, the set of Eisenstein integers, and so $\nu$ exists.  (More explicitly, we can set
\[ \nu := (5+i\sqrt 3)/2 = 2+\lambda = (\lambda\xl 1)\xl 2 = (\lambda\xl 1)\xl ((\lambda\xl 1)\xl 1) \in Q_\lambda\;. \]
For the second value, $\lambda = i$, one can see that $Q_\lambda = \ints[i]$, the Gaussian integers, so $\nu$ clearly exists.  Explicitly, we can take
\[ \nu := 2+i = (i\xl 0)\xl 1 \in Q_\lambda\;. \]
For the third value, $\lambda = (-1+i\sqrt 3)/2$, we have $1-\lambda = (3 - i\sqrt 3)/2$, and we can take
\[ \nu := (5+i\sqrt 3)/2 = (1\xl (1-\lambda))\xl 1 \in Q_\lambda\;. \]
Thus from now on, we can assume that $|\lambda| > 1$.

If $\frac{\arg\lambda}{\pi}$ is irrational, then we can take $\nu$ to be some appropriate positive power of $\lambda$, so we can henceforth assume that $\frac{\arg\lambda}{\pi}$ is rational.  Let $n\in\ints$ be least such that $n>1$ and $\lambda^n>0$ (in fact, we must have $n>2$ since $\lambda\notin\reals$).  Then for any nonnegative $k\in\ints$, we have $\arg(\lambda^{kn}) = 0$, and thus $\arg(\lambda^{kn+1}) = \arg\lambda$.  Set $\gamma_k := 1 - \lambda^{kn+1}$.  We have $\gamma_k \in Q_\lambda$ for all integers $k\ge 0$, and $|\gamma_k| \ge |\lambda|^{kn+1} - 1 > 1$ if $k$ is sufficiently large.

Now consider $\arg\gamma_k$ for any $k$ such that $|\gamma_k|>1$.  If $\frac{\arg\gamma_k}{\pi}$ is irrational, then we can let $\nu$ be some appropriate positive power of $\gamma_k$, as we did above with $\lambda$.  Otherwise, let $m\in\ints$ be least such that $m>1$ and $(\gamma_k)^m > 0$.  Thus $\arg\gamma = 2\pi a/m$ for some $a\in\ints$ such that $-m/2\le a < m/2$ and $a$ is coprime with $m$.  Then there exists $b\in\ints$ such that $b>0$ and $m\divides ab-1$ ($b$ is a modular reciprocal of $a$ modulo $m$).  This gives $\arg((\gamma_k)^b) = 2\pi/m$.  If $m>12$, then we can set $\nu := (\gamma_k)^b$, giving $0 < \arg\nu = 2\pi/m < \pi/6$.  Thus the only unresolved case is where $m\le 12$.  Note that there are only finitely many possible values of $\arg\gamma_k$ with $m\le 12$.  Hence we finish by showing that there exists $k$ such that this case does not happen.

It is evident on geometrical grounds that
\[ (\arg\lambda) - \pi = (\arg(\lambda^{kn+1})) - \pi = \arg(-\lambda^{kn+1}) < \arg(1-\lambda^{kn+1}) = \arg\gamma_k < 0\;. \]
Letting $\theta_k := (\arg\gamma_k) - (\arg\lambda) + \pi$, we see that $\theta_k$ is one of the interior angles of the triangle $(0,-\lambda^{kn+1},\gamma_k)$, namely, the angle at the origin.  The interior angle at $-\lambda^{kn+1}$ is $\arg\lambda$.  We have $0<\theta_k<\pi/2$, and by the Law of Sines,
\[ \sin\theta_k = \frac{\sin(\arg\lambda)}{|\gamma_k|}\,. \]
Since $|\gamma_k| \rightarrow \infty$ as $k\rightarrow\infty$, it follows that $\theta_k \rightarrow 0$ as $k\rightarrow\infty$.  From this we see that there are infinitely many values of $\theta_k$, and thus of $\arg(\gamma_k)$, for different $k$, and so for some positive $k\in\ints$ we have $|\gamma_k|>1$ and $\arg((\gamma_k)^m) \not> 0$ for all integers $1\le m\le 12$.
%
\end{proof}

\section{$R_\lambda$ when $\Re(\lambda) = 1/2$}

As in previous sections, we use $\x$ without subscript to mean $\xl$.

Here we look at $R_\lambda$ for some $\lambda$ with real part $1/2$.  For these $\lambda$, we have $1-\lambda = \lambda^*$, and so $R_\lambda$ is closed under complex conjugate.  Furthermore, $R_\lambda$ is convex iff $R_{\lambda^*}$ is convex, and so we can assume throughout this section that $\Im(\lambda) \ge 0$.  We also have in particular, $\lambda\x 0 = (1-\lambda)\lambda = |\lambda|^2$, and so $|\lambda|^2 \in R_\lambda$.

\begin{proposition}\label{prop:weak}
Suppose $\Re(\lambda) = 1/2$ and $|\lambda| \le \sqrt 3$.  Then $R_\lambda$ is convex if and only if $|\lambda| \notin \{1, \sqrt 2,\sqrt 3\}$.
\end{proposition}

\begin{proof}
If $|\lambda| = \sqrt n$ for any $n\in\posints$, then $\lambda = (1 + i\sqrt{4n-1})/2$.  In this case, $\lambda$ is a nonreal quadratic integer.  (In
fact, $\lambda$ is a root of the monic, quadratic polynomial $x^2 - x
+ n \in \ints[x]$, which is irreducible over $\rats$.)  Thus $R_\lambda
\subseteq \ints[\lambda] = \ints +\lambda\ints$, which is a discrete subring of $\complexes$.

Now suppose $|\lambda| \notin \{1,\sqrt 2,\sqrt 3\}$.  We have $1/4 \le |\lambda|^2 < 3$ but $|\lambda|^2 \notin \{1,2\}$.  If, in addition, $|\lambda|^2\neq 1+\p$, then $R_{|\lambda|^2}$ is convex by previous results.  Since $|\lambda|^2 \in R_\lambda$, we have that $R_\lambda$ is convex for these $\lambda$ by Corollary~\ref{cor:convex-extend}.

Finally, we consider the case where $|\lambda|^2 = 1+\p$, or equivalently, $\lambda = (1+i\sqrt{5+2\sqrt 5})/2$.  We have $|\lambda| = \p$ in this case, and in fact, $\lambda = \p e^{i\tau/5}$.  The points $0,1,\lambda$ form the vertices of an acute Robinson triangle, i.e., a triangle with side lengths $(1,\p,\p)$.  Given what we know about $R_{|\lambda|^2} = R_{1+\p}$, it may come as a surprise that $R_\lambda$ is indeed convex.  We show below via an explicit derivation that the point $\mu := e^{i\tau(7/10)}$ is in $R_\lambda$.  (The derivation below was found by a computer-assisted search.)  It then follows from Corollary~\ref{cor:convex-extend} and Proposition~\ref{prop:unit-circle} that $R_\lambda$ is convex.

We first note that $\lambda$ is an algebraic integer of degree $4$ with minimum polynomial $x^4 - 2x^3 + 4x^2 - 3x + 1$.  Thus
\begin{equation}\label{eqn:reduce}
\lambda^4 = -1 + \lambda - 4\lambda^2 + 2\lambda^3\;.
\end{equation}
It can also be readily checked (on purely geometric grounds, even) that
$\mu = 2\lambda - \lambda^2 + \lambda^3$.  The derivation of $\mu$ follows:
\begin{align*}
x_0 &:= 0 \\
x_1 &:= 1 \\
x_2 &:= x_0\x x_1 = \lambda \\
x_3 &:= x_1\x x_0 = 1 - \lambda \\
x_4 &:= x_0\x x_2 = \lambda^2 \\
x_5 &:= x_2\x x_1 = \lambda(2-\lambda) = 2\lambda - \lambda^2 \\
x_6 &:= x_1\x x_3 = (1+\lambda)(1-\lambda) = 1 - \lambda^2 \\
x_7 &:= x_4\x x_0 = \lambda^2(1-\lambda) = \lambda^2 - \lambda^3 \\
x_8 &:= x_1\x x_6 = (1+\lambda+\lambda^2)(1-\lambda) = 1 - \lambda^3 \\
x_9 &:= x_7\x x_5 = \lambda^2(3-3\lambda+\lambda^2) = -1 + 3\lambda - \lambda^2 - \lambda^3 & \mbox{(using (\ref{eqn:reduce}))} \\
x_{10} &:= x_2\x x_8 = \lambda(1-\lambda)(2+\lambda+\lambda^2) = 1 - \lambda + 3\lambda^2 - 2\lambda^3 & \mbox{(using (\ref{eqn:reduce}))} \\
\mu = x_{11} &:= x_9\x x_{10} = \lambda^2(1-\lambda)(5-2\lambda+2\lambda^2) = 2\lambda - \lambda^2 + \lambda^3 & \mbox{(using (\ref{eqn:reduce}))}
\end{align*}
\end{proof}

\begin{corollary}
Suppose $\Re(\lambda) = 1/2$, and let $y := |\Im(\lambda)|$.   If $y < \sqrt{11}/2$ and $y\notin \left\{\sqrt 3/2, \sqrt 7/2\right\}$, then $R_\lambda$ is convex.
\end{corollary}

\section{$R_\lambda$ for $\lambda$ contained in a discrete subring of $\complexes$}\label{sec:discrete-ring}

In this section, consider the case (hinted at in the previous section) where $\lambda$ belongs to a discrete subring of $\complexes$.  We start with two standard lemmas that characterize the discrete subrings of $\complexes$.

\begin{lemma}\label{lem:discrete-ring}
Suppose $D$ is a subring of $\complexes$ that is discrete in the induced topology.  Then no two distinct elements of $D$ are less than unit distance apart.  Consequently, $D$ is (topologically) closed, and $D\intersect\reals = \ints$.
\end{lemma}

\begin{proof}
Suppose for the sake of contradiction that $a,b\in D$ are such that $0<|a-b|<1$.  Then $(a-b)^n\in D \cmpl \{0\}$ for all integers $n>0$, and $\lim_{n\rightarrow\infty} (a-b)^n = 0$.  This means that $0\in D$ is an accumulation point of $D$, and hence $D$ is not discrete.  The other two consequences follow immediately.
\end{proof}

\begin{lemma}\label{lem:discrete-subring}
Suppose $D$ is a subring of $\complexes$ that is discrete in the induced topology.  Then either $D = \ints$ or $D = \ints[\alpha] = \ints + \alpha\ints$, where $\alpha$ is a nonreal quadratic integer.  Equivalently, either $D = \ints$ or there exists $n\in\posints$ such that $D = \ints + \alpha\ints$, where $\alpha$ is either $i\sqrt n$ or $(1+i\sqrt{4n-1})/2$.
\end{lemma}

\begin{proof}
$\ints$ is the smallest subring of $\complexes$ and is discrete.  If $D \ne \ints$, then choose some $\alpha \in D \cmpl \ints$.  Then $\alpha$ cannot be real by Lemma~\ref{lem:discrete-ring}.  Since $D\cmpl\ints$ is closed, we can choose $\alpha$ to have minimum norm.  This implies that $-1/2 \le \Re(\alpha) \le 1/2$, for otherwise, we could add some appropriate integer to $\alpha$ to reduce its norm.  We also have $|\alpha|\ge 1$ by Lemma~\ref{lem:discrete-ring}.

Since $\alpha\in D$, we have $\ints + \alpha\ints \subseteq D$.  We now show that $D \subseteq \ints + \alpha\ints$, and thus equality holds.  Suppose otherwise, and let $\beta$ be some element of $D \cmpl (\ints + \alpha\ints)$.  By adding some appropriate member of $\ints + \alpha\ints$ to $\beta$, we can assume that $\beta$ lies somewhere in the parallelogram $P$ with corners $(1+\alpha)/2$, $(1-\alpha)/2)$, $(-1+\alpha)/2$, and $(-1-\alpha)/2$, but not the origin.  $P$ is included in the larger parallelogram $P'$ with corners $\pm 1$ and $\pm\alpha$.  The norm of any non-corner point in $P'$ is strictly bounded by the norm of one of the corners of $P'$, which is $|\alpha|$, since $|\alpha| \ge 1$.  The corners of $P'$ are not included in $P$, and so we must have $|\beta| < |\alpha|$, contradicting the minimality of $|\alpha|$.

Thus $D = \ints[\alpha] = \ints + \alpha\ints$, and it follows that $\alpha$ is a quadratic integer.  By the quadratic formula, all quadratic integers are of the form $(m \pm \sqrt{m^2 - 4n})/2$ for some $m,n\in\ints$.  Since $\alpha\notin\reals$, we must have $n > 0$, and without loss of generality, we can assume $-1/2 < \Re(\alpha) \le 1/2$ and so $m \in \{0,1\}$, which gives the result.
\end{proof}

From Fact~\ref{fact:discrete-ring} and Lemma~\ref{lem:discrete-subring} it follows that if $D$ is discrete and $\lambda\in D$, then $R_\lambda = Q_\lambda \subseteq D$ and is discrete as well.  Sometimes equality holds in the inclusion above.  For example,

\begin{fact}
$R_2 = R_{-1} = \ints$.
\end{fact}

Usually, equality does not hold; $R_2$ is the only case where equality holds for $D := \ints$.  $R_\lambda$ is a proper subset of $\ints$ for all $\lambda\ge 3$, as the next general lemma implies.

\begin{lemma}\label{lem:ideal}
Let $D$ be any subring of $\complexes$.  For any $\lambda\in D$, let $I_\lambda := \lambda(1-\lambda)D = \{a\lambda(1-\lambda) \mid a\in D\}$ be the ideal of $D$ generated by $\lambda(1-\lambda)$.  Then
\[ Q_\lambda \subseteq I_\lambda + \{0,1,\lambda,1-\lambda\}
\;. \]
If $D$ is discrete, then the same inclusion holds for $R_\lambda$.
\end{lemma}

\begin{proof}[Proof sketch]
One merely checks that the right-hand side is $\lambda$-convex.
\end{proof}

The next corollary follows from Lemma~\ref{lem:ideal} by the Chinese Remainder Theorem.

\begin{corollary}\label{cor:ideal}
Let $D$ and $\lambda$ be as in Lemma~\ref{lem:ideal}, above.  Then
\[ Q_\lambda \subseteq (\lambda D + \{0,1\}) \intersect ((1-\lambda)D + \{0,1\})
\;. \]
If $D$ is discrete, then the same inclusion holds for $R_\lambda$.
\end{corollary}

When applying Lemma~\ref{lem:ideal} with $D := \ints$, it suffices to consider $\lambda > 1$, so in this case, we will assume $\lambda \ge 2$.

\begin{corollary}\label{cor:mod}
For all $\lambda\in\ints$ such that $\lambda \ge 2$ and all $n\in\ints$, if $n\in R_\lambda$, then $n \equiv d\pmod{\lambda(\lambda-1)}$ for some $d\in\{0,1,\lambda,1-\lambda\}$.  Equivalently, if $n$ is in $R_\lambda$ then $n$ is congruent to either $0$ or $1$ modulo both $\lambda$ and $\lambda-1$.  In particular, if $n\in R_\lambda$, then $n\equiv n^2 \pmod{\lambda(\lambda-1)}$.
\end{corollary}

One would generally like to know when equality holds in Lemma~\ref{lem:ideal} for discrete $D$.  We show that it holds at least for $D := \ints$ (Theorem~\ref{thm:period}, below), but we currently have no general proof for all discrete $D$.

We can at least prove a sufficient condition for equality (Theorem~\ref{thm:discrete-ring}, below).  First, a definition, which is justified by the lemma that follows it.

\begin{definition}\label{def:translation-point}
We will say that a point $x\in\complexes$ is a \emph{translation point} of $R_\lambda$ iff $\{x,x+1\} \subseteq R_\lambda$.
\end{definition}

\begin{lemma}\label{lem:translation-point}
If $x$ is a translation point for $R_\lambda$, then so is $-x$, and furthermore, $\rho_{x,x+1}(R_\lambda) = \rho_{-x,-x+1}(R_\lambda) = R_\lambda$.
\end{lemma}

\begin{proof}
$R_\lambda = 1 - R_\lambda$ by Corollary~\ref{cor:outer-dual}, so if $x$ is a translation point of $R_\lambda$, then so is $-x$.  We have $\rho_{x,x+1}(R_\lambda) \subseteq R_\lambda$ and $\rho_{-x,-x+1}(R_\lambda) \subseteq R_\lambda$ by Corollary~\ref{cor:translate}.  To get the reverse containments, we observe that $\rho_{x,x+1}$ and $\rho_{-x,-x+1}$ are inverses of each other, and so, applying $\rho_{-x,-x+1}$ to both sides of the first containment, we get
\[ R_\lambda = \rho_{-x,-x+1}(\rho_{x,x+1}(R_\lambda)) \subseteq \rho_{-x,-x+1}(R_\lambda)\;, \]
and applying $\rho_{x,x+1}$ to the second containment similarly yields $R_\lambda \subseteq \rho_{x,x+1}(R_\lambda)$.
\end{proof}

\begin{corollary}\label{cor:translation-point}
For any $\lambda\in\complexes$, the translation points of $R_\lambda$ form a subgroup of $\complexes$ under addition.
\end{corollary}

\begin{theorem}\label{thm:discrete-ring}
Let $D$ be a discrete subring of $\complexes$, and let $S\subseteq D$ generate the additive group of $D$.  Suppose $\lambda\in D$ is such that $a\lambda(1-\lambda)$ is a translation point of $R_\lambda$ for every $a\in S$.  Then
\begin{equation}\label{eqn:discrete-ring}
R_\lambda = I_\lambda + \{0,1,\lambda,1-\lambda\}
\;,
\end{equation}
where $I_\lambda := \lambda(1-\lambda)D \subseteq D$ is the ideal generated by $\lambda(1-\lambda)$.
\end{theorem}

\begin{proof}
Since $D$ is discrete, we have $R_\lambda = Q_\lambda$, and so the $\subseteq$-inclusion holds by Lemma~\ref{lem:ideal}.  For the reverse inclusion, we have by assumption and Corollary~\ref{cor:translation-point} that every element of $I_\lambda$ is a translation point of $R_\lambda$.  Now suppose $x \in I_\lambda + b$ for some $b\in \{0,1,\lambda,1-\lambda\}$.  Note that $b\in R_\lambda$.  Writing $x := y + b$ for $y\in I_\lambda$, we have
\[ x = \rho_{y,y+1}(b) \in \rho_{y,y+1}(R_\lambda) = R_\lambda \]
by Lemma~\ref{lem:translation-point}, because $y$ is a translation point of $R_\lambda$.  This proves the $\supseteq$-inclusion.
\end{proof}

Theorem~\ref{thm:discrete-ring} is useful because the rings in question are finitely generated $\ints$-modules, and so Equation~(\ref{eqn:discrete-ring}) can be verified by testing a finite number of points.  For example, Figure~\ref{fig:R_2i} shows $R_{2i}$.  Equation~(\ref{eqn:discrete-ring}) holds for $\lambda := 2i$, because $\ints[2i]$ is spanned by $\{1,2i\}$, and it is evident from the picture that both $4+2i = \lambda(1-\lambda)$ and $-4+8i = 2i\lambda(1-\lambda)$ are both translation points of $R_{2i}$.
\begin{figure}
\centering
\includegraphics[width=0.5\textwidth]{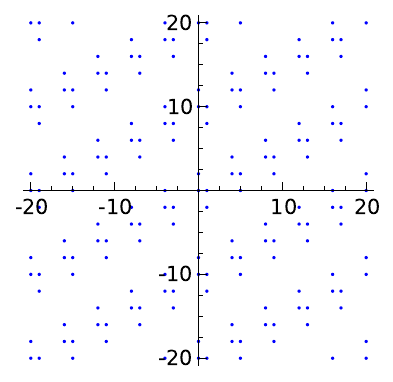}
\caption{A plot of $R_{2i}$.}\label{fig:R_2i}
\end{figure}

We end this section by showing (Lemma~\ref{lem:period} and Theorem~\ref{thm:period}) that Equation~(\ref{eqn:discrete-ring}) holds for all integer $\lambda\ge 2$ (that is, when $D = \ints$  and $I_\lambda = \lambda(1-\lambda)\ints$ in Theorem~\ref{thm:discrete-ring}).  It follows immediately that $R_\lambda$ is periodic for all $\lambda\in\ints\cmpl\{0,1\}$ and that the period is $\lambda(\lambda-1)$ if $\lambda \ge 3$ (Corollary~\ref{cor:periodic}).  Most of the technical difficulty is in proving Lemma~\ref{lem:period}, so we defer that proof until after Theorem~\ref{thm:period}.

\begin{lemma}\label{lem:period}
For every $\lambda\in\ints$ with $\lambda \ge 3$, the value $\lambda(\lambda-1)$ is contained in $R_\lambda$; in fact, it is in $Q_\lambda^{(\lambda-1)}$ unless $\lambda = 3$ or $\lambda = 5$, in which case, it is in $Q_\lambda^{(\lambda)}$.
\end{lemma}

\begin{theorem}\label{thm:period}
For all integers $\lambda \notin \{0,1\}$,
\[ R_\lambda = \lambda(\lambda-1)\ints + \{0,1,\lambda,1-\lambda\}\;. \]
\end{theorem}

\begin{proof}
We can assume WLOG that $\lambda \ge 2$, since both sides of the equation are unchanged by substituting $1-\lambda$ for $\lambda$ everywhere (q.v.\ Fact~\ref{fact:inner-dual}).  The $\lambda = 2$ case is obvious, so assume $\lambda \ge 3$.

Along with $0$ and $1$, the following are all elements of $R_\lambda$:
\begin{align*}
0\x 1 &= \lambda\;, \\
1\x 0 &= 1-\lambda\;, \\
\lambda\x 0 &= \lambda(1-\lambda)\;, \\
1 - \lambda(1-\lambda) &= \lambda(\lambda-1) + 1\;.
\end{align*}
Thus by Lemma~\ref{lem:period}, $\lambda(\lambda-1)$ is a translation point of $R_\lambda$.  Since the additive group of $\ints$ is generated by $\{1\}$, the theorem follows by Theorem~\ref{thm:discrete-ring}.
\end{proof}

\begin{proof}[Proof of Lemma~\ref{lem:period}]
Let $n>1$ be an integer.  By Lemma~\ref{lem:level-characterization}, $\lambda(\lambda-1)\in Q_\lambda^{(n)}$ if and only if there exist integers $b_0,\ldots,b_n$ such that $0\le b_i \le \binom{n}{i}$ for all $0\le i\le n$ and
\[ \lambda(\lambda-1) = \sum_{i=0}^n b_i \lambda^i(1-\lambda)^{n-i}\;. \]
Letting $b_0$ and $b_n$ both be zero\footnote{This is necessary, as can be seen by reducing the above equation modulo $\lambda$ and $\lambda-1$, respectively, and noting that $b_0,b_n\in\{0,1\}$.} and dividing both sides by $\lambda(1-\lambda)$, we get the equivalent equation
\begin{equation}\label{eqn:period-reduced-proof}
-1 = \sum_{i=1}^{n-1} b_i\lambda^{i-1}(1-\lambda)^{n-i-1} = \sum_{j=0}^m c_j \lambda^j\mu^{m-j}\;,
\end{equation}
where we define $m := n-2$ and $c_j := b_{j+1}$ for all $0\le j\le m$, and we set $\mu := 1-\lambda$ for convenience.  The range requirement of each $c_j$ is then
\begin{equation}\label{eqn:range-requirement}
0\le c_j \le \binom{m+2}{j+1}\hspace{0.5in}(\forall j,\;0\le j\le m)\;.
\end{equation}
Thus the theorem is proved for $\lambda$ if we can find an appropriate $m$ and integer values for $c_0,\ldots,c_m$ satisfying both (\ref{eqn:period-reduced-proof}) and (\ref{eqn:range-requirement}).

Here are explicit values for $m$ and $c_0,\ldots,c_m$ satisfying (\ref{eqn:period-reduced-proof}) and (\ref{eqn:range-requirement}) for $3\le\lambda\le 8$:
\[ \begin{array}{|c||c||c|c|c|c|c|c|}\hline
\lambda & m & c_0 & c_1 & c_2 & c_3 & c_4 & c_5 \\\hline\hline
3       & 1 & 2   & 1   &     &     &     &     \\\hline
4       & 1 & 3   & 2   &     &     &     &     \\\hline
5       & 3 & 4   & 1   & 2   & 3   &     &     \\\hline
6       & 3 & 5   & 2   & 3   & 4   &     &     \\\hline
7       & 4 & 6   & 9   & 0   & 3   & 5   &     \\\hline
8       & 5 & 7   & 2   & 6   & 4   & 3   & 6   \\\hline
\end{array} \]
Thus from now on, we can assume that $\lambda \ge 9$.  We then must set $n := \lambda-1$, whence $m = \lambda-3 \ge 6$.

To find $c_0,\ldots,c_m$ satisfying (\ref{eqn:period-reduced-proof}) and (\ref{eqn:range-requirement}), we first find values for the $c_j$ that satisfy (\ref{eqn:period-reduced-proof}) but ignore the range requirements (\ref{eqn:range-requirement}).  We then make a series of adjustments to the $c_j$ in a way that leaves the right-hand side of (\ref{eqn:period-reduced-proof}) unchanged, until all coefficients are in their required ranges.

Initially, we set $c_j := -\binom{m}{j}$ for all $0\le j\le m$.\footnote{It will be convenient in this proof to treat the $c_j$ as variables whose values can change, as in a computer algorithm, rather than choosing new symbols to denote changed values.  Which values of the $c_j$ we are referring to will always be clear from the context.}  These values satisfy (\ref{eqn:period-reduced-proof}) by the Binomial Theorem, using the fact that $\lambda + \mu = 1$.  Now what kind of adjustment to the $c_j$ preserves (\ref{eqn:period-reduced-proof})?  Choose some $j$ with $0\le j<m$.  If we simultaneously add $\lambda$ to $c_j$ and subtract $\mu$ from $c_{j+1}$ (equivalently, add $\lambda-1$ to $c_{j+1}$), then these two changes clearly cancel, and the right-hand side of (\ref{eqn:period-reduced-proof}) is unchanged.  We call this a \emph{$j$-adjustment}:
\begin{algo}
$j$\textbf{-adjust:} \\
\>$c_j \assn c_j+\lambda$ \\
\>$c_{j+1} \assn c_{j+1} + \lambda-1$ \\
\myend
\end{algo}
Note that this adjustment increases the values of both $c_j$ and $c_{j+1}$.  Our strategy is then to choose different values of $j$ in some order and, for each $j$, make just enough $j$-adjustments so that either $c_j$ is positive or $c_{j+1}$ is positive, depending on the value of $j$.

The order of our choices of $j$ is important.  The strictest range requirements are at the ``ends,'' i.e., for $j$ close to $0$ or $m$.  We make those adjustments first, working our way inward, finishing somewhere in the middle.  Recalling that $m\ge 6$ and $\lambda = m+3$, we define the middle index to be
\begin{equation}\label{eqn:p}
p := \bigceiling{\frac{\lambda m-\lambda+1}{2\lambda-1}} = \bigceiling{\frac{m^2+2m-2}{2m+5}} = \bigceiling{\frac{m}{2}}\;.
\end{equation}
Note that $3 \le p \le m-3$.  We choose this value because setting $j:=p$ maximizes the function $t(j)$, defined as
\[ t(j) := \binom{m}{j}\left(\frac{\lambda}{\lambda-1}\right)^j\;, \]
giving us the most leeway where we need it.  One can readily check that the sequence $t(0),\ldots,t(m)$ is bitonic, ascending monotonically from $t(0)$ to $t(p)$, then descending monotonically from $t(p)$ to $t(m)$.

Here is the algorithm to obtain $c_0,\ldots,c_m$ satisfying (\ref{eqn:period-reduced-proof}) and all range requirements (\ref{eqn:range-requirement}), with explanation afterwards:
\begin{algo}
// Initialization \\
$\myfor j\assn 0 \myto m \mydo$ \\
\>$c_j \assn -\binom{m}{j}$ \\
\myend\myfor \\
// Adjusting to the left of index $p$ \\
$\myfor j\assn 0 \myto p-1 \mydo$ \\
\>$\mywhile c_j < 0 \mydo$ \\
\>\>$j$\textbf{-adjust} \\
\>\myend\mywhile \\
\>// $0\le c_j < \lambda$, satisfying (\ref{eqn:range-requirement}), and $c_j$ is not changed subsequently. \\
\myend\myfor \\
// $c_p < \lambda$ \\
// Adjusting to the right of index $p$ \\
$\myfor j\assn m \mydownto p+2 \mydo$ \\
\>$\mywhile c_j < 0 \mydo$ \\
\>\>$(j-1)$\textbf{-adjust} \\
\>\myend\mywhile \\
\>// $0\le c_j < \lambda$, satisfying (\ref{eqn:range-requirement}), and $c_j$ is not changed subsequently. \\
\myend\myfor \\
// $c_{p+1} < \lambda$ \\
// Adjusting at index $p$ \\
$\mywhile c_p < 0 \textbf{ or } c_{p+1} < 0 \mydo$ \\
\>$p$\textbf{-adjust} \\
\myend\mywhile
\end{algo}

To show that this algorithm is correct, we first justify the assertions made in comments after the first two inner while-loops.  First, $c_0$ is changed by a single $0$-adjustment from $-1$ to $\lambda-1$.  For $k=1,2,\ldots,p-1$ in that order, $c_k$ is first increased by a sequence of $(k-1)$-adjustments followed by zero or more $k$-adjustments.  Since $c_{k-1} \ge -\binom{m}{k-1}$ immediately before the $(k-1)$-adjustments, and each $(k-1)$-adjustment increases $c_{k-1}$ by $\lambda$, the number $a$ of $(k-1)$-adjustments satisfies $a \le \bigceiling{\lambda^{-1}\binom{m}{k-1}}$.  Thus after all $(k-1)$-adjustments, we have (recalling that $t(k-1)\le t(k)$)
\begin{align*}
c_k &= -\binom{m}{k} + (\lambda-1)a < -\binom{m}{k} + (\lambda-1)\left(\lambda^{-1}\binom{m}{k-1}+1\right) \\
&= \binom{m}{k}\left(-1+\frac{t(k-1)}{t(k)}\right) + \lambda - 1 \le \lambda-1\;.
\end{align*}
Then the $k$-adjustments (if any) leave $0 \le c_k < \lambda$ as required, and $c_k$ is not changed subsequently.  This justifies the comment after the first inner while-loop.  Note that this reasoning also applies to $c_p$, showing that $c_p < \lambda$ after the second for-loop (although $c_p$ may still be negative at that point).

Justifying the comment after the second inner while-loop is similar.  First, $c_m$ is changed by a single $(m-1)$-adjustment, going from $-1$ to $\lambda - 2$.  Then for $k = m-1,m-2,\ldots,p+2$ in that order, $c_k$ is first increased by a sequence of $k$-adjustments followed by zero or more $(k-1)$-adjustments.  Since $c_{k+1} \ge -\binom{m}{k+1}$ immediately before the $k$-adjustments, and each $k$-adjustment increases $c_{k+1}$ by $\lambda-1$, the number $b$ of $k$-adjustments satisfies $b \le \bigceiling{(\lambda-1)^{-1}\binom{m}{k+1}}$.  Thus after all $k$-adjustments, we have (recalling that $t(k+1)\le t(k)$)
\begin{align*}
c_k &= -\binom{m}{k} + \lambda b < -\binom{m}{k} + \lambda\left((\lambda-1)^{-1}\binom{m}{k+1}+1\right) \\
&= \binom{m}{k}\left(-1+\frac{t(k+1)}{t(k)}\right) + \lambda \le \lambda\;.
\end{align*}
Then the $(k-1)$-adjustments (if any) leave $0 \le c_k < \lambda$ as required, and $c_k$ is not changed subsequently.  This justifies the comment after the second inner while-loop.  Like before, this reasoning also applies to $c_{p+1}$, showing that $c_{p+1} < \lambda$ after the third for-loop (although $c_{p+1}$ may still be negative at that point).

It remains to show that after the last while-loop, $c_p$ and $c_{p+1}$ satisfy (\ref{eqn:range-requirement}).  That loop results in both $c_p$ and $c_{p+1}$ being nonnegative.  Thus we are done if $c_p \le \binom{m+2}{p+1}$ and $c_{p+1} \le \binom{m+2}{p+2}$ in the end.  Let $r$ and $s$ be the number of $p$-adjustments needed to get $c_p \ge 0$ and $c_{p+1} \ge 0$, respectively.  Since $c_p \ge -\binom{m}{p}$ and $c_{p+1} \ge -\binom{m}{p+1}$ just before the final while-loop, we have
\begin{align*}
r &\le \bigceiling{\lambda^{-1}\binom{m}{p}} = \bigceiling{t(p)\frac{(\lambda-1)^p}{\lambda^{p+1}}} \;, & s &\le \bigceiling{(\lambda-1)^{-1}\binom{m}{p+1}} = \bigceiling{t(p+1)\frac{(\lambda-1)^p}{\lambda^{p+1}}}\;.
\end{align*}
Since $t(p) \ge t(p+1)$, the last while-loop runs at most $\bigceiling{\lambda^{-1}\binom{m}{p}}$ times, and since $c_p < \lambda$ and $c_{p+1} < \lambda$ just before this loop, the final values of $c_p$ and $c_{p+1}$ then satisfy
\begin{align*}
c_p &< \lambda + \lambda \bigceiling{\lambda^{-1}\binom{m}{p}} < \lambda + \lambda\left(\lambda^{-1}\binom{m}{p}+1\right) = \binom{m}{p} + 2\lambda\;, \\
c_{p+1} &< \lambda + (\lambda-1) \bigceiling{\lambda^{-1}\binom{m}{p}} \le \lambda + \lambda\bigceiling{\lambda^{-1}\binom{m}{p}} < \lambda + \lambda\left(\lambda^{-1}\binom{m}{p}+1\right) = \binom{m}{p} + 2\lambda\;.
\end{align*}
Since each quantity in the inequalities above is an integer, and there are two strict inequalities in each chain, we have both $c_p$ and $c_{p+1}$ ending up less than or equal to $f := \binom{m}{p} + 2(\lambda-1) = \binom{m}{p} + 2(m+2)$.  To finish the proof, we show that $f \le \binom{m+2}{p+1}$ and $f \le \binom{m+2}{p+2}$.

By a straightforward calculation,
\[ \binom{m+2}{p+2} = \frac{m-p}{p+1}\binom{m+2}{p+1} = \frac{(m+2)(m+1)}{(p+2)(p+1)}\binom{m}{p}\;. \]
By (\ref{eqn:p}), we have $m-p \le p+1$, whence $\binom{m+2}{p+2} \le \binom{m+2}{p+1}$, and so it suffices to show that $f \le \binom{m+2}{p+2}$, or equivalently,
\begin{align*}
\binom{m}{p} + 2(m+2) &\le \frac{(m+2)(m+1)}{(p+2)(p+1)}\binom{m}{p}\;, \\
\intertext{i.e.,}
2 &\le \binom{m}{p}\left(\frac{m+1}{(p+2)(p+1)}-\frac{1}{m+2}\right)\;.
\end{align*}
Finally, to show this last inequality, recalling that $m\ge 6$ and $3\le p\le m-3$, we have
\begin{align*}
\binom{m}{p}\left(\frac{m+1}{(p+2)(p+1)}-\frac{1}{m+2}\right) &= \binom{m}{p}\frac{(m+2)(m+1) - (p+2)(p+1)}{(m+2)(p+2)(p+1)} \\
&\ge \binom{m}{3} \frac{(m+2)(m+1) - (m-1)(m-2)}{(m+2)(m-1)(m-2)} \\
&= \binom{m}{3}\frac{6m}{(m+2)(m-1)(m-2)} = \frac{m^2}{m+2} \ge \frac{9}{2} \ge 2\;.
\end{align*}
\end{proof}

\begin{remark}
The ranks given in Theorem~\ref{thm:period} are tight, i.e., $6\notin Q_3^{(2)}$ and $20\notin Q_5^{(4)}$ (which can be checked by exhaustive search), and $\lambda(\lambda-1)\notin Q_\lambda^{(\lambda-2)}$ for any $\lambda\ge 3$.  The latter can be seen by reducing (\ref{eqn:period-reduced-proof}) modulo $\lambda$, which gives $c_0 \equiv -1 \pmod{\lambda}$.  Since $0\le c_0 \le \binom{m+2}{1} = m+2$, we must have $\lambda - 1 \le c_0 \le m+2$.  In particular, $m \ge \lambda-3$, which implies $n \ge \lambda-1$.  This also shows that no single choice of $m$ can satisfy (\ref{eqn:period-reduced-proof}) for all $\lambda$.
\end{remark}


\begin{corollary}\label{cor:periodic}
$R_\lambda$ is periodic with period $\lambda(\lambda-1)$ for all $\lambda\in\ints$ such that $\lambda\ge 3$.
\end{corollary}


\section{A characterization of $Q_{[x]}$ with some applications}
\label{sec:Q-x}

In this section we prove a simple characterization of $Q_{[x]}$ (see Definition~\ref{def:Q-x}) beyond the characterization given in Lemma~\ref{lem:level-characterization} (and by extension, a new characterization of $Q_\lambda$).  This lets us, among other things, list all the polynomials in $Q_{[x]}$ of degree $\le 2$ and get a finite upper bound on the number of polynomials in $Q_{[x]}$ of any given degree bound.

Recall that $Q_{[x]} \subseteq \ints[x]$.

\begin{theorem}\label{thm:Q-x}
Let $f$ be any polynomial in $\ints[x]$.  Then $f\in Q_{[x]}$ if and only if either $f=0$ or $f=1$ or $0<f(\lambda)<1$ for all $0<\lambda<1$.
\end{theorem}

\begin{corollary}\label{cor:Q-x-mult}
$Q_{[x]}$ is closed under multiplication and the operator $p\mapsto 1-p$.
\end{corollary}

\begin{corollary}\label{cor:Q-lambda-characterization}
For any $\lambda\in\complexes$,
\[ Q_\lambda = \{0,1\} \union \{ p(\lambda) \mid p\in\ints[x] \myand \mbox{$0<p(\mu)<1$ for all $0<\mu<1$} \}\;. \]
\end{corollary}

\begin{corollary}\label{cor:arith-progressions}
For every $\lambda\in\complexes$, \ $Q_\lambda$ contains arbitrarily long finite arithmetic progressions.  In particular, for every integer $n\ge 0$,
\[ \left\{ k\lambda^n(1-\lambda)^n \mid k\in\ints \myand 0\le k < 2^{2n} \right\} \subseteq Q_\lambda\;. \]
\end{corollary}

Before proving Theorem~\ref{thm:Q-x}, we need a definition and a few lemmas.
 We extend the definition of the binomial coefficient $\binom{n}{k}$ in the usual way for all $n,k\in\ints$ with $n\ge 0$, namely, by defining $\binom{n}{k} := 0$ if $k<0$ or $k>n$.  Then the recurrence $\binom{n+1}{k} = \binom{n}{k} + \binom{n}{k-1}$ holds for all such $n$ and $k$.

\begin{definition}\label{def:f-n}
Let $f\in\complexes[x]$ be any polynomial, and let $n$ be any nonnegative integer such that $\deg(f) \le n$.  We let $f^{(n)}_0,f^{(n)}_1,\ldots,f^{(n)}_n \in \complexes$ denote the unique coefficients such that $f(x) = \sum_{k=0}^n f^{(n)}_k x^k(1-x)^{n-k}$ (cf.\ Lemma~\ref{lem:lin-indep}).  We define $f^{(n)}_k := 0$ for all $k<0$ and $k>n$.
\end{definition}

The next two lemmas relate the $f^{(n)}$-coefficients for different $n$.  The first lemma says that the $f^{(n)}$-coefficients satisfy the same ``Pascal's triangle'' recurrence as the binomial coefficients.

\begin{lemma}\label{lem:f-n-recurrence}
Let $f$ and $n$ be as in Definition~\ref{def:f-n}.  Then for any $k\in\ints$, \ $f^{(n+1)}_k = f^{(n)}_k + f^{(n)}_{k-1}$.
\end{lemma}

\begin{proof}
This is clearly true for $k<0$ and $k>n+1$, since both sides are $0$.  Moreover, we have
\begin{align*}
f(x) &= \sum_{k=0}^n f^{(n)}_k x^k(1-x)^{n-k} = (x + (1 - x))\sum_{k=0}^n f^{(n)}_k x^k(1-x)^{n-k} \\
&= \sum_{k=0}^n f^{(n)}_k x^{k+1}(1-x)^{n-k} + \sum_{k=0}^n f^{(n)}_k x^k(1-x)^{n+1-k} \\
&= \sum_{k=1}^{n+1} f^{(n)}_{k-1} x^k(1-x)^{n+1-k} + \sum_{k=0}^n f^{(n)}_k x^k(1-x)^{n+1-k} \\
&= \sum_{k=0}^{n+1} \left(f^{(n)}_{k-1} + f^{(n)}_k\right) x^k(1-x)^{n+1-k}\;.
\end{align*}
Comparing coefficients with the equation $f(x) = \sum_{k=0}^{n+1} f^{(n+1)}_k x^k(1-x)^{n+1-k}$, we see that $f^{(n+1)}_k = f^{(n)}_{k-1} + f^{(n)}_k$ for all $0\le k\le n+1$.
\end{proof}

The next lemma extends the previous one in a natural way.

\begin{lemma}\label{lem:different-n}
Let $f\in\complexes[x]$ be any polynomial, and let $m$ be any natural number such that $m \ge \deg(f)$.  For any integers $n$ and $k$ such that $n\ge m$,
\[ f^{(n)}_k = \sum_{i=0}^m f^{(m)}_i\binom{n-m}{k-i}\;. \]
\end{lemma}

\begin{proof}
We proceed by induction on $n$.  If $n = m$, then for $0\le k\le m$ we have $\sum_{i=0}^m f^{(m)}_i\binom{n-m}{k-i} = f^{(m)}_k\binom{0}{0} = f^{(m)}_k$.  Now suppose the lemma holds for some $n\ge m$.  We have $f^{(n+1)}_k = f^{(n)}_k + f^{(n)}_{k-1}$ for any $k\in\ints$ by Lemma~\ref{lem:f-n-recurrence}, and so by the inductive hypothesis,
\begin{align*}
f^{(n+1)}_k &= f^{(n)}_{k-1} + f^{(n)}_k = \sum_{i=0}^m f^{(m)}_i\binom{n-m}{k-1-i} + \sum_{i=0}^m f^{(m)}_i \binom{n-m}{k-i} \\
&= \sum_i f^{(m)}_i\left[\binom{n-d}{k-i}+\binom{n-d}{k-i-1}\right] = \sum_i f^{(m)}_i \binom{n+1-d}{k-i}\;.
\end{align*}
Thus the lemma holds for $n+1$.
\end{proof}

\begin{lemma}\label{lem:there-is-a-t}
Let $f\in\complexes[x]$ be any polynomial, and let $d := \deg(f)$.  For every $\eps > 0$ there exists a $t\ge 0$ such that, for all natural numbers $n \ge d$ and all integers $k$ such that $t < k < n-t$,
\[ \bigabs{f(k/n) - \frac{f^{(n)}_k(1-k/n)^d}{\binom{n-d}{k}}} \le \eps\;. \]
\end{lemma}

\begin{proof}
We can assume WLOG that $f \ne 0$ (and thus $d\ge 0$).  Set $\lambda := k/n$.  From Lemma~\ref{lem:different-n}, we have that if $d < k < n-d$, then $0<\lambda<1$ and
\begin{align*}
f^{(n)}_k &= \sum_{i=0}^d f^{(d)}_i\binom{n-d}{k-i} \\
&= \binom{n-d}{k}\left(f^{(d)}_0 + f^{(d)}_1\frac{k}{n-d-k+1} + f^{(d)}_2\frac{k(k-1)}{(n-d-k+1)(n-d-k+2)} + \cdots\right) \\
&= \binom{n-d}{k} \sum_{i=0}^d f^{(d)}_i \prod_{j=1}^i \frac{k-j+1}{n-d-k+j} = \binom{n-d}{k} \sum_{i=0}^d f^{(d)}_i \prod_{j=1}^i \frac{\lambda-\frac{j-1}{n}}{1-\lambda-\frac{d-j}{n}}\;.
\end{align*}
We then have
\begin{align}
\bigabs{f(\lambda) - \frac{f^{(n)}_k(1-\lambda)^d}{\binom{n-d}{k}}} &= \bigabs{\sum_{i=0}^d f^{(d)}_i \lambda^i(1-\lambda)^{d-i} -  (1-\lambda)^d\sum_{i=0}^d f^{(d)}_i \prod_{j=1}^i \frac{\lambda-\frac{j-1}{n}}{1-\lambda-\frac{d-j}{n}}} \\
&\le \sum_{i=0}^d\bigabs{f^{(d)}_i}\cdot\bigabs{\lambda^i(1-\lambda)^{d-i} -  (1-\lambda)^d\prod_{j=1}^i \frac{\lambda-\frac{j-1}{n}}{1-\lambda-\frac{d-j}{n}}} \\
&= \sum_{i=0}^d\bigabs{f^{(d)}_i}\cdot\lambda^i(1-\lambda)^{d-i}\cdot\bigabs{1 - \prod_{j=1}^i \left(\frac{\lambda-\frac{j-1}{n}}{\lambda}\right)\left(\frac{1-\lambda}{1-\lambda-\frac{d-j}{n}}\right)} \\
&\le \sum_{i=0}^d\bigabs{f^{(d)}_i}\cdot\bigabs{1 - \prod_{j=1}^i \left(\frac{\lambda-\frac{j-1}{n}}{\lambda}\right)\left(\frac{1-\lambda}{1-\lambda-\frac{d-j}{n}}\right)}\;. \label{eqnitem:rhs}
\end{align}
The product $\prod_{j=1}^i\cdots$ above is positive.  If $k$ and $n-k$ are both large compared to $d$, then it is also close to $1$, but it may be less than or greater than $1$, depending on $\lambda$.  We will bound it from above and below.  Note that, for $1\le j \le d$,
\[ 1- \frac{d}{k} = 1 - \frac{d}{\lambda n} < \frac{\lambda - \frac{j-1}{n}}{\lambda} < 1 \]
and
\[ 1 < \frac{1-\lambda}{1-\lambda-\frac{d-j}{n}} < 1 + \frac{d/n}{1-\lambda - d/n} =  1 + \frac{d}{n-k-d}\;. \]
Letting $M := \sum_{i=0}^d \bigabs{f^{(d)}_i}$, we see that (\ref{eqnitem:rhs}) above is then less than or equal to
\begin{align*}
M\cdot \max&\left\{1 - \left(1 - \frac{d}{k}\right)^d,\;\left(1 + \frac{d}{n - k - d}\right)^d - 1\right\} \\
&\le M \cdot \max\left\{\left(1 + \frac{d}{k}\right)^d - 1,\;\left(1 + \frac{d}{n-k - d}\right)^d - 1\right\} \\
&\le M \cdot \left(\max\left\{\exp(d^2/k),\; \exp(d^2/(n-k-d))\right\} - 1 \right)
\end{align*}
Now we just need to let $t\ge d$ be large enough so that this quantity is at most $\eps$ when $t < k < n-t$.  Letting
\[ t := \bigceiling{\frac{d^2}{\log\left(\frac{\eps}{M} + 1\right)} + d} \]
suffices.  (Note that $M>0$ (because $f\ne 0$) and that $t$ only depends on $f$ and $\eps$ and not on $n$.)
\end{proof}

\begin{lemma}\label{lem:coeffs-ge-zero}
Let $f\in\reals[x]$ be a polynomial such that $f(\lambda) > 0$ for all $0 < \lambda < 1$.  Then for all sufficiently large $n$, \ $f^{(n)}_0, \ldots, f^{(n)}_n \ge 0$.
\end{lemma}

\begin{proof}
Notice that if $f^{(n)}_0, \ldots, f^{(n)}_n \ge 0$ for \emph{some} natural number $n\ge \deg(f)$, then by Lemma~\ref{lem:different-n}, we have $f^{(n')}_0, \ldots, f^{(n')}_n \ge 0$ for all $n' \ge n$ as well.  Thus it suffices to find some $n$ such that $f^{(n)}_0, \ldots, f^{(n)}_n \ge 0$.

We use induction on $d := \deg(f)$.  If $d = 0$, then $f(x)$ is some constant $c>0$.  We then have $f^{(0)}_0 = c > 0$, and so taking $n := 0$ suffices.  Now suppose $d > 0$ and the lemma holds for all polynomials of degree less than $d$.  If $f(0) = 0$, then $f(x) = xg(x)$ for some polynomial $g\in\reals[x]$ of degree $d-1$.  By the inductive hypothesis, there is an $n \ge d-1$ such that $g^{(n)}_0,\ldots,g^{(n)}_n \ge 0$.  Then $f^{(n+1)}_0 = f(0) = 0$, and for all $1\le k\le n+1$, we see that $f^{(n+1)}_k = g^{(n)}_{k-1} \ge 0$.  Thus the lemma holds for $f$ witnessed by $n+1$.

A similar argument applies if $f(1) = 0$: we get $f(x) = (1-x)h(x)$ for some $h$ of degree $d-1$.  Letting $n$ be such that $h^{(n)}_0,\ldots,h^{(n)}_n \ge 0$, we have $f^{(n+1)}_k = h^{(n)}_k \ge 0$ for all $0\le k\le n$, and furthermore, $f^{(n+1)}_{n+1} = f(1) = 0$.

We can now assume (by the continuity of $f$) that $f(0) > 0$ and $f(1) > 0$.  By compactness, there exists $\eps > 0$ such that $f(\lambda) \ge \eps$ for all $0\le \lambda\le 1$.  Now given $f$ and $\eps$, let $t$ be the number obtained from Lemma~\ref{lem:there-is-a-t}.  Then for all $n \ge d$ and all $k$ such that $t < k < n - t$, that lemma implies $f^{(n)}_k \ge 0$, because $f^{(n)}_k(1-k/n)^d/\binom{n-d}{k} \ge 0$.  (The latter quantity is within $\eps$ of $f(k/n)$, which itself is at least $\eps$.)

It remains to show that if $n$ is sufficiently large, then $f^{(n)}_k \ge 0$ and $f^{(n)}_{n-k} \ge 0$ for all $0\le k\le t$.  To this end, it suffices to prove the following statement for all $k$, which we do by induction on $k$:
\begin{quote}
There exists an integer $n_k \ge d$ such that, for all integers $n \ge n_k$, \ $f^{(n)}_k \ge \eps$ and $f^{(n)}_{n-k} \ge \eps$.
\end{quote}
For $k=0$, we have $f^{(n)}_0 = f(0) \ge \eps$ and $f^{(n)}_n = f(1) \ge \eps$ for all $n\ge d$, and so we can set $n_0 := d$.  Now let $k>0$, and assume the statement holds for $k-1$.  Let $y := f^{(n_{k-1})}_k$.  (It could be that $y<0$.)  By Lemma~\ref{lem:f-n-recurrence},
\[ f^{(n_{k-1}+1)}_k = f^{(n_{k-1})}_k + f^{(n_{k-1})}_{k-1} \ge y + \eps\;. \]
Similarly,
\[ f^{(n_{k-1}+2)}_k = f^{(n_{k-1}+1)}_k + f^{(n_{k-1}+1)}_{k-1} \ge y + 2\eps\;, \]
and so on, yielding, for all $\ell \ge n_{k-1}$,
\[ f^{(\ell)}_k \ge y + (\ell - n_{k-1})\eps\;. \]
Thus $f^{(\ell)}_k \ge \eps$ for all $\ell$ large enough.  By a similar argument, letting $z := f^{(n_{k-1})}_{n_{k-1}-k}$, we get
\[ f^{(m)}_{m-k} \ge z + (m-n_{k-1})\eps \]
for all $m \ge n_{k-1}$.  Now setting
\[ n_k := \bigceiling{\max\left\{n_{k-1},\;n_{k-1} + 1 - \frac{y}{\eps},\;n_{k-1} + 1 - \frac{z}{\eps}\right\}}\;, \]
the statement holds for $k$.
\end{proof}

\begin{proof}[Proof of Theorem~\ref{thm:Q-x}]
First we show the ``only if'' part.  If $f\in Q_{[x]}$, then $f\in Q_{[x]}^{(n)}$ for some $n$.  By Lemma~\ref{lem:level-characterization}, there exist  integers $b_0,\ldots,b_n$ such that $0\le b_k\le\binom{n}{k}$ for all $0\le k\le n$ and $f(x) = \sum_{k=0}^n b_k x^k(1-x)^{n-k}$.  (In fact, $b_k = f^{(n)}_k$ by Lemma~\ref{lem:lin-indep}.)  If $b_k = 0$ for all $k$, then $f=0$.  If $b_k = \binom{n}{k}$ for all $k$, then $f = (x+(1-x))^n = 1$ by the Binomial Theorem.  Otherwise, some $b_i>0$, and this clearly implies $f(\lambda) > 0$ for all $0<\lambda<1$; also, some $b_j < \binom{n}{j}$, which similarly implies $f(\lambda) < 1$ for all $0<\lambda<1$.

Now we show the ``if'' part.  If $f=0$ or $f=1$, then $f\in Q_{[x]}^{(0)}$, and we are done.  Otherwise, if $0<f(\lambda)<1$ for all $0<\lambda<1$, then by Lemma~\ref{lem:coeffs-ge-zero}, there exists $n'$ such that for all $n\ge n'$ and $0\le k\le n$, we have $f^{(n)}_k \ge 0$.  Letting $g := 1-f$, we have $0<g(\lambda)$ for all $0<\lambda<1$, and so also by Lemma~\ref{lem:coeffs-ge-zero}, there exists $n''$ such that $g^{(n)}_k \ge 0$ for all  $n\ge n''$ and $0\le k\le n$.  Now let $n := \max\{n',n''\}$.  From the Binomial Theorem,
\[ 1 = (x + (1-x))^n = \sum_{k=0}^n \binom{n}{k} x^k(1-x)^{n-k}\;, \]
viewed as a polynomial in $x$.  Then
\[ g(x) = 1 - f(x) = \sum_{k=0}^n\left[\binom{n}{k}-f^{(n)}_k\right]x^k(1-x)^{n-k}\;. \]
Comparing coefficients, we have $0\le g^{(n)}_k = \binom{n}{k}-f^{(n)}_k$, and thus $0\le f^{(n)}_k \le \binom{n}{k}$, for all $0\le k\le n$.

We can now apply Lemma~\ref{lem:level-characterization} again (letting $b_k := f^{(n)}_k$) to put $f$ into $Q_{[x]}^{(n)}$ and be done, provided $f^{(n)}_0,\ldots,f^{(n)}_n$ are all integers.  This is true, and one way to see it is as follows: Consider the $\complexes$-linear map $\map{\p}{\complexes^{n+1}}{\complexes^{n+1}}$ that maps any vector $(b_0,\ldots,b_n)$ to the unique vector $(c_0,\ldots,c_n)$ such that
\[ \sum_{k=0}^n b_k x^k(1-x)^{n-k} = \sum_{k=0}^n c_k x^k\;. \]
(The left-hand side is a polynomial in $x$ of degree $\le n$, so this map is well-defined and easily seen to be linear.)  Let $M$ be the $(n+1)\times(n+1)$ matrix representing $\p$.  By expanding the left-hand side above for various choices of $(b_0,\ldots,b_n)$, one sees that $M$ is an integer matrix.  Note that if $b_0 = \cdots = b_{i-1} = 0$ for some $i\le n$, then $x^i$ divides the left-hand side, and thus $c_0 = \ldots c_{i-1} = 0$.  This means that $M$ is triangular.  If, in addition, $b_i = 1$, then $c_i = 1$ as well, and this means that all diagonal entries of $M$ are $1$.  Therefore, $\det M = 1$, which implies $M^{-1}$ is an integer matrix.  Now since $f\in\ints[x]$, we have $f(x) = \sum_{k=0}^n c_k x^k$ for integers $c_0,\ldots,c_n$.  If follows that $f^{(n)}_0,\ldots,f^{(n)}_n$ are all integers, since $(f^{(n)}_0,\ldots,f^{(n)}_n) = \p^{-1}(c_0,\ldots,c_n)$.
\end{proof}

\begin{proposition}\label{prop:degree-2}\ 
\begin{enumerate}
\item
There are exactly four elements of $Q_{[x]}$ of degree $\le 1$, namely,
\begin{align*}
P_0 &:= 0\;, & P_1 &:= 1\;, & P_2 &:= x\;, & P_3 &:= 1-x\;.
\end{align*}
\item
There are exactly ten elements of $Q_{[x]}$ of degree $2$, namely,
\begin{align*}
P_4 &:= x^2\;, & P_5 &:= -x^2+1\;, \\
P_6 &:= -x^2+2x\;, & P_7 &:= x^2-2x+1\;, \\
P_8 &:= -x^2+x\;, & P_9 &:= x^2-x+1\;, \\
P_{10} &:= -2x^2+2x\;, & P_{11} &:= 2x^2-2x+1\;, \\
P_{12} &:= -3x^2+3x\;, & P_{13} &:= 3x^2-3x+1\;.
\end{align*}
\end{enumerate}
\end{proposition}

\begin{proof}
For (1.), we note that these are the only four polynomials $P\in\ints[x]$ of degree $\le 1$ satisfying the conclusion of Theorem~\ref{thm:Q-x}.

Any polynomial $P$ of degree $\le 2$ is uniquely determined by its values on three distinct inputs.  We consider $P(0)$, $P(1/2)$, and $P(1)$.  If, in addition, $P\in Q_{[x]}$ and is nonconstant, then by Theorem~\ref{thm:Q-x}, we have: (i) $P(0) \in \{0,1\}$; (ii) $P(1) \in \{0,1\}$; and (iii) $0 < P(1/2) < 1$.  Since $P\in\ints[x]$, \ $P(1/2)$ is a multiple of $1/4$, and thus (iii) implies $P(1/2) \in \{1/4,1/2,3/4\}$.  Taking all possible combinations, there are then at most $2\cdot 2\cdot 3 = 12$ many $P\in Q_{[x]}$ with degree $1$ or $2$ satisfying (i), (ii), and (iii).  Two of these have degree $1$ ($P_2$ and $P_3$, above).  The other ten have degree $2$ and are listed above as $P_4,\ldots,P_{13}$.  We verify that they are all in $Q_{[x]}$ by giving explicit derivations:
\begin{align*}
   P_4 &= (1-x)P_0 + xP_2\;, &    P_5 &= (1-x)P_1 + xP_3\;, \\
   P_6 &= (1-x)P_2 + xP_1\;, &    P_7 &= (1-x)P_3 + xP_0\;, \\
   P_8 &= (1-x)P_0 + xP_3\;, &    P_9 &= (1-x)P_1 + xP_2\;, \\
P_{10} &= (1-x)P_2 + xP_3\;, & P_{11} &= (1-x)P_3 + xP_2\;, \\
P_{12} &= (1-x)P_6 + xP_5\;, & P_{13} &= (1-x)P_7 + xP_4\;.
\end{align*}
\end{proof}

We can use the same technique to get finite upper bounds on the number of elements of $Q_{[x]}$ with any given degree.  If the degree is at least four, then slightly better bounds can be obtained by using a classic theorem of Chebyshev \cite{Chebyshev:oeuvresI} (see \cite[Chapter~21]{AZ:proofs-from-the-book4}) to bound the leading coefficient by $4^{n-1}$ in absolute value.  We also can eliminate some polynomials from $Q_{[x]}$ using the following fact:

\begin{fact}\label{fact:critical-points}
Let $P\in\reals[x]$ be any real polynomial such that $\{P(0),P(1)\} \subseteq \{0,1\}$.  Then $0<P(\lambda)<1$ for all $0<\lambda<1$ if and only if $0<P(r)<1$ for every root $r$ of $P'$ (the derivative of $P$) such that $0<r<1$.
\end{fact}

\newpage

\begin{center}
\noindent{\Large\bf Part II: Aperiodic Order}\addcontentsline{toc}{part}{Part II: Aperiodic Order}
\end{center}

\section{$R_\lambda$ for some algebraic integers $\lambda$}
\label{sec:alg-int}

In this section, we prove a result (Theorem~\ref{thm:main}, below) that gives some sufficient conditions for $R_\lambda$ to be discrete for certain algebraic integers $\lambda$, including some (e.g., $1+\p$) not belonging to any discrete subring of $\complexes$.  In fact, all cases we currently know of where $R_\lambda$ is discrete follow from Theorem~\ref{thm:main}.

A \emph{Pisot-Vijayaraghavan number} (or \emph{PV number} for short) is an algebraic integer $\alpha > 1$ whose Galois conjugates $\alpha'$ (other than $\alpha$) all lie inside the unit disk in $\complexes$, i.e., satisfy $|\alpha'| < 1$.  This notion can be relaxed to allow for non-real $\alpha$ by excluding both $\alpha$ and its complex conjugate $\alpha^*$ from the norm requirement.  We need a stronger definition.

\begin{definition}\label{def:strong-PV-number}
We call an algebraic integer $\alpha\in\complexes$ a \emph{strong PV number} iff its Galois conjugates, other than $\alpha$ and $\alpha^*$, all lie in the unit interval $\opop{0,1}$.  In this case, we also say that \emph{$\alpha$ is sPV\@}.  We say that a strong PV number $\alpha$ is \emph{trivial} if it has no conjugates other than $\alpha$ and $\alpha^*$ (i.e., no conjugates in $\opop{0,1}$).  Otherwise, $\alpha$ is \emph{nontrivial}.
\end{definition}

Nontrivial strong PV numbers include $1+\p$ and $2+\sqrt 2$, and there are infinitely many real, irrational---hence nontrivial---strong PV numbers (Corollary~\ref{cor:degree2}, below).  Every strong PV number greater than $1$ is a PV number, but not conversely; for example, $\p$ and $1+\sqrt 2$ are PV numbers but not sPV\@.  Theorem~\ref{thm:main} implies that $R_\lambda$ is discrete for all strong PV numbers $\lambda$.  This result does not extend to all PV numbers; for example, $R_\p = R_{1+\sqrt 2} = \reals$.

\begin{fact}\label{fact:spv-basic}
If $\alpha$ is sPV, then so are $\alpha^*$ and $1-\alpha$; furthermore, $\alpha \notin\opop{0,1}$.
\end{fact}

The trivial strong PV numbers come in two types:

\begin{fact}\label{fact:trivial-spv}
By Lemma~\ref{lem:discrete-subring}, every trivial sPV number is of one of the following two types:
\begin{enumerate}
\item
The strong PV numbers of degree $1$ coincide with the integers.
\item
All non-real algebraic integers of degree $2$ are (trivial) sPV\@.  These numbers coincide with the non-real members of discrete subrings of $\complexes$ and are of the form either $a+ib\sqrt n$ or $a+b(1+i\sqrt{4n-1})/2$, for integers $a,b,n$ with $b\ne 0$ and $n>0$.
\end{enumerate}
\end{fact}

Pinch proved that if $\lambda$ is a real sPV number, then $R_\lambda$ is discrete~\cite[Proposition~12]{Pinch}.

\begin{proposition}[Pinch]\label{prop:pinch-discrete}
If $\lambda$ is a real sPV number, then $R_\lambda$ is discrete.
\end{proposition}

His proof idea generalizes to complex $\lambda$ is a straightforward way.

\begin{proposition}\label{prop:complex-Pinch}
If $\lambda$ is a nonreal sPV number, then $R_\lambda$ is discrete.
\end{proposition}

\begin{proof}
Suppose $R_\lambda$ is not discrete.  Then $R_\lambda = \complexes$ by Corollary~\ref{cor:general-equivalences} and Theorem~\ref{thm:convex-is-easy}.  Since $Q_\lambda$ is dense in $R_\lambda$, there is a point $z\in Q_\lambda$ such that $0<|z|<1$.  Let polynomial $p\in Q_{[x]}$ be such that $p(\lambda) = z$ (Lemma~\ref{lem:poly-eval}).  Since $\lambda$ is an algebraic integer, so is $z$, and it follows that $z$ has integer norm in the algebraic sense---that is, $N(z)\in\ints$, where $N(z)$ is the product of the conjugates of $z$.  By Corollary~\ref{cor:norm}, there is a positive integer $m$ such that $\prod_\nu p(\nu) = N(z)^m$, where $\nu$ runs over the conjugates of $\lambda$.  These conjugates include $\lambda$ itself and its complex conjugate $\lambda^*$, and the rest of the conjugates (if any) all satisfy $0 < \nu < 1$.  We thus have
\[ N(z)^m = \prod_\nu p(\nu) = p(\lambda)p(\lambda^*)\prod_{\nu\,:\, 0<\nu<1} p(\nu) = |z|^2\prod_{\nu\,:\,0<\nu<1} p(\nu)\;. \]
We have $0 < p(\nu) < 1$ if $0 < \nu < 1$ (Theorem~\ref{thm:Q-x}), and thus $0 < N(z)^m < 1$.
But $N(z)^m$ is an integer.  Contradiction.
\end{proof}


Theorem~\ref{thm:main}, below, substantially strengthens Proposition~\ref{prop:complex-Pinch}, and leads to sharper results about the nature of $R_\lambda$, particularly, the cut-and-project schemes of Meyer (see Sections~\ref{sec:meyer-sets} and \ref{sec:cut-and-project-connection} below).  Its proof was found independently of Pinch's paper~\cite{Pinch} as were the connections to cut-and-project schemes explored in \cite{BeMo} for one specific $\lambda$.  Propositions~\ref{prop:pinch-discrete} and \ref{prop:complex-Pinch}, however, can be combined with results on affine embedding to give a short proof of Theorem~\ref{thm:main}.  (See Theorems~\ref{thm:sub-Meyer} and \ref{thm:sub-meyer-spv} in Section~\ref{sec:concise-proof}.)  The short proof has the disadvantage of not connecting with cut-and-project schemes at all, and for this reason we include the longer, more informative proof of Theorem~\ref{thm:main} in Section~\ref{sec:main}, below, which includes concepts used elsewhere in the paper.

\subsection{Meyer sets and cut-and-project schemes}
\label{sec:meyer-sets}

Before stating and proving Theorem~\ref{thm:main}, we recall some concepts from discrete geometry, particularly concepts relating to ordered but aperiodic point sets in Euclidean space.  The theory of such sets has gained intense interest recently, following the discovery of so-called ``quasicrystals''---materials whose atomic arrangement shares many properties of crystals (e.g., sharp spikes in X-ray diffraction patterns), but---unlike with true crystals---lacks translational symmetry.  Much of the mathematical theory of the corresponding point sets is due to Meyer~\cite{Meyer:sets72,Meyer:quasicrystals}, and it also relates closely to aperiodic tilings of the plane \cite{Penrose:tiling,deBruijn:Penrose-tilings,GS:tilings}.  This section draws somewhat from the recent exposition of Baake \& Grimm \cite{BG:aperiodic-order} as well as papers by Moody \cite{Moody:meyer-sets}.  Also consulted are some related papers by Lagarias \cite{Lagarias:meyer,Lagarias:quasicrystals}.  We do not need to present the concepts in their full generality.

\begin{definition}
Let $X$ be any metric space, and let $S$ be any subset of $X$ with at least two elements.
\begin{itemize}
\item
$S$ is \emph{uniformly discrete} iff there exists an $r>0$ such that $B_r(x) \intersect B_r(y) = \emptyset$ for all distinct $x,y\in S$.  The supremum of the set of such $r$ is the \emph{packing radius} of $S$, denoted $\packrad(S)$.
\item
$S$ is \emph{relatively dense} (in $X$) iff there exists an $R>0$ such that $\bigcup_{x\in S} B_R(x) = X$.  The infimum of the set of such $R$ is the \emph{covering radius} of $S$ (in $X$), denoted $\coverrad(S)$.
\item
$S$ is a \emph{Delone set} (in $X$) iff it is both uniformly discrete and relatively dense (in $X$).
\end{itemize}
\end{definition}

The packing radius is also $1/2$ times the infimum of interpoint distances in $S$.  Thus if $R_\lambda$ is discrete, then it has packing radius $1/2$ by Corollary~\ref{cor:general-equivalences}.  We will use the following lemma a couple of times in the proof of Theorem~\ref{thm:main}.  Stronger statements are possible, but this is all we need.

\begin{lemma}\label{lem:uniformly-discrete-cap-bounded}
If $A\subseteq\reals^n$ is uniformly discrete and $B\subseteq\reals ^n$ is bounded, then $A\cap B$ is finite.
\end{lemma}

\begin{proof}
Let $r>0$ be the packing radius of $A$, and let $R>0$ be such that $B$ is included in some closed ball of radius $R$.  Then
\[ |A\cap B| \le \left(1+\frac{R}{r}\right)^n\;, \]
because the open balls of radius $r$ centered at the points of $A\cap B$ are pairwise disjoint subsets of some fixed ball of radius $r+R$.
\end{proof}

The next definition is not Meyer's original definition, but was shown equivalent to it by Jeffrey Lagarias~\cite{Lagarias:meyer}.

\begin{definition}\label{def:Meyer-set}
A subset $S\subseteq\reals^d$ is a \emph{Meyer~set} iff $S$ is relatively dense in $\reals^d$ and $S-S$ is uniformly discrete.
\end{definition}

There are many equivalent characterizations of Meyer sets; see, for example, Moody~\cite[Theorem~9.1]{Moody:meyer-sets}.  All Meyer sets are Delone sets, but Meyer sets have additional properties, such as finite local complexity, not shared by all Delone sets.  Other properties of interest include repetivity, diffractivity, and aperiodicity.  A Meyer set may or may not possess any of these additional properties.  For definitions, see \cite{Moody:meyer-sets,BG:aperiodic-order}.

One way of producing aperiodic Meyer sets is through a \emph{cut-and-project scheme}.  Our definition here follows Moody \cite{Moody:meyer-sets} with some minor alterations.
Recall that if $\map{f}{X}{Y}$ is some function with domain $X$, and $A$ is any subset of $X$, then we let $f|A$ denote the function $f$ restricted to domain $A$.

\begin{definition}\label{def:cut-and-project-scheme}
A \emph{cut-and-project scheme} is a tuple $(G,\reals^n,L)$, where
\begin{itemize}
\item
$G$ is a locally compact abelian (topological) group,
\item
$n$ is a positive integer,
\item
$L$ is a discrete subgroup of $G\times\reals^n$ such that the quotient group $(G\times\reals^n)/L$ is compact,
\item
letting $\map{\pi_1}{G\times\reals^n}{G}$ and $\map{\pi_2}{G\times\reals^n}{\reals^n}$ be the two canonical projection maps,
\begin{enumerate}
\item
$\pi_1(L)$ is dense in $G$, and
\item
$\pi_2|L$ is one-to-one.
\end{enumerate}
\end{itemize}
$G$ is called the \emph{internal space} and $\reals^n$ the \emph{physical space} of the cut-and-project scheme.
\end{definition}

\begin{definition}\label{def:model-set}
Fix a cut-and-project scheme $\cM = (G,\reals^n,L)$.  A \emph{window} of $\cM$ is any relatively compact\footnote{A subspace $Y$ of a topological space $X$ is \emph{relatively compact} iff the closure of $Y$ is compact.} subset $W\subseteq G$ with nonempty interior.  A set $M\subseteq\reals^n$ is a \emph{model set} (of $\cM$) iff $M = \pi_2(L\intersect\pi_1^{-1}(W))$ for some window $W$.
\end{definition}

All model sets are Delone sets.  Model sets can also be used to characterize Meyer sets.  Citing work of Meyer~\cite{Meyer:sets72} and Lagarias~\cite{Lagarias:meyer} as well as his own work, Moody~\cite{Moody:meyer-sets} gives various characterizations of Meyer sets, two of which are the following:

\begin{fact}[Meyer, Moody~\cite{Moody:meyer-sets}]\label{fact:meyer-set-characterization}
Let $S$ be a relatively dense subset of $\reals^n$.  Then $S$ is a Meyer set if and only if one (or both) of the following equivalent conditions holds:
\begin{enumerate}
\item
$S$ is a subset of a model set.
\item
$S-S$ is uniformly discrete.
\end{enumerate}
\end{fact}

In this paper, $G$ will always be $\reals^m$ (for some positive $m$) with the usual vector addition, and $L$ will always be an integer lattice spanning $\reals^m\times\reals^n \cong \reals^{m+n}$.

\begin{remark}
Another means of generating Delone and Meyer sets is via substitution and inflation (see Lagarias \& Wang~\cite{LW:substitution-delone-sets}).
\end{remark}

We end this subsection with a standard result about cut-and-project schemes used to show aperiodicity of the corresponding model sets.


\begin{lemma}\label{lem:no-ap}
Let $\cM = (\reals^m,\reals^n,L)$ be a cut-and-project scheme with internal space $\reals^m$ and associated projectors $\pi_1$ and $\pi_2$ as in Definition~\ref{def:cut-and-project-scheme}.  If $\pi_1|L$ is one-to-one, then no model set of $\cM$ intersects any arithmetic progression at infinitely many points.
\end{lemma}


\begin{proof}
Let $M$ be a model set with corresponding window $W\subseteq\reals^m$ as in Definition~\ref{def:model-set}.  Let $A := \tuple{a_0, a_1, a_2, \ldots}$ be an infinite arithmetic progression in $\reals^n$ that intersects $\pi_2(L)$ in at least two distinct points, say $a_i$ and $a_j$ for some least $i<j$.  Since $\pi_2(L)$ is an additive subgroup of $\reals^n$, we get that $A\intersect\pi_2(L)$ is an arithmetic subprogression of $A$.  Letting $B := \tuple{b_0,b_1,b_2,\ldots}$ be this subprogression, we see that $b_k = a_{i+k(j-i)} = a_i+k(a_j-a_i)$ for $k=0,1,2,\ldots\,$.  Furthermore, $M\intersect A = M\intersect B$, because $M\subseteq\pi_2(L)$.

Let $\bv b_0,\bv b_1,\bv b_2,\ldots \in L$ be the unique vectors such that $\pi_2(\bv b_k) = b_k$ for $k=0,1,2,\ldots\,$.  For any such $k$ we have by linearity
\[ \pi_2(\bv b_0 + k(\bv b_1 - \bv b_0)) = b_0 + k(b_1 - b_0) = b_k = \pi_2(\bv b_k)\;, \]
and so $\bv b_k = \bv b_0 + k(\bv b_1 - \bv b_0)$ for all $k$, making $\tuple{\bv b_0, \bv b_1,\ldots}$ a proper arithmetic progression in $L$.  Define $c_k := \pi_1(\bv b_k) \in \reals^m$ for $k = 0,1,2,\ldots\,$.  Again, this time by linearity of $\pi_1$, we have $c_k = c_0 + k(c_1 - c_0)$ for all $k$.  Moreover, $c_0 \ne c_1$ because $\pi_1|L$ is one-to-one, making $\tuple{c_0,c_1,c_2,\ldots}$ a proper arithmetic progression.  Since $W$ is bounded, for all but finitely many $k$ we have $c_k \notin W$, which puts $\bv b_k$ out of $\pi_1^{-1}(W)$, and so puts $b_k$ out of $M$.  Thus $M$ intersects $B$ (hence $A$) at only finitely many points.
\end{proof}

We explore the connection between cut-and-project schemes and $\lambda$-convex sets in Section~\ref{sec:cut-and-project-connection}, below.  As a warm-up, we have the following proposition, which is an easy corollary to Corollary~\ref{cor:Q-lambda-characterization}:

\begin{proposition}\label{prop:Q-lambda-cut-and-project}
For any algebraic integer $\lambda$ of degree $d>0$,
\[ Q_\lambda \subseteq \{0,1\} \union \{ p(\lambda) \mid p\in\ints[x] \myand \deg(p) < d \myand \mbox{$0<p(\mu)<1$ for all $0<\mu<1$ conjugate to $\lambda$} \}\;. \]
\end{proposition}

\begin{proof}
Let $m\in\ints[x]$ be the minimal (monic) polynomial of $\lambda$ of degree $d$.  Corollary~\ref{cor:Q-lambda-characterization} says that $Q_\lambda = \{0,1\} \union \{ p(\lambda) \mid p\in\ints[x] \myand \mbox{$0<p(\mu)<1$ for all $0<\mu<1$} \}$.  Since $m$ is monic, we can write any $p\in\ints[x]$ as $p = qm+r$ for some $q,r\in\ints[x]$ with $\deg(r) < d$.  Moreover, if $\mu$ is conjugate to $\lambda$, then $p(\mu) = r(\mu)$, which means that we can restrict the degree of the polynomial in the set-former to be less than $d$ by substituting $r$ for $p$.
\end{proof}

We will see below that if $\lambda$ is an sPV number, then the right-hand side of Proposition~\ref{prop:Q-lambda-cut-and-project} is discrete.  In fact, it is the model set of a cut-and-project scheme.

\subsection{Main results}
\label{sec:main}

Proposition~\ref{prop:pinch-discrete}~\cite{Pinch} is a special case of the following theorem, the main theorem of this section:

\begin{theorem}\label{thm:main}
If $\lambda$ is sPV, then $R_\lambda(S)$ is uniformly discrete for any finite set $S\subseteq \rats(\lambda)$.
\end{theorem}

Let $D$ be a discrete subring of $\complexes$.  $D$ is closed by Lemma~\ref{lem:discrete-ring}.  If $\lambda$ belongs to $D$, then $R_\lambda$ is discrete for an ``easy'' reason: $R_\lambda\subseteq D$.  Proposition~\ref{prop:1-plus-phi} gave us our first example of a discrete $R_\lambda$ that is \emph{not} contained in any discrete subring of $\complexes$.  Theorem~\ref{thm:main} will give us infinitely many other examples, i.e., values of $\lambda$ such that $R_\lambda$ is discrete but $\ints[\lambda]$ is not.

To prove Theorem~\ref{thm:main}, we build up some more machinery using some results of linear algebra---especially the spectral decomposition of an operator.

Some of our proof technique resembles work on generalized Fibonacci sequences done by Kalman \cite{Kalman:generalized-fib}, who found a closed form for the $\ordth{n}$ element $a_n$ of a sequence satisfying the $\ordth{k}$ order linear recurrence $a_n = c_1 a_{n-1} + \cdots + c_k a_{n-k}$, where the $c_i$ are fixed integers.
Both proofs use the spectral decomposition of the companion matrix of an irreducible polynomial.

Throughout this section, for convenience, we start the indexing of vectors and matrices with $0$ instead of with $1$.  If $E$ is some expression of matrix type, we let $[E]_{ij}$ denote the $\ordth{(i,j)}$ entry of $E$, for any appropriate integers $i,j\ge 0$.  We let an expression of the form $(x_0,x_1,x_2,\ldots)$ denote either a row vector or a column vector, depending on the context.

We first extend our definition of $\xl$ and $\rho_{a,b}$ to a more general setting.  As in previous sections, we use $\x$ without a subscript to denote $\xl$.

\begin{definition}\label{def:E-rho-general}\rm
Let $R$ be a commutative ring and let $M$ be an $R$-module.  For any $\lambda\in R$ and $u,v\in M$, we define $u\x v$ and $\rho_{u,v}(\lambda)$ to be the point $(1-\lambda)u + \lambda v \in M$.\footnote{Since $u\x v$ is an affine combination of $u$ and $v$, the operation $\x$ is well-defined on any $R$-affine space.  We will not need this additional generality here, however.}

A set $S\subseteq M$ is \emph{$\lambda$-convex} iff $S$ is closed under the $\x$ operation.  For any $T\subseteq M$, we define $Q_\lambda(T)$ to be the least $\lambda$-convex superset of $T$.
\end{definition}

Part~I of this paper is mostly concerned with the case where $R = M = \complexes$, except in Sections~\ref{sec:c-is-open} and \ref{sec:Q-x}, where $R = M = \ints[x]$.  Many of the basic results of that part carry over to the more general setting.  For example, Lemma~\ref{lem:translate-general} below is analogous to Lemma~\ref{lem:translate-Q-R}.

\begin{fact}\label{fact:E-homomorphism}
Let $R$ be a commutative ring, let $M$ and $N$ be $R$-modules, and let $\map{f}{M}{N}$ be a homomorphism of $R$-modules.  Then for any $\lambda\in R$ and $u,v\in M$,
\[ f(u\x v) = f(u)\x f(v)\;. \]
\end{fact}

\begin{lemma}\label{lem:translate-general}
Let $R$, $M$, $N$, and $f$ be as in Fact~\ref{fact:E-homomorphism}.  Then for any $\lambda\in R$ and $S\subseteq M$,
\[ f(Q_\lambda(S)) = Q_\lambda(f(S))\;. \]
\end{lemma}

\begin{proof}
Given any $x,y\in f(Q_\lambda(S))$, let $u,v\in Q_\lambda(S)$ be such that $x = f(u)$ and $y = f(v)$.  Then $x\x y = f(u)\x f(v) = f(u\x v) \in f(Q_\lambda(S))$.  This shows that $f(Q_\lambda(S))$ is $\lambda$-convex, and, since $f(S)\subseteq f(Q_\lambda(S))$, this proves the $\supseteq$-inclusion of the lemma.  For the $\subseteq$-inclusion, we first show that $T := f^{-1}(Q_\lambda(f(S)))$ is a $\lambda$-convex superset of $S$.  We have $S\subseteq f^{-1}(f(S)) \subseteq T$.  Now let $u,v \in T$ be arbitrary.  Then  because $f(u)$ and $f(v)$ are both in $Q_\lambda(f(S))$, we have $f(u\x v) = f(u)\x f(v) \in Q_\lambda(Q_\lambda(f(S))) = Q_\lambda(f(S))$.  This then puts $u\x v$ into $T$, and so $T$ is $\lambda$-convex.  By minimality, $Q_\lambda(S) \subseteq T$.  Applying $f$ to both sides gives $f(Q_\lambda(S)) \subseteq f(T) \subseteq Q_\lambda(f(S))$.
\end{proof}

We will need to consider another special case of Definition~\ref{def:E-rho-general}: Let $V$ be any vector space over some field $k$, and fix a $k$-linear map $\map{\Lambda}{V}{V}$.  Then $\Lambda$ naturally turns $V$ into a $k[x]$-module, where scalar multiplication is defined for all $g\in k[x]$ and $v\in V$ as $gv := (g(\Lambda))(v)$  (see, for example, Jacobson \cite[Chapter~3]{Jacobson:BasicAlgebraI}).  We denote this $k[x]$-module by $V_\Lambda$.  We will only use the case where $\lambda = x\in k[x]$, and thus we can identify $\lambda$ with $\Lambda$.

\begin{definition}\label{def:rho-extend}
Let $V$ be a vector space over a field $k$, and let $\map{\Lambda}{V}{V}$ be a $k$-linear map.
\begin{enumerate}
\item
For any vectors $u,v\in V$, define $u\xL v := (I-\Lambda)u + \Lambda v = u + \Lambda(v - u)$.
\item
A set $S\subseteq V$ is \emph{$\Lambda$-convex} iff $u\xL v \in S$ whenever $u,v\in S$.
\item
For any set of vectors $S\subseteq V$, we define $Q_\Lambda(S)\subseteq V$ as the least $\Lambda$-convex superset of $S$.
\end{enumerate}
\end{definition}

We could easily define $\Lambda$-clonvexity and $R_\Lambda(S)$ for vector spaces over $\reals$ or $\complexes$, but we will not need this notion here.

\begin{fact}\label{fact:rho-sum}
Let $V$, $k$, and $\Lambda$ be as in Definition~\ref{def:rho-extend}.  For any $u,v,w,x\in V$ and any $a\in k$, we have
\begin{align*}
(u+v)\xL(w+x) = u\xL w + v\xL x && \mbox{ and } && au\xL av = a(u\xL v)\;.
\end{align*}
\end{fact}

\begin{lemma}\label{lem:rho-map}
Let $U$ and $V$ be vector spaces over some field $k$, and let $\map{A}{U}{U}$, $\map{B}{V}{V}$, and $\map{t}{U}{V}$ be $k$-linear maps such that $B \circ t = t \circ A$.  For any $S\subseteq U$, we have
\[ t(Q_A(S)) = Q_B(t(S))\;. \]
\end{lemma}

\begin{proof}
This is a special case of Lemma~\ref{lem:translate-general}.  The combined condition that $t$ is linear and $B \circ t = t \circ A$ is equivalent to $t$ being a homomorphism of $k[x]$-modules $U_A\rightarrow V_B$.
\end{proof}

We use Lemma~\ref{lem:translate-general} one more time to get a generalization of Lemma~\ref{lem:poly-eval}.

\begin{lemma}\label{lem:1-d-module}
Let $R$ be a commutative ring, $\lambda\in R$, and $M$ an $R$-module.  For any $u\in M$,
\[ Q_\lambda(\{0,u\}) = \left\{ P(\lambda)u \mid P\in Q_{[x]} \right\}\;, \]
where $P(\lambda)$ is interpreted in $R$ for each $P\in Q_{[x]}$ (see Definition~\ref{def:Q-x}).
\end{lemma}

\begin{proof}
The lemma follows from Lemma~\ref{lem:translate-general} provided we set things up the right way.  We can view $M$ as a $\ints[x]$-module (which we denote by $M_\lambda$) by defining scalar multiplication of a vector $w\in M$ with a scalar $P\in\ints[x]$ by $Pw := P(\lambda)w$, where $P(\lambda)$ is interpreted in $R$ as usual and $P(\lambda)w$ is scalar multiplication in $M$.  Then $\xl$, interpreted in $M$, is the same map as $\xx$ interpreted in $M_\lambda$.  Consequently, for any $S\subseteq M$, \ $Q_\lambda(S)$ interpreted in $M$ is the same set as $Q_x(S)$ interpreted in $M_\lambda$.  Now given $u\in M$, consider the map $\map{g}{\ints[x]}{M_\lambda}$ that sends $P\in\ints[x]$ to $Pu\in M_\lambda$.  One checks that $g$ is a homomorphism of $\ints[x]$-modules.  Using Lemma~\ref{lem:translate-general}, we now have, working in $M_\lambda$,
\[ Q_x(\{0,u\}) = Q_x(g(\{0,1\})) = g(Q_x(\{0,1\})) = g(Q_{[x]}) = \{ g(P) \mid P\in Q_{[x]} \} = \{ Pu \mid P\in Q_{[x]}\}\;. \]
The lemma follows by interpreting both sides in $M$.
\end{proof}

The first extended proof of this section considers the case where $k = \reals$, \ $V = \reals^d$ for some finite $d>0$, and $\map{\Lambda}{V}{V}$ is an $\reals$-linear map given by some \emph{diagonal} matrix.  It gives a good illustration of the general concepts above.  For $0\le i<d$, let $\mu_i := [\Lambda]_{ii}$ be the diagonal entries of $\Lambda$.  So on the $\ordth{i}$ coordinate, $\Lambda$ acts as scalar multiplication by $\mu_i$.  One can consider Lemma~\ref{lem:cube} below to be a $d$-dimensional generalization of Theorem~\ref{thm:convex-is-easy}(1).

\begin{lemma}\label{lem:cube}
Let $d$, $\Lambda$, and $\mu_0,\ldots,\mu_{d-1}$ be as in the last paragraph, and suppose that all of the $\mu_i$ are in $\clcl{0,1}$.  Let $\one = (1,1,\ldots,1)$ be the vector of $d$ ones, and let $S \eqdf \{0,\one\}$.  Then
\begin{enumerate}
\item
$Q_\Lambda(S) \subseteq \clcl{0,1}^{\times d}$, and
\item
$Q_\Lambda(S)$ is dense in $\clcl{0,1}^{\times d}$ if and only if all the $\mu_i$ are pairwise distinct and lie strictly between $0$ and $1$.
\end{enumerate}
\end{lemma}

\begin{proof}
Let $\map{\pi_i}{\reals^d}{\reals}$ be the projection onto the $\ordth{i}$ coordinate, for all $0\le i<d$.  For part~(1.): by Lemma~\ref{lem:rho-map}, for all $i$, we have $\pi_i(Q_\Lambda(S)) = Q_{\mu_i}(\pi_i(S)) = Q_{\mu_i}$, which is contained in $\clcl{0,1}$ by Theorem~\ref{thm:convex-is-easy}(1).

For part~(2.):  By permuting coordinates if necessary, we can assume without loss of generality that $0\le \mu_0 \le \mu_1 \le \cdots \le \mu_{d-1}\le 1$.  We show the ``only if'' part first.  If $\mu_0 = 0$, then $\pi_0(Q_\Lambda(S)) = Q_0 = \{0,1\}$, and so every point in $Q_\Lambda(S)$ has zeroth coordinate $0$ or $1$.  Similarly for the $\ordst{(d-1)}$ coordinate if $\mu_{d-1} = 1$.  If $\mu_i = \mu_{i+1}$ for some $i$, then every point in $Q_\Lambda(S)$ has equal $\ordth{i}$ and $\ordst{(i+1)}$ coordinates.  This follows from the observation that both points in $S$ have equal $\ordth{i}$ and $\ordst{(i+1)}$ coordinates and that this property is preserved by $\xL$.  In none of these cases is $Q_\Lambda(S)$ dense in $\clcl{0,1}^{\times d}$.

For the ``if'' part, now assume that $0<\mu_0<\mu_1<\cdots<\mu_{d-1}<1$, and let $R \eqdf \overline{Q_\Lambda(S)}$.  We show that $R = \clcl{0,1}^{\times d}$.  By a proof similar to that of Lemma~\ref{lem:topo-closure} in Part~I, we have that $R$ is $\Lambda$-convex.  Our proof now is in two steps: we first show that $R$ is convex and then show that $R$ contains all the corners of $\clcl{0,1}^{\times d}$, i.e., all points in $\{0,1\}^{\times d}$.  This suffices for the ``if'' part.

For convexity, fix any two points $a,b\in R$.
Define $a_0 := a$ and $b_0 := b$, and for every integer $n\ge 0$, inductively define
\begin{align*}
a_{n+1} &\eqdf b_n\xL a_n\;, & b_{n+1} &\eqdf a_n\xL b_n\;.
\end{align*}
All these points are in $R$, because $R$ is $\Lambda$-convex.  Set  $m = (a+b)/2$, the midpoint of $a$ and $b$, and set $\sigma := \max_{0\le i<d}|2\mu_i-1|$, noting that $0\le\sigma<1$.  We then check by induction that for all $n\ge 0$,
\begin{align*}
a_n + b_n &= 2m \;, & \|a_n-b_n\|_\infty &\le \sigma^n \|a-b\|_\infty\;,
\end{align*}
where for all $x\in\reals^d$, \ $\|x\|_\infty \eqdf \max_{0\le i<d}|\pi_i(x)|$ denotes the $\ell_\infty$-norm of $x$.  This certainly holds for $n = 0$.  For the inductive case, given $n\ge 0$, we have
\[ a_{n+1} + b_{n+1} = b_n\xL a_n + a_n\xL b_n = (b_n+a_n)\xL(a_n+b_n) = a_n+b_n = 2m\;. \]
Also,
\[ a_{n+1}-b_{n+1} = b_n\xL a_n - a_n\xL b_n = (b_n-a_n)\xL(a_n-b_n)\;, \]
and thus for all $0\le i<d$, letting $\delta \eqdf a_n-b_n$,
\begin{align*}
|\pi_i(a_{n+1}-b_{n+1})| &= |\pi_i((-\delta)\xL\delta)| = |(-\pi_i(\delta))\xmi\pi_i(\delta))| = |(2\mu_i-1)\pi_i(\delta)| \le \sigma|\pi_i(\delta)| \le \sigma\|\delta\|_\infty \\
\le \sigma^{n+1}\|a-b\|_\infty\;,
\end{align*}
and from this it follows that $\|a_{n+1}-b_{n+1}\|_\infty \le \sigma^{n+1}\|a-b\|_\infty$, which finishes the induction.  Now for all $n$,
\[ 2\|a_n-m\|_\infty = \|2a_n-2m\|_\infty = \|a_n-b_n\|_\infty \le \sigma^n\|a-b\|_\infty\;, \]
and since $|\sigma| < 1$, we then have $m = \lim_{n\rightarrow\infty} a_n \in R$, because $R$ is closed.  All of this shows that $R$ is closed under midpoints, and it follows from the closedness of $R$ that $R$ is convex.

To finish, we show that every corner $\vec b = (b_0,b_1,\ldots,b_{d-1})$---where each $b_i \in \{0,1\}$---is in $R$.  The idea is to use the polynomials $T_{\gamma,\eps}$ of Lemma~\ref{lem:threshold} to construct a polynomial $P\in Q_{[x]}$ such that for all $i$, \ $P(\mu_i)$ approximates $b_i$ as closely as we want.  More precisely, for every $\delta > 0$, we find some $P_\delta\in Q_{[x]}$ such that for all $0\le i<d$, \ $|P_\delta(\mu_i) - b_i| \le \delta$.  Then we invoke Lemma~\ref{lem:1-d-module} to get a point in $Q_\Lambda(S)$ close to $\vec b$.  We set $\mu_{-1} \eqdf 0$ and let $\mu_d \eqdf 1$ for convenience.  For $0\le i \le d$, we let $\gamma_i \eqdf (\mu_{i-1}+\mu_i)/2$ be the midpoint of $\mu_{i-1}$ and $\mu_i$.  Notice that $0 < \gamma_0 < \mu_0 < \gamma_1 < \mu_1 < \cdots < \gamma_{d-1} < \mu_{d-1} < \gamma_d < 1$.  Now let $\delta > 0$ be arbitrary.  For a certain $\eps > 0$ that we choose later, we define
\[ P_\delta \eqdf \prod_{0\le i < d\;:\; b_i = 0}(1 - T_{\gamma_{i+1},\eps}(1-T_{\gamma_i,\eps}))\;, \]
where the $T_{\cdot,\cdot}$ polynomials are given by Lemma~\ref{lem:threshold}.  $P_\delta$ is in $Q_{[x]}$ by Corollary~\ref{cor:Q-x-mult}.  If we choose $\eps$ not to exceed $(\mu_i - \mu_{i-1})/2$ for any $0\le i\le d$, then no interval $\opop{\gamma_i-\eps,\gamma_i+\eps}$ contains any of the $\mu_j$.  This means that we can apply Lemma~\ref{lem:threshold} to $\mu_i$ for each $0\le i< d$: Since $\gamma_i \le \mu_i-\eps$ and $\mu_i+\eps \le \gamma_{i+1}$, we have $T_{\gamma_j,\eps}(\mu_i) \le \eps$ for all $j\le i$ and $T_{\gamma_j,\eps}(\mu_i) \ge 1-\eps$ for all $j>i$.  Hence, $1-T_{\gamma_{i+1},\eps}(\mu_i)(1-T_{\gamma_i,\eps}(\mu_i)) \le 1-(1-\eps)^2 \le 2\eps$, but $1-T_{\gamma_{j+1},\eps}(\mu_i)(1-T_{\gamma_j,\eps}(\mu_i)) \ge 1-\eps$ for all $j\ne i$.  Now considering $P_\delta(\mu_i)$:
\begin{itemize}
\item
If $b_i = 1$, then all terms in the product are $\ge 1-\eps$, and so $P_\delta(\mu_i) \ge (1-\eps)^d \ge 1- d\eps$.
\item
If $b_i = 0$, then one term in the product is $\le 2\eps$, and the rest are $\le 1$, and so $P_\delta(\mu_i) \le 2\eps$.
\end{itemize}
So choosing $\eps$ to be the smaller of $\delta/(d+2)$ and $\min_{0\le i <d}((\mu_i - \mu_{i-1})/2)$, we obtain $|P_\delta(\mu_i) - b_i| \le \delta$ for all $i$.

Finally, we apply Lemma~\ref{lem:1-d-module} with the ring being $\ints[x]$, with $\lambda \eqdf x\in\ints[x]$ as always, with $M \eqdf V_\Lambda = (\reals^d)_\Lambda$ (restricted to being a $\ints[x]$-module), and with $u \eqdf \one$.  That lemma says that $Q_\Lambda(S) = \{ P(\Lambda)\one \mid P\in Q_{[x]}\}$, and so $Q_\Lambda(S)$ contains the points $P_\delta(\Lambda)\one$ for all $\delta>0$.  Now, $P_\delta(\Lambda)$ is the $d\times d$ diagonal matrix with diagonal entries $P_\delta(\mu_0),\ldots,P_\delta(\mu_{d-1})$, and so $P_\delta(\Lambda)\one = (P_\delta(\mu_0),\ldots,P_\delta(\mu_{d-1}))$, which is within distance $\delta$ of $\vec b$ ($\ell_\infty$ distance) by the previous paragraph.  So we have that $\vec b$ is arbitarily close to points in $Q_\Lambda(S)$, which puts $\vec b \in R$ as required.
\end{proof}

\begin{corollary}\label{cor:cube}
Let $d$, $\Lambda$, and $\mu_0,\ldots,\mu_{d-1}$ be as in Lemma~\ref{lem:cube}.  Let $\bv{x} \eqdf (x_0,\ldots,x_{d-1})\in\reals^d$ be a vector.  If $\mu_0,\ldots,\mu_{d-1}$ are pairwise distinct elements of $\opop{0,1}$, then $Q_\Lambda(\{0,\bv{x}\})$ is a dense subset of $\clcl{0,x_0}\times\cdots\times\clcl{0,x_{d-1}}$.
\end{corollary}

\begin{proof}
Let $D$ be the $d\times d$ diagonal matrix with diagonal elements $x_0,\ldots,x_{d-1}$.  Then $\bv{x} = D\one$, and so by Lemma~\ref{lem:rho-map} and the fact that $\Lambda$ and $D$ commute, we get
\[ Q_\Lambda(\{0,\bv{x}\}) = Q_\Lambda(D\{0,\one\}) = D(Q_\Lambda\{0,\one\})\;, \]
which is dense in $D\left(\clcl{0,1}^d\right) = \clcl{0,x_0}\times\cdots\times\clcl{0,x_{d-1}}$ by Lemma~\ref{lem:cube}.
\end{proof}

\begin{definition}\label{def:lambda-mult}\rm
Let $p(x) = x^d + \sum_{j=0}^{d-1} c_jx^j$ be some monic polynomial in $\complexes[x]$ of degree $d>0$ with coefficients $c_0,\ldots,c_{d-1}\in\complexes$.  We let $\Lambda_p$ denote the companion matrix of $p$, that is, the $d\times d$ matrix
\begin{equation}\label{eqn:lambda-mult}
\Lambda_p := \mat{
0 & 0 & 0 & \cdots & 0 & -c_0 \\
1 & 0 & 0 & \cdots & 0 & -c_1 \\
0 & 1 & 0 & \cdots & 0 & -c_2 \\
0 & 0 & 1 & \cdots & 0 & -c_3 \\
\vdots & \vdots & \vdots & \ddots & \vdots & \vdots \\
0 & 0 & 0 & \cdots & 1 & -c_{d-1}}\;.
\end{equation}
\end{definition}

We can view $\Lambda_p$ as representing a linear map $\complexes^d \rightarrow \complexes^d$ relative to the standard basis $\{\bv{e}_0,\ldots,\bv{e}_{d-1}\}$ of $\complexes^d$.  It is well-known that $\Lambda_p$ has characteristic polynomial $\pm p$, and so its eigenvalues are the roots of $p$.

The next lemma is standard.

\begin{lemma}
\label{lem:lambda-spectral-decomp}
Let $p(x)$ and $\Lambda_p$ be as in Definition~\ref{def:lambda-mult} where $p$ has degree $d>0$.  Let $\mu_0,\ldots,\mu_{d-1}\in\complexes$ be the (not necessarily distinct) roots of $p$.  Let
\begin{equation}\label{eqn:vandermonde}
V := V(\vec\mu) := \mat{
1 & \mu_0 & (\mu_0)^2 & \cdots & (\mu_0)^{d-1} \\
1 & \mu_1 & (\mu_1)^2 & \cdots & (\mu_1)^{d-1} \\
\vdots & \vdots & \vdots & \ddots & \vdots \\
1 & \mu_{d-1} & (\mu_{d-1})^2 & \cdots & (\mu_{d-1})^{d-1}
}
\end{equation}
be the $d\times d$ Vandermonde matrix with respect to $\vec\mu := (\mu_0,\ldots,\mu_{d-1})$ (that is, $[V]_{ij} = (\mu_i)^j$ for all $i,j\in\{0,\ldots,d-1\}$), and let $D$ be the $d\times d$ diagonal matrix with diagonal entries $[D]_{ii} := \mu_i$ for all $i\in\{0,\ldots,d-1\}$.  Then
\[ V\Lambda_p = DV\;. \]
\end{lemma}

\begin{proof}
For any $i,k\in\{0,\ldots,d-1\}$, the $\ordth{(i,k)}$ entry of $V\Lambda_p$ is given by
\[ [V\Lambda_p]_{ik} = \sum_{j=0}^{d-1}[V]_{ij}[\Lambda_p]_{jk} = \sum_{j=0}^{d-1}(\mu_i)^j[\Lambda_p]_{jk} = \left\{
\begin{array}{ll}
(\mu_i)^{k+1} & \mbox{if $k<d-1$,} \\
-\sum_{j=0}^{d-1} c_j(\mu_i)^j & \mbox{if $k=d-1$.}
\end{array} \right. \]
Since $\mu_i$ is a root of $p$, we have $-\sum_{j=0}^{d-1} c_j(\mu_i)^j = (\mu_i)^d$, and so in either case we get $[V\Lambda_p]_{ik} = (\mu_i)^{k+1} = \mu_i(\mu_i)^k = [D]_{ii}[V]_{ik} = [DV]_{ik}$.
\end{proof}

\begin{remark}
If $V$ is invertible (which is true when all the $\mu_i$ are pairwise distinct), then Lemma~\ref{lem:lambda-spectral-decomp} says that the columns of $V^{-1}$ are eigenvectors of $\Lambda_p$ with respective eigenvalues $\mu_0,\ldots,\mu_{d-1}$.
\end{remark}

We now turn to the proof of Theorem~\ref{thm:main}, the main result of this section.  We essentially prove that $R_\lambda(S)$ is a subset of a certain model set (this is the content of
Eq.~(\ref{eqn:Q-in-model-set}) below).  Although it would then follow directly from previous results of Meyer~\cite{Meyer:sets72} (also see~\cite{Moody:meyer-sets}) that $R_\lambda(S)$ is uniformly discrete, we give a self-contained proof for the sake of completeness.  We will give more details about the connection with model sets after the proof.  The case where $\lambda$ is not real (and $S = \{0,1\}$) was proved by Rohit Gurjar~\cite{Gurjar:nonreal-lambda}.

To denote the model set containing $R_\lambda(S)$, we use the following
notation 
generalizing that of Berman \& Moody~\cite{BeMo}:

\begin{definition}\label{def:Sigma_P}
Let $\lambda\in\complexes\cmpl\opop{0,1}$ be an algebraic integer and let $\mu_0,\ldots,\mu_{k-1}$ be the conjugates of $\lambda$ that are in $\opop{0,1}$.
For any $x \in \ints[\lambda]$, let $\bv{x}'$ denote the vector of
conjugates $(x'_0,\dots,x'_{k-1}) \in \reals^k$, where $x'_i$ is the $\ordth{i}$ conjugate of $x$, i.e., the image of $x$ under the ring isomorphism $\ints[\lambda]\rightarrow\ints[\mu_i]$ sending $\lambda$ to $\mu_i$, for all $0\le i < k$.
Let $P$ be any interval of $\reals^k$, defined as $P := \clcl{\ell_0,h_0}\times\cdots\times\clcl{\ell_{k-1},h_{k-1}}$,
where for $0 \le i < k$, we have $\ell_i, h_i \in \reals$ and $\ell_i < h_i$. 
 We define
\begin{align}\label{eqn:Sigma_P}
\Sigma_P^{(\lambda)}
:= \{ x\in\ints[\lambda] \mid \bv{x}' \in P\}\;.
\end{align}
If $k=0$, then we take $\bv{x}'$ to be the empty tuple for every $x\in\ints[\lambda]$, and we take $P$ to be the singleton set containing this tuple, whence $\Sigma_P^{(\lambda)} = \ints[\lambda]$.  (We leave off the superscript if $\lambda$ is clear from the context.)
\end{definition}

 For any set $S \subseteq \ints[\lambda]$, let
$S' := \{\bv{s}' \mid s \in S\}$ denote the image of $S$ under
the mapping $\map{\cdot'}{\ints[\lambda]}{\ints[\mu_0]\times\cdots\times\ints[\mu_{k-1}]}$ defined above.


When we use Definition~\ref{def:Sigma_P}, we generally assume
$P$ contains some set $S'$ of conjugates of a finite set $S$.
Furthermore, if $S = \{0,1\}$, it turns out that
 $P = \clcl{0,1}^{\times k}$ is the appropriate choice. The
 reason for this
 should become clear as we use the definition and the result below.
Also, although here we are assuming $S \subseteq \ints[\lambda]$,
we will see in the proof of
Theorem~\ref{thm:main} that some choice for $P$ is sufficient for proving that
$R_\lambda(S)$ is uniformly discrete for any finite $S \subseteq \rats[\lambda]$.

A key result is that %
$Q_\lambda(S)$ is always contained in a set of the
form $\Sigma_P^{(\lambda)}$.

\begin{theorem}\label{thm:new-containment}
For any algebraic integer $\lambda \in\complexes\cmpl\clcl{0,1}$ and any finite $S \subseteq \ints[\lambda]$,
 let $P$ be any interval of $\reals^k$---as
 defined in Definition~\ref{def:Sigma_P}---that includes $S'$. Then
\begin{eqnarray}\label{eqn:Q-in-model-set}
Q_\lambda(S) \subseteq \Sigma_P^{(\lambda)}.
\end{eqnarray}
\end{theorem}
\begin{proof}
For convenience, we fix $\lambda$ and 
write $\Sigma_P$ for $\Sigma_P^{(\lambda)}$.
 First we observe that $S \subseteq \Sigma_P$, since by hypothesis
$S' \subseteq P$, and hence $s \in \Sigma_P$ for each $s \in S$.
We now show that $\Sigma_P$ is closed under $\x_\lambda$.
Let $x, y \in \Sigma_P$. Then the respective
vectors of conjugates $\bv{x}'$ and $\bf{y}'$ are contained in $P$. 
Now let $z := x \xl y$. The $\ordth{i}$ coordinate of $\bv z'$ is
$z'_i = (1-\mu_i)x'_i + \mu_i y'_i$. But by definition,
$\mu_i \in \clcl{0,1}$ for each $i$ between $0$ and $k-1$.
Thus due to the fact that $[\ell_i,h_i]$ is convex, we conclude that
$\ell_i \le (1-\mu_i)x'_i + \mu_i y'_i \le h_i$, so that $z'_i \in [\ell_i,h_i]$ for all $i$. Clearly, $z \in \ints[\lambda]$,
and hence $z \in \Sigma_P$,
so $\Sigma_P$ is closed under $\xl$ as claimed. It therefore contains
the closure of $S$ under $\xl$, i.e.,
$\Sigma_P \supseteq Q_\lambda(S)$. 
\end{proof}

The proof of Theorem~\ref{thm:main} now 
reduces to proving that $\Sigma_P^{(\lambda)}$ is uniformly discrete for sPV $\lambda$.

\begin{proof}[Proof of Theorem~\ref{thm:main}]
We are given a finite set $S\subseteq\rats(\lambda)$, which we can assume contains $0$ and $1$ (so that $Q_\lambda\subseteq Q_\lambda(S)$).  We can also assume without loss of generality that $S\subseteq\ints[\lambda]$.  This can be seen as follows: since $S$ is finite, there exists a positive $\delta\in\ints$ such that $\delta S \subseteq \ints[\lambda]$, but by Lemma~\ref{lem:translate-Q-R}, $R_\lambda(\delta S) = R_\lambda(\rho_{0,\delta}(S)) = \rho_{0,\delta}(R_\lambda(S)) = \delta R_\lambda(S)$, so $R_\lambda(S)$ is uniformly discrete if and only if $R_\lambda(\delta S)$ is uniformly discrete.  We can thus substitute $\delta S$ for $S$ in the theorem.

Now it suffices to show that $Q_\lambda(S)$ is uniformly discrete; in that case, $R_\lambda(S) = Q_\lambda(S)$ by Lemma~\ref{lem:topo-closure}, because $Q_\lambda(S)$ is closed.  We can assume without loss of generality that $\lambda$ is nontrivial (cf.\ Definition~\ref{def:strong-PV-number}), for if $\lambda$ is trivial, then $\ints[\lambda]$ is a (uniformly) discrete subring of $\complexes$ including $Q_\lambda(S)$ (by Fact~\ref{fact:trivial-spv}), and we are done. 

Let $p(x) := x^d + \sum_{i=0}^{d-1}c_ix^i$ be the minimal polynomial of $\lambda$, where $d\ge 2$ and each $c_i$ is an integer.  By Fact~\ref{fact:spv-basic} and $\lambda$ being nontrivial, we have $\lambda\notin\clcl{0,1}$.  Let $\mu_0,\ldots,\mu_{k-1}$ be the conjugates of $\lambda$ that are in $\clcl{0,1}$, as in Definition~\ref{def:Sigma_P}.  In fact, $\mu_i\in\opop{0,1}$ for all $0\le i<k$.  Since $\lambda$ is sPV, we have two cases:
(1) $\lambda\in\reals$, or (2) $\lambda\notin\reals$.  Since $\lambda$ is sPV, we have $k=d-1$ in case~1 and $k = d-2$ in case~2.  In either case, define $\mu_k := \lambda$.  In case~2, also set $\mu_{k+1} := \lambda^*$.  Thus $\mu_0,\ldots,\mu_{d-1}$ are all the conjugates of $\lambda$ in either case.
Since $\lambda\notin\clcl{0,1}$, \ $Q_\lambda(S)$ is unbounded (Corollary~\ref{cor:unbounded}), a fact we will need toward the end of
 this proof.\\

Each $x \in \ints[\lambda]$
is a polynomial $x = \sum_{i=0}^{d-1} a_i\lambda^i$, where $a_i \in \ints$ for each $i$, and $d$ is the degree of $\lambda$. We may represent
the set of coefficients $a_i \in \ints$ as a column vector $\bv{a} := (a_0,\dots,a_{d-1}) \in \ints^d$.
If we regard the powers of $\lambda$ as a row vector
$\pmb{\lambda} := (1,\lambda,\dots,\lambda^{d-1})$, then $x = {\pmb{\lambda}}\bv{a}$. Similarly, for the conjugates $\mu_i$ ($0 \le i < d-1$) of $\lambda$,
we define the row vector $\pmb{\mu}_i := (1,\mu_i,\dots,\mu_i^{d-1})$, and thus
$(\pmb{\lambda}\bv{a})' = ({\pmb{\mu}}_0\bv{a},\dots,{\pmb{\mu}}_{k-1}\bv{a})$, where we understand the last vector to be a column vector
in $\reals^k$.
As $S \subseteq \ints[\lambda]$, any element of
$S$ can be written in the same form ${\pmb{\lambda}}\bv{a}$ for some $\bv{a}\in\ints^d$. Let $\pi_i : \complexes^d \rightarrow \complexes$ denote the linear map projecting an element of $\complexes^d$ onto the 
$\ordth{i}$ coordinate. Thus, for example, $\pi_i((\pmb{\lambda}\bv{a})')
= {\pmb{\mu}}_{i}{\bf a}$. Let $S'_i$ denote the set of projections
of elements of $S'$ onto the $\ordth{i}$ coordinate: $S'_i := \{\pi_i(\bv{s}') \mid s \in S\}$. 
Put $\ell_i := \min (S'_i)$ and $h_i := \max (S'_i)$ and
$P := \clcl{\ell_0,h_0}\times\cdots\times\clcl{\ell_{k-1},h_{k-1}}$
as in Definition~\ref{def:Sigma_P} (note that if $S = \{0,1\}$,
then $P = \clcl{0,1}^{\times k})$. Then clearly $P \supseteq S'$,
and hence
by Theorem~\ref{thm:new-containment}, $Q_{\lambda}(S) \subseteq \Sigma_P$.  (This inclusion also holds when $\lambda$ is trivial, i.e., when $k=0$.)

\bigskip

Thus it now suffices to show that
$\Sigma_P$ is uniformly discrete. By the containment
$Q_{\lambda}(S) \subseteq \Sigma_P$ it then follows
that $Q_{\lambda}(S)$ is uniformly discrete.

\bigskip

 Our first goal is to put the set
$\Sigma_P 
= \{ x\in\ints[\lambda] \mid \bv{x}' \in P\}$ in a form 
(Eq.~(\ref{eq:SigmaP-final-form}) below) convenient for
proving its uniform discreteness. 
By the notation introduced above, $\Sigma_P = \{\pmb{\lambda}\bv{a} \mid \bv{a} \in \ints^d \myand (\pmb{\lambda}\bv{a})' \in  P \}$.
 Now recall
 the Vandermonde matrix $V = V(\mu_0,\ldots,\mu_{d-1})$ defined in 
Eq.~(\ref{eqn:vandermonde}). 
For any vector $\bv{v}\in\complexes^d$, let $\bv{v}_F$ denote the vector consisting of the $F$irst
$k$ elements of $\bv{v}$. Using this notation,
$(\pmb{\lambda}\bv{a})' = (V\bv{a})_F$. Then we have
\begin{equation}\label{eqn:Sigma-vector-form1}
\Sigma_P = \{\pmb{\lambda}\bv{a} \mid  \bv{a} \in \ints^d \myand (\pmb{\lambda}\bv{a})' \in P\} 
= \{\pmb{\lambda}\bv{a} \mid \bv{a} \in \ints^d
\myand (V\bv{a})_F \in P \}\;.
\end{equation}

The first $k$ rows of $V$ can be chopped up into two matrices: a $k\times k$ matrix $W$ consisting of the first $k$ rows and columns of $V$, followed by a $k\times(d-k)$ matrix $X$ consisting of the first $k$ rows and remaining columns of $V$.  (Both $W$ and $X$ are real matrices.)  Thus $V$ looks as follows when $\lambda\in\reals$ (case~1 where $k = d-1$) and $\lambda\notin\reals$ (case~2 where $k = d-2$), respectively:
\begin{align*}
V &= \left[\begin{array}{cccc|c}
& \multicolumn{2}{c}{\vdots} & & \vdots \\
\cdots & \multicolumn{2}{c}{W} & \cdots & X \\
& \multicolumn{2}{c}{\vdots} & & \vdots \\ \hline
1 & \lambda & \lambda^2 & \cdots & \lambda^{d-1}
\end{array}\right]\;, &
V &= \left[\begin{array}{cccc|cc}
&  \multicolumn{2}{c}{\vdots} & & \multicolumn{2}{c}{\vdots} \\
\cdots & \multicolumn{2}{c}{W} & \cdots & \multicolumn{2}{c}{X} \\
& \multicolumn{2}{c}{\vdots} & & \multicolumn{2}{c}{\vdots} \\ \hline
1 & \lambda & \lambda^2 & \cdots & \lambda^{d-2} & \lambda^{d-1} \\
1 & \lambda^* & (\lambda^*)^2 & \cdots & (\lambda^*)^{d-2} & (\lambda^*)^{d-1}
\end{array}\right]\;.
\end{align*}

Decompose the vector $\bv{a}$ as $(\bv{a}_F; \bv{a}_L)$, where
as above $\bv{a}_F$ denotes the first $k$ coordinates of $\bv{a}$,
and $\bv{a}_L$ denotes the $L$ast $d-k$ coordinates of $\bv{a}$.
Thus $(V\bv{a})_F = W\bv{a}_F + X\bv{a}_L$, 
and Eq.~(\ref{eqn:Sigma-vector-form1}) becomes
\begin{eqnarray*}
\Sigma_P  & = &\{\pmb{\lambda}\bv{a} \mid \bv{a}_F \in \ints^k \myand \bv{a}_L \in \ints^{d-k} \myand W\bv{a}_F + X\bv{a}_L \in 
P \} \\
         & = &\{\pmb{\lambda}\bv{a} \mid \bv{a}_F \in \ints^k
         \myand \bv{a}_L \in \ints^{d-k} \myand W\bv{a}_F \in P-X\bv{a}_L
          \}\;.
\end{eqnarray*}

Note that $W = V(\mu_0,\ldots,\mu_{k-1})$ is invertible and $W^{-1}$ is a real matrix. 
Hence, defining $\Omega = W^{-1}P$ and $Y = W^{-1}X$, we have, 
\begin{equation}\label{eqn:Sigma-vector-form2}
\Sigma_P = 
 \{\pmb{\lambda}{\bf a} \mid \bv{a}_F \in \ints^k\cap (\Omega - Y\bv{a}_L) \myand \bv{a}_L \in \ints^{d-k}\}\;.
\end{equation} 
Clearly, $\Omega \subseteq\reals^k$ is a bounded, convex parallelepiped depending only on $S$ and $\mu_0,\ldots,\mu_{k-1}$. 

Letting $\pmb{\lambda}_F := (1,\lambda,\ldots,\lambda^{k-1})$ and $\pmb{\lambda}_L := (\lambda^k,\ldots,\lambda^{d-1})$ (both row vectors which concatenate to
$\pmb{\lambda}$, so that 
$\pmb{\lambda} = (\pmb{\lambda}_F;\pmb{\lambda}_L)$ is the $\ordth{k}$ row of $V$), Equation~(\ref{eqn:Sigma-vector-form2})  establishes that
\begin{equation}\label{eqn:not-done-yet}
\Sigma_P = \bigcup_{\bv a_L \in \ints^{d-k}} \left\{ ({\pmb{\lambda}}_{F};{\pmb{\lambda}}_{L})(\bv a_F;\bv a_L) \mid \bv a_F \in \ints^k \intersect \left( \Omega - Y\bv a_L \right) \right\}\;.
\end{equation}
We will use the $F$- and $L$-subscripts for new, stand-alone vectors we introduce to indicate their types: $F$-subscripts for $k$-dimensional vectors and $L$-subscripts for $(d-k)$-dimensional vectors.  For any $\bv c_L\in\ints^{d-k}$, define
\begin{equation}\label{eqn:Omega-c-L}
\Omega_{\bv c_L} := \left(\ints^k + Y\bv c_L\right) \intersect \Omega\;.
\end{equation}
For any $\bv c = (\bv c_F;\bv c_L)\in\ints^d$, let $\hat{\bv c}_F$ denote $\bv c_F + Y\bv c_L$ (column vector in $\complexes^k$).  Then (\ref{eqn:not-done-yet}) becomes
\begin{eqnarray*}
\Sigma_P &= & 
      \bigcup\limits_{\bv a_L \in \ints^{d-k}} \left\{ \pmb{\lambda}_F \bv a_F + \pmb{\lambda}_L \bv a_L  \mid \bv a_F \in \ints^k \intersect \left(\Omega - Y\bv a_L\right)\right\}\\
 &= & \bigcup\limits_{\bv a_L \in \ints^{d-k}} \left\{ \pmb{\lambda}_F \bv a_F + \pmb{\lambda}_L \bv a_L  \mid \hat{\bv a}_F \in \left(\ints^k + Y\bv a_L\right) \intersect \Omega\right\}\\
 & = & \bigcup\limits_{\bv a_L \in \ints^{d-k}} \left\{\pmb{\lambda}_L \bv a_L + \pmb{\lambda}_F (\hat{\bv a}_F - Y \bv a_L) \mid \hat{\bv a}_F \in \Omega_{\bv a_L} \right\} \\
&=& \bigcup\limits_{\bv a_L \in \ints^{d-k}} \left\{ \left(\pmb{\lambda}_L - \pmb{\lambda}_F Y \right) \bv a_L + \pmb{\lambda}_F \hat{\bv a}_F \mid \hat{\bv a}_F \in \Omega_{\bv a_L} \right\}\\
 & = & \bigcup\limits_{\bv a_L \in \ints^{d-k}} \left[ \left(\pmb{\lambda}_L - \pmb{\lambda}_F Y \right) \bv a_L + \pmb{\lambda}_F \Omega_{\bv a_L} \right].
\end{eqnarray*}
The set $\pmb{\lambda}_F \Omega_{\bv a_L}$ in the right-hand side is bounded, independent of $\bv a_L$, because $\Omega_{\bv a_L}\subseteq\Omega$.  Set
\begin{equation}\label{eqn:Delta}
\Delta := \pmb{\lambda}_L - \pmb{\lambda}_F Y\;,
\end{equation}
noting that $\Delta$ is a $(d-k)$-dimensional row vector.  With these definitions, we have
\begin{equation}\label{eq:SigmaP-final-form}
\Sigma_P = \bigcup_{\bv a_L\in \ints^{d-k}} \left(\Delta \bv a_L +\pmb{\lambda}_F \Omega_{\bv a_L}\right)\;.
\end{equation}

Consider the set $A := \Delta\ints^{d-k} = \left\{ \Delta \bv a_L \mid \bv a_L\in\ints^{d-k} \right\}$.  $A$ is a subgroup of the additive group of $\complexes$, being the image of $\ints^{d-k}$ under the homomorphism given by $\Delta$.

\begin{claim}\label{claim:Delta-one-to-one}
Suppose that (1) $A$ has no accumulation points in $\complexes$ (whence $A$ is uniformly discrete), and (2) the map given by $\Delta$ is one-to-one\footnote{This restriction can be relaxed to ``finite-to-one,'' and the proof of the claim still goes through.  It does not matter here, because every nontrivial subgroup of $\ints^{d-k}$ is infinite, but this relaxation may be useful if the claim is ever generalized.} restricted to $\ints^{d-k}$, i.e., for any point $p \in A$, there is only one $\bv a_L \in \ints^{d-k}$ such that $p = \Delta \bv a_L$.  Then 
$\Sigma_P$ is uniformly discrete.
\end{claim}

\begin{proof}[Proof of the claim]
We show $\Sigma_P$ as expressed in Eq.~(\ref{eq:SigmaP-final-form}) is uniformly discrete. We must find an $r>0$ such that any two distinct points in $\Sigma_P$ are at least $r$ distance apart.  Let $x$ and $y$ be any two distinct elements of $\Sigma_P$, and assume that $|x-y|\le 1$, i.e., $x-y\in \overline D$, where $\overline D$ is the closed unit disk in $\complexes$.  
By Eq.~(\ref{eq:SigmaP-final-form}), we can write
\begin{align*}
x &= \Delta \bv a_L + \pmb{\lambda}_F \bv x_F\;, \\
y &= \Delta \bv b_L + \pmb{\lambda}_F \bv y_F\;,
\end{align*}
for some $\bv a_L,\bv b_L\in\ints^{d-k}$ and some $\bv x_F \in \Omega_{\bv a_L}$ and $\bv y_F \in \Omega_{\bv b_L}$.  By the definition of $\Omega_{\bv c_L}$ (Eq.~(\ref{eqn:Omega-c-L})), we can write $\bv x_F = {\bv a}_F + Y {\bv a}_L$ and $\bv y_F = \bv b_F + Y{\bv b}_L$ for some $\bv a_F,\bv b_F \in \ints^k$.
Thus we find,
\begin{eqnarray}
   x-y & = & \Delta \bv a_L - \Delta \bv b_L + 
             \pmb{\lambda}_F(\bv x_F - \bv y_F)
             \label{eqn:x-yform1}\\
       & = & \Delta\bv a_L - \Delta\bv b_L 
              +\pmb{\lambda}_F\bv a_F-\pmb{\lambda}_F\bv b_F
              +\pmb{\lambda}_FY\bv a_L - \pmb{\lambda}_F Y \bv b_L \\
       & = & \pmb{\lambda}_L\bv a_L -\pmb{\lambda}_L\bv b_L   
             -\pmb{\lambda}_FY\bv a_L + \pmb{\lambda}_FY\bv b_L
             +\pmb{\lambda}_F\bv a_F-\pmb{\lambda}_F\bv b_F
              +\pmb{\lambda}_FY\bv a_L - \pmb{\lambda}_F Y \bv b_L \\
       & = & \pmb{\lambda}_L(\bv a_L -\bv b_L)
             + \pmb{\lambda}_F(\bv a_F-\bv b_F),
             \label{eqn:x-yform2}
\end{eqnarray}
where in the third equality we used the definition of $\Delta$. The 
remainder of the proof 
is essentially to show that $\pmb{\lambda}_L(\bv a_L -\bv b_L)$ and $\pmb{\lambda}_F(\bv a_F-\bv b_F)$ take
on finitely many possible values, from which we can conclude
the same for $x-y$.

To see this, note that by rearranging Eq.~(\ref{eqn:x-yform1}), we find,
$x-y - \pmb{\lambda}_F(\bv x_F - \bv y_F) = \Delta (\bv a_L - \bv b_L)$.
From $\bv x_F - \bv y_F \in \Omega_{\bv a_L} - \Omega_{\bv b_L}
\subseteq \Omega-\Omega$, and since $x - y \in \overline{D}$, we conclude
that $\Delta (\bv a_L - \bv b_L) \in \overline{D} - \pmb{\lambda}_F(\Omega-\Omega)$.
We also have   $\Delta (\bv a_L - \bv b_L) \in A$, so
$\Delta (\bv a_L - \bv b_L) \in A \cap (\overline{D} - \pmb{\lambda}_F(\Omega-\Omega))$.
But $A$ is discrete and $\overline{D} - \pmb{\lambda}_F(\Omega-\Omega)$ is bounded, so
by Lemma~\ref{lem:uniformly-discrete-cap-bounded}, $A \cap (\overline{D} - \pmb{\lambda}_F(\Omega-\Omega))$ is finite. For this reason,
and because $\Delta$ is 1-1, the set $G$ defined as
\[ G := \left\{ \bv a_L\in\ints^{d-k} \mid \Delta \bv a_L \in \overline {D} - {\pmb{\lambda}}_{F}(\Omega - \Omega) \right\} \]
is finite.  Because $\Delta (\bv a_L - \bv b_L) \in  \overline{D} - \pmb{\lambda}_F(\Omega-\Omega)$,
we have that $\bv a_L - \bv b_L \in G$,
and hence $\pmb{\lambda}_L(\bv a_L - \bv b_L)$
is contained in $\pmb{\lambda}_LG$, which
is finite as well.
Now we know $\bv x_F - \bv y_F = \bv a_F - \bv b_F
+ Y(\bv a_L - \bv b_L)$, so that 
$\bv a_F - \bv b_F = \bv x_F - \bv y_F - Y(\bv a_L - \bv b_L)
\in \Omega - \Omega - YG$. However, because $G$ is finite
and $\Omega -\Omega$ is bounded, the set
$H$ defined by $H := \ints^k \cap (\Omega-\Omega - YG)$ is finite,
again by Lemma~\ref{lem:uniformly-discrete-cap-bounded}.
We thus have $\bv a_F - \bv b_F \in H$, and hence $\pmb{\lambda}_F(\bv a_F - \bv b_F)
\in \pmb{\lambda}_F H$, where
$\pmb{\lambda}_F H$ is finite. Then by Eq.~(\ref{eqn:x-yform2}),
$x - y \in \pmb{\lambda}_L G + \pmb{\lambda}_F H$, where, by the
foregoing, the sets $G$ and $H$ are independent of $x$ and $y$. Since
$ \pmb{\lambda}_L G + \pmb{\lambda}_F H$ is finite,
 its nonzero elements are bounded away from $0$ in absolute value, and so we may set $r := \min\left(\{1\} \cup \{ |z| : z\in \pmb{\lambda}_L G + \pmb{\lambda}_F H \mbox{ and } z \ne 0 \}\right)$.  This establishes the Claim.
\end{proof}

Continuing with the proof of Theorem~\ref{thm:main}, we just need to see when the vector $\Delta$ and the set $A = \Delta\ints^{d-k}$ satisfy the assumptions of the Claim.  
Start with the observation that, by Eq.~(\ref{eq:SigmaP-final-form}), we have
\begin{eqnarray}\label{eqn:model-set}
\Sigma_P & = & \bigcup_{\bv a_L\in \ints^{d-k}} \left(\Delta \bv a_L + \pmb{\lambda}_F \Omega_{\bv a_L}\right)\\
 & \subseteq & \bigcup_{\bv a_L\in \ints^{d-k}} \left(\Delta \bv a_L + \pmb{\lambda}_F \Omega\right)\\
 & = & \{\Delta \bv a_L \mid \bv a_L\in\ints^{d-k} \} + \pmb{\lambda}_F\Omega\\
 & = & A + \pmb{\lambda}_F\Omega\;.
 \label{eqn:sigma-containment}
\end{eqnarray}
We now show that the assumptions (1) and (2) of the Claim are satisfied in both of the two cases given at the start of this proof.
 
\begin{enumerate}
\item
When $k=d-1$ (whence $\lambda\in\reals$), $X$ and $Y$ are both $(d-1)$-dimensional column vectors, $\pmb{\lambda}_F$ is a $(d-1)$-dimensional row vector, $\pmb{\lambda}_L$ is the scalar $\lambda^{d-1}$, and $\Delta$ (cf.\ Eq.~(\ref{eqn:Delta})) is the scalar $\lambda^{d-1} - \pmb{\lambda}_F Y$.  The set $A = \Delta \ints$ is clearly discrete.  Also, we must have $\Delta \ne 0$, for otherwise, $A = \{0\}$, and so by Theorem~\ref{thm:new-containment} and Equations~(\ref{eqn:Q-in-model-set}) and (\ref{eqn:model-set}--\ref{eqn:sigma-containment}), $Q_\lambda(S)\subseteq \Sigma_P \subseteq \pmb{\lambda}_F\Omega$, which is bounded, but we know that $Q_\lambda(S)$ is unbounded.  Thus, for any point $p \in A$ there is exactly one $a \in \ints$ such that $p = \Delta a$.
\item
When $k=d-2$ (and $\lambda\notin\reals$), $W$ is a $k\times k$ matrix, $X$ and $Y$ are $k\times 2$ matrices, $\pmb{\lambda}_F = (1,\ldots,\lambda^{d-3})$, and $\pmb{\lambda}_L = (\lambda^{d-2},\lambda^{d-1})$ (both row vectors).  Let $\bv x_{d-2} := ((\mu_0)^{d-2},\ldots,(\mu_{d-3})^{d-2})$ and $\bv x_{d-1} := ((\mu_0)^{d-1},\ldots,(\mu_{d-3})^{d-1})$ be the columns of $X$, and let $\bv y_{d-2} := W^{-1}\bv x_{d-2}$ and $\bv y_{d-1} := W^{-1}\bv x_{d-1}$ be the columns of $Y$.  Then $\Delta = (\alpha,\beta)\in\complexes^2$, where $\alpha := \lambda^{d-2} - \pmb{\lambda}_F \bv y_{d-2}$ and $\beta := \lambda^{d-1} - \pmb{\lambda}_F \bv y_{d-1}$, and this gives $A = \Delta\ints^2 = \alpha \ints + \beta \ints$.  It is easy to see that if $\alpha\ne 0$ and $\beta / \alpha \notin \reals$ (i.e., $\alpha$ and $\beta$ are in different directions), then the set $A$ is uniformly discrete, and also, for any $p \in A$ there is exactly one pair $(a_{d-2}, a_{d-1}) \in \ints^2$ such that $p=\alpha a_{d-2} + \beta a_{d-1}$, satisfying the hypotheses of Claim~\ref{claim:Delta-one-to-one}.  Now, we show that $\alpha \ne 0$ and $\beta / \alpha \notin \reals$, which finishes the proof.

By replacing $\bv y_{d-2}$ and $\bv y_{d-1}$ by their definitions, we can write
\begin{align}\label{eqn:alpha-beta}
\alpha &= \lambda^{d-2} - (\lambda^{0},\ldots , \lambda^{d-3}) W^{-1}\bv x_{d-2} & \mbox{and} && \beta &= \lambda^{d-1} - (\lambda^{0},\ldots , \lambda^{d-3}) W^{-1}\bv x_{d-1}\;.
\end{align}
We know that $W W^{-1} = I$, so for any $0\leq i \leq d-3$, $((\mu_i)^0, (\mu_i)^1, \ldots, (\mu_i)^{d-3}) W^{-1} = \tpp{\bv{e}}{i}$. And hence,
\begin{align}
\label{eqn:roots-alpha}
((\mu_i)^0, (\mu_i)^1, \dots, (\mu_i)^{d-3}) W^{-1} \bv x_{d-2} &= \tpp{\bv{e}}{i}\bv x_{d-2} = (\mu_i)^{d-2}\;, \\
\label{eqn:roots-beta}
((\mu_i)^0, (\mu_i)^1, \dots, (\mu_i)^{d-3}) W^{-1} \bv x_{d-1} &= \tpp{\bv{e}}{i}\bv x_{d-1} = (\mu_i)^{d-1}\;.
\end{align}
Let us now look at the $\alpha$ and $\beta$ of (\ref{eqn:alpha-beta}) as polynomials in $\lambda$.  With some abuse of notation let us define the monic polynomial $\alpha(z) := z^{d-2} - (z^0, z^1, \ldots, z^{d-3}) W^{-1} \bv x_{d-2}$.  Equation~(\ref{eqn:roots-alpha}) tells us that $\alpha(z)$ has $\mu_0, \mu_1, \ldots, \mu_{d-3}$ as its roots.  As its degree is $d-2$, we can write
\[ \alpha(z) = (z-\mu_0)(z-\mu_1) \cdots (z-\mu_{d-3})\;. \]
Similarly, by Equation~(\ref{eqn:roots-beta}), $\mu_0, \ldots, \mu_{d-3}$ are also roots of the monic polynomial $\beta(z) := z^{d-1} - (z^0, \ldots, z^{d-3}) W^{-1} \bv x_{d-1}$.  But $\beta(z)$ has degree $d-1$, so it has one more root.  The sum of all its roots is equal to minus the coefficient on $z^{d-2}$ in $\beta(z)$, which is zero, and so the other root is $(-\mu_0 - \mu_1 - \cdots -\mu_{d-3})$.  Hence we can write the following:
\begin{align*}
\alpha &= \alpha(\lambda) = (\lambda-\mu_0)(\lambda-\mu_1) \cdots (\lambda-\mu_{d-3})\;, \\
\beta &= \beta(\lambda) = (\lambda-\mu_0)(\lambda-\mu_1) \cdots (\lambda-\mu_{d-3}) (\lambda + \mu_0 + \ldots +\mu_{d-3})\;.
\end{align*}
Now it is clear that $\alpha \ne 0$ and $\beta / \alpha = \lambda + \mu_0 + \ldots +\mu_{d-3}$, which is non-real because $\lambda$ is non-real.
\end{enumerate}
Thus in both cases, the assumptions of Claim~\ref{claim:Delta-one-to-one} are satisfied, and the 
proof of Theorem~\ref{thm:main} is complete.
\end{proof}

\begin{remark}\label{rem:no-new-lambda}
Case~1 of the proof of Theorem~\ref{thm:main} (the case where $k=d-1$) would still go through if we allowed $p\in D[x]$ for any discrete subring $D\subseteq\complexes$ (instead of insisting that $p\in\ints[x]$), that is, we could relax the definition of strong PV number to allow algebraic $D$-integers $\lambda$ rather than $\ints$-integers, and we would still get discrete $Q_\lambda$ in Case~1.  This generalization does not yield any new values of $\lambda$ that are not already sPV by our original definition, however.  See Proposition~\ref{prop:no-new-lambda} in the Appendix.
\end{remark}

We summarize the most important findings of
Theorem~\ref{thm:new-containment} and the proof of Theorem~\ref{thm:main}---Equation~(\ref{eqn:Q-in-model-set}) in particular---in the following fact (cf.\ Equation~(\ref{eqn:Sigma-vector-form2})) for any finite $S\subseteq\ints[\lambda]$.

\begin{fact}\label{fact:main}
Let $\lambda\in\complexes\cmpl\clcl{0,1}$ be an algebraic integer of degree $d>0$.  Let $\mu_0,\ldots,\mu_{d-1}\in\complexes$ be the conjugates of $\lambda$, and assume that $\mu_0,\ldots,\mu_{k-1}$ are elements of $\clcl{0,1}$, for some $k \ge 1$, and that $\lambda = \mu_k$.  Let $S\subseteq\ints[\lambda]$ be finite.  Then letting
\begin{itemize}
\item
$W := V(\mu_0,\ldots,\mu_{k-1})$,
\item
$\bv x_j := \left((\mu_0)^j,\ldots,(\mu_{k-1})^j\right)$ (column vector) for all $k\le j < d$,
\item
$\bv y_j  := W^{-1}\bv x_j\in\reals^k$ for all $k\le j < d$,
\item
for all $0\le j <k$, \ $\ell_j := \min S_j$ and $h_j := \max S_j$, where
\[ S_j := \left\{\sum_{i=0}^{d-1} a_i(\mu_j)^i \;:\; a_0,\ldots,a_{d-1}\in\ints \myand \sum_{i=0}^{d-1} a_i\lambda^i \in S \right\}\;, \]
and
\item
$\Omega := W^{-1}P$, where $P := \clcl{\ell_0,h_0}\times\clcl{\ell_1,h_1}\times\cdots\times\clcl{\ell_{k-1},h_{k-1}}$,
\end{itemize}
we have
\begin{align}\label{eqn:main1}
Q_\lambda(S) &\subseteq \Sigma_P
 = \left\{({\pmb{\lambda}}{\bf a}) ~|~ 
{\bf a} \in \ints^d \myand {\bf a}_F \in \ints^k\cap (\Omega - Y{\bf a}_L)\right\}\\
             & = \left\{\left. \sum_{j=0}^{d-1} a_j\lambda^j \;\right|\; (a_0,\ldots,a_{d-1}) \in \ints^d \myand (a_0,\ldots,a_{k-1}) \in  \Omega - \sum_{j=k}^{d-1} a_j\bv y_j \right\} \\ \label{eqn:main2}
&= \left\{\left. \sum_{j=0}^{d-1} a_j\lambda^j \;\right|\; (a_0,\ldots,a_{d-1}) \in \ints^d \myand W(a_0,\ldots,a_{k-1}) + \sum_{j=k}^{d-1} a_j\bv x_j\in P \right\} \\ \label{eqn:main3general}
&= \left\{\left. \sum_{j=0}^{d-1} a_j\lambda^j \;\right|\; (a_0,\ldots,a_{d-1}) \in \ints^d \myand \ell_i \le \sum_{j=0}^{d-1} a_j(\mu_i)^j \le h_i \mbox{ for all $0 \le i < k$} \right\}\;.
\end{align}
In the special case where $S = \{0,1\}$, we have $\ell_0 = \cdots = \ell_{k-1} = 0$ and $h_0 = \cdots = h_{k-1} = 1$, whence
\begin{equation}\label{eqn:main3}
Q_\lambda \subseteq \Sigma_P = \left\{\left. \sum_{j=0}^{d-1} a_j\lambda^j \;\right|\; (a_0,\ldots,a_{d-1}) \in \ints^d \myand 0 \le \sum_{j=0}^{d-1} a_j(\mu_i)^j \le 1 \mbox{ for all $0 \le i < k$} \right\}\;.
\end{equation}
\end{fact}

\subsection{Connection to cut-and-project schemes}
\label{sec:cut-and-project-connection}

Here we show that the set $\Sigma_P^{(\lambda)}$ of Definition~\ref{def:Sigma_P}, used explicitly in Theorem~\ref{thm:main} and Fact~\ref{fact:main} and which includes $Q_\lambda(S)$ ($= R_\lambda(S)$) as a subset
by Theorem~\ref{thm:new-containment}, is always the model set of a particular cut-and-project scheme (see Section~\ref{sec:meyer-sets}).  Let $\lambda$ be sPV with minimal polynomial $p(x)$ of degree $d$, and let $S$ be some finite subset of $\ints[\lambda]$, as in the proof of Theorem~\ref{thm:main}.  As in that proof, we can assume that $\lambda$ is nontrivial with $k\ge 1$ conjugates $\mu_0,\ldots,\mu_{k-1}$ in $\opop{0,1}$.  If $\lambda\in\reals$, then $k=d-1$, and otherwise, $k=d-2$ and we set $\mu_{d-1} \eqdf \lambda^*$.  In either case, we set $\mu_k \eqdf \lambda$.  We define $P$ and $\Sigma_P = \Sigma_P^{(\lambda)}$ accordingly.

In what follows, $V$ will be the $(k+1)\times d$ matrix consisting of the first $k+1$ rows of the Vandermonde matrix $V' \eqdf V(\mu_0,\ldots,\mu_{d-1})$.  Our conventions dictate that if $\lambda\in\reals$, then $V = V'$, and otherwise, $V$ is missing the last row of $V'$---the row corresponding to $\lambda^*$.  In either case, the last row of $V$ contains powers of $\mu_k = \lambda$, and all the other entries of $V$ are real.  We also let $F$ be $\reals$ if $\lambda\in\reals$, and $F \eqdf \complexes$, otherwise.  We identify $\complexes$ with $\reals^2$ via real and imaginary parts, and so the columns of $V$ can naturally be viewed as vectors in $d$ real dimensions.

\begin{definition}\label{def:cut-and-project-with-lambda}
Given $\lambda$, etc.\ as above, define the tuple
\[ \cM \eqdf (\reals^k,F,L)\;, \]
where $L$ is the $d$-(real)-dimensional integer lattice spanned by the columns of $V$:
\[ L \eqdf V\ints^d = \left\{ \sum_{j=0}^{d-1} a_j\bv{u}_j : a_0,\ldots,a_{d-1}\in\ints \right\}\;, \]
where, for all $0\le j < d$, \ $\bv{u}_j \eqdf ((\mu_0)^j,\ldots,(\mu_{k})^j)$ is the $\ordth{j}$ column of $V$.

We also have the corresponding projection maps $\pi_1$ and $\pi_2$, where $\map{\pi_1}{\reals^k\times F}{\reals^k}$ takes a vector $\bv v \in \reals^k\times F$ to the vector of its first $k$ components (all real), and $\map{\pi_2}{\reals^k\times F}{F}$ takes $\bv v$ to its last component (either real or complex, depending on $\lambda$).
\end{definition}

We can give a short, self-contained proof that the tuple $\cM$ is a cut-and-project scheme (see Definition~\ref{def:cut-and-project-scheme}), based on some of our previous results.

\begin{proposition}\label{prop:M-is-cut-and-project}
Let $\lambda\in\complexes$ be sPV of degree $d$, and let $\cM$ be as in Definition~\ref{def:cut-and-project-with-lambda}.  Then $\cM$ is a cut-and-project scheme.  Further, if $\lambda$ is nontrivial, then $\pi_1|L$ is also one-to-one.
\end{proposition}

\begin{proof}
The vectors $\bv{u}_j$ can be seen to be $\reals$-linearly independent.  (This is clearly true when $\lambda\in\reals$ and so $V = V'$, since the conjugates of $\lambda$ are pairwise distinct.  If $\lambda\notin\reals$, then $V$ is missing the last row (powers of $\lambda^*$) of $V'$.  But any $\reals$-linear combination of the columns of $V'$ makes the $\ordst{(k+1)}$ (i.e., the last) component the complex conjugate of the $\ordth{k}$ component, and so the former is nonzero whenever the latter is.)  Since the $\bv u_j$ are linearly independent, $L$ spans $\reals^d \cong \reals^k\times F$, and thus $\reals^d/L$ is compact (the $d$-dimensional torus).

The only other nonobvious things to prove are that: (a) $\pi_1(L)$ is dense in $\reals^k$, (b) $\pi_2|L$ is one-to-one, and (c) $\pi_1|L$ is one-to-one if $\lambda$ is nontrivial. 

For part~(a), let $p(x) = x^d + \sum_{i=0}^{d-1} c_ix^i \in \ints[x]$ be the minimal polynomial of $\lambda$, and let $\Lambda = \Lambda_p$ be the $d\times d$ companion matrix of $p$ as in Definition~\ref{def:lambda-mult}.  Let $D$ be the $(k+1)\times(k+1)$ diagonal matrix where $[D]_{ii} = \mu_i$ for all $0\le i\le k$.  By adapting the proof of Lemma~\ref{lem:lambda-spectral-decomp}, we see that $V\Lambda = DV$, which implies
\[ DL = DV\ints^d = V\Lambda\ints^d \subseteq V\ints^d = L\;, \]
the inclusion following from the fact that $\Lambda$ is an integer matrix.  Thus $D$ maps $L$ into $L$, and hence the same is true for $I-D$.  It follows that $L$ is closed\footnote{
Here is an alternative proof that does not use the companion matrix:
Let $\bv{x}, \bv{y} \in L = V\ints^d$. Then for some $\bv{a},\bv{b} \in \ints^d$,
we have
\begin{align*}
\bv{x} &= \left(\sum_{j=0}^{d-1}a_j\mu_0^j,\ldots,
\sum_{j=0}^{d-1}a_j\mu_{d-1}^j\right)\;, & \bv{y} = \left(\sum_{j=0}^{d-1}b_j\mu_0^j,\ldots, \sum_{j=0}^{d-1}b_j\mu_{d-1}^j\right)\;.
\end{align*}
The $\ordth{i}$ component of
$\bv{x} \x_D \bv{y}$ is then $(1-\mu_i)\sum_{j=1}^{d-1}a_j\mu_i^j
+ \mu_i\sum_{j=1}^{d-1}b_j\mu_i^j = \sum_{j=1}^{d-1}\left(a_j\mu_i^j + (b_j - a_j)\mu_i^{j+1}\right)$. But $\mu_i^d$ is an integer combination of $\mu_i^0,\ldots,\mu_i^{d-1}$, so this is of the form $z_i = \sum_{j=0}^{d-1}c_j\mu_i^j$, where each $c_j\in\ints$. Thus the resulting
vector $\bv{z} = \bv{x}\x_D\bv{y}$ is contained in $L$.}
 under the binary operation $\xD$ (cf.\ Definition~\ref{def:rho-extend}), and so $Q_D(S)\subseteq L$ for any $S\subseteq L$.

Define $\bv{w} \eqdf V\one$, where $\one$ is the $d$-dimensional vector of all ones.  Note that $\bv{w}$ is in $L$.  Set $\bv{x} \eqdf \pi_1(\bv{w})$.  The entries $x_0,\ldots,x_{k-1}$ of $\bv{x}$ (i.e., all but the last entry of $\bv{w}$) are all positive reals; indeed, $x_i = (1-\mu_i^d)/(1-\mu_i)$ for all $0\le i<k$, and each such $\mu_i$ is in $\opop{0,1}$.  Also note that $\pi_1\circ D = D' \circ \pi_1$, where $D'$ is the $k\times k$ diagonal matrix whose diagonal elements are $\mu_0,\ldots,\mu_{k-1}$, i.e., the first $k$ elements of the diagonal of $D$, which are pairwise distinct elements of $\opop{0,1}$ (both compositions above are maps $\reals^k\times F\rightarrow\reals^k$).  Now by Lemma~\ref{lem:rho-map}, we have $\pi_1(Q_D(\{0,\bv{w}\})) = Q_{D'}(\pi_1(\{0,\bv{w}\})) = Q_{D'}(\{0,\bv{x}\})$, which is dense in the $k$-dimensional box $B \eqdf \clcl{0,x_0}\times \cdots \times\clcl{0,x_{k-1}}$ by Corollary~\ref{cor:cube}.  But $\pi_1(Q_D(\{0,\bv{w}\})) \subseteq \pi_1 L$ by the previous paragraph, and thus $\pi_1 L \cap B$ is dense in $B$ as well.  By this fact and the translation invariance of $L$, \ $\pi_1 L$ must be dense in all of $\reals^k$, which proves part~(a).

For part~(b), consider two distinct points $\bv x, \bv y \in L$.  Then $\bv x - \bv y \in L\cmpl\{0\}$, and we can write $\bv x - \bv y = V\bv a$ for some nonzero $\bv a = (a_0,\ldots,a_{d-1})\in\ints^d$.

Not only is $V\bv a$ nonzero, but all of its components are nonzero as well.  Consider any component of $V\bv a$, say, the $\ordth{j}$ component (for some $j$).  This component is $\sum_{i=0}^{d-1} a_i\mu_j^i$, which cannot be $0$, for that would contradict the fact that $\mu_j$ has degree $d$ (being a conjugate of $\lambda$).

We have $\pi_2(\bv x) - \pi_2(\bv y) = \pi_2(\bv x - \bv y) = \pi_2(V\bv a)$.  The $\ordth{k}$ component of $V\bv a$ is $\pi_2(V\bv a)$, which is nonzero, and thus $\pi_2|L$ is one-to-one.

For part~(c), suppose $\lambda$ is nontrivial.  For $\bv x$, $\bv y$, and $\bv a$ as above, $\pi_1(\bv x) - \pi_1(\bv y) = \pi_1(\bv x - \bv y) = \pi_1(V\bv a)$.    Since $\lambda$ is nontrivial, $k\ge 1$.  Thus the vector $\pi_1(V\bv a)$ has positive dimension, i.e., the $\ordth{0}$ component of $V\bv a$ is included in $\pi_1(V\bv a)$ and is nonzero.  From this we conclude that $\pi_1|L$ is one-to-one.
\end{proof}


\begin{proposition}\label{prop:subset-of-model-set}
Let $\lambda$ and $\cM$ be as in Proposition~\ref{prop:M-is-cut-and-project}.  For any finite $S\subseteq \ints[\lambda]$, the set $Q_\lambda(S)$ is a subset of some model set of $\cM$ (cf.\ Definition~\ref{def:model-set}).
\end{proposition}

\begin{proof}
We show that the set $\Sigma_P$ (Definition~\ref{def:Sigma_P}) corresponding to $\lambda$ and $S$  is a model set of $\cM$.  To see how $\Sigma_P$ fits into Definition~\ref{def:cut-and-project-with-lambda}, we essentially perform a change of basis.  First, we recall the notation used in the proof of Theorem~\ref{thm:main} and in Definition~\ref{def:cut-and-project-with-lambda}:
\begin{itemize}
\item
$V$ is the $(k+1)\times d$ matrix of the first $k+1$ rows of the Vandermonde matrix $V(\mu_0,\ldots,\mu_{d-1})$, indexed by $0$ through $k$.  We have no use for the $\ordst{(k+1)}$ row of the Vandermonde matrix (if it exists).
\item
The last (i.e., the $\ordth{k}$) row of $V$ is the row vector $\tp{\bv e_k}V = (1,\lambda,\ldots,\lambda^{d-1})$ of powers of $\lambda = \mu_k$.  (Here, $\bv e_k$ is the $(k+1)$-dimensional column vector $(0,0,\ldots,0,1)$.)
\item
The projection maps $\pi_1$ and $\pi_2$ and the lattice $L = V\ints^d$ are as in Definition~\ref{def:cut-and-project-with-lambda}.
\end{itemize}
Now we can start with Equation~(\ref{eqn:main2}), using the notation of Fact~\ref{fact:main}.  Let $V'$ be the first $k$ rows of $V$ (i.e., all but the last row).  Noting that $V'$ is formed by appending the column vectors $\bv x_k,\ldots,\bv x_{d-1}$ onto the right end of $W$, we have
\begin{align*}
\Sigma_P &= \left\{\left. \sum_{j=0}^{d-1} a_j\lambda^j \;\right|\; (a_0,\ldots,a_{d-1}) \in \ints^d \myand W(a_0,\ldots,a_{k-1}) + \sum_{j=k}^{d-1} a_j\bv x_j\in P \right\} \\
&= \left\{ \tp{\bv{e}_k}V\bv a \mid \bv a\in \ints^d \myand V'\bv a \in P \right\}\;.
\end{align*}
Letting $\bv b := V\bv a$, we finally get
\[ \Sigma_P = \left\{ \tp{\bv{e}_k}\bv b \mid \bv b\in V\ints^d \myand \pi_1(\bv b) \in P \right\} = \left\{ \pi_2(\bv b) \mid \bv b \in L \myand \pi_1(\bv b) \in P \right\} = \pi_2\left( L \intersect \pi_1^{-1}(P) \right)\;. \]
The right-hand side is evidently the model set with window $P$ in the cut-and-project scheme of Definition~\ref{def:cut-and-project-with-lambda}.
\end{proof}

\begin{remark}
Unfortunately, Proposition~\ref{prop:subset-of-model-set} does not show that the various discrete $Q_\lambda(S)$ are Meyer sets, only because we do not know in general whether any particular $Q_\lambda(S)$ is relatively dense.  We address the issue of relative density in Section~\ref{sec:relative-density}, below, where we prove relative density in many cases.
\end{remark}

The next proposition speaks to the aperiodicity of various $Q_\lambda(S)$.  It follows immediately from Lemma~\ref{lem:no-ap} and Proposition~\ref{prop:M-is-cut-and-project}.

\begin{proposition}\label{prop:no-infinite-APs}
Let $\lambda$ be sPV and $S\subseteq\rats(\lambda)$ be finite.  If $\lambda$ is nontrivial (by Definition~\ref{def:strong-PV-number}), then $Q_\lambda(S)$ intersects any arithmetic progression in only finitely many points (and is therefore aperiodic).
\end{proposition}

\begin{proof}
By scaling everything up by some appropriate integer as in the proof of Theorem~\ref{thm:main}, we may assume that $S\subseteq\ints[\lambda]$.

By Proposition~\ref{prop:subset-of-model-set}, $Q_\lambda(S)$ is a subset of $\Sigma_P$, the model set of a cut-and-project scheme $\cM$ as in Proposition~\ref{prop:M-is-cut-and-project}, and since $\lambda$ is nontrivial, the map $\pi_1|L$ is one-to-one.  Then $\Sigma_P$ intersects any arithmetic progression in only finitely many points by Lemma~\ref{lem:no-ap}.
\end{proof}

\section{Further properties of strong PV numbers}

In this section, we collect a some additional facts about strong PV numbers, and we characterize all strong PV numbers of degree $\le 3$.

\subsection{The topology of sPV}

Recall from Part~I that $\convex$ is the set of all $\lambda$ such that $R_\lambda$ is convex.  Theorem~\ref{thm:main} says that no strong PV number is in $\convex$.  Thus Theorems~\ref{thm:deleted-neighborhood} and \ref{thm:main} immediately imply the following topological fact about the set of sPV numbers, which contrasts with the set of PV numbers, which is known to have infinitely many accumulation points~\cite{Vijayaraghavan:fractional-parts-I,Vijayaraghavan:fractional-parts-II}.

\begin{corollary}
The strong PV numbers form a discrete subset of $\complexes$.
\end{corollary}

\subsection{Abundance of strong PV numbers}

In this section, we show how an sPV number gives rise to infinitely many more sPV numbers.

\begin{proposition}\label{prop:Sigma_P-sPV}
Let $\lambda$ be a strong PV number, and let $k\ge 0$ be the number of conjugates of $\lambda$ that are in $\opop{0,1}$.  Every element of $\Sigma_P^{(\lambda)}$ is a strong PV number, where $P := \clcl{0,1}^{\times k}$.
\end{proposition}

\begin{proof}
Let $z$ be any element of $\Sigma_P := \Sigma_P^{(\lambda)}$.  Then $z\in\ints[\lambda]$.  By standard results in algebra (see Corollary~\ref{cor:Z-algebraic} in the Appendix), $z$ is an algebraic integer, and for every conjugate $c$ of $z$ there exists a conjugate $\mu$ of $\lambda$ such that $c$ is the image $h_\mu(z)$ of $z$ under the ring isomorphism $\map{h_\mu}{\ints[\lambda]}{\ints[\mu]}$ that maps $\lambda$ to $\mu$.  If $\mu = \lambda$ or $\mu = \lambda^*$, then $c = z$ or $c = z^*$ accordingly.  Equivalently, if $c\notin\{z,z^*\}$, then $c = h_\mu(z)$ for some $\mu \notin \{\lambda,\lambda^*\}$ conjugate to $\lambda$.  Since $\lambda$ is sPV, this $\mu$ is in $\opop{0,1}$.  This implies that $c$ is one of the coordinates of $\bv z'$ (see Definition~\ref{def:Sigma_P}).  Since $z\in\Sigma_P$, we have $\bv z' \in P$, whence $c\in\clcl{0,1}$.  It follows that $z$ is sPV.
\end{proof}

\begin{corollary}\label{cor:R-lambda-is-sPV}
If $\lambda$ is a strong PV number, then so are all elements of $R_\lambda$.
\end{corollary}

\begin{proof}
By Theorem~\ref{thm:new-containment} with $S := \{0,1\}$, Proposition~\ref{prop:Sigma_P-sPV}, and the fact that $R_\lambda = Q_\lambda$.
\end{proof}

One of our main conjectures is the converse of Theorem~\ref{thm:main}, that is, if $R_\lambda$ is discrete, then $\lambda$ is sPV\@.  All examples of discrete $R_\lambda$ that we currently know of are where $\lambda$ is sPV\@.  We have a natural way of generating new elements of $\complexes\cmpl\convex$ from old ones: by Corollary~\ref{cor:extend}, if $R_\lambda$ is discrete, then $R_\mu$ is discrete for all $\mu\in R_\lambda$.  Corollary~\ref{cor:R-lambda-is-sPV} says that this process cannot be used to disprove the conjecture.

Suppose $\lambda$ is nontrivial sPV and $\Sigma_P$ is as in Proposition~\ref{prop:Sigma_P-sPV}.  Although every element of $\Sigma_P$ is sPV by that proposition, it is not the case that every sPV element of $\ints[\lambda]$ is in $\Sigma_P$.  For example, $2\in\ints[\lambda]$ and is sPV, but $\bv 2' = (2,2,\ldots,2)\notin P$, and so $2\notin\Sigma_P$.  We do have the following, however:

\begin{proposition}
Suppose $\lambda$ is a strong PV number of prime degree, with $k$ many conjugates in $\opop{0,1}$.  Then every sPV number in $\ints[\lambda]\cmpl\ints$ is in $\Sigma_P^{(\lambda)}$, where $P := \clcl{0,1}^{\times k}$.
\end{proposition}

\begin{proof}
Let $d$ be the degree of $\lambda$, let $z$ be any strong PV number in $\ints[\lambda]\cmpl\ints$, and let $n$ be the degree of $z$.  We have $n>1$ and $n|d$, hence $n=d$.  If $\lambda\in\reals$, then so is $z$, and thus $z$ has $k = d-1$ many conjugates in $\opop{0,1}$.  In any case, $z$ has at least $k$ many conjugates in $\opop{0,1}$.\footnote{In fact, it has exactly $k$ many such conjugates.  This can be seen as follows: The only way $z$ can have more conjugates in $\opop{0,1}$ than $\lambda$ is if $z\in\reals$ but $\lambda\notin\reals$.  But then $\rats(z)\subseteq\reals$ and $\rats(\lambda)\not\subseteq\reals$.  Since $\rats(z)$ is properly included in $\rats(\lambda)$, we would thus have $n=[\rats(z):\rats] < [\rats(\lambda):\rats] = d$.  Contradiction.}  Since all of these conjugates are coordinates of $\bv z'$, it follows that all coordinates of $\bv z'$ are in $\opop{0,1}$, whence $\bv z'\in P$, which puts $z$ into $\Sigma_P^{(\lambda)}$.
\end{proof}

%

\subsection{Characterizing nontrivial strong PV numbers of low degree}

The next corollary characterizes the strong PV numbers $\lambda\in\reals$ of degree $2$ and gives bounds on the corresponding $R_\lambda$ sets.  It uses case~1 of Theorem~\ref{thm:main} with $d=2$.  It only applies to negative $\lambda$; for positive numbers, we can use the fact that $R_{1-\lambda} = R_\lambda$ and the fact (Fact~\ref{fact:spv-basic}) that $\lambda$ is sPV if and only if $1-\lambda$ is.

\begin{corollary}\label{cor:degree2}
Let $\lambda$ be any negative (real) number.  Then $\lambda$ is sPV of degree $2$ if and only if $\lambda = (-m - \sqrt{m^2+4n})/2$ for some integers $m$ and $n$ with $0 < n \le m$.  If this is the case, then
\begin{equation}\label{eqn:degree2}
R_\lambda \subseteq \Sigma_P^{(\lambda)} = \{ a - b\lambda \mid a,b\in\ints \myand b\mu \le a \le b\mu + 1\} = \{1\} \cup \left\{ \ceiling{b\mu} - b\lambda \mid b\in\ints\right\}\;,
\end{equation}
where $\mu := (-m + \sqrt{m^2+4n})/2$ is the conjugate of $\lambda$ and $\Sigma_P^{(\lambda)}$ is as in Definition~\ref{def:Sigma_P} with $P := \clcl{0,1}$.  Furthermore, $\lambda < -1$, and except for $0$ and $1$, any two adjacent elements of $R_\lambda$ differ by either $-\lambda$ or at least $1-\lambda$.
\end{corollary}

\begin{proof}
The quadratic polynomial $p(x) := x^2 + mx - n \in \ints[x]$ has roots $(-m\pm\sqrt{m^2+4n})/2$.  If $\lambda$ and $\mu$ are the roots of $p(x)$ as above, then the inequality $0 < n \le m$ is equivalent to $0 < \mu < 1$, and it guarantees that $p(x)$ is irreducible (and so $\lambda$ and $\mu$ are conjugates).  We can thus apply case~1 of Theorem~\ref{thm:main} with $(\mu_0,\mu_1) := (\mu,\lambda)$, which says that $R_\lambda$ is uniformly discrete (and thus $R_\lambda = Q_\lambda$).  Applying Equation~(\ref{eqn:main3}) in Fact~\ref{fact:main} with $d:=2$, \ $k:=1$, and $(\mu_0,\mu_1) := (\mu,\lambda)$, we have (using $(a,b)$ instead of $(a_0,a_1)$ as the index)
\[ Q_\lambda \subseteq \left\{\left. a + b\lambda \;\right|\; a,b\in \ints \myand 0 \le a + b\mu \le 1 \right\}\;. \]
By switching $b$ with $-b$, this set inclusion is seen to be equivalent to the first inclusion of (\ref{eqn:degree2}).  The subsequent set equality in (\ref{eqn:degree2}) follows from the fact that $\mu$ is irrational.

It is clear that $\lambda < -1$.  The quantity $\ceiling{b\mu} - b\lambda$ increases strictly monotonically in $b$, so adjacent points of $R_\lambda\cmpl\{1\}$ correspond (at least) to adjacent integer values of $b$.  When $b$ increases by $1$, \ $\ceiling{b\mu}$ increases by $0$ or $1$, giving a difference of either $-\lambda$ or $1-\lambda$.  Finally, the closest point to $1$ in $R_\lambda$, other than $0$, is $\ceiling{\mu} - \lambda = 1 - \lambda$, which is $-\lambda$ away from $1$.
\end{proof}

The case where $m = n = 1$ was already shown in Proposition~\ref{prop:1-plus-phi}.  In that case, $\lambda = -\p$ and $R_\lambda = R_{-\p} = R_{1+\p}$.  Here are a few other cases:
\begin{description}
\item[$m=2$ and $n=1$:] $\lambda = -1-\sqrt 2$ and $R_\lambda = R_{1-\lambda} = R_{2+\sqrt 2}$.
\item[$m=2$ and $n=2$:] $\lambda = -1-\sqrt 3$ and $R_\lambda = R_{1-\lambda} = R_{2+\sqrt 3}$.
\item[$m=3$ and $n=1$:] $\lambda = -(3+\sqrt{13})/2$ and $R_\lambda = R_{1-\lambda} = R_{(5+\sqrt{13})/2}$.
\end{description}

The next corollary characterizes the non-real strong PV numbers of degree $3$.  It applies case~2 of Theorem~\ref{thm:main} with $d=3$.

\begin{corollary}\label{cor:case2}
A non-real number $\lambda$ is sPV of degree $3$ if and only if $\lambda$ is a root of a polynomial $p(x) := x^3 + ax^2 + bx + c$ for some $a,b,c\in\ints$ such that
\begin{enumerate}
\item $c<0$,
\item $a+b+c \ge 0$, and
\item the discriminant $\Delta < 0$, where $\Delta := a^2b^2 - 4b^3 - 4a^3c - 27c^2 + 18abc$.
\end{enumerate}
If this is the case, then let $\mu$ be the (unique) root of $p$ in $\opop{0,1}$, and let $\lambda$ be one of the non-real roots of $p$.  Then
\begin{align*}
R_\lambda &\subseteq \left\{ a_0 + a_1\lambda + a_2\lambda^2 \mid a_0,a_1,a_2\in\ints \myand a_0 + a_1\mu + a_2\mu^2 \in \clcl{0,1} \right\} \\
&= \left\{ a_0 +a_1\lambda + a_2\lambda^2 \mid a_0,a_1,a_2\in\ints \myand -a_1\mu -a_2\mu^2 \le a_0 \le -a_1\mu - a_2\mu^2 + 1 \right\} \\
&= \left\{1\right\} \union \left\{ m\lambda + n\lambda^2 - \floor{m\mu+n\mu^2} \mid m,n\in\ints\right\}\;.
\end{align*}
\end{corollary}

\begin{proof}[Proof sketch]
The first two conditions give $p(0) = c < 0$ and $p(1) = 1+a+b+c > 0$, respectively, and together they force $p$ to have a root in $\opop{0,1}$.  The negative discriminant implies that the other two roots are non-real.  (The only other way to force a single real root in $\opop{0,1}$ is for $p(0)>0$ and $p(1)<0$, but this makes $p$ have three real roots.)  Thus the conditions of case~2 of Theorem~\ref{thm:main} are satisfied if and only if the conditions of the corollary are satisfied.  If this is the case, we apply Fact~\ref{fact:main}  (Equation~(\ref{eqn:main3})) with $d := 3$, with $k:=1$, and with $\mu_0 := \mu$, with $k := 1$ to get the first inclusion of the corollary.  The next equality is immediate.  The final equality follows from the fact that $p$ must be irreducible, and thus the set $\{1,\mu,\mu^2\}$ is linearly independent over $\rats$, whence $m\mu + n\mu^2\notin\ints$ unless $m=n=0$.
\end{proof}

Figure~\ref{fig:case2} shows two applications of Corollary~\ref{cor:case2}.
\begin{figure}
\centering
\includegraphics[width=0.5\textwidth]{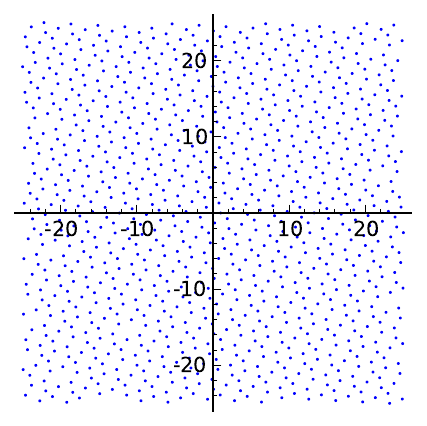}\includegraphics[width=0.5\textwidth]{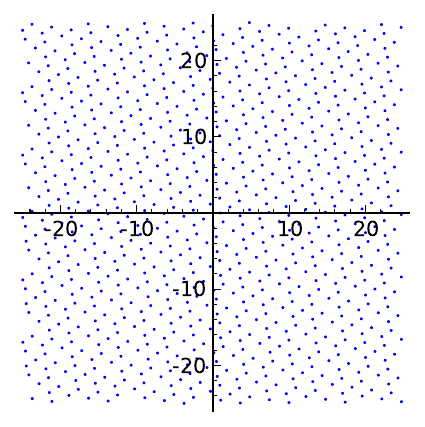}
\caption{Two plots illustrating Corollary~\ref{cor:case2}.  The left plot is of $R_\lambda$, where $\lambda$ is the root of the polynomial $x^3 + x^2 - 1$ closest to $-0.877 + 0.745i$.  The right plot is of $R_\mu$, where $\mu$ is the root of the polynomial $x^3+x-1$ closest to $-0.341+1.162i$.  One can easily show that $R_\lambda = R_{\lambda^2} = R_{\lambda^3}$.}\label{fig:case2}
\end{figure}
Let $r := \lambda - \mu$, let $s := \lambda^2 - \mu^2 = r(\lambda+\mu)$, and consider the lattice $r\ints + s\ints$.  This lattice is discrete, since $r$ and $s$ are $\reals$-linearly independent, and it is interesting to notice that each point in $R_\lambda$ differs from a point in this lattice by some real displacement between $0$ and $1$. 
In Propositions~\ref{prop:cubicomplexcandp1} and \ref{prop:cubicomplexcandp2} we show that equality holds in Corollary~\ref{cor:case2} for the values of $\lambda$ and $\mu$ indicated in Figure~\ref{fig:case2}.  It would be nice to know in general whether or when equality holds.

The next corollary covers the case of strong PV numbers of degree $3$ which are real.  It considers only negative $\lambda$; the positive numbers are then all of the form $1-\lambda$.

\begin{corollary}
A negative real number $\lambda$ is sPV of degree $3$ if and only if $\lambda$ is the least root of a polynomial $p(x) := x^3 + ax^2 + bx + c$ for some $a,b,c\in\ints$ such that
\begin{enumerate}
\item $c>0$,
\item $a+b+c \ge 0$,
\item $-2a-3 < b < 0$, and
\item the discriminant $\Delta > 0$, where $\Delta := a^2b^2 - 4b^3 - 4a^3c - 27c^2 + 18abc$.
\end{enumerate}
If this is the case, then $a>0$ as well, and letting $\mu < \nu$ be the two roots of $p$ in $\opop{0,1}$ and letting $v_\mymax(a_1,a_2)$ and $v_\mymin(a_1,a_2)$ be $\max\{-a_1\mu -a_2\mu^2, -a_1\nu -a_2\nu^2\}$ and $ \min\{-a_1\mu - a_2\mu^2, -a_1\nu - a_2\nu^2\}$, respectively, we have
\begin{align*}
R_\lambda &\subseteq \left\{ a_0 + a_1\lambda + a_2\lambda^2 \mid a_0,a_1,a_2\in\ints \myand a_0 + a_1\mu + a_2\mu^2 \in \clcl{0,1} \myand a_0 + a_1\nu + a_2\nu^2 \in \clcl{0,1} \right\} \\
&= \left\{ a_0 +a_1\lambda + a_2\lambda^2 \mid a_0,a_1,a_2\in\ints \myand v_\mymax(a_1,a_2) \le a_0 \le v_\mymin(a_1,a_2) + 1 \right\} \\
&= \left\{1\right\} \union \left\{ m\lambda + n\lambda^2 + \ceiling{v_\mymax(m,n)} \mid m,n\in\ints \myand \ceiling{v_\mymax(m,n)} = \ceiling{v_\mymin(m,n)} \right\} \\
&= \left\{1\right\} \union \left\{ m\lambda + n\lambda^2 - \floor{m\mu+n\mu^2} \mid m,n\in\ints \myand \floor{m\mu+n\mu^2} = \floor{m\nu + n\nu^2} \right\}\;.
\end{align*}
\end{corollary}

\begin{proof}[Proof sketch]
The discriminant $\Delta$ is positive if and only if $p$ has three distinct real roots: $\lambda<0$ and its conjugates $\mu$ and $\nu$ such that $\mu <\nu$.  As in Corollary~\ref{cor:case2}, conditions~1 and 2 are equivalent to $p(0),p(1) > 0$.  The derivative of $p(x)$ is $p'(x) = 3x^2 +2ax + b$ and has roots $x_\pm = (-a \pm \sqrt{a^2-3b})/3$.  Condition~4 implies $x_+,x_-\in\reals$ and $p(x_+)<0$.
\begin{description}
\item[($\Longrightarrow$):]
Suppose $\lambda$ is sPV of degree $3$ with minimal polynomial $p$ as above (and thus $\lambda < 0 < \mu < \nu < 1$).  Then $\Delta > 0$ (Condition~4), and $\lambda$ is the least root of $p$.  Additionally, $0<\mu < x_+ < \nu < 1$ and $p(0),p(1)>0$, the latter implying Conditions~1 and 2.  From $0<x_+<1$ we get $a < \sqrt{a^2-3b} < a+3$ (and in particular, $a > -3$ and $b\le a^2/3$).  Squaring both sides of the inequality $\sqrt{a^2-3b} < a+3$ yields $-2a-3<b$.  If $a\ge 0$, then squaring both sides of the inequality $a < \sqrt{a^2-3b}$ gives $b<0$, which implies Condition~3.  We have left to consider when $a<0$, and here there are only two possible cases: when $a = -1$ (whence $b=0$, because $-2a-3<b\le a^2/3$) or when $a=-2$.  We cannot have $a=-2$, because this violates $-2a-3<b\le a^2/3$ for any integer $b$.  If $a=-1$ and $b=0$, however, then $\Delta = 4c-27c^2 < 0$ for all integers $c > 0$ (cf.\ Condition~1), so this case cannot happen either.  This establishes the forward direction.  Finally, we also have $a > 0$, for if $a=0$, then $b\in\{-2,-1\}$ by Condition~3, and in either case, $\Delta = -4b^3-27c^2 < 0$ for any integer $c\ge -b$ (cf.\ Condition~2).
\item[($\Longleftarrow$):]
Suppose $\lambda$ is the least root of $p$ as above satisfying Conditions~1--4, with the other real roots $\mu<\nu$ (by Condition~4).  We have $p(0),p(1)>0$ by Conditions~1 and 2, and $0<x_+<1$ by Condition~3.  By Condition~4, $p(x_+)<0$, and thus there are two roots of $p$ strictly between $0$ and $1$.  These are $\mu$ and $\nu$, because $\lambda < 0$.  Finally, $p$ is irreducible (over $\rats$), for otherwise, $p$ has an integral root (which must be $\lambda$), which makes $\mu$ and $\nu$ conjugate roots of some monic quadratic polynomial in $\ints[x]$, but this is impossible.  Thus $\lambda$ is sPV of degree~$3$.
\end{description}
The rest of the corollary follows from Equation~(\ref{eqn:main3}) of Fact~\ref{fact:main} where $d=3$ and $k=2$.
\end{proof}

Notice that the condition $\floor{m\mu+n\mu^2} = \floor{m\nu + n\nu^2}$ in the corollary implies $|m\nu + n\nu^2-m\mu-n\mu^2|<1$, which implies $|m+n(\nu+\mu)| < (\nu-\mu)^{-1}$.  Thus given $m\in\ints$, there are only finitely many $n\in\ints$ such that the condition is satisfied, and vice versa.

\section{$\lambda$-convex closures of some regular shapes}

In this section, we consider the $\lambda$-convex closures of some point sets, specifically, regular polygons and the regular polyhedra (Platonic solids).  We apply Theorem~\ref{thm:main} to find combinations of finite sets $S$ and values $\lambda$ that make $Q_\lambda(S)$ uniformly discrete.  In all cases in this section, $\lambda$ will be a real number.  A challenge for future research is to find familiar point sets in $\complexes$ whose $\lambda$-convex closures are discrete for nonreal $\lambda$.

Recall that if $R_\lambda$ is convex, then so is $R_\lambda(S)$ for any $S$, by Proposition~\ref{prop:convex-implies-convex}.  Thus to find discrete, nontrivially generated $\lambda$-convex sets, we must have $R_\lambda = Q_\lambda$ discrete.  We thus confine ourselves to considering sPV $\lambda$.

\subsection{Regular polygons}
\label{sec:reg-polygons}

Here we will prove facts about discrete sets of the form $Q_\lambda(P_n)$, where $P_n$ is the set of vertices of a regular $n$-gon for $n\ge 3$, and $\lambda$ is chosen appropriately, depending on $n$.  We will also show plots of some of these sets.

We first review the facts about cyclotomic field extensions of $\rats$ that we will need.  Fix an integer $n > 1$, and let
\begin{equation}\label{eqn:zeta}
\zeta := e^{i\tau/n}
\end{equation}
be the principal $\ordth{n}$ root of unity.  Then $\zeta$ is an algebraic integer of degree $\phi(n)$ (where $\phi$ is Euler's totient function) whose minimum polynomial is the $\ordth{n}$ cyclotomic polynomial $\Phi_n$ and whose conjugates are the primitive $\ordth{n}$ roots of $1$.  The field extension $\rats(\zeta)$ of $\rats$ is a Galois extension of degree $[\rats(\zeta):\rats] = \phi(n)$ whose Galois group---isomorphic to $(\ints/n\ints)^\times$---consists of those automorphisms $\eta_a$ that map $\zeta$ to $\zeta^a$ (and thus map $\zeta^j$ to $\zeta^{ja}$ for any $j\in\ints$) for all $0<a<n$ such that $\gcd(a,n)=1$.  Thus the $\eta_a$ transitively permute the primitive $\ordth{n}$ roots of $1$.

We will use the following lemma repeatedly to lift discreteness results on the real line up to the complex plane.

\begin{lemma}\label{lem:rot-symmetry}
Let $\lambda\in\reals$ and $S\subseteq\complexes$.  Suppose that: (1) $Q_\lambda(\Re(S))$ is uniformly discrete; and (2) there exist $a,b\in\complexes$ such that $b-a\notin\reals$ and $S = \rho_{a,b}(S)$.  Then $Q_\lambda(S)$ is uniformly discrete.
\end{lemma}

\begin{proof}
Suppose otherwise.  Let $d := \max(|b-a|,1)$.  For any $\eps > 0$, we show that there exist distinct $u,v\in Q_\lambda(\Re(S))$ such that $|u-v| < \eps$, contradicting the assumption that $Q_\lambda(\Re(S))$ is uniformly discrete.  By assumption, there exist distinct $x,y\in Q_\lambda(S)$ such that $|x-y|<\eps/d$.  We have
\[ \{\Re(x),\Re(y)\} \subseteq \Re(Q_\lambda(S)) = Q_\lambda(\Re(S))\;, \]
the last equality by Lemma~\ref{lem:rho-map}.\footnote{We use Lemma~\ref{lem:rho-map} with $U = \complexes$, \ $V = k = \reals$, \ $t = \Re$, and the $\reals$-linear maps $A$ and $B$ being scalar multiplication by $\lambda$ on $\complexes$ and $\reals$, respectively.}
If $\Re(x) \ne \Re(y)$, then letting $u := \Re(x)$ and $v := \Re(y)$ gives the contradiction: $0<|u-v| \le |x-y|<\eps/d\le \eps$.  If $\Re(x) = \Re(y)$, then $x-y = ir$ for some nonzero $r\in\reals$ with $|r| < \eps/d$.  Letting $u := \Re(\rho_{a,b}(x))$ and $v := \Re(\rho_{a,b}(y))$, we have
\[ u-v = \Re(\rho_{a,b}(x)-\rho_{a,b}(y)) = \Re((x-y)(b-a)) = \Re(ir(b-a)) = -r\Im(b-a) \ne 0\;, \]
and so $0 < |u-v| = |r||\Im(b-a)| \le |r||b-a| < (\eps/d)|b-a| \le \eps$, and
\[ \{u,v\} \subseteq \Re(\rho_{a,b}(Q_\lambda(S))) = \Re(Q_\lambda(\rho_{a,b}(S))) = \Re(Q_\lambda(S)) = Q_\lambda(\Re(S))\;. \]
Contradiction.  The first equality is by Lemma~\ref{lem:translate-Q-R}; the last is by Lemma~\ref{lem:rho-map}.
\end{proof}

The primary question we ask in this section is: For which $n$ and which $\lambda$ is $Q_\lambda(P_n)$ discrete, and if so, what does this set look like?  We also wish to show that, if $Q_\lambda(P_n)$ is discrete, then it is a Meyer set.  For this it suffices to show uniform discreteness of $Q_\lambda(P_n) - Q_\lambda(P_n)$ and relative density of $Q_\lambda(P_n)$ in $\complexes$ (see Fact~\ref{fact:meyer-set-characterization}).  We show the former in this section and the latter in Section~\ref{sec:relative-density}.\footnote{We cannot apply Proposition~\ref{prop:subset-of-model-set} directly here, because $P_n$ is not a subset of $\ints[\lambda]$.}

We will use the next lemma in conjunction with Lemma~\ref{lem:rot-symmetry}:

\begin{lemma}\label{lem:minkowski-difference}
For any $a,b\in\complexes$ and $S\subseteq\complexes$, if $S = \rho_{a,b}(S)$, then $S-S = \rho_{0,b-a}(S-S)$.
\end{lemma}

\begin{proof}
For all $x,y\in\complexes$, we have $\rho_{0,b-a}(x-y) = (x-y)(b-a) = \rho_{a,b}(x) - \rho_{a,b}(y)$.  Therefore,
\[ \rho_{0,b-a}(S-S) = \{\rho_{0,b-a}(x-y) \mid x,y\in S\} = \{\rho_{a,b}(x) - \rho_{a,b}(y) \mid x,y\in S\} = \rho_{a,b}(S) - \rho_{a,b}(S) = S-S\;. \]
\end{proof}

It is most natural to define $P_n\subseteq\complexes$ formally as the set of all $\ordth{n}$ roots of~$1$, which orients the polygon on the unit circle.  Purely for reasons of technical convenience, however (including some relating to the computer code generating the plots shown below), we instead orient $P_n$ so that one of its sides coincides with the unit interval and the rest of the polygon lies in the upper halfplane:
\begin{equation}\label{eqn:P-n}
P_n := \left\{ \sum_{j=0}^{\ell-1} \zeta^j : 0\le \ell < n \right\} = \left\{\frac{\zeta^\ell-1}{\zeta-1} : 0\le \ell < n \right\}\;,
\end{equation}
for $\zeta$ as in (\ref{eqn:zeta}).  This definition also has the advantage that $\{0,1\}\subseteq P_n$.

We first prove a straightforward result for the equilateral triangle, square, and regular hexagon.

\begin{proposition}\label{prop:easy-n-gons}
If $\lambda\in\reals$ is sPV, then $Q_\lambda(S)-Q_\lambda(S)$ is uniformly discrete for all $S\in\{P_3,P_4,P_6\}$.
\end{proposition}

\begin{proof}
For $S\in\{P_3,P_4,P_6\}$, we note that $\rho_{a,b}(S) = S$, where $a = 1$, and either $b = (1+i\sqrt 3)/2$ for $S=P_3$, \ $b = 1+i$ for $S=P_4$, or $b = (3+i\sqrt 3)/2$ for $S=P_6$.  It follows by Lemma~\ref{lem:minkowski-difference} that $\rho_{0,c}(S-S) = S-S$, where $c=(-1+i\sqrt 3)/2$ or $c=i$ or $c=(1+i\sqrt 3)/2$, respectively.  In all cases, $\Re(S-S) = \Re(S)-\Re(S) \subseteq \left\{-2,-\frac{3}{2},-1,-\frac{1}{2},0,\frac{1}{2},1,\frac{3}{2},2\right\} \subseteq \rats \subseteq \rats(\lambda)$, and so by Theorem~\ref{thm:main}, $Q_\lambda(\Re(S-S))$ is uniformly discrete.  It follows that $Q_\lambda(S-S)$ is uniformly discrete by Lemma~\ref{lem:rot-symmetry}, and in addition, $Q_\lambda(S-S) = Q_\lambda(S)-Q_\lambda(S)$ by Lemma~\ref{lem:distribute}.
%
%
%
\end{proof}


We obtain most of the other values of $\lambda$ that we consider here, corresponding to the various $P_n$, by perhaps a somewhat arbitrary type of geometric construction: choose two nonparallel lines $L_1$ and $L_2$ passing through pairs of points of $P_n$, then let $\lambda$ be a certain ratio of distances on $L_1$ between the two points on the polygon and the point of intersection of $L_1$ and $L_2$.  For the moment we will choose pairs of adjacent points on the polygon, and so we can assume without loss of generality that $L_1$ is the real axis.

\subsection{Constructions with odd $n$}

Let $n\ge 3$ be odd, and consider the following construction using $P_n$:
\begin{center}
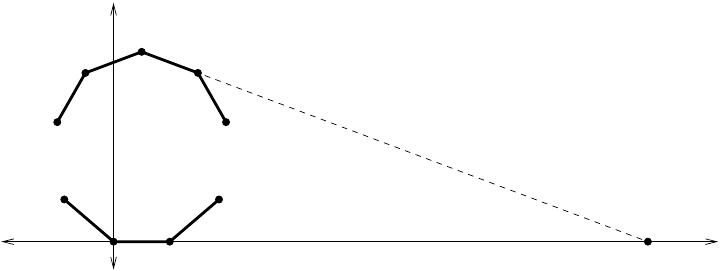
\end{center}
The line through the base of $P_n$ is the real axis, and the other line passes through the apex $A$ of $P_n$ and the point $B$ adjacent to it to the right.  The two lines intersect at the point
\begin{equation}\label{eqn:lambda-n}
\lambda_n := \frac{1}{2(1-\cos(\pi/n))} = \frac{1}{2+\zeta^k+\zeta^{-k}}\;,
\end{equation}
where $k := (n-1)/2$ and $\zeta$ is as in Equation~(\ref{eqn:zeta}).  Note that $\gcd(n,k) = 1$.  (One way to see that this formula for $\lambda_n$ is correct is to drop a vertical line segment from $A$ down to the base at the point $1/2$, forming a right triangle $T$ with the point $\lambda_n$.  The hypotenuse of $T$ has length $\lambda_n$ by symmetry, and the base of $T$ has length $\lambda_n - 1/2$, and finally, the acute angle at $\lambda_n$ is $\pi/n$.  To see the second equation, note that $(\zeta^k + \zeta^{-k})/2 = \cos(2k\pi/n) = \cos(\pi - \pi/n) = -\cos(\pi/n)$.)  Of course, the first equation of (\ref{eqn:lambda-n}) makes sense for all $n\in\posints$, not necessarily odd, and we will on occasion refer to $\lambda_n$ for even $n$, but for the time being, we assume that $n$ is odd.

Now consider the set $Q_{\lambda_n}(P_n)$.  Since $0,1\in P_n$, this set includes $Q_{\lambda_n}$ as a subset, and so it can be discrete only if $Q_{\lambda_n}$ is discrete.  The next proposition comes close to a converse.

\begin{proposition}\label{prop:discrete-reg-n-gon}
For any odd $n\ge 3$, let $\lambda_n$ be given by Equation~(\ref{eqn:lambda-n}).  Suppose $\lambda_n$ is sPV\@.  Then $Q_{\lambda_n}(P_n)-Q_{\lambda_n}(P_n)$ is uniformly discrete; in fact, $Q_{\lambda_n}(P_{2n})-Q_{\lambda_n}(P_{2n})$ is uniformly discrete.
\end{proposition}

\begin{remark}
Note that if $Q_{\lambda_n}(P_{2n})-Q_{\lambda_n}(P_{2n})$ is uniformly discrete, then so is $Q_{\lambda_n}(P_n)-Q_{\lambda_n}(P_n)$: since $P_n$ can be embedded into $P_{2n}$ by a $\complexes$-affine transformation, the same holds for $P_n-P_n$ into $P_{2n}-P_{2n}$, and thus for $Q_{\lambda_n}(P_n)-Q_{\lambda_n}(P_n)$ into $Q_{\lambda_n}(P_{2n})-Q_{\lambda_n}(P_{2n})$ by Lemma~\ref{lem:distribute}.
\end{remark}

\begin{proof}[Proof of Proposition~\ref{prop:discrete-reg-n-gon}]
Assume that $n\ge 3$ is odd and that $\lambda_n$ is sPV\@.  By the previous remark, it suffices to show that $Q_{\lambda_n}(P_{2n})-Q_{\lambda_n}(P_{2n})$ is uniformly discrete.  By symmetry, we have $\rho_{1,1+\eta}(P_{2n}) = P_{2n}$, where $\eta := e^{i\tau/(2n)}$ is the principal $(2n)$-th root of $1$.  Then by Lemmas~\ref{lem:distribute}, \ref{lem:rot-symmetry}, and \ref{lem:minkowski-difference}, it suffices to show that $Q_{\lambda_n}(\Re(P_{2n}-P_{2n}))$ is uniformly discrete.

By Equation~(\ref{eqn:P-n}),
\[ \Re(P_{2n}) = \left\{ \sum_{j=0}^{\ell-1} \cos(j\pi/n) : 0\le \ell < n \right\}\;. \]
All the points in this finite set are elements of $\rats(\lambda_n)$.  This is because $\cos(\pi/n) = 1 - 1/(2\lambda_n)$ is clearly in $\rats(\lambda_n)$, and all the terms in the sum above are of the form $\cos(j\pi/n)$ for integer $j$, and these can in turn be expressed as (Chebychev) polynomials $T_j(\cos(\pi/n))$ of $\cos(\pi/n)$, where $T_j(x)\in\ints[x]$.  It follows that $\Re(P_{2n}-P_{2n}) = \Re(P_{2n})-\Re(P_{2n}) \subseteq \rats(\lambda_n)$ and is finite.  Theorem~\ref{thm:main} then says that $Q_{\lambda_n}(\Re(P_{2n}-P_{2n}))$ is uniformly discrete.
\end{proof}

The ``good'' news regarding Proposition~\ref{prop:discrete-reg-n-gon} is that there exist sVP $\lambda_n$, which yield interesting sets $Q_{\lambda_n}(P_n)$ and $Q_{\lambda_n}(P_{2n})$.  The ``bad'' news is that there are only finitely many such $n$.

\begin{proposition}\label{prop:lambda-n-PV}
Let $n\ge 3$ be odd.  Then $\lambda_n$ is sPV if and only if $n \in\{3,5,7,9,15\}$.
\end{proposition}

We prove Proposition~\ref{prop:lambda-n-PV} via the following three lemmas.  Define $\ints_n^* := \{j\in\ints \mid 0<j<n \myand \gcd(j,n)=1\}$ as is customary.

\begin{lemma}
Let $n \ge 3$ be odd.  Then $\lambda_n$ (see Eq.~(\ref{eqn:lambda-n})) is an algebraic integer of degree $\phi(n)/2$, and its Galois conjugates are of the form
\[ \frac{1}{2\left(1+\cos(2j\pi/n)\right)} \]
for all integers $0 < j < n/2$ such that $\gcd(j,n)=1$.
\end{lemma}

\begin{proof}
From Equation~(\ref{eqn:lambda-n}), letting $k := (n-1)/2$, we see that $1/\lambda_n = 2+\zeta^k+\zeta^{-k}$ is an algebraic integer, since $\zeta^{-k} = \zeta^{n-k}$ and the algebraic integers form a ring.  Let $m$ be the degree of $1/\lambda_n$, and let $p(x) := c_0+c_1x+\cdots+x^m\in\ints[x]$ be its (monic) minimal polynomial.  All numbers of the form $\eta_j(1/\lambda_n)$ for $j\in\ints_n^*$ must be conjugates of $1/\lambda_n$, since the $\eta_j$ are all field automorphisms.  It is more convenient to work with one such conjugate, $\mu := \eta_2(1/\lambda_n) = 2+\zeta^{2k}+\zeta^{-2k}=2+\zeta+\zeta^{-1} = 2(1+\cos(2\pi/n))$, rather than $1/\lambda_n$ itself.  Then, for $j\in\ints_n^*$, the numbers
\[ \eta_j(\mu) = 2+\zeta^j+\zeta^{-j} = 2(1+\cos(2j\pi/n)) \]
are all conjugates of $\mu$ (and hence of $1/\lambda_n$), and these numbers are pairwise distinct for $0<j<n/2$, since $\cos(x)$ is decreasing on $\clcl{0,\pi}$.  There are exactly $\phi(n)/2$ such $j$, and so $\mu$ has at least this many conjugates, and these are also conjugates of $1/\lambda_n$.  It follows that $m\ge \phi(n)/2$.  On the other hand, since $\mu\in\rats(\zeta)\cap\reals$, we have that $\rats(\mu)$ is a proper subfield of $\rats(\zeta)$, and so $[\rats(\zeta):\rats(\mu)] \ge 2$, and because $[\rats(\zeta):\rats(\mu)][\rats(\mu):\rats] = [\rats(\zeta):\rats] = \phi(n)$, it follows that $[\rats(\mu):\rats] \le \phi(n)/2$, which implies $m\le \phi(n)/2$.  Thus $m = \phi(n)/2$ is the degree of $\mu$, which is also the degree of $1/\lambda_n$ and of $\lambda_n$.  Furthermore, $\mu$ and $1/\lambda_n$ share the same set of conjugates $\{\eta_j(\mu) : j\in\ints_n^*\myand j<n/2\} = \{2(1+\cos(2j\pi/n)) : j\in\ints_n^*\myand j<n/2\}$, and hence the conjugates of $\lambda_n$ are exactly the reciprocals of these, being the roots of the polynomial $q(x) := c_0x^m + c_1x^{m-1}+\cdots+1$.

It remains to show that the constant term $c_0$ of $p(x)$ is $\pm 1$, which implies that $\lambda_n$ is an algebraic integer, since $c_0$ is also the leading coefficient of $q(x)$.  Using some trigonometric identities, we have
\[ \frac{1}{\lambda_n} = 2(1-\cos(\pi/n)) = \frac{2\sin^2(\pi/n)}{1+\cos(\pi/n)} = \frac{1-\cos(2\pi/n)}{1+\cos(\pi/n)} = \frac{2-\zeta-\zeta^{-1}}{2 -\zeta^k - \zeta^{-k}} = \frac{\nu}{\eta_k(\nu)}\;, \]
where we have set $\nu := 2 - \zeta - \zeta^{-1}$.  Then (recalling that $k := (n-1)/2$),
\[ \mu = \eta_2\left(\frac{1}{\lambda_n}\right) = \frac{\eta_2(\nu)}{\eta_{2k}(\nu)} = \frac{\eta_2(\nu)}{\eta_{n-1}(\nu)} = \frac{\eta_2(\nu)}{\eta_{-1}(\nu)} = \frac{\eta_2(\nu)}{\nu}\;. \]
Now $c_0$ is, up to a change of sign, the product of all the roots of $p(x)$, i.e., the conjugates of $1/\lambda_n$ (or of $\mu$).  We then have, noticing that $\eta_j(\mu) = \eta_{n-j}(\mu)$ for all $j\in\ints_n^*$,
\begin{align*}
c_0^2 &= \left(\prod_{j\in\ints_n^*\myand j<n/2}\eta_j(\mu)\right)^2 = \prod_{j\in\ints_n^*\myand j<n/2}\eta_j(\mu)\prod_{j\in\ints_n^*\myand j<n/2}\eta_{n-j}(\mu) = \prod_{j\in\ints_n^*} \eta_j(\mu) = \prod_j\frac{\eta_{2j}(\nu)}{\eta_j(\nu)}\;,
\end{align*}
where the index $2j$ is assumed to be reduced modulo $n$.  But the right-hand side is $1$, because the numerators and denominators both run through the same values.  Thus $c_0 = \pm 1$, and we are done.
\end{proof}

\begin{lemma}\label{lem:forbidden-interval}
Let $n\ge 3$ be odd.  Then $\lambda_n$ is sPV if and only if there are no $j\in\ints_n^*$ such that $n/3 \le j \le 2n/3$, except for $(n\pm 1)/2$.
\end{lemma}

\begin{proof}
Set $k:= (n-1)/2$.  By the previous lemma, $\lambda_n$ is an algebraic integer with conjugates $(2+2\cos(2j\pi/n))^{-1} = (2+\zeta^j+\zeta^{-j})^{-1}$ for $j\in\ints_n^*$ with $j<n/2$.  We can drop the requirement that $j<n/2$, because $\cos(2j\pi/n) = \cos(2(n-j)\pi/n)$ for all $j$.  Since $\lambda_n = (2+\zeta^k+\zeta^{-k})^{-1} = (2+\zeta^{k+1}+\zeta^{-(k+1)})^{-1}$, the conjugates of $\lambda_n$ other than $\lambda_n$ itself are of the form $(2+\zeta^j+\zeta^{-j})^{-1} = (2+2\cos(2j\pi/n))^{-1}$ for all $j\in\ints_n^* \cmpl \{k,k+1\}$.  It follows by definition that $\lambda_n$ is sPV if and only if $0 < (2+2\cos(2j\pi/n))^{-1} < 1$, or equivalently,  $\cos(2j\pi/n) > -1/2$ for all such $j$.  This latter inequality is equivalent to $j < n/3$ or $j > 2n/3$, for all $j\in\ints_n^*$ other than $k$ or $k+1$.
\end{proof}

\begin{lemma}\label{lem:n-sufficiently-large}
If $n>21$ and $n$ is odd, then there exists a $j\in\ints_n^*$ such that $n/3 < j < 2n/3$ and $j\notin\{(n-1)/2,(n+1)/2\}$.
\end{lemma}

\begin{proof}
We give an elementary proof using only the Bertrand-Chebyshev theorem, which states that for all integers $m>1$, there exists a prime $p$ with $m<p<2m$.  It follows that this must also be true for all \emph{real} $m>1$, by applying the theorem to $\floor{m}$ if $m\ge 2$.  Set $k := (n-1)/2$ as usual.
\begin{description}
\item[Case 1:] $n = qp$ where $p$ is an odd prime and $q\in\{5,7,9,11\}$.  Choose two numbers $j_1 < j_2$ depending on $p$ and $q$ according to the following table:
\[ \begin{array}{r||c|c}
q & j_1 & j_2 \\ \hline\hline
5 & 2p-1 & 2p+1 \\
7 & 3p-1 & 3p+1 \\
9 & 4p-2 & 4p-1 \\
11 & 5p-2 & 5p-1
\end{array} \]
Since $n$ is sufficiently large, one can readily check that: (i) $n/3 < j_1 < j_2 < k$; (ii) neither $j_1$ nor $j_2$ is a multiple of $p$; and (iii) at least one of $j_1$ and $j_2$ is coprime with $q$ and hence coprime with $n$, satisfying the lemma.
\item[Case 2:] not Case~1 and neither $k$ nor $k+1$ is prime.  By Bertrand-Chebyshev, there exists a prime $j$ such that $n/3 < j < 2n/3$.  Then $j\notdiv n$ and $j$ is neither $k$ nor $k+1$, so $j$ satisfies the lemma.
\item[Case 3:] not Case~1, one of $k$ and $k+1$ is prime, and the other is not of the form $2p$ for any prime $p$.  By Bertrand-Chebyshev, there exists a prime $r$ such that $n/6 < r < n/3$.  We must have $r\notdiv n$, because we are not in Case~1 (and $n$ is odd).  Then setting $j := 2r$ satisfies the lemma.
\item[Case 4:] not Case~1, one of $k$ and $k+1$ is prime, and the other is of the form $2p$ for some prime $p$.  By Bertrand-Chebyshev, choose a prime $s$ such that $n/12 < s < n/6$.  Again, $s\notdiv n$ since we are not in Case~1 (and $n$ is odd).  Then setting $j := 4s$ satisfies the lemma.
\end{description}
Cases~1--4 are clearly exhaustive.
\end{proof}

\begin{proof}[Proof of Proposition~\ref{prop:lambda-n-PV}]
Lemmas~\ref{lem:forbidden-interval} and \ref{lem:n-sufficiently-large} imply that $\lambda_n$ cannot be sPV if $n>21$.  For $3\le n\le 21$, one can check case by case that Lemma~\ref{lem:forbidden-interval} is satisfied if and only if $n\in\{3,5,7,9,15\}$.
\end{proof}

The following table gives basic information about $\lambda_n$ for the values of $n$ we are interested in:
\[ \begin{array}{r|l|r}
n & \mbox{minimal polynomial of $\lambda_n$} & \mbox{approximate value of $\lambda_n$} \\ \hline
 3 & x-1                  &  1.000 \\
 4 & 2x^2-4x+1            &  1.707 \\
 5 & x^2-3x+1             &  2.618 \\
 6 & x^2-4x+1             &  3.732 \\
 7 & x^3-6x^2+5x-1        &  5.049 \\
 9 & x^3-9x^2+6x-1        &  8.291 \\
15 & x^4-24x^3+26x^2-9x+1 & 22.881
\end{array} \]
All are sPV except $\lambda_4$, which is not an algebraic integer.  $2\lambda_4 = 2+\sqrt 2$ is sPV, however.

Figures~\ref{fig:pentagon}--\ref{fig:30gon}
on the following pages show plots of the uniformly discrete sets $Q_{\lambda_n}(P_m)$ for $n=5,7,9,15$ as in Proposition~\ref{prop:lambda-n-PV} and for selected $m \ge 3$ dividing $2n$ (the $n=3$ case is trivial, as $\lambda_3 = 1$).  We order them by increasing $n$ then increasing $m$ for each $n$.  Each displayed set has dihedral $D_m$ symmetry, but none has any translational symmetry.  Each displayed set is a Meyer set by Corollary~\ref{cor:P-n-meyer-set}.
\begin{figure}
\centering
\includegraphics[width=1\textwidth]{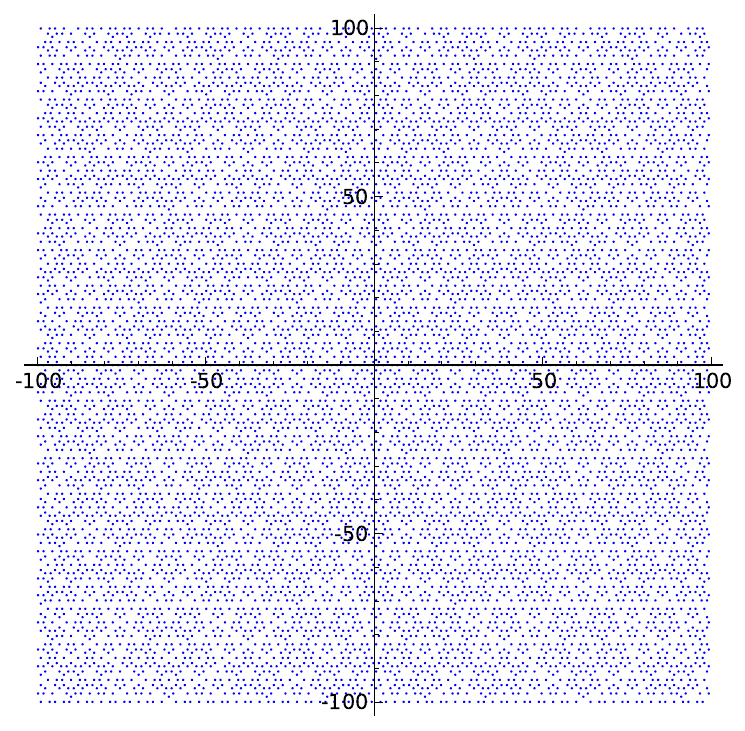}
\caption{$Q_{\lambda_5}(P_5)$, the $\lambda_5$-convex closure of a regular pentagon.  $\lambda_5 = 1+\p$.}\label{fig:pentagon}
\end{figure}
\begin{figure}
\centering
\includegraphics[width=1\textwidth]{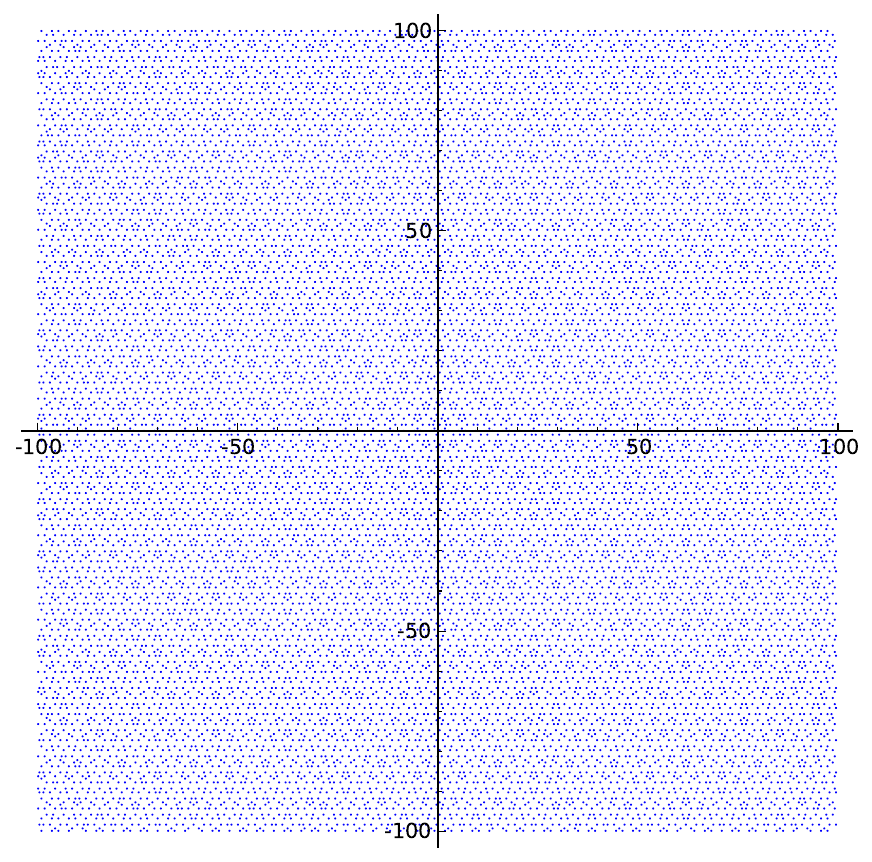}
\caption{$Q_{\lambda_5}(P_{10})$, the $\lambda_5$-convex closure of a regular decagon.}\label{fig:decagon}
\end{figure}
\begin{figure}\label{fig:heptagon-show3}
\centering
\includegraphics[width=1\textwidth]{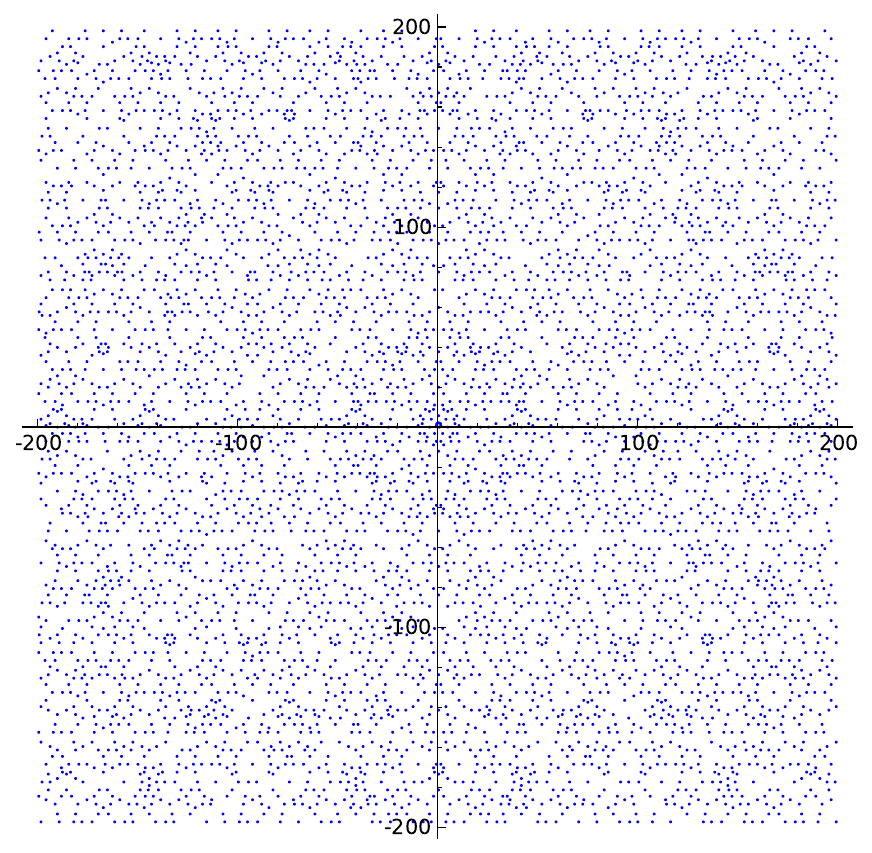}
\caption{$Q_{\lambda_7}(P_7)$, the $\lambda_7$-convex closure of a regular heptagon.  $\lambda_7 \approx 5.049$ and has minimal polynomial $x^3-6x^2+5x-1$.}\label{fig:heptagon}
\end{figure}
\begin{figure}
\centering
\includegraphics[width=1\textwidth]{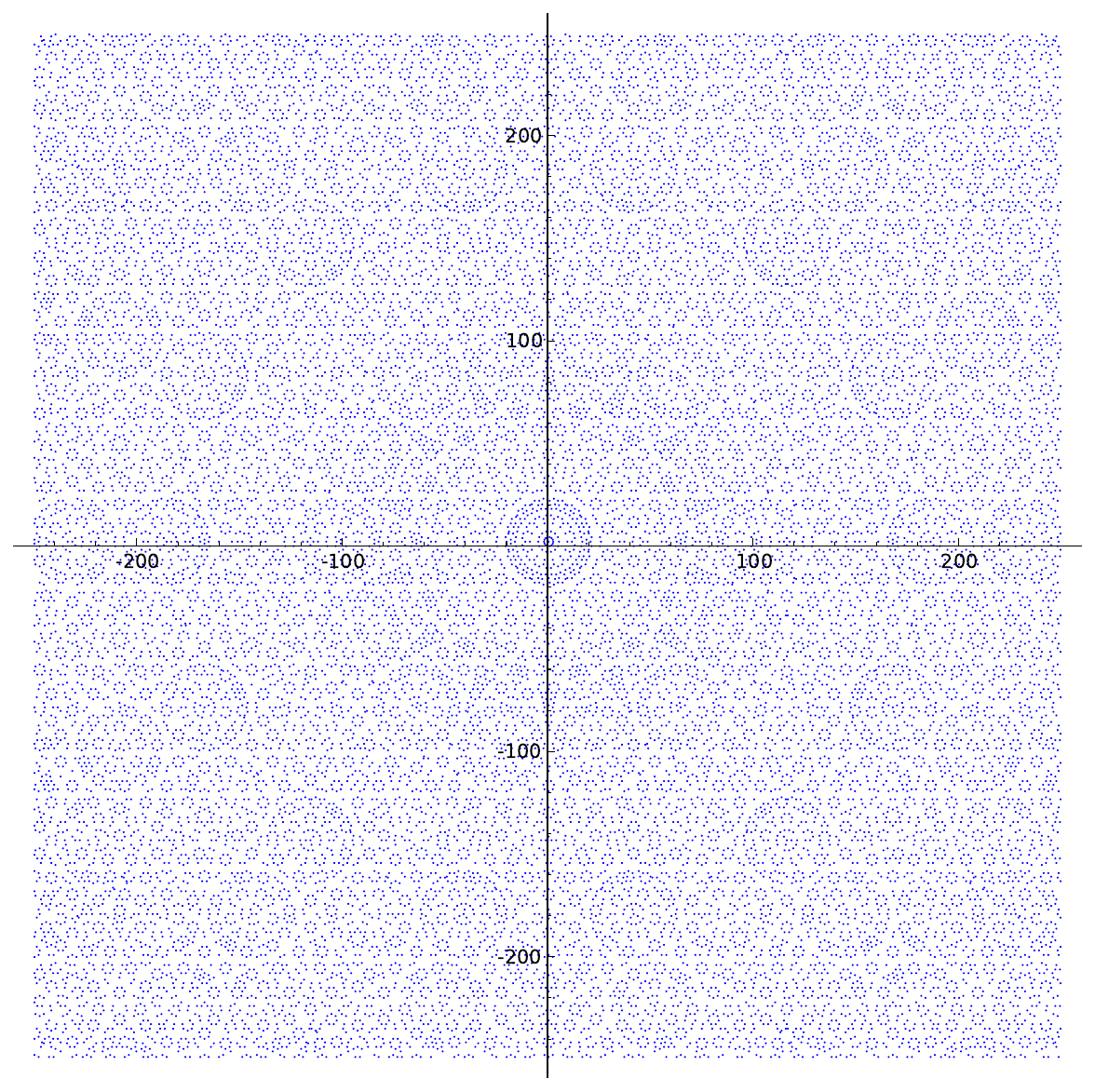}
\caption{$Q_{\lambda_7}(P_{14})$, the $\lambda_7$-convex closure of a regular $14$-gon.}\label{fig:14-gon}
\end{figure}
\begin{figure}
\centering
\includegraphics[width=1\textwidth]{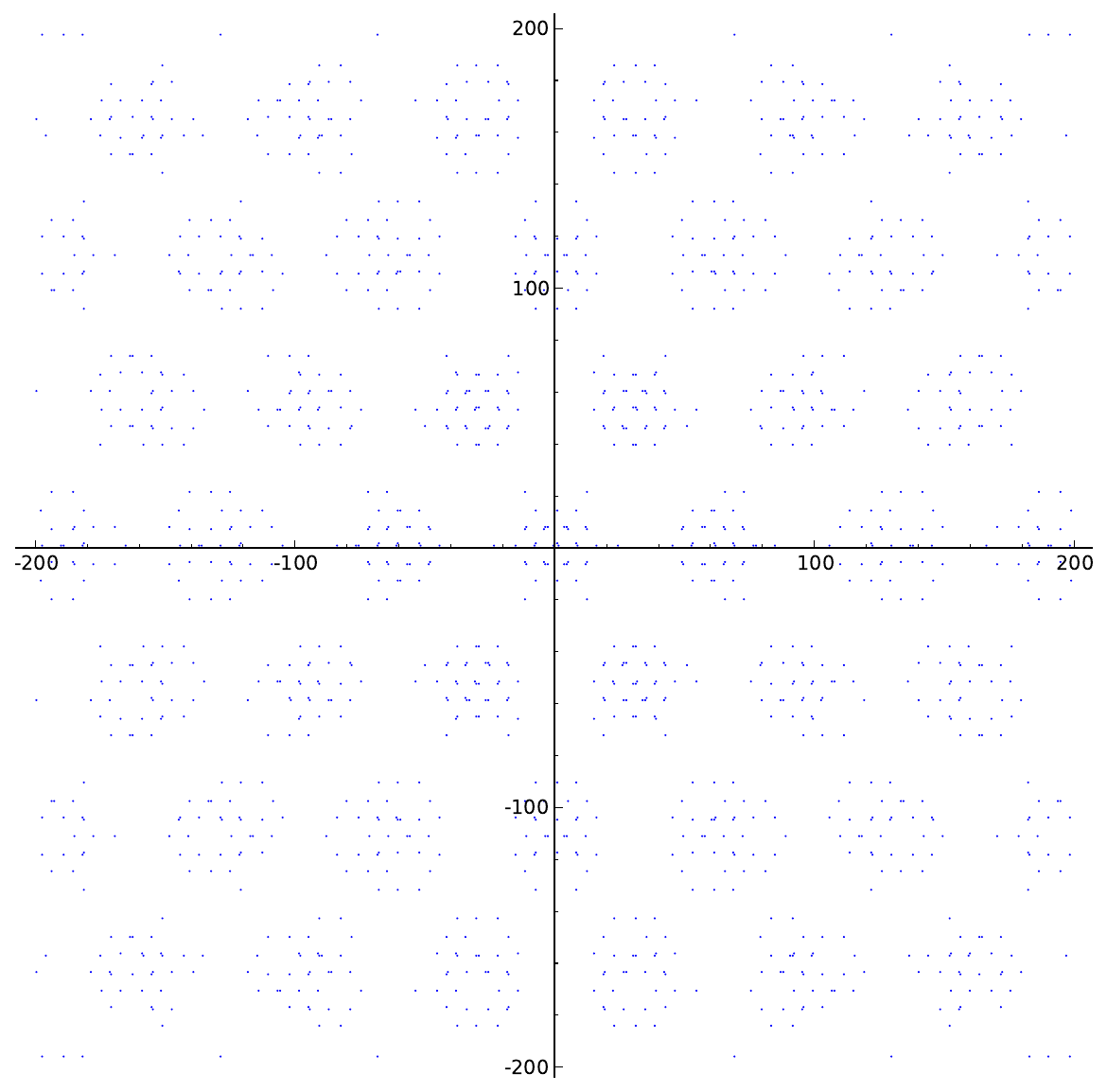}
\caption{$Q_{\lambda_9}(P_6)$, the $\lambda_9$-convex closure of a regular hexagon.}\label{fig:hexagon-lambda9-close}
\end{figure}
\begin{figure}
\centering
\includegraphics[width=1\textwidth]{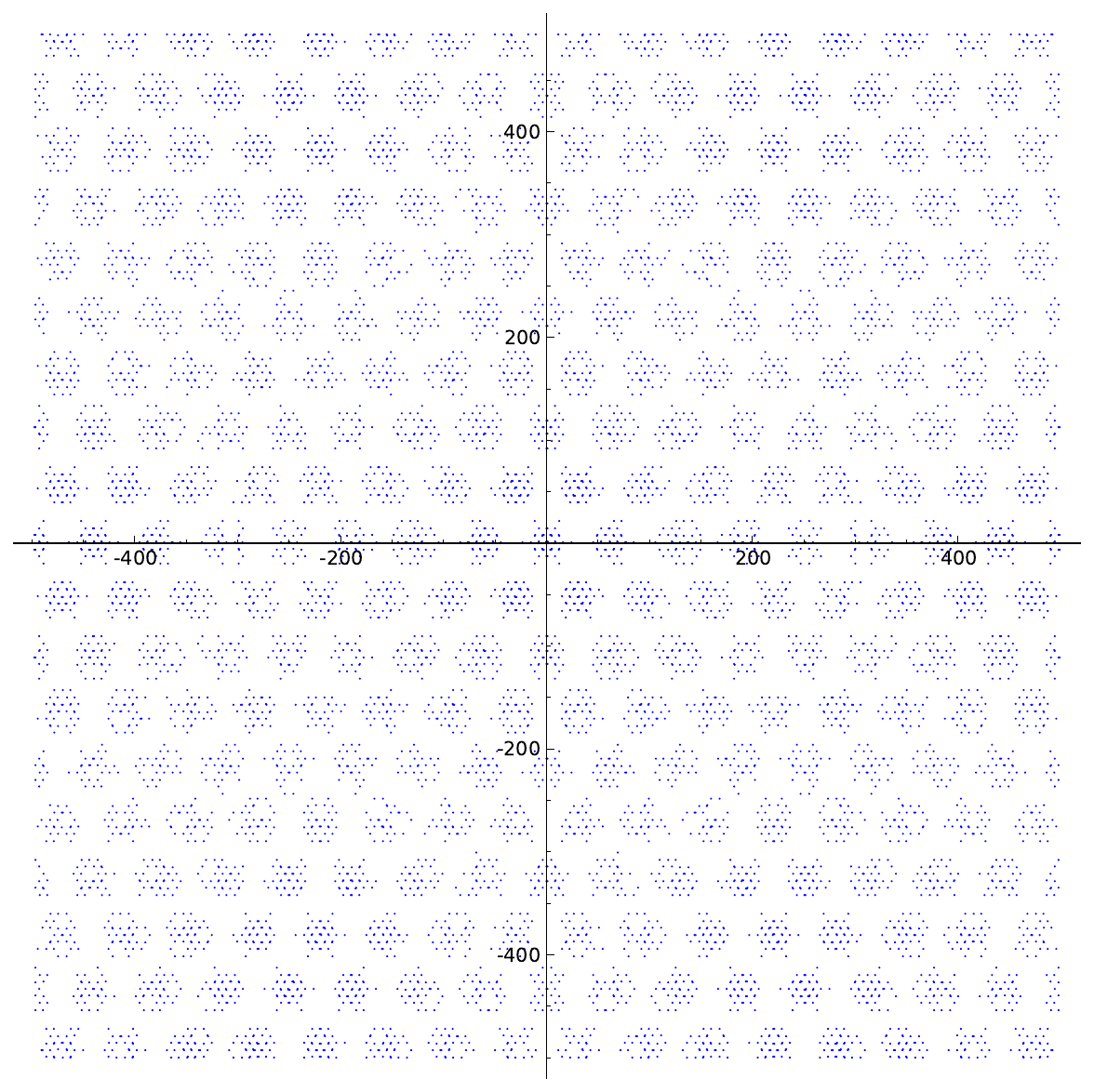}
\caption{A wider view of $Q_{\lambda_9}(P_6)$.}\label{fig:hexagon-lambda9-wide}
\end{figure}
\begin{figure}
\centering
\includegraphics[width=1\textwidth]{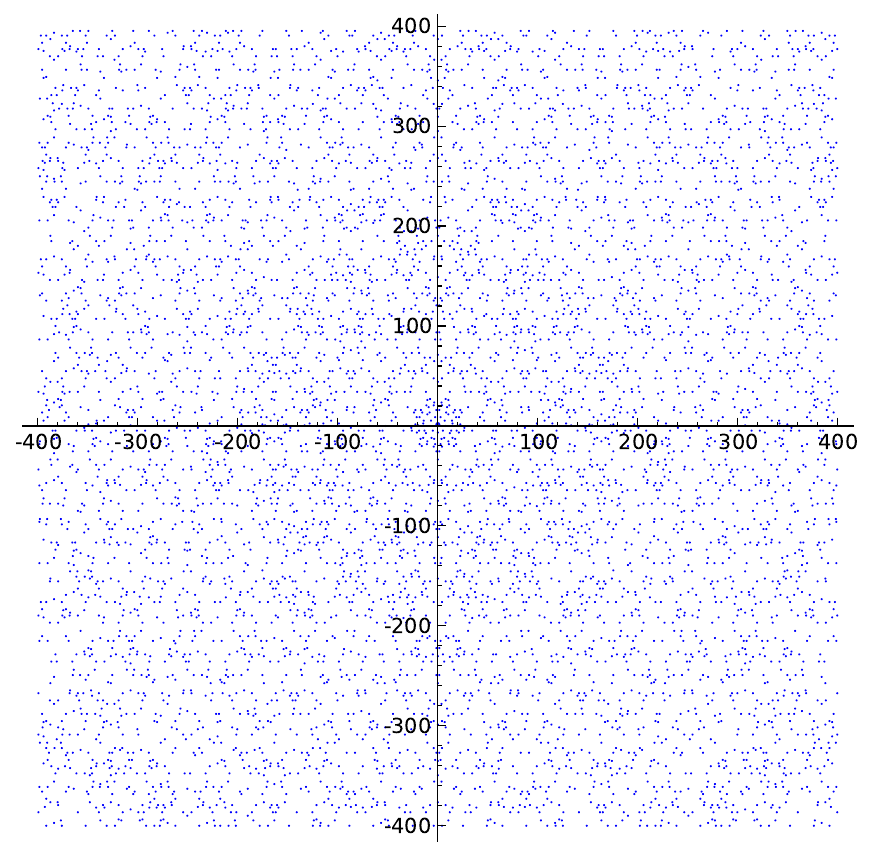}
\caption{$Q_{\lambda_9}(P_9)$, the $\lambda_9$-convex closure of a regular enneagon (nonagon).}\label{fig:enneagon-close}
\end{figure}
\begin{figure}
\centering
\includegraphics[width=1\textwidth]{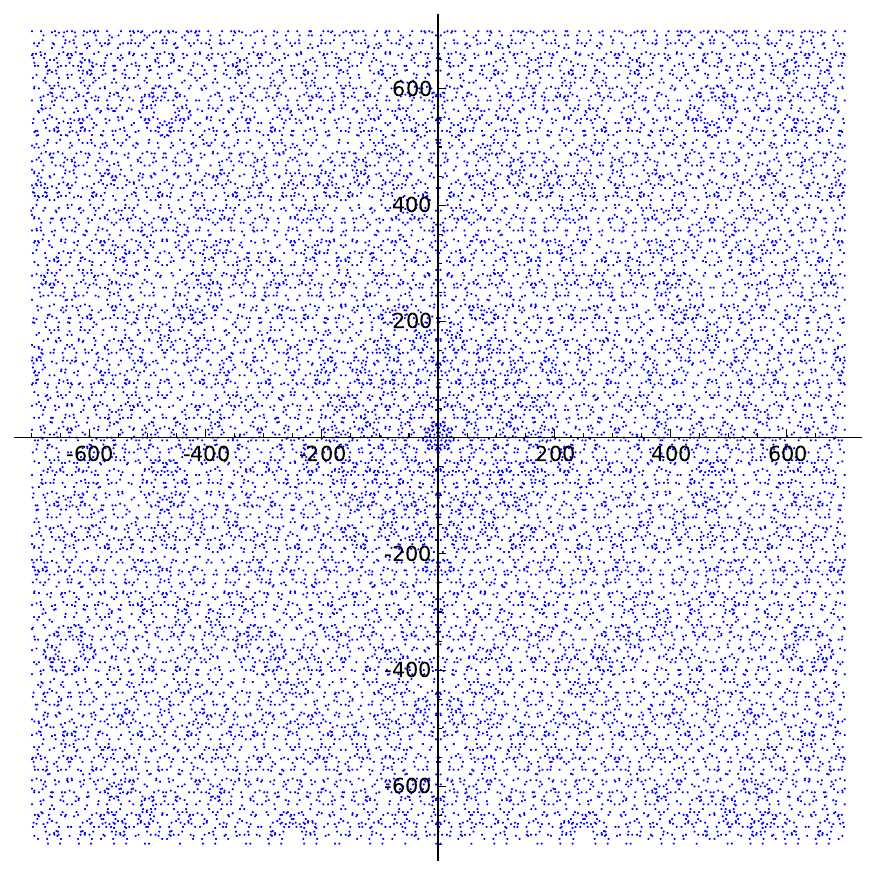}
\caption{A wider view of $Q_{\lambda_9}(P_9)$.}\label{fig:enneagon-wide}
\end{figure}
\begin{figure}
\centering
\includegraphics[width=1\textwidth]{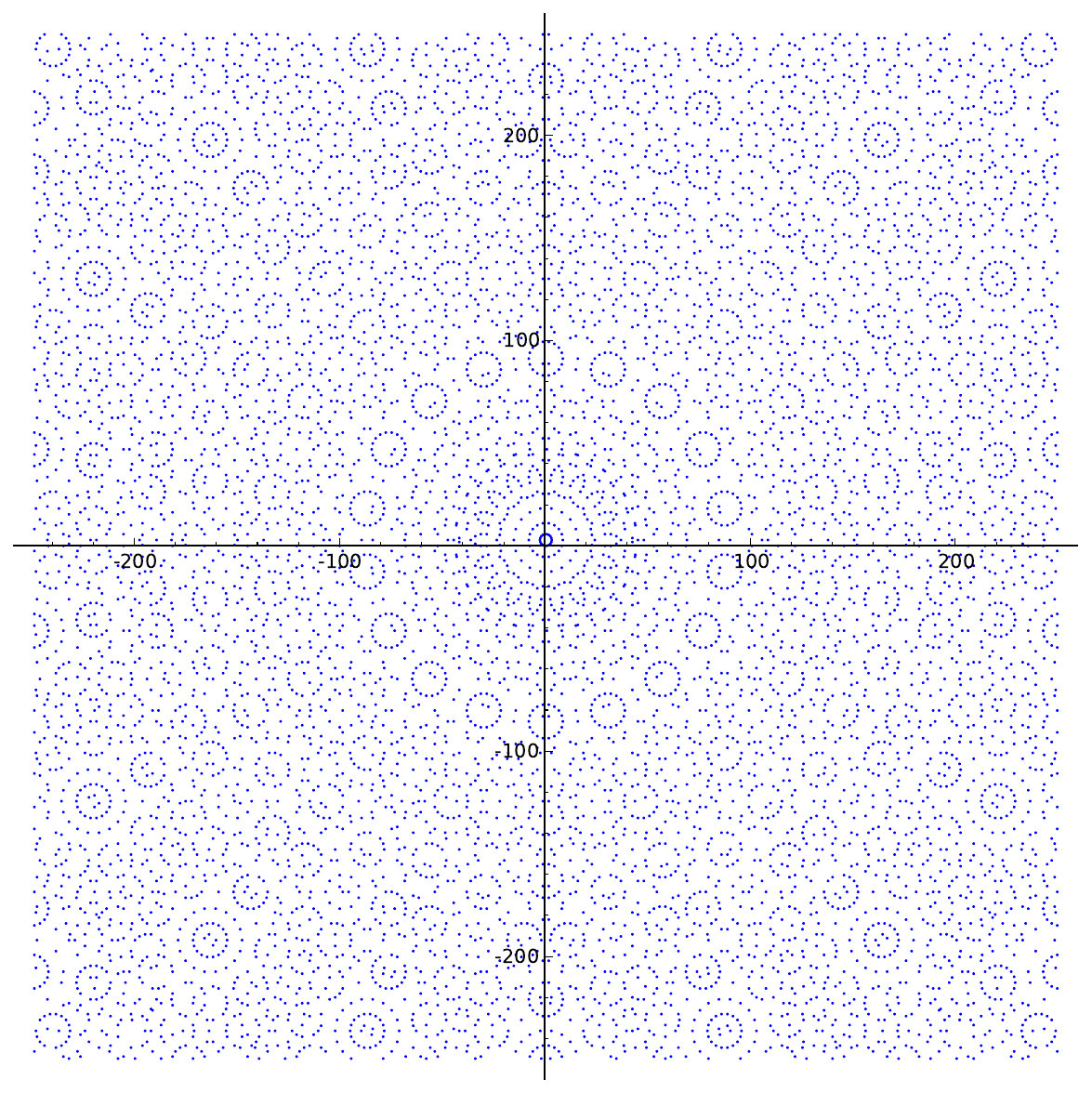}
\caption{$Q_{\lambda_9}(P_{18})$, the $\lambda_9$-convex closure of a regular $18$-gon.}\label{fig:18gon-close}
\end{figure}
\begin{figure}
\centering
\includegraphics[width=1\textwidth]{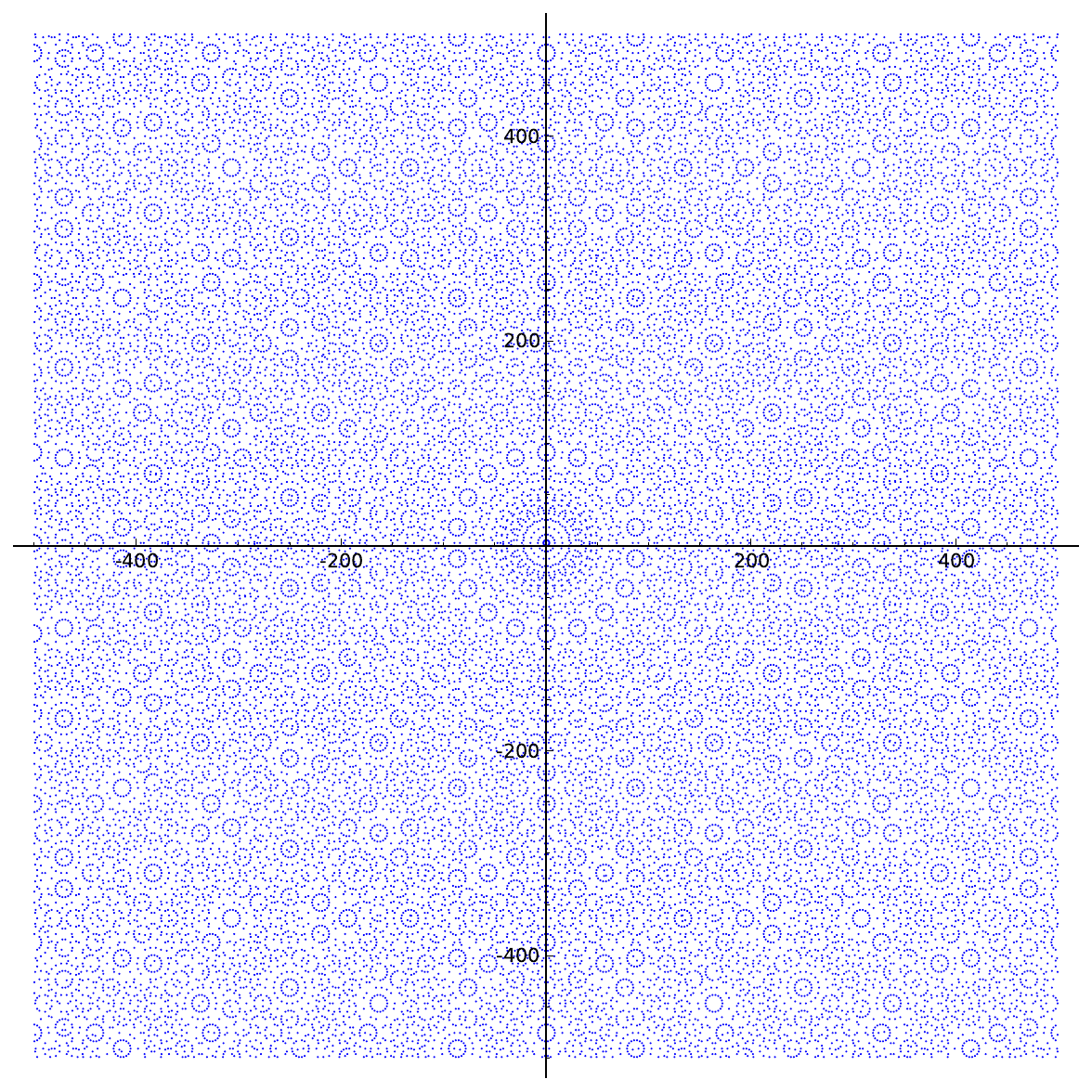}
\caption{A wider view of $Q_{\lambda_9}(P_{18})$.}\label{fig:18gon-wide}
\end{figure}
\begin{figure}
\centering
\includegraphics[width=1\textwidth]{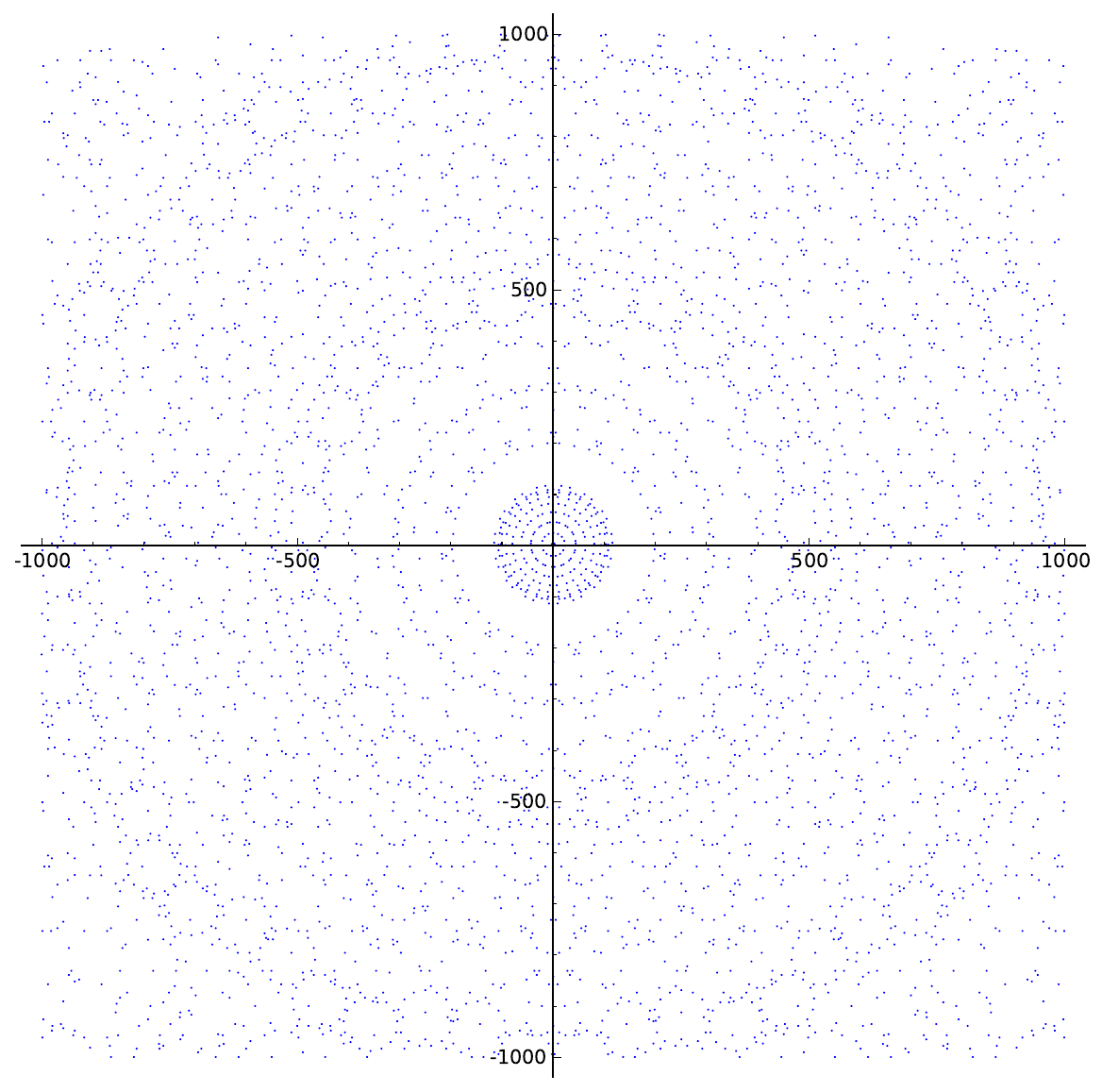}
\caption{$Q_{\lambda_{15}}(P_{15})$, the $\lambda_{15}$-convex closure of a regular $15$-gon.}\label{fig:15gon}
\end{figure}
\begin{figure}
\centering
\includegraphics[width=1\textwidth]{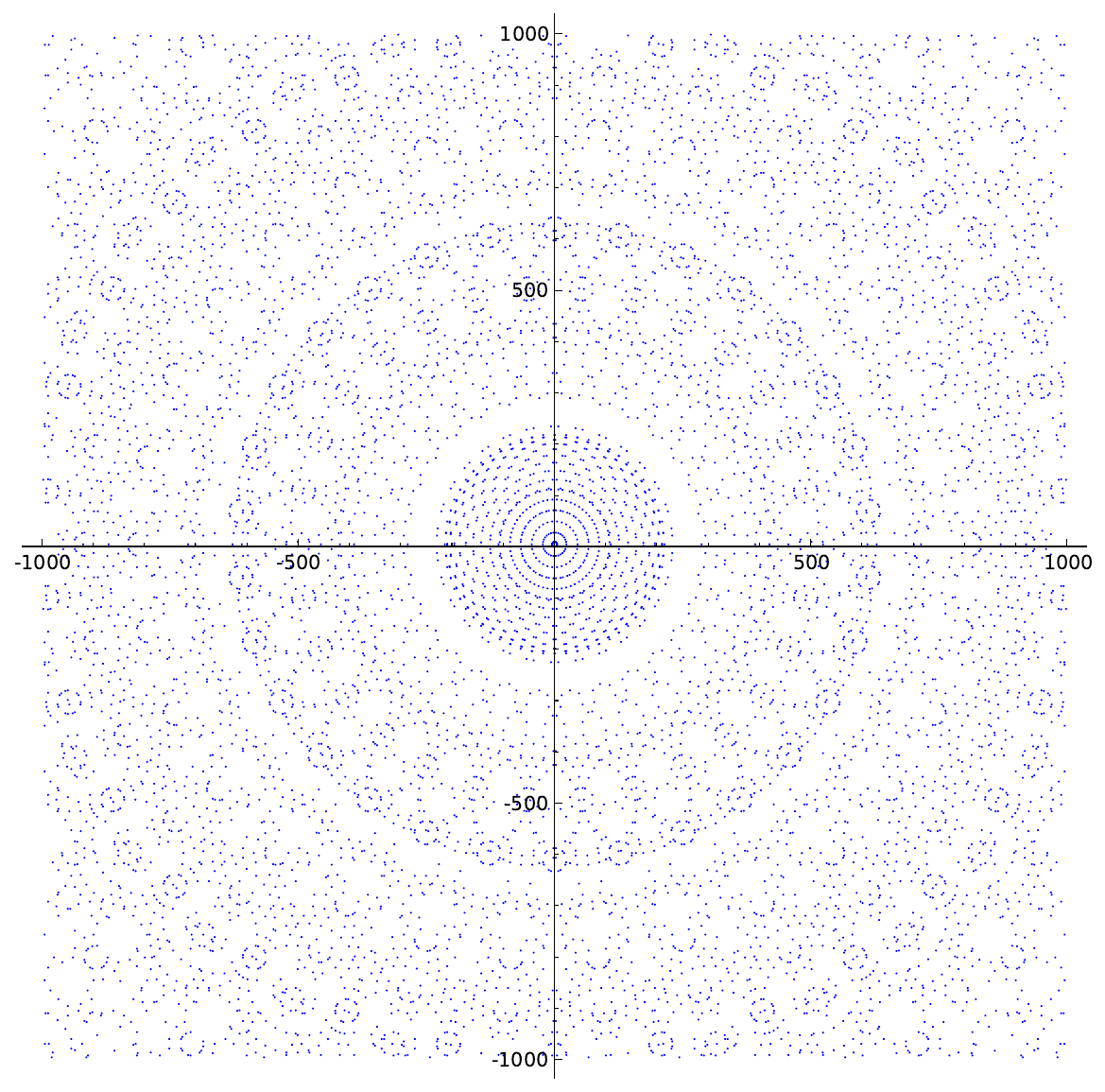}
\caption{$Q_{\lambda_{15}}(P_{30})$, the $\lambda_{15}$-convex closure of a regular $30$-gon.}\label{fig:30gon}
\end{figure}

\subsection{Constructions with even $n$}

Although the first part of Equation~(\ref{eqn:lambda-n}) makes sense when $n$ is even, $\lambda_n$ may or may not be an algebraic integer in this case.  For example, $\lambda_4 = 1 + \sqrt{2}/2$ is not an algebraic integer, but $2\lambda_4 = 2+\sqrt 2$ and $\lambda_6 = 2 + \sqrt{3}$ both are, and are two of the strong PV numbers given after Corollary~\ref{cor:degree2}.

\begin{proposition}\label{prop:octagon}
The sets $Q_{2+\sqrt 2}(P_8)-Q_{2+\sqrt 2}(P_8)$ and $Q_{2+\sqrt 3}(P_{12})-Q_{2+\sqrt 3}(P_{12})$ are both uniformly discrete.
\end{proposition}

\begin{proof}
By symmetry, we have $P_8 = \rho_{1,1+c}(P_8)$ for $c := e^{i\tau/8} = \frac{\sqrt 2}{2}(1 + i)$, and $P_{12} = \rho_{1,1+d}(P_{12})$ for $d := e^{i\tau/12} = \frac{1}{2}(\sqrt 3 + i)$.  We have $\Re(P_8) = \left\{-\frac{\sqrt 2}{2},0,1,1+\frac{\sqrt 2}{2}\right\}\subseteq \rats\left(2+\sqrt 2\right)$, and thus $\Re(P_8-P_8)\subseteq\rats\left(2+\sqrt 2\right)$ as well.  Similarly, $\Re(P_{12}) = \left\{-\frac{1+\sqrt 3}{2},-\frac{\sqrt 3}{2},0,1,1+\frac{\sqrt 3}{2},1+\frac{1+\sqrt 3}{2}\right\}\subseteq \rats\left(2+\sqrt 3\right)$, and thus $\Re(P_{12}-P_{12})\subseteq\rats\left(2+\sqrt 3\right)$ also.

As in the proofs of Propositions~\ref{prop:easy-n-gons} and \ref{prop:discrete-reg-n-gon}, the statement now follows from Theorem~\ref{thm:main} and Lemmas~\ref{lem:distribute}, \ref{lem:rot-symmetry}, and \ref{lem:minkowski-difference}.
\end{proof}

Figure~\ref{fig:hex-oct} shows three point sets that do not fit the ``odd $n$'' pattern---the $2$-convex closure of $P_6$ (upper left), the $(2+\sqrt 2)$-convex closure of $P_8$ (upper right), and the $(2+\sqrt 3)$-convex closure of $P_{12}$.  Note that $2+\sqrt 2 = 2\lambda_4$ and $2+\sqrt 3 = \lambda_6$.
These sets are all Meyer sets by Corollary~\ref{cor:P-n-meyer-set}.
\begin{figure}
\centering
\includegraphics[width=0.4\textwidth]{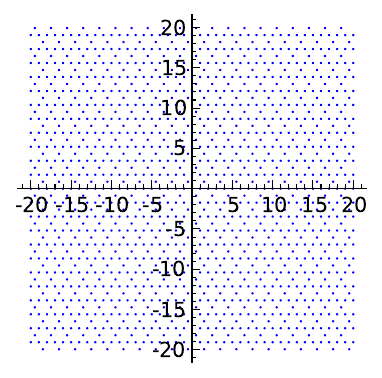}
\includegraphics[width=0.4\textwidth]{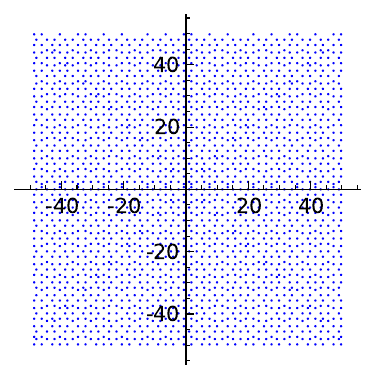}
\includegraphics[width=0.8\textwidth]{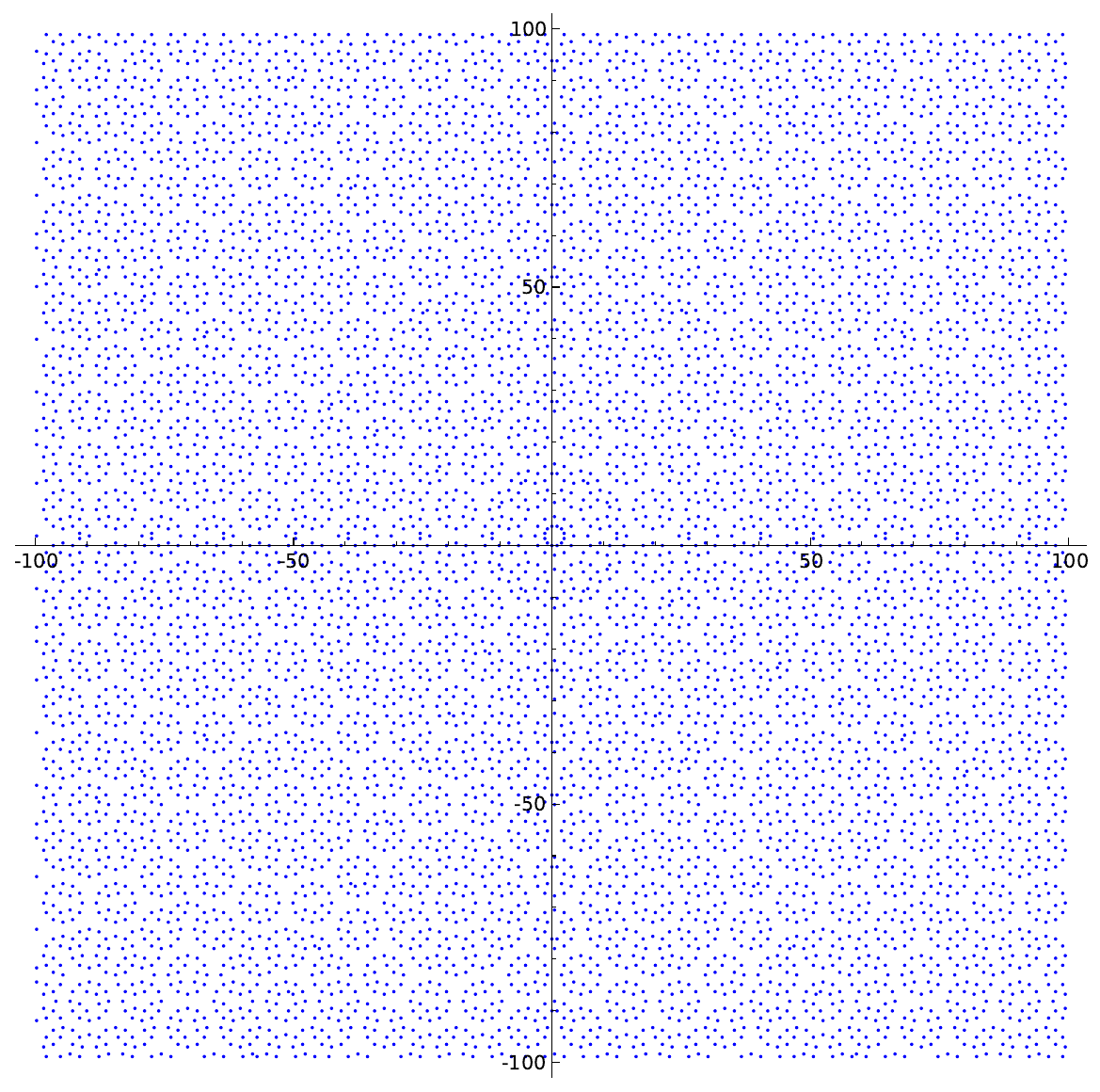}
\caption{$Q_2(P_6)$ (top left), $Q_{2+\sqrt 2}(P_8)$ (top right), and $Q_{2+\sqrt 3}(P_{12})$ (bottom).  $Q_2(P_6) = \{a+b\eta \mid a,b\in\ints \myand ab \equiv 0 \pmod{2} \}$, where $\eta := e^{i\tau/3}$ is the principal third root of unity.  All three sets are uniformly discrete, with any two points of any set at least unit distance apart.  Of these three, only $Q_2(P_6)$ is periodic.}\label{fig:hex-oct}
\end{figure}

\subsection{Regular convex polyhedra}

In this section, we investigate cases where $\lambda$-convex closures of the five regular convex polyhedra (Platonic solids) in $\reals^3$ are uniformly discrete, for real $\lambda$.  We show that for any real sPV $\lambda$, the $\lambda$-convex closure of the tetrahedron, cube, and octahedron are uniformly discrete.  Berman \& Moody showed that the $(1+\p)$-convex closures of the dodecahedron and icosahedron are both uniformly discrete using icosians (a finite subgroup of the multiplicative group of the quaterions)~\cite{BeMo}.  We give an independent, geometry-based proof.  We prove all these results using the projection technique we first used in the proof of Lemma~\ref{lem:rot-symmetry}.  In fact, one can think of the next proposition as a three-dimensional version of Lemma~\ref{lem:rot-symmetry}.

\begin{proposition}
For any sPV $\lambda\in\reals$, the $\lambda$-convex closures of the sets of corners of the regular tetrahedron, cube, and octahedron are all uniformly discrete point sets in $\reals^3$.
\end{proposition}

\begin{proof}
We prove the result for the regular tetrahedron; the other two have similar proofs.  For concreteness---and consistency with how we handled the regular polygons above---we orient the tetrahedron in $\reals^3$ so that
\begin{itemize}
\item
one of its edges coincides with the set $\{(x,0,0) \mid 0\le x\le 1\}$,
\item
one of the faces incident to this edge lies in the half-plane $\{(x,y,0) \mid y \ge 0\}$, and
\item
the tetrahedron itself lies in the half-space $\{(x,y,z) \mid z \ge 0\}$.
\end{itemize}
We let $T\subseteq\reals^3$ be the four corners of this tetrahedron.  The perpendicular projection of $T$ into the $x,y$-plane is
\[ T' := \{(0,0), (1,0), (1/2,\sqrt 3/2), (1/2,\sqrt 3/6)\} \subseteq \reals^2\;, \]
shown below, with edges included:
\begin{center}
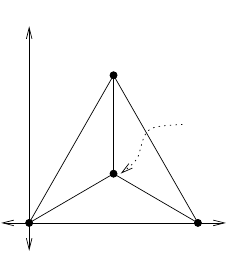
\end{center}
Projecting $T'$ perpendicularly onto the $x$-axis yields the set $T'' = \{0,\frac{1}{2},1\}$.  Since $T'' \subseteq \rats \subseteq \rats(\lambda)$, it follows from Theorem~\ref{thm:main} that $Q_\lambda(T'')$ is uniformly discrete, and from this it follows (using Lemma~\ref{lem:rho-map}) that $Q_\lambda(T')$ is included in the union $U$ of a discrete set of vertical lines with a uniform positive lower bound on interline spacing.  Projecting $T'$ perpendicularly into a line through one of the other outer edges---the line through the origin and $\left(\frac{1}{2},\frac{\sqrt 3}{2}\right)$, say---shows that $Q_\lambda(T')$ is included in a rotated copy $U'$ of $U$ whose lines are not parallel with those of $U$.  Thus $Q_\lambda(T')\subseteq U\intersect U'$, which is uniformly discrete.  But now $Q_\lambda(T')$ is the perpendicular projection of $Q_\lambda(T)$ onto the $x,y$-plane (by Lemma~\ref{lem:rho-map} again), and so $Q_\lambda(T)$ is confined to the union $V$ of a discrete set of lines in $\reals^3$ parallel to the $z$-axis with a uniform positive lower bound on interline spacing.  Projecting $Q_\lambda(T)$ perpendicularly into the plane containing one of $T$'s other faces shows that $Q_\lambda(T)$ is included in a rotated copy $V'$ of $V$.  Thus $Q_\lambda(T)\subseteq V\intersect V'$, which is uniformly discrete.
\end{proof}

We turn now to geometry-based proofs of the cases of the dodecahedron and the icosahedron.

\begin{proposition}[Berman \& Moody~\cite{BeMo}]
The $(1+\p)$-convex closure of the regular dodecahedron and regular icosahedron in $\reals^3$ are both uniformly discrete.
\end{proposition}

\begin{proof}
We set $\lambda := 1+\p = (3+\sqrt 5)/2$.  Let $D$ be the vertices of a regular dodecahedron in $\reals^3$ (its orientation and location are not important).  It turns out, quite fortunately, that if we project $Q_\lambda(D)$ perpendicularly into the plane containing one of $D$'s faces, we get a point-set equal (up to similarity) to $Q_\lambda(P_{10})$, shown in Figure~\ref{fig:decagon}, which is discrete by Propositions~\ref{prop:discrete-reg-n-gon} and \ref{prop:lambda-n-PV}.  (Note that $\lambda = \lambda_5$.)

To see this, consider the perpendicular projection of $D$ into the plane containing one of its faces, shown in Figure~\ref{fig:dodecahedron}.
\begin{figure}
\centering
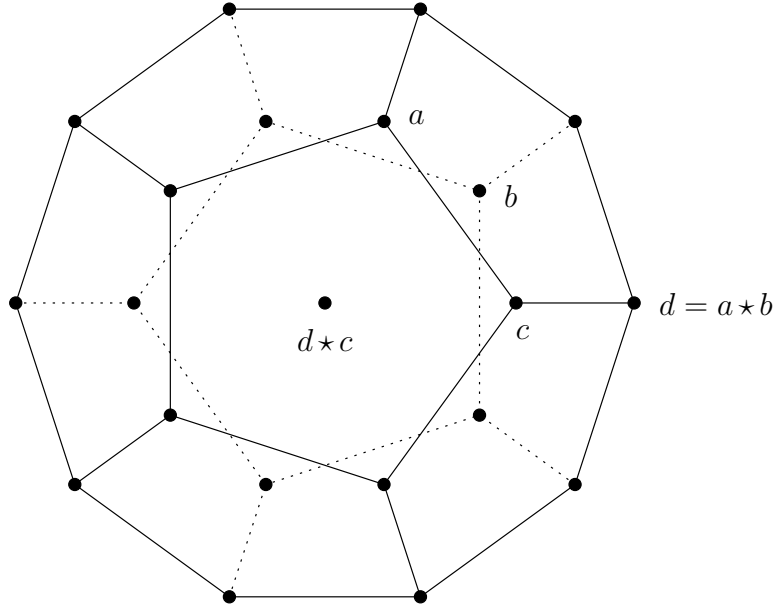
\caption{The projection of a dodecahedron $D$ onto the plane of one of its faces.  The point in the middle is not part of the projection of $D$ itself, but of a point in $Q_\lambda(D)$ lying directly above the top face.  Here, $\x = \x_{1+\p}$.}\label{fig:dodecahedron}
\end{figure}
Call this projection $D' \subseteq \reals^2$.  $D'$ consists of two concentric decagons.  Points in the outer decagon are $\lambda$-extrapolants of pairs of adjacent points of the inner decagon; for example, $d = a \x b$ in the figure.  Letting $S$ be the inner decagon (a copy of $P_{10}$), we see that all the points of $D'$ are therefore in $Q_\lambda(S)$.  Moreover, we know that $Q_\lambda(S)$ is uniformly discrete by Propositions~\ref{prop:discrete-reg-n-gon} and \ref{prop:lambda-n-PV}.  Thus $Q_\lambda(D')$ is uniformly discrete.  Notice that the center point $d\x c$ shown in Figure~\ref{fig:dodecahedron} is also in $Q_\lambda(D')$ (though not in $D'$ itself).

As with previous arguments, we conclude that $Q_\lambda(D)$ is confined to the union $U$ of a discrete set of parallel lines in $\reals^3$ with a uniform positive lower bound on interline spacing.  By symmetry, we can project onto another, nonparallel face to see that $Q_\lambda(D)$ is also contained in a rotated copy of $U$, hence $Q_\lambda(D)$ must be uniformly discrete.  This finishes the case of the dodecahedron.

Consider the top pentagonal face of $D$ whose projection is shown in Figure~\ref{fig:dodecahedron} centered and bounded by solid lines.  Call this face $F$.  The center point $d\x c \in Q_\lambda(D')$, shown in Figure~\ref{fig:dodecahedron}, is actually the projection of a point $e\in\reals^3$ lying directly above $F$.  The point $e$ is the intersection of lines extending the five edges of $D$ that are incident to the vertices of $F$ but are not edges of $F$ itself (their projections appear as the solid radial lines in the figure).  The point $e$ is in $Q_\lambda(D)$, because $e$ is on each of these lines, and its projection $d\x c$ is the $\lambda$-extrapolation of points in $Q_\lambda(D')$.  By symmetry, there are points lying similarly above the other eleven faces of $D$.  These twelve points, all in $Q_\lambda(D)$ are the vertices of a regular icosahedron $I$.  Since $I\subseteq Q_\lambda(D)$, we have $Q_\lambda(I)\subseteq Q_\lambda(D)$, and thus $Q_\lambda(I)$ is uniformly discrete.
\end{proof}


\subsection{Other results}
\label{sec:other-results}

In this section we give some \textit{ad hoc} results about specific point sets.  For example, we get a lower bound of $1/2$ on the packing distances of some of the uniformly discrete sets we considered in the last subsection.  We make no attempt at generalizing our results, although such generalizations are most likely possible.

If $\lambda = 2$, then $R_2(P_6) = Q_2(P_6)$ is a proper subset of the Eisenstein integers (the regular triangular lattice) that has both translational symmetry (in six directions) and $D_6$ rotational symmetry\footnote{$D_n$ is the dihedral group of order $2n$.} about the center of the hexagon.  For another example, if $S = \{0,1,i,1+i\}$ is the set of vertices of a square (similarly positioned), then $R_2(S) = Q_2(S) = \ints[i]$, the set of Gaussian integers.  If $L = \{0,1,i\}$, then $R_2(L)$ is a proper subset of $\ints[i]$, however, as the following lemma shows:

\begin{lemma}\label{lem:Q2-gamma}
Let $\gamma$ be any element of $\complexes \cmpl \rats$.  Then
\begin{equation}\label{eqn:Q2-gamma}
Q_2(\{0,1,\gamma\}) = \{ a + b\gamma \mid a,b\in\ints \myand ab\equiv 0 \pmod{2} \}\;.
\end{equation}
If $\gamma\in\reals$, then $R_2(\{0,1,\gamma\}) = \reals$; otherwise, $R_2(\{0,1,\gamma\}) = Q_2(\{0,1,\gamma\})$, which is a discrete set.
\end{lemma}

\begin{proof}
It is easy to check, using the fact that $x \x_2 y = 2y - x$ for any $x$ and $y$, that the right-hand side of (\ref{eqn:Q2-gamma}) contains $\{0,1,\gamma\}$ and is $2$-convex.  This proves $\subseteq$.

For the reverse inclusion, we first note that $Q_2(\{0,1\}) = \ints$ and $Q_2(\{0,\gamma\}) = \gamma\ints$, and thus $\ints \union \gamma\ints \subseteq Q_2(\{0,1,\gamma\})$.  Now let $a,b\in\ints$ be arbitrary.  If $b$ is even, then $a + b\gamma = (-a) \x_2 (b\gamma/2)$.  If $a$ is even, then $a + b\gamma = (-b\gamma) \x_2 (a/2)$.  In either case, this shows that $a+b\gamma \in Q_2(\ints\union\gamma\ints) = Q_2(\{0,1,\gamma\})$, which establishes the reverse containment.

If $\gamma$ is a real, irrational number, then (\ref{eqn:Q2-gamma}) implies that $Q_2(\{0,1,\gamma\})$ is a dense subset of $\reals$, and so $R_2(\{0,1,\gamma\}) = \reals$.  If $\gamma\notin\reals$, then $Q_2(\{0,1,\gamma\})$ is a subset of the discrete lattice generated by $1$ and $\gamma$.
\end{proof}

Considering $Q_2(P_6)$ again, one can easily check that $P_6 \subseteq Q(\{0,1,\eta\})$ where $\eta := e^{i\tau/3}$ is the principal third root of unity.  Thus, by Lemma~\ref{lem:Q2-gamma}, $Q_2(P_6) = Q_2(\{0,1,\eta\}) = \{a+b\eta \mid a,b\in\ints \myand ab\equiv 0 \pmod{2} \}$, and this is one of the sets shown in Figure~\ref{fig:hex-oct}.

\bigskip

We now reconsider the set $Q_{1+\p}(P_5)$, illustrated in Figure~\ref{fig:pentagon}.  Like $Q_2(P_6)$, this set is uniformly discrete by Propositions~\ref{prop:discrete-reg-n-gon} and \ref{prop:lambda-n-PV} and has $D_5$ symmetry about the center of $P_5$.  Unlike $Q_2(P_6)$, however, $Q_{1+\p}(P_5)$ has no translational symmetry.

\begin{theorem}\label{thm:pentagon}
Except for adjacent points in $P_5$, all points of $Q_{1+\p}(P_5)$ are farther than unit distance apart.
\end{theorem}

\begin{proof}
%
Let $\lambda := 1+\p$ and let $\x$ mean $\xl$ as usual.  Let $S := Q_\lambda(P_5)$.  We again use the projection idea from the proof of Proposition~\ref{prop:easy-n-gons}, but for technical convenience, we project horizontally onto the imaginary axis via $\Im$ rather than vertically onto the real axis via $\Re$.  A little elementary geometry shows that $\Im(P_5) = \{0, \zeta, \kappa\}$, where
\begin{align*}
\zeta &:= \sin(\tau/5) = \frac{\sqrt 2}{4}\sqrt{5+\sqrt 5}\approx 0.951\;, & \kappa &:= \p\zeta = \frac{\sqrt{5+2\sqrt 5}}{2}\approx 1.539\;.
\end{align*}
From this we readily get $0 = -\p\kappa + \lambda\zeta = \kappa \x \zeta \in Q_\lambda(\{\kappa,\zeta\})$, and therefore
\[ \Im(S) = Q_\lambda(\Im(P_5)) = Q_\lambda(\{0,\zeta,\kappa\}) = Q_\lambda(\{\kappa,\zeta\}) = \rho_{\kappa,\zeta}(Q_\lambda)\;. \]
The right-hand side is discrete by Proposition~\ref{prop:1-plus-phi}, which shows that $\Im(S)$ is discrete.


Fix any $a,b\in S$ such that $0 < |b-a| \le 1$.  We now show that $|b-a| = 1$ and $a,b\in P$.  Let $C$ be the center of $P_5$, and let $G$ be the group of rotations about $C$ through angles that are multiples of $\tau/5$.  ($G$ is the $5$-element cyclic group generated by $\rho_{\xi,0}$, where $\xi := e^{i3\tau/10}$.)  By symmetry, each element of $G$ leaves $P_5$, and therefore $S$, invariant.

\begin{claim}
There exist distinct $g_1,g_2\in G$ such that
\begin{enumerate}
\item
$\{\Im(g_1(a)),\Im(g_1(b))\} = \{\Im(g_2(a)),\Im(g_2(b))\} = \{\zeta,\kappa\}$,
\item
the slope of the line through $g_1(a)$ and $g_1(b)$ is $\tan(\tau/10)$, and
\item
the slope of the line through $g_2(a)$ and $g_2(b)$ is $-\tan(\tau/10)$.
\end{enumerate}
Furthermore, $|b-a| = 1$.
\end{claim}

\begin{proof}[Proof of the Claim]
Let $\theta$ be the argument of $b-a$, and let $r = |b-a|$.  By assumption, $0<r\le 1$.  The elements of $G$ rotate the line segment connecting $a$ with $b$ to form line segments with length $r$ and with arguments $\theta + k\tau/5$, for $k\in\{0,1,2,3,4\}$.  The vertical displacement of each such line segment (that is, the absolute difference between the imaginary parts of the two endpoints) is thus $r|\sin(\theta + k\tau/5)|$.  A simple geometric argument shows that there exist distinct $k_1,k_2\in\{0,1,2,3,4\}$ such that $0<|\sin(\theta + k_1\tau/5)| \le \sin(\tau/10)$ and $0<|\sin(\theta + k_2\tau/5)| \le \sin(\tau/10)$.  Let $g_1$ and $g_2$ be the corresponding elements of $G$, respectively.  Now we note that
\[ \sin(\tau/10) = \frac{\sqrt 2}{4}\sqrt{5-\sqrt 5} = \kappa - \zeta \approx 0.588\;. \]
We thus have
\begin{align}\label{eqn:vertical-displace}
0 < |\Im(g_1(b)) - \Im(g_1(a))| &\le \kappa - \zeta\;, & 0 < |\Im(g_2(b)) - \Im(g_2(a))| &\le \kappa - \zeta\;.
\end{align}
We know that $g_1(a),g_1(b),g_2(a),g_2(b)\in S$, and we established earlier that $\Im(S) = \rho_{\kappa,\zeta}(Q_\lambda)$.  Thus $\Im(g_j(a))$ and $\Im(g_j(b))$ are both in $\rho_{\kappa,\zeta}(Q_\lambda)$ for $j\in\{1,2\}$.  Now Proposition~\ref{prop:1-plus-phi} implies that, except for $\kappa$ and $\zeta$, any two points of $\rho_{\kappa,\zeta}(Q_\lambda)$ differ by at least $(\kappa - \zeta)\p = \zeta$, which is strictly greater than $\kappa - \zeta$.  This establishes the first item of the Claim, and it also implies that equality must hold for both inequalities in (\ref{eqn:vertical-displace}).  Therefore, for both $j\in\{1,2\}$,
\begin{equation*}
\sin(\tau/10) = \kappa - \zeta = |\Im(g_j(b)) - \Im(g_j(a))| = r|\sin(\theta + k_j\tau/5)| \le |\sin(\theta + k_j\tau/5)| \le \sin(\tau/10)\;.
\end{equation*}
The only way this can occur is when $r = 1$ and $|\sin(\theta+k_j\tau/5)| = \sin(\tau/10)$.  Since $g_1\ne g_2$, it must then be that the line through $g_1(a)$ and $g_1(b)$ and the line through $g_2(a)$ and $g_2(b)$ have oppositely signed slopes, both with absolute value $\tan(\tau/10)$.  By swapping $g_1$ and $g_2$ if necessary, we can assume that the line through $g_1(a)$ and $g_1(b)$ has positive slope.  This establishes the rest of the Claim.
\end{proof}

It remains to show that $a,b\in P_5$, and for this it suffices that $g_1(a)$ and $g_1(b)$ are both in $P_5$.  By swapping $a$ and $b$ if necessary, we can assume that $\Im(g_1(a)) = \zeta$ and $\Im(g_1(b)) = \kappa$.  To get the slope of the line between $g_2(a)$ and $g_2(b)$ to be $-\tan(\tau/10)$, as asserted by the Claim, it must be that $g_2$ results from applying $g_1$ followed by a clockwise rotation around $C$ through angle $\tau/5$.  Thus we have $\Im(g_2(a)) = \kappa$ and $\Im(g_2(b)) = \zeta$, and this is only possible if $g_1(b) = g_2(a)$ is the apex of $P_5$ and $g_1(a)$ is the element of $P_5$ immediately to its left.  This finishes the proof of Theorem~\ref{thm:pentagon}.
\end{proof}

\begin{corollary}
All $\reals$-affine transformations of $\complexes$ that are symmetries of $Q_{1+\p}(P_5)$ are length-preserving and leave $P_5$ invariant.
\end{corollary}

We turn to the case of the regular octagon.  The $(2+\sqrt 2)$-clonvex closure of $P_8$ is illustrated in the upper right portion of Figure~\ref{fig:hex-oct}.  As with $Q_{1+\p}(P_5)$, this set is discrete and has no translational symmetry, although it has $D_8$ symmetry about the center of $P_8$.  The same techniques are used to prove the following, whose proof we only sketch:

\begin{theorem}\label{thm:octagon}
All points in $Q_{2+\sqrt 2}(P_8)$ are at least unit distance apart.
\end{theorem}

Unlike $Q_{1+\p}(P_5)$, \ $Q_{2+\sqrt 2}(P_8)$ has infinitely many pairs of points that are unit distance apart; they radiate out from $P_8$ in the eight directions whose angles are multiples of $\tau/8$, as can be seen in Figure~\ref{fig:hex-oct}.

\begin{proof}[Proof sketch of Theorem~\ref{thm:octagon}]
Let $S := S_\lambda(P_8)$.  $\Im(S)$ is a subset of $\rho_{\kappa,1}(Q_\lambda)$, where this time, $\kappa := \sqrt 2/2$.  Similarly, $\Re(S)$ is a subset of $\rho_{0,1-\kappa}(Q_\lambda)$.  $Q_\lambda$ is uniformly discrete by Corollary~\ref{cor:degree2}, where $m := -2$ and $n := 1$.  This makes both the real and imaginary parts of elements of $S$ drawn from discrete sets.  Thus $S$ is discrete.

By Corollary~\ref{cor:degree2} with $\lambda' = 1 - \lambda = -1-\sqrt 2$, any two distinct elements of $\Im(S)$ or of $\Re(S)$ differ by at least $(1-\kappa)(-\lambda') = (2-\sqrt 2)(1+\sqrt 2)/2 = \sqrt 2/2$.  It follows that if distinct $x,y\in S$ have different real parts and different imaginary parts, then $|x-y| \ge 1$.  If $x$ and $y$ have the same real part and different imaginary parts, or vice versa, then by symmetry we can rotate the plane about the center of $P_8$ through angle $\tau/8$ to obtain images $x'$ and $y'$, both in $S$, whose real parts and imaginary parts both differ.  Then we have $|x-y| = |x'-y'| \ge 1$.
\end{proof}

We conjecture that $Q_{2+\sqrt 2}(P_8)$ has no $\reals$-affine symmetries except those that leave $P_8$ invariant.

\begin{theorem}\label{thm:dodecagon}
All points in $Q_{2+\sqrt 3}(P_{12})$ are at least unit distance apart.
\end{theorem}

\begin{proof}
Let $\lambda := -1-\sqrt 3$.  Then $Q_{2+\sqrt 3}(P_{12}) = Q_{1-\lambda}(P_{12}) = Q_\lambda(P_{12})$ by Fact~\ref{fact:inner-dual}.  As in previous proofs we consider the projection $\Re(Q_\lambda(P_{12})) = Q_\lambda(\Re(P_{12}))$ onto the real axis.  We have
\[ \Re(P_{12}) = \{-(1+\sqrt 3)/2,-\sqrt 3/2,0,1,(2+\sqrt 3)/2,(3+\sqrt 3)/2\} = \frac{1}{2} S\;, \]
where $S \eqdf \{\lambda,1+\lambda,0,2,1-\lambda,2-\lambda\} \subseteq \ints[\lambda]$.  Since $\frac{1}{2}S$ is a finite subset of $\rats(\lambda)$, we know that $Q_\lambda(\Re(P_{12})) = Q_\lambda(\Re(P_{12})) = Q_\lambda\left(\frac{1}{2}S\right)$ is uniformly discrete by Theorem~\ref{thm:main}.  We first get a lower bound on the interpoint distances of $Q_\lambda\left(\frac{1}{2}S\right)$, then combine this with symmetry to get a lower bound on the interpoint distances of $Q_\lambda(P_{12})$.


The only conjugate of $\lambda$ is $\mu := \sqrt 3 - 1 \approx 0.732$.  Now we can apply Fact~\ref{fact:main} with $d := 2$, \ $k := 1$, \ $\mu_0 := \mu$, and
\[ S_0 := \left\{ a_0 + a_1\mu : a_0,a_1\in\ints \myand a_0 + a_1\lambda \in S \right\} = \{\mu,1+\mu,0,2,1-\mu,2-\mu\}\;, \]
and this makes $\ell_0 = \min(S_0) = 0$ and $h_0 = \max(S_0) = 2$.  Thus by Equation~(\ref{eqn:main3general}), we have
\begin{equation}\label{eqn:dodecagon}
Q_\lambda(S) \subseteq \{a_0 + a_1\lambda : a_0,a_1\in\ints \myand 0 \le a_0 + a_1\mu \le 2 \}\;.
\end{equation}

Now consider two distinct points of $Q_\lambda(S)$, say $a = a_0 + a_1\lambda$ and $b = b_0 + b_1\lambda$, where $a_0,a_1,b_0,b_1\in\ints$.  We show here that $|b-a| \ge 1$.  Let $\delta := b-a = \delta_0 + \delta_1\lambda$ where $\delta_0 = b_0 - a_0$ and $\delta_1 = b_1 - a_1$ are not both zero.  If $\delta_1 = 0$, then $\delta$ is a nonzero integer and we are done, so we can assume that $\delta_1 \ne 0$.  Also assume without loss of generality that $\delta_0 + \delta_1\mu > 0$ (otherwise swap the roles of $a$ and $b$).  Since $a_0 + a_1\mu$ and $b_0 + b_1\mu$ are both in $\clcl{0,2}$, we then have $0 < \delta_0 + \delta_1\mu < 2$, or equivalently, $-\delta_1\mu < \delta_0 < 2 - \delta_1\mu$.  Adding $\delta_1\lambda$ to all sides gives us
\[ -\delta_1\mu + \delta_1\lambda < \delta_0 + \delta_1\lambda = \delta < 2  - \delta_1\mu + \delta_1\lambda\;, \]
that is,
\[ -2\sqrt 3\, \delta_1 < \delta < 2 - 2\sqrt 3\,\delta_1\;, \]
using the fact that $\lambda - \mu = -2\sqrt 3$.  If $\delta_1 < 0$, then $\delta > 2\sqrt 3 \ge 1$.  If $\delta_1 > 0$, then $\delta < 2 - 2\sqrt 3 \le -1$.  Thus in any case, $|\delta| = |b-a| \ge 1$.

It follows immediately that the distance between any two distinct points of $Q_\lambda\left(\frac{1}{2}S\right) = \frac{1}{2}Q_\lambda(S)$ is at least $1/2$.  Thus $Q_\lambda(P_{12})$ is confined to the union of a discrete set of vertical lines with interline distance at least $1/2$.  Now consider any two distinct points $x, y \in Q_\lambda(P_{12})$ and let $z := y - x$.  Writing $z = re^{i\theta}$ where $r := |z|>0$ and $\theta := \arg z$, it suffices to show that $r\ge 1$.  We see by the symmetry of $P_{12}$ that $Q_\lambda(P_{12})$ is invariant under rotation about the center of $P_{12}$ through angle $\tau/12 = \pi/6$.  Thus for each $k\in\ints$, \ $Q_\lambda(P_{12})$ contains a pair of points that differ by $e^{i k\pi/6}z = r e^{i (\theta + k\pi/6)}$.  Choose $k$ to minimize $|\Re(e^{i k\pi/6}z)| = r|\cos(\theta + k\pi/6)|$, and set $\psi := \theta + k\pi/6$ for our chosen $k$.  Let $x',y'\in Q_\lambda(P_{12})$ be such that $y' - x' = e^{i k\pi/6}z = r e^{i \psi}$ (whence $\Re(y'-x') = r\cos\psi$).  Then either $\cos\psi = 0$ (if $x'$ and $y'$ lie on the same vertical line) or $|\cos\psi| \ge 1/(2r)$ (if $x'$ and $y'$ lie on different vertical lines).  By swapping $x'$ and $y'$ if necessary, we can assume that $0\le \psi < \pi$.  Since $k$---and thus $\psi$---was chosen to minimize $|\cos\psi|$, we get that $5\pi/12 \le \psi \le 7\pi/12$.

We have two cases:
\begin{description}
\item[Case~1:] $\cos\psi \ne 0$.  Then putting the two inequalities above together, we get
\[ \frac{1}{2r} \le |\cos\psi| \le \cos(5\pi/12)\;, \]
and thus
\[ r \ge \frac{1}{2\cos(5\pi/12)} > \frac{1}{2\cos(\pi/3)} = 1\;. \]
\item[Case~2:] $\cos\psi = 0$, that is, $\psi = \pi/2$, and $x'$ and $y'$ lie on the same vertical line.  Using the rotational symmetry of $Q_\lambda(P_{12})$ again, we can find points $x''$ and $y''$ such that $y'' - x'' = e^{-i\pi/6}(y'-x') = r e^{i\pi/3}$.  Thus
\[ \Re(y'' - x'') = r\cos(\pi/3) = \frac{r}{2} \ge \frac{1}{2}\;, \]
because $x''$ and $y''$ lie on different vertical lines.  This also implies $r \ge 1$.
\end{description}
\end{proof}


%


\section{Relative density of discrete $Q_\lambda(S)$}
\label{sec:relative-density}

Proposition~\ref{prop:subset-of-model-set} says that $R_\lambda(S)$ is a subset of the model set of a certain cut-and-project scheme, where $\lambda$ is sPV and $S\subseteq \ints[\lambda]$ is finite.  One of our chief conjectures is that these $R_\lambda(S)$ are all Meyer sets.  To establish the conjecture it suffices by Fact~\ref{fact:meyer-set-characterization} to show that $R_\lambda(S)$ is relatively dense (in either $\reals$ or $\complexes$ as appropriate).  While we have not succeeded in finding a general proof, we have a number of partial results,
as well as a sufficient condition for $R_\lambda$ to be Meyer: This is the case if
$\lambda$ is sPV and
$R_\lambda$ contains a unit of $\ints[\lambda]$ other than $1$ (see Corollary~\ref{cor:meyer-if-unit}).

We begin with the case of real quadratic
sPV numbers (Section~\ref{sec:quadratics}), and, generalizing
techniques of Berman and Moody~\cite{BeMo} and Mas{\'a}kov{\'a} et al.~\cite{MPP:sconvexjournal}, characterize all such numbers for which $R_\lambda$ \emph{equals}
a naturally corresponding model set (which is stronger than the Meyer property). 
 As model sets
are relatively dense, relative density is immediate.
We find
that the model set property can only hold for four values of $\lambda$ (or the corresponding $1-\lambda$): 
$-\p$, $-1-\sqrt{2}$, $-1-\sqrt{3}$, and $\frac{-3-\sqrt{17}}{2}$. We then prove
it holds for each of these values; the proof for $\frac{-3-\sqrt{17}}{2}$ is somewhat
more difficult and appears separately, as outlined in
the next paragraph. We thus obtain a characterization 
that strictly strengthens a previous result
of Mas{\'a}kov{\'a} et al.~\cite{MPP:sconvexjournal}, which
gave a similar characterization for quadratic unitary Pisot numbers (a subset of
quadratic sPV numbers).

Certain aspects of
Berman and Moody's technique, which we call the \emph{covering method}, 
are then progressively strengthened in Section~\ref{sec:coveringmethod},
where we complete the proof of our characterization of Section~\ref{sec:quadratics}
for the remaining case of $\frac{-3-\sqrt{17}}{2}$.
Sufficient conditions for obtaining a model set are
 further extended to higher degree $\lambda$. This allows
us to conclude that $R_{\lambda}$ is relatively dense for
three cubic $\lambda$'s: two non-real values, and $\lambda_7$, the
latter being the appropriate parameter for generating the set
based on the regular heptagon (see Figure~\ref{fig:heptagon}).
These proofs explicitly show
that $R_\lambda$ {\it contains} (but does not necessarily equal) a model set. 

Another notion, which we call {\it affine embedding},
provides a flexible tool for proving both relative density {\it and}
uniform discreteness (Section~\ref{sec:affine}).
Affine embeddings combined with the covering method leads to the
strongest sufficient condition we have for $R_\lambda$ to be a Meyer set (Corollary~\ref{cor:meyer-if-unit},
as mentioned above). It also yields a method to determine if
certain elements of
$\Sigma_P$ are {\it not} contained in $R_\lambda$, for certain $\lambda$
(Section~\ref{sec:case-study-root13}).

 Finally, in Section~\ref{sec:higher-dimensions} we use the 1-dimensional (real) results outlined above to prove relative density results in higher dimensions.
 Thus, for example, the relative density of
 $R_{\lambda_7}$ implies the relative density
  of the set of Figure~\ref{fig:heptagon} derived from the heptagon.

\subsection{Relative Density in $\reals$ of $R_\lambda$ for Real Quadratic sPV Numbers $\lambda$}\label{sec:quadratics}

We assume throughout this subsection that $\lambda$ is a real quadratic sPV number.
These numbers are characterized in Corollary~\ref{cor:degree2}.

In Proposition~\ref{prop:subset-of-model-set}, we showed that $R_\lambda$ is a subset of
a cut-and-project, or model, set. In the quadratic case, according to
Eq.~(\ref{eqn:degree2}) in
Corollary~\ref{cor:degree2} (for $\lambda < 0$ without loss of generality),
we can write this relation as follows:
\begin{eqnarray*}
R_\lambda \subseteq \{ a - b\lambda \mid a,b\in\ints \mbox{\rm ~and~} a-b\mu \in [0,1]\}
\end{eqnarray*}
where $\mu = \lambda'$ is the Galois conjugate of $\lambda$.  (Following Berman \& Moody, we use $x\mapsto x'$ as the field automorphism of $\rats(\lambda) = \rats(\mu)$ that swaps $\lambda$ with $\mu$.)
Thus, more compactly,
\[ R_\lambda \subseteq \{ x \mid x \in \ints[\lambda] \textup{~and~} x' \in [0,1]\}. \]
Following the notation of Definition~\ref{def:Sigma_P}, which is based in turn on the notation of
Berman and Moody~\cite{BeMo}, we denote $P := [0,1]$ and
$\Sigma_P := \{ x \mid x \in \ints[\lambda] \mbox{\rm ~and~} x' \in [0,1]\}$. Then
the above containment is stated simply as,
\[ R_\lambda \subseteq \Sigma_P\;. \]
Our goal in this section is to determine for which $\lambda$ the above containment
is an equality.  This we do in Theorem~\ref{theorem:main-candp}.  When equality holds, given
that $\Sigma_P$ is a model set, it then follows that
$R_\lambda$ is relatively dense. This is one of only two ways we have been able
to prove that $R_\lambda$ is relatively dense; the other is used for Corollary~\ref{cor:13-relatively-dense} later on.

 The development closely follows prior work of
Berman and Moody~\cite{BeMo} and Mas{\'a}kov{\'a} et al.~\cite{MPP:sconvexjournal}.

Recall the concept of the \emph{field norm} $N(x) = xx'$, for $x \in \ints[\lambda]$.\footnote{More generally, the field norm (over $\rats$) of an algebraic number $x$ is the product of the conjugates of $x$ over $\rats$.  If $x$ is an algebraic integer, then $N(x)$ is also an integer.}
More concretely, if $x = a + b\lambda$, then $N(x) = (a + b\lambda)(a + b\mu)$. It is known
that $N$ is multiplicative, and that $x$ is a unit (has
an inverse) in $\ints[\lambda]$
iff $N(x) = \pm 1$. 

\begin{definition}\label{def:unitarypisot}
An algebraic integer $\alpha$ \emph{unitary} iff $\alpha$ is a unit of $\ints[\alpha]$.  A quadratic integer $\alpha \in \reals$ is \emph{unitary Pisot} iff it is a PV number which is a unit in $\ints[\alpha]$. 
\end{definition}

Most quadratic sPV numbers are not unitary Pisot (as they can be irrationals
with non-unit field norm, cf.\ Corollary~\ref{cor:degree2} for $n > 1$).
In \cite{MPP:sconvexjournal}, Mas{\'a}kov{\'a} et al.\ characterize the real quadratic unitary Pisot numbers that lead to model sets.
By a result of Pinch 
 (in \cite{Pinch}, Theorem 15)
for real quadratic integer $\alpha$, $R_\alpha$ is discrete iff $\alpha$ is sPV\@. Thus  
real quadratic unitary Pisots can lead to model sets only if
they are also sPV\@. Our characterization below (Theorem~\ref{theorem:main-candp})
thus strengthens  that of
\cite{MPP:sconvexjournal}, as it deals with 
a strict superset of those unitary Pisots that may
result in model sets.

Let $R'_{\lambda}$ denote the set of all conjugates of elements in $R_{\lambda}$.
It is clear that $R'_{\lambda} = R_{\lambda'} = R_\mu$. 
Similarly, $\Sigma'_P = \{x' \mid  x \in \ints[\lambda] \wedge x' \in P\}$. Since $\mu$
generates the same ring as $\lambda$, $x \in \ints[\lambda]$ iff
$x' \in \ints[\lambda]$, so it is plain that $\Sigma'_P = \ints[\lambda]
\cap P$. 
Because the algebraic structures are the same,
 that implies $R_{\lambda} \subseteq \Sigma_P$ iff $R_\mu \subseteq
 \ints[\lambda]\cap P$.  Further, $R_\lambda = \Sigma_P$
 iff $R_\mu = \ints[\lambda]\cap P = \Sigma'_P$. 

The following result and proof
is identical to an analogous one in~\cite{MPP:sconvexjournal}. 
The only observation we make here is that it's not necessary to assume
that $\lambda$ is unitary Pisot. 

\begin{lemma}\label{lem:dividesy}
If $R_\mu = \Sigma'_P$, then for any $y \in \Sigma'_P$, both of the following hold:
\begin{enumerate}[(i)]
\item $\mu \divides y$ or $\mu \divides (y-1)$.
\item $(1-\mu) \divides y$ or $(1-\mu) \divides (y-1)$. 
\end{enumerate}
(Divisibility is with respect to $\ints[\lambda] = \ints[\mu]$.)
\end{lemma}
\begin{proof}
By the hypothesis, $y \in \Sigma'_P$ implies that $y \in R_\mu$. By Lemma~\ref{lem:level-characterization}, this implies that $y = \sum_{i=0}^n b_i \mu^i(1-\mu)^{n-i}$ for some $n \ge 0$ and
$0 \le b_i \le {n \choose i}$. If $b_0 = 0$, then clearly $\mu \divides y$. If $b_0 = 1$,
then $y = 1 + \mbox{terms divisible by $\mu$}$, so $\mu \divides (y-1)$, which yields
item (i). For item (ii), the proof is the same, but based on the $i= n$ term in
the sum.
\end{proof}

 We recall some well-known facts about
  rings of real quadratic integers (i.e., rings of algebraic integers in
  real quadratic
 fields) that are of the form $\ints[\sqrt{d}]$, where $d > 0 $ is square-free.
 See \cite{Edwards} and \cite{Goldman} for further details and proofs.
It is known that all such rings have
infinitely many units. 
Even for modest values of $d$, the so-called {\it fundamental units}
(those of smallest size and $> 1$)
can already be huge. In fact, all units in $\ints[\sqrt{d}]$
are positive or negative powers $\pm u_d^k$ of the fundamental unit $u_d$.
Now our rings $\ints[\lambda]$ are not quite of this form, since (in the negative
case), $\lambda = \frac{-m-\sqrt{D}}{2}$ where $D = m^2+4n$, and 
$D$ might not be square-free. Clearly, $\ints[\sqrt{D}] \subseteq\ints[\lambda]$,
and any units in $\ints[\sqrt{D}]$ are also units in $\ints[\lambda]$.
Indeed, the rings $\ints[\sqrt{D}]$ also
have infinitely many units. This is because any element $a+b\sqrt{D} \in 
\ints[\sqrt{D}]$ is a unit iff $N(a+b\sqrt{D}) = (a+b\sqrt{D})(a-b\sqrt{D}) = a^2 - Db^2 = 1$.
This is Pell's equation, which is known by a theorem of Lagrange to have infinitely
many integer solutions when $D$ is not a perfect square (as is true when $\lambda$ is
real quadratic sPV). The solution $a_0, b_0\in \ints$ with the smallest
value $a_0 + b_0\sqrt{D} > 1$
corresponds to the fundamental unit $u_D$
in  $\ints[\sqrt{D}]$. Suppose $D = s^2d$ where $d$ is square-free.
 Since $\ints[\sqrt{D}] = \ints[s\sqrt{d}] \subseteq \ints[\sqrt{d}]$, and all units in $\ints[\sqrt{d}]$
 are powers of a fundamental unit $u_d \in \ints[\sqrt{d}]$, it follows that
 the fundamental unit $u_D$ of $\ints[\sqrt{D}]$ is a power of $u_d > 1$: for some
 $k>0$, $u_D = u_d^k$, and the units of $\ints[\sqrt{D}]$
 are powers of $u_D$. This fact is used in the next proof.

   Given these considerations, in the 
   next lemma we need only rely on the existence
of arbitrarily large units in $\ints[\lambda]$, 
which usually coincide with neither $\lambda$ nor $1-\lambda$
(so that $\lambda$ need not be unitary).
Otherwise the result and the proof
are essentially those of a similar lemma
in~\cite{MPP:sconvexjournal}.

\begin{lemma}\label{lem:divides2} Let $\lambda$ be a real quadratic sPV number. 
Suppose that $R_\mu = \Sigma'_P$.
Then in the ring $\ints[\lambda]$, $\mu \divides 2$ and $(1-\mu) \divides 2$.  
\end{lemma}
\begin{proof}
For any unit $z\in \ints[\lambda]$, we know that it has an inverse $z^{-1} \in \ints[\lambda]$
so that $zz^{-1} = 1$. We may take $z$ to be a power of a fundamental unit
 $> 1$, and $z^{-1} > 0$.
Since $z, z^{-1} \in \reals$ and $z > 1$, the equation $zz^{-1} = 1$ then implies
that $z^{-1} \in \opop{0, 1}$. Furthermore, we choose $z$ sufficiently
large that $z^{-1} \in \opop{0, 1/2}$.
Let $u = z^{-1}$ denote this unit. Thus $u, 2u \in \ints[\lambda] \cap P = \Sigma'_P$,
and
Lemma~\ref{lem:dividesy} can be applied to $y = u$ and $y = 2u$.



If $\mu$ is a unit, then $\mu \divides 2$ and we are done. Suppose then that
$\mu$ is not a unit, and that $\mu \notdiv 2$. Since $\mu$ is not a unit,
 we also have $\mu \notdiv u$.
Now if $\mu \divides 2u$, given that $u$ is a unit, 
 we would have $2 = \mu yu^{-1}$, $u^{-1} \in \ints[\lambda]$, for some $y \in \ints[\lambda]$,
 so that
  $\mu \divides 2$, a contradiction.
Thus $\mu \notdiv 2$
 and $\mu \notdiv u$ also imply that
  $\mu \notdiv 2u$.

Since $\mu \notdiv 2u$, by Lemma~\ref{lem:dividesy} part (i),
$\mu \divides (2u-1)$.  From the fact that $\mu \notdiv u$ we conclude from
the same Lemma that
$\mu \divides (u-1)$. Then $\mu \divides (2u-1-u+1)$, i.e., $\mu \divides u$, a contradiction. 
Hence $\mu \divides 2$. 

The exact same argument applies to $1-\mu$, appealing to part (ii) of 
Lemma~\ref{lem:dividesy}. 
\end{proof}

\begin{lemma}\label{lem:pre-characterization}
For any real quadratic sPV number $\lambda$ such that 
in the ring $\ints[\lambda]$, $\mu \divides 2$ and $(1-\mu) \divides 2$ (equivalently, $\lambda \divides 2$ and $(1-\lambda) \divides 2$),
one of the following must hold:
\begin{enumerate}[(a)]
\item $\lambda = -\varphi$ or $\lambda = 1+\varphi$
\item $\lambda = -1-\sqrt{2}$ or $\lambda = 2+\sqrt{2}$
\item $\lambda = -1-\sqrt{3}$ or $\lambda = 2+\sqrt{3}$
\item $\lambda = \frac{-3-\sqrt{17}}{2}$ or $\lambda = \frac{5+\sqrt{17}}{2}$
\end{enumerate}
\end{lemma}
\begin{proof}
We consider $\lambda < 0$; the postive case is similar. Then by Corollary~\ref{cor:degree2}, $\lambda$ and $\mu$ have the minimal polynomial $x^2 + mx - n$, for $0 < n \le m$.
Since $\mu \divides 2$, there are $a, b \in \ints$
such that $(a + b\mu) \mu = 2$. Given that $\mu^2 + m\mu- n = 0$, a little
calculation gives the constraint $nb = 2$. This is only possible if $n = 1$
or $n = 2$.

Since we also have $(1-\mu) \divides 2$, there are $c, d \in \ints$
such that $(c + d\mu) (1-\mu) = 2$. Again, a little
calculation yields $d = \frac{2}{m-n+1}$. Since $d \in \ints$,
this can only hold if $m-n = 0$ or $1$. 

If $n = 1$, we have the unitary Pisot case, which is essentially
treated in~\cite{MPP:sconvexjournal}. For the sake of completeness,
we give the argument here: From $n=1$ and the constraint $m-n = 0$ or $1$,
we conclude that $m = 1$ or $m=2$. In the first case we obtain $\lambda = -\varphi$ (item (a)),
 and in the second, $\lambda=-1-\sqrt{2}$ (item
(b)).

Now suppose $n = 2$. Then $m = 2$ or $m = 3$. The former
case corresponds to $\lambda = -1-\sqrt{3}$, item (c). The latter case corresponds
to $\lambda = \frac{-3-\sqrt{17}}{2}$, item (d). 
\end{proof}

From Lemmas~\ref{lem:divides2} and \ref{lem:pre-characterization} we immediately get the following:

\begin{lemma}\label{lem:characterization}
For any real quadratic sPV number $\lambda$ such that $R_\mu = \Sigma'_P$ (equivalently, $R_\lambda = \Sigma_P$),
one of the items (a), (b), (c), (d) of Lemma~\ref{lem:pre-characterization} must hold.
\end{lemma}

In each of the items (a), (b), and (c) in Lemma~\ref{lem:pre-characterization}, at least one of the two alternative values of $\lambda$ is unitary Pisot.  We next prove that $R_\lambda = \Sigma_P$ for these values of $\lambda$.  
This was first
proved in \cite{MPP:sconvexjournal} using the notion of $\beta$-expansions
due to R{\'e}nyi~\cite{Renyi1957}. Here we include a 
self-contained proof that does not use
$\beta$-expansions. The original method of Berman and
Moody~\cite{BeMo} can only be applied directly to $\lambda = -\varphi$
or $1+\varphi$. We extend the technique for the other two values.

\begin{theorem}[Mas{\'a}kov{\'a} et al.~\cite{MPP:sconvexjournal}]\label{thm:cutandproject} 
If $\lambda$ or $1-\lambda$ is as in Lemma~\ref{lem:pre-characterization} items (a), (b), or (c),
then $R_\lambda = \Sigma_{P}$. 
\end{theorem}
\begin{proof}

We find it convenient to have $\mu > 1/2$. This can always be arranged either by using $\lambda$ or $1-\lambda$, and since $R_\lambda = R_{1-\lambda}$, there is no
loss of generality. Furthermore, whenever we make this choice for the relevant $\lambda$'s
(Lemma~\ref{lem:pre-characterization}), then $1-\lambda$ is a unit. To be concrete,
this dictates that we use $\lambda = 2+\sqrt{2}$ or $\lambda = -1-\sqrt{3}$. In those
cases, $\lambda$ is {\it not} a unit. Only when
 $\lambda = 1+\varphi$ does it hold that {\it both} $\lambda$ and $1-\lambda$ 
are units; that is what makes that case
easier.

For any set $A$, define $x\x A = \{x\x a \mid  a \in A\}$, and similarly
$A\x x = \{a \x x \mid  a \in A\}$. 

Our primary task is to prove that 
\begin{eqnarray}\label{eqn:replicating}
\Sigma_P = 0\x \Sigma_P \cup 1\x \Sigma_P \cup \Sigma_P\x 0\cup \Sigma_P \x 1.
\end{eqnarray}
This is a generalization of Berman and Moody's ``replication" property. It says that any element
in $\Sigma_P$ can be obtained by extrapolation on the left or the right by $0$ or $1$
with another element in $\Sigma_P$. The main result
then follows easily, as explained below.

We compute $0\x \Sigma_P$, $1\x \Sigma_P$, $\Sigma_P\x 0$, and 
$\Sigma_P \x 1$ in turn, starting with $0\x \Sigma_P$:
\begin{eqnarray*}
0\x \Sigma_P & = & \{0 \x x \mid  x \in \ints[\lambda] \wedge x' \in P\}\\
             & = & \{\lambda x \mid  x \in \ints[\lambda] \wedge x' \in P\}\\
             & = & \{\lambda x \mid  \lambda x \in \lambda\ints[\lambda] \wedge 
              (\lambda x)' \in \mu P\}
\end{eqnarray*}  
Note at this point that, because $\lambda$ is not necessarily a unit
(unless $\lambda = -\varphi$ or $1+\varphi$), the ideal
$\lambda\ints[\lambda]$ is not equal to the ring $\ints[\lambda]$. 
We address this issue later. 
After changing variables $\lambda x \mapsto x$,
\begin{eqnarray}\label{eqn:0ontheleft}
0\x \Sigma_P & = & \{x \mid  x \in \lambda\ints[\lambda] \wedge x' \in \mu P\}.
\end{eqnarray}  
Note the important fact that
$\mu P = \clcl{0, \mu}$.

Proceeding similarly,
\begin{eqnarray*}
1\x \Sigma_P & = & \{1 \x x \mid  x \in \ints[\lambda] \wedge x' \in P\}\\
             & = & \{1-\lambda + \lambda x \mid  x \in \ints[\lambda] \wedge x' \in P\}\\
             & = & \{1-\lambda + \lambda x \mid  1-\lambda + \lambda x \in 1-\lambda + 
                      \lambda\ints[\lambda] \wedge (1-\lambda + \lambda x)' \in
                                 1-\mu + \mu P\},
\end{eqnarray*}
and hence,
\begin{eqnarray}\label{eqn:1ontheleft}
     1\x \Sigma_P=  \{x \mid   x \in 1-\lambda + \lambda\ints[\lambda] \wedge 
              x' \in 1-\mu + \mu P\}.
\end{eqnarray} 

Observe that 
$1-\mu + \mu P = \clcl{1-\mu, 1}$. Further note that, because $\mu > 1/2$,
\begin{eqnarray*}
  \mu P \cup (1-\mu + \mu P) = \clcl{0, \mu} \cup \clcl{1-\mu, 1} = \clcl{0,1} = P.
\end{eqnarray*} 
It is crucial that the intervals $\mu P$ and $1-\mu + \mu P$ cover $P$.

Next we extrapolate $\Sigma_P$ on the {\it right} with $0$:
\begin{eqnarray*}
 \Sigma_P \x 0 & = & \{ x \x 0 \mid  x \in \ints[\lambda] \wedge x' \in P\}\\
               & = & \{ (1-\lambda) x \mid  x \in \ints[\lambda] \wedge x' \in P\}\\
               & = & \{ (1-\lambda) x \mid  (1-\lambda) x 
                      \in (1-\lambda)\ints[\lambda] \wedge ((1-\lambda) x)' \in (1-\mu) P\}.
\end{eqnarray*}
Because $1-\lambda$ {\it is} a unit,
we have $(1-\lambda)\ints[\lambda] = \ints[\lambda]$. 
Therefore,
\begin{eqnarray}\label{eqn:0ontheright}
 \Sigma_P \x 0 & = & \{ x \mid  x 
                      \in \ints[\lambda] \wedge x' \in (1-\mu) P\}.
\end{eqnarray}

  Observe that $(1-\mu) P = \clcl{0,1-\mu}$ is a \emph{proper} subset of
  $\clcl{0,1/2}$ because of $\mu > 1/2$.

Finally we extrapolate $\Sigma_P$ on the right with $1$:
\begin{eqnarray*}
 \Sigma_P \x 1 & = & \{ x \x 1 \mid  x \in \ints[\lambda] \wedge x' \in P\}\\
               & = & \{ \lambda + (1-\lambda) x \mid  x \in \ints[\lambda] \wedge x' \in P\}\\
               & = & \{\lambda + (1-\lambda) x \mid  \lambda + (1-\lambda) x 
                      \in \lambda + (1-\lambda)\ints[\lambda] \wedge (\lambda + (1-\lambda) x)' \in \mu + (1-\mu) P\}.
\end{eqnarray*}
Using the fact that $\lambda + (1-\lambda)\ints[\lambda] = \ints[\lambda]$,
\begin{eqnarray}\label{eqn:1ontheright}
  \Sigma_P \x 1  =  \{ x \mid  x 
                      \in \ints[\lambda] \wedge x' \in \mu + (1-\mu) P\},
\end{eqnarray}
Observe that $\mu + (1-\mu) P  = \clcl{\mu, 1}$. 
Further note that, because $\mu > 1/2$,
\begin{eqnarray*}
    (1-\mu) P \cup (\mu + (1-\mu) P) = \clcl{0, 1-\mu} \cup \clcl{\mu, 1} = P\cmpl \opop{1-\mu,\mu}.
\end{eqnarray*}
The fact that $(1-\mu) P$ and $\mu + (1-\mu) P$ fail to cover $P$ in general
necessitates extrapolation on the {\it left} as well as the right.

We now address the nature of the ideal $\lambda\ints[\lambda]$. From Lemma~\ref{lem:divides2},
we know that $\lambda \divides 2$. As mentioned previously, {\it $\lambda$ is not a unit},
 unless $\lambda = 1+\varphi$ or $-\varphi$. When $\lambda = 1+\varphi$ or $-\varphi$,
  $\lambda$ is a unit,
 so $\lambda\ints[\lambda] = \ints[\lambda]$, and the problem goes away (more 
 on this below). 
But in the other cases $N(\lambda) \not = 1$: For
$\lambda = 2+\sqrt{2}$, $N(\lambda) = 2$, and for $\lambda = -1-\sqrt{3}$,
$N(\lambda) = -2$. 
Now for any $x = a + b\lambda \in \ints[\lambda]$, we have 
$x = a + b\lambda 
 \equiv a~\pmod{\lambda}$. That is,
$x \equiv k~\pmod{\lambda}$ where $k$ is an integer\footnote{To clarify,
the mod $\lambda$ notation means equivalence up to $\lambda$-multiples of
elements of $\ints[\lambda]$. Thus a more precise (but
also more cumbersome) notation
would be $x \equiv k~\pmod{\lambda\ints[\lambda]}$.}.
However, because $\lambda \divides 2$, we find that $2 \equiv 0~\pmod{\lambda}$.
 Hence for any
$x \in \ints[\lambda]$, we have either $x \equiv 0~\pmod{\lambda}$
or $x \equiv 1~\pmod{\lambda}$.
Hence the ideal $\lambda\ints[\lambda]$ partitions $\ints[\lambda]$ into
two classes, one equivalent to $0$ mod $\lambda$, one equivalent to $1$ mod
$\lambda$, and preserving addition and multiplication
when $N(\lambda)=\pm 2$. Then $\ints[\lambda]/\lambda\ints[\lambda] \cong \ints_2$. 
As this is a field, $\lambda\ints[\lambda]$ is a maximal ideal; it is also
possible to prove the latter fact directly. However, in this proof we really
only need the partition of $\ints[\lambda]$ into the ``$0$" class and
the ``$1$" class.

The fact that $\lambda\ints[\lambda]$ induces this partitioning
(for $\lambda \not = -\varphi$ or $1+\varphi$) leads to the necessity
of extrapolating on the {\it right} as well as the left.

We now turn to the task of establishing Eq.~(\ref{eqn:replicating}). First, because
each is a subset of $\ints[\lambda]$, and the relevant convex sets (e.g.,
$P$, $\mu P$, etc.) are contained in $P$, it is clear that each of
$0\x \Sigma_P$, $1\x \Sigma_P$, $\Sigma_P\x 0$, $\Sigma_P\x 1$ are contained
in $\Sigma_P$. For the reverse containment, consider any $x \in \Sigma_P$. There are
three cases:
\begin{enumerate}[(i)]
  \item $x' \in \clcl{0, 1-\mu}$: By Eq.~(\ref{eqn:0ontheright}), $\Sigma_P\x 0$ contains all
   $ x \in \ints[\lambda]$ such that $ x' \in \clcl{0,1-\mu}$. Thus this implies that
      $x \in \Sigma_P\x 0$. 
  \item $x' \in \clcl{\mu, 1}$: By Eq.~(\ref{eqn:1ontheright}),  $\Sigma_P\x 1$ contains all
   $ x \in \ints[\lambda]$ such that $ x' \in \clcl{\mu,1}$. Thus this implies that
      $x \in \Sigma_P\x 1$.
  \item $x' \in \opop{1-\mu, \mu}$: By the partition of $\ints[\lambda]$,
       we have that either $x \equiv 0~\pmod{\lambda}$ or $x \equiv 1~\pmod{\lambda}$.
       If the former, then $x\in \lambda\ints[\lambda]$, and hence (since
        $x' \in \mu P = \clcl{0,\mu}$), by
        Eq.~(\ref{eqn:0ontheleft}), we have $x\in 0\x \Sigma_P$. If the latter, then
        $x \in 1-\lambda + \lambda\ints[\lambda]$, and hence (since
          $x' \in 1-\mu + \mu P = \clcl{1-\mu, 1}$), by
          Eq.~(\ref{eqn:1ontheleft}),  we have $x \in 1\x\Sigma_P$.
\end{enumerate} 
This establishes Eq.~(\ref{eqn:replicating}). 
If $\lambda = 1+\varphi$, we can take $\mu < 1/2$, and we would
only have cases (i) and (ii) above. That is, for $\lambda = 1+\varphi$,
we find $\Sigma_P = \Sigma_P\x 0 \cup \Sigma_P \x 1$. 


The proof is completed by an induction to show that $\Sigma_P \subseteq R_\lambda$. We first consider $\lambda = 2+\sqrt{2}$
and $-1-\sqrt{3}$. For these values,
$|\lambda|, |1-\lambda| > 2$.
Let $x$ be any element of $\Sigma_P$.  By Equation~(\ref{eqn:replicating}), there exists $y\in\Sigma_P$ such that $x\in\{0\x y, 1\x y, y\x 0, y\x 1\}$.  Suppose for the moment that $|x| > \max(|\lambda|,|1-\lambda|)$.  There are four cases:
\begin{enumerate} 
\item If $x = 0\x y$,
then $x = \lambda y$ and $|\lambda|\cdot|y| = |x|$, hence $|y| = \frac{|x|}{|\lambda|}
< |x|$.
\item If $x = 1\x y$, then $x = 1-\lambda+\lambda y$, so 
$\lambda y = x - (1-\lambda)$, and $|\lambda|\cdot|y| \le |x| + |1-\lambda| < 2|x|$. 
Thus $|y| < \frac{2}{|\lambda|}|x| < |x|$. 
\item If $x = y\x 0$ then $x = (1-\lambda)y$,
and $|y| = \frac{|x|}{|1-\lambda|} < |x|$. 
\item If $x = y\x 1$, then $x = (1-\lambda)y + \lambda$.
This implies $(1-\lambda) y = x - \lambda$, and hence 
$|y| \le \frac{|x| + |\lambda|}{|1-\lambda|} < \frac{2|x|}{|1-\lambda|} < |x|$.
\end{enumerate} 
Thus for any $x\in\Sigma_P$ with $|x| > \max(|\lambda|,|1-\lambda|)$, we obtain a $y\in\Sigma_P$ with norm less than $|x|$ which yields $x$ under extrapolation with $0$ or $1$ on the left or right (and thus $y\in R_\lambda$ implies $x\in R_\lambda$).
Since $\Sigma_P$ is uniformly discrete, we can repeat this process a finite number of times to reduce to a $y\in\Sigma_P$ with $|y| \le \max(|\lambda|,|1-\lambda|)$.
It only remains to verify that the finite number of $y \in \Sigma_P$ satisfying  $|y| \le \max(|\lambda|,|1-\lambda|)$ are in $R_{\lambda}$. 
But these are exactly the points $0, 1$, $1-\lambda$, and $\lambda$
itself, all of which are in $R_{\lambda}$.

For $\lambda = 1+\varphi$, while $|\lambda| = 1+ \varphi > 2$, we have $1 < |1-\lambda|= \varphi < 2$. Thus
the above argument doesn't work.  This case is still simpler, however, since we only need cases (i) and (ii) above. This is Berman and Moody's
argument, which we review here. We consider $x\in\Sigma_P$ with $|x| > |1-\lambda|^3
=\varphi^3$. If $x = y\x 0$, then $x = (1-\lambda)y$, so 
$|y| = \frac{|x|}{|1-\lambda|} < |x|$. If $x  = y\x 1$, then
$x = (1-\lambda)y + \lambda$ and $|y| \le \frac{|x| + |\lambda|}{|1-\lambda|}
= \frac{|x| + 1+\varphi}{\varphi} < \frac{1}{\varphi}\left(1+\frac{1+\varphi}{\varphi^3}\right)|x|$. Now observe that,
\begin{eqnarray*}
 \frac{1}{\varphi}\left(1+\frac{1+\varphi}{\varphi^3}\right) = \frac{\varphi^3+\varphi+1}{\varphi^4}
          = \frac{\varphi^3+\varphi^2}{\varphi^2(\varphi+1)}=
            \frac{\varphi^2(\varphi+1)}{\varphi^2(\varphi+1)} = 1,
\end{eqnarray*}
 so $|y| < |x|$. As before, we reduce to a $y\in\Sigma_P$ with norm $\le \varphi^3 = 1+2\varphi$.
 The relevant points in $\Sigma_P$ are $0, 1, \lambda, 1-\lambda$ and $-1-\lambda
 = \lambda\x 0$,
 all of which are in $R_\lambda$.
\end{proof}


\vskip 0.2in

The foregoing proof requires either $\lambda$ or $1-\lambda$ to be a unit.
This does not hold for $\lambda = \frac{-3-\sqrt{17}}{2}$,
and hence this requires a different technique (although still not relying on
$\beta$-expansions),
developed in Section~\ref{sec:coveringmethod}. It turns out that $R_\lambda = \Sigma_P$ for 
$\lambda = \frac{-3-\sqrt{17}}{2}$ as well
(see Proposition~\ref{prop:sqrt17} in Section~\ref{sec:sqrt17} below). Thus from Lemma~\ref{lem:characterization}, Theorem~\ref{thm:cutandproject}, and Proposition~\ref{prop:sqrt17},
the following is immediate:

\begin{theorem}\label{theorem:main-candp}
If $\lambda$ is a real quadratic sPV number, then 
$R_\lambda = \Sigma_{P}$ if and only if $\lambda$ is one of the
cases (a), (b), (c), or (d) of Lemma~\ref{lem:pre-characterization}.
\end{theorem}

This can be further strengthened. In \cite[Theorem~15]{Pinch}, Pinch shows that
if $\lambda$ is real and not an integer, and $R_\lambda$ is discrete,
then $\lambda$ has a conjugate in $(0,1)$. Hence
if $\lambda$ is real and quadratic and $R_\lambda$ is discrete,
we can conclude that $\lambda$ is sPV. Conversely, we know that
if $\lambda$ is sPV, then $R_\lambda$ is discrete. Thus for
real, quadratic $\lambda$, $R_\lambda$ is discrete iff 
$\lambda$ is sPV. We thus have the following corollary of
Theorem~\ref{theorem:main-candp} (recall that if $\lambda\in\reals$ and $R_\lambda$ is discrete, then $\lambda$ is an algebraic integer by Theorem~\ref{thm:real-case}):

\begin{corollary}
If $\lambda$ is real and quadratic, and $R_\lambda$ is discrete,
then $R_\lambda = \Sigma_{P}$ if and only if $\lambda$ is one of the
cases (a), (b), (c), or (d) of Lemma~\ref{lem:pre-characterization}.
\end{corollary}

As per the discussion just after Definition~\ref{def:unitarypisot}, 
Theorem~\ref{theorem:main-candp} also implies the following immediately:

\begin{corollary}[Mas{\'a}kov{\'a} et al.~\cite{MPP:sconvexjournal}]\label{cor:cutandproject}
If $\lambda$ is unitary Pisot, then 
$R_\lambda = \Sigma_{P}$ iff $\lambda$ is one of the
 cases (a), (b), and (c) of Lemma~\ref{lem:pre-characterization}.
\end{corollary}

\subsection{The Covering Method and its Applications}\label{sec:coveringmethod}

The underlying strategy
of the proof of Theorem~\ref{thm:cutandproject}
is to find a finite cover of the interval $P = [0,1]$
(in that case, simply the intervals $[0,1-\mu]$,
$[1-\mu,\mu]$, $[\mu, 1]$),
such that for any $x \in \Sigma_P$, depending on which element of
the cover $x'$ is contained in, we can write $x$ as $b\xl y$ or
$y \xl b$, for $b \in \{0,1\}$. Since, in all of these four cases,
$|y| < |x|$, an inductive proof then suffices to show that $\Sigma_P
\subseteq R_\lambda$.

In this section we generalize that proof, which yields further relative
density results,
including for case (d) of Theorem~\ref{theorem:main-candp}, $\lambda = \frac{-3-\sqrt{17}}{2}$. That case
is covered in 
Proposition~\ref{prop:sqrt17} below.
 
In particular, the technique allows us to use a fixed set of ``seed" points other
than $\{0,1\}$. Indeed, it turns out that when $\Sigma_P$ has a unit,
Theorem~\ref{thm:add-points} below shows that 
it is always possible to adjoin finitely many points $Y$ in $\Sigma_P$ 
to $R_\lambda$ such that it equals $\Sigma_P$, i.e., for some $Y \subseteq
\Sigma_P$, 
$\Sigma_P = R_\lambda(Y)$. Thus, if we can explicitly find such
a set $Y$ that is a subset of $R_\lambda$, then $R_\lambda$ itself
is cut and project: $\Sigma_P = R_\lambda$.
However, the result yields more than relative
density results. As we will see later on, even when
$Y \not \subseteq R_\lambda$, the results of this section enable us to prove
that certain elements of $\Sigma_P$ are {\it not} in $R_\lambda$, again in
certain cases. 

The proof of Theorem~\ref{thm:add-points} extends that of Theorem~\ref{thm:cutandproject} in two ways. 
First, the size of the cover of $P$ is larger, but
still finite (depending on the
size of the fundamental unit in $\Sigma_P$).  Secondly, the proof
that $\Sigma_P \subseteq R_\lambda$ proceeds by inducting on $|z|$
as before, but now a larger finite set $Y$ of seed points plays the role that
$\{0,1\}$ played in the proof of
Theorem~\ref{thm:cutandproject}.
In one respect, the induction is simpler, since we find $z$ can be written
as $x\xl y$ where $x \in Y$ and hence (as shown below) $|z| > |y|$. That is,
$z$ is the extrapolant of one of a finite number of seed points on the \emph{left
only} with a smaller point
on the \emph{right}.

\subsubsection{Quadratics and Non-Real Cubics}

We begin by considering quadratic sPV numbers and non-real cubic (degree~$3$) sPV numbers.  In the latter case---as with the real quadratic case---$\lambda$ has a unique conjugate $\mu\in\opop{0,1}$.  It also has $\lambda^*$ as its other conjugate.  In this section, we again let $x \mapsto x'$ be the unique field isomorphism\footnote{If $\lambda\notin\reals$, then $\rats(\lambda) \ne \rats(\mu)$, and thus this map is not an automorphism as in the real quadratic case.} $\rats(\lambda)\rightarrow\rats(\mu)$ that maps $\lambda$ to $\mu$, and we define $P := \clcl{0,1}$ and $\Sigma_P = \Sigma_{\clcl{0,1}} := \{x\in\ints[\lambda] \mid x'\in\clcl{0,1}\}$ as before.  So if $\lambda\notin\reals$, we have
\begin{equation}\label{eqn:cubic-Sigma-P}
\Sigma_P = \{a+b\lambda+c\lambda^2 \mid a,b,c\in\ints \myand 0 \le a+b\mu+c\mu^2 \le 1\}\;.
\end{equation}
Since $\Sigma_P$ is a cut-and-project (model) set, it is relatively dense in $\complexes\cong\reals^2$.

%
%

The next theorem shows that, for any sPV $\lambda$ with a unique conjugate in $\opop{0,1}$, if there exists an $\alpha\in\Sigma_P$ that is a unit of $\ints[\lambda]$ other than $1$, then there exists a finite set $Y\subseteq \Sigma_P$ such that $R_\alpha(Y) = \Sigma_P$.  If, in addition, $\alpha\in R_\lambda$, then $R_\lambda(Y) = \Sigma_P$.  The theorem also gives sufficient conditions on $Y$ such that $R_\alpha(Y) = \Sigma_P$.  These conditions are sometimes not completely necessary, however, as one can often get by with a smaller ``seed'' set $Y$.  This will be true, for example, in the case where $\lambda = -(3+\sqrt{13})/2$ (Proposition~\ref{prop:root-13}, below).  This theorem is our main tool for proving relative density of $R_\lambda$ for such $\lambda$.

\begin{theorem}\label{thm:add-points}
Let $\lambda$ be a strong PV number with a unique conjugate in $\opop{0,1}$ (which implies $\lambda$ is either real quadratic or non-real cubic).  Let $P := \clcl{0,1}$ and $\Sigma_P := \{x\in\ints[\lambda] \mid x'\in\clcl{0,1}\}$, where the map $x\mapsto x'$ is as defined above.  Suppose $\Sigma_P$ contains a unit $\alpha$ of $\ints[\lambda]$ other than $1$.  Let $\beta := \alpha' \in P$.  Then $|\alpha|>1$, \ $0<\beta<1$, and there exists a finite set $X\subseteq\Sigma_P$ such that
\begin{equation}\label{eqn:X-condition}
\clcl{0,1} \subseteq \bigcup_{x\in X} \clcl{(1-\beta)x',(1-\beta)x' + \beta}\;.
\end{equation}
Moreover, for any such $X$, letting
\begin{equation}\label{eqn:seed-upper-bound}
M := \left(\frac{|\alpha-1|}{|\alpha|-1}\right)\max_{x\in X} |x|\;,
\end{equation}
the set
\begin{equation}\label{eqn:seed-set}
Y := \Sigma_P \intersect \clcl{-M,M}
\end{equation}
is finite, and $R_\alpha(Y) = \Sigma_P$.  If, in addition, $\alpha\in R_\lambda$, then $R_\lambda(Y) = \Sigma_P$.
\end{theorem}

\begin{proof}
Observe that $\beta\in\opop{0,1}$, since $\alpha\in\Sigma_P$ and $\alpha \notin \{0,1\}$.  This in turn implies $\alpha\notin\ints$, and thus, since $\lambda$ has prime degree and $\alpha\in\ints[\lambda]$, it must be that $\alpha$ has the same degree as $\lambda$, and by standard facts of algebra, the conjugates of $\alpha$ are $\alpha$, $\beta$, and $\alpha^*$ (see, e.g., Corollary~\ref{cor:Q-conjugates}).  If $\lambda$ is real quadratic, then $\alpha\in\reals$ and so $\alpha = \alpha^*$ making $\alpha$ real quadratic.  If $\lambda$ is non-real qubic, then $\alpha$ is also qubic, whence $\alpha \ne \alpha^*$ and so $\alpha\notin\reals$.  Since $\alpha$ is a unit, we have $N(\alpha) = \pm 1$.  If $\alpha\in\reals$, then $N(\alpha) = \alpha\beta$, and $N(\alpha) = \alpha\alpha^*\beta$ otherwise.  Thus $\beta = 1/|\alpha|$ if $\alpha\in\reals$ and $\beta = 1/|\alpha|^2$ otherwise.  In either case, we have $|\alpha| > 1$.

Also observe that $\Sigma_P$ is $\alpha$-convex: for any $u,v\in\Sigma_P$, \ $u\xa v\in\ints[\lambda]$ and $(u\xa v)' = u'\xb v'\in P$ (because $u',v',\beta\in P$).  Then since $Y\subseteq\Sigma_P$, we have $R_\alpha(Y)\subseteq\Sigma_P$.  It therefore suffices for the first part to show that $\Sigma_P\subseteq R_\alpha(Y)$.

There exists a finite $X$ satisfying Equation~(\ref{eqn:X-condition}), because $\Sigma_P'$ is dense in $\clcl{0,1}$.\footnote{See the discussion following this proof for bounds on how big $X$ needs to be in the case where $\lambda$ is real quadratic.}  (Notice, however, that $X$ must contain both $0$ and $1$.)
From the fact that $|\beta| < 1$ and either $|\alpha||\beta| = 1$ or $|\alpha|^2|\beta| = 1$, we must have $|\alpha|>1$.
Then $M$ is well-defined, and in addition, $X\subseteq Y$ because $|\alpha-1|/(|\alpha|-1) \ge 1$.

We show for any $z\in\Sigma_P$ that $z$ is in $R_\alpha(Y)$.  We do this by induction on $|z|$, which is allowed, because $\Sigma_P$ is uniformly discrete.  If $|z| \le M$, then $z\in Y$ and we are done, so suppose $|z|>M$.  Choose $x\in X$ such that $z' \in \clcl{(1-\beta)x',(1-\beta)x' + \beta}$, and let $y := x \xainv z$.  One can quickly check that $x\xa y = z$.  Moreover, $y\in\ints[\lambda]$, since $\alpha$ is a unit in $\ints[\lambda]$.  We next show that $|y| < |z|$.  The condition $|y| < |z|$ is equivalent to
\[ \left|\left(1-\frac{1}{\alpha}\right)x + \frac{1}{\alpha}z\right| < |z|\;, \]
This holds provided $|(1-1/\alpha)x| + |z/\alpha| < |z|$, or equivalently, $|\alpha-1||x| + |z| < |\alpha||z|$.  Since $x\in X$, we have $|x| \le M(|\alpha|-1)/|\alpha-1|$, and so the inequality is satisfied.

It remains to show that $0 \le y' \le 1$, thus putting $y$ in $\Sigma_P$.  This suffices: applying the inductive hypothesis to $y$ to get $y\in R_\alpha(Y)$, we get $z = x\xa y \in R_\alpha(Y)$ as desired.  We have
\[ y' = (x\xainv z)' = x' \xbinv z' = \frac{1}{\beta}\left((\beta-1)x'+z'\right)\;. \]
Since $z' \in \clcl{(1-\beta)x',(1-\beta)x' + \beta}$, we then have
\[ \frac{1}{\beta}\left((\beta-1)x' + (1-\beta)x'\right) \le y' \le \frac{1}{\beta}\left((\beta-1)x'+(1-\beta)x' + \beta\right)\;; \]
that is, $0\le y' \le 1$.  Thus by induction, we get $\Sigma_P\subseteq R_\alpha(Y)$, and hence, $\Sigma_P = R_\alpha(Y)$.

If $\alpha\in R_\lambda$, then by Lemma~\ref{lem:extend} and the fact that $\Sigma_P$ is $\lambda$-convex,
\[ \Sigma_P = R_\alpha(Y) \subseteq R_\lambda(Y) \subseteq \Sigma_P\;, \]
and thus all sets above are equal.
\end{proof}


\begin{remark}
We have proved something stronger.  Define $Y^{(0)}, Y^{(1)}, Y^{(2)}, \ldots$ inductively as follows: $Y^{(0)} := Y$ and $Y^{(n+1)} := X \xa Y^{(n)}$ for $n\ge 0$.  Then the proof shows that $\Sigma_P = \bigcup_{n=0}^\infty Y^{(n)}$.  This was shown for the case where $\lambda = \alpha = -\p$ by Berman \& Moody~\cite{BeMo}.
\end{remark}

If $\lambda$ 
is real quadratic, we can get sufficient bounds on the size of $X$ based on results in the theory of uniformly distributed distributions (see, for example, Allouche \& Shallit~\cite{AS:auto-sequences}).  The following concepts and most of the following facts are from \cite[Chapter~2]{AS:auto-sequences}.  Let $[a_0,a_1,a_2,\ldots]$ be the continued fraction expansion of $\mu := \lambda'$ (noting that $a_0=0$).\footnote{Generally speaking, $a_0 = \floor{\mu}$ and $\mu = a_0 + 1/[a_1,a_2,\ldots]$, etc.}  For $i\ge 0$, define $A_i := \mat{a_i&1\\1&0}$, and for $k\ge -1$ define
\begin{equation}\label{eqn:rational-approximants}
\mat{p_k\\q_k} := A_0A_1\cdots A_k\mat{1\\0}\;.
\end{equation}
The $q_k$ satisfy the recurrence $q_{k+1} = a_{k+1}q_k + q_{k-1}$ and form a strictly increasing sequence that is bounded below by the Fibonacci sequence; also, $\lim_{k\rightarrow\infty} p_k/q_k = \mu$.  Letting
\begin{equation}\label{eqn:eta-k}
\eta_k := (-1)^k(q_k\mu - p_k)\;,
\end{equation}
one can show that $\eta_k > 0$ for all $k\ge -1$ and that $\eta_{k+1} < \eta_k$ for all $k\ge 1$.  In fact, $1/q_{k+2} < \eta_k < 1/q_{k+1}$ for all $k\ge 1$.

Then we have
\[ \Sigma_P' = \{1\}\union\{\ceiling{b\mu}-b\mu \mid b\in\ints\} = \{1\}\union\{\{a\mu\} \mid a\in\ints\}\;, \]
where $\{a\mu\} := a\mu - \floor{a\mu}$ is the fractional part of $a\mu$.  
(To see this, set $a := -b$.)  For integer $n\ge 0$, let
\[ X_n := \{1\} \union \{ \ceiling{b\mu}-b\lambda \mid b\in\ints \myand -n\le b\le 0\} \subseteq \Sigma_P\;. \]
Then $X_n' = \{1\}\union \{\{a\mu\} \mid a\in\ints \myand 0\le a\le n\}$.  Much is known about sets of this form for $n\ge 0$.  By the three-distance theorem (see~\cite[Section~2.6]{AS:auto-sequences}), there are at most three possible distances between adjacent points of $X_n'$---the smallest distance is $\eta_k$, where $k\ge -1$ is largest such that $q_k \le n$~\cite[Theorem~2.6.2]{AS:auto-sequences}; the largest distance $d_n$ satisfies the bound
\begin{equation}\label{eqn:distance-bound}
d_n \le \eta_k + \eta_{k-1}\;,
\end{equation}
where $k\ge 0$ is largest such that $q_k \le n+1$~\cite[Theorem~2.6.3]{AS:auto-sequences}.

We now consider choosing $X$ to be $X_n$ for $n$ sufficiently large so that Equation~(\ref{eqn:X-condition}) is satisfied.  The intervals on the right cover $\clcl{0,1}$ just when there are no gaps between adjacent intervals.  That is, for any $x,y\in X_n$ such that $x' < y'$ are adjacent in $X_n'$, is suffices that $(1-\beta)x' + \beta \ge (1-\beta)y'$, or equivalently, $y' - x' \le \beta/(1-\beta)$.  Thus it suffices to choose $n$ so that $d_n \le \beta/(1-\beta)$, and by Equation~(\ref{eqn:distance-bound}) this will be true if
\begin{equation}\label{eqn:X-bound}
\eta_k+\eta_{k-1} \le \frac{\beta}{1-\beta}\;,
\end{equation}
where $k\ge 0$ is largest such that $q_k \le n+1$.  Thus we first choose the least $k$ satisfying (\ref{eqn:X-bound}), then we set $n := q_k - 1$.

Letting $X$ be $X_n$, although sufficient, is not an optimal choice, because it does not minimize $\max_{x\in X}|x|$ and thus the value of $M$ in Equation~(\ref{eqn:seed-upper-bound}).  To minimize $M$, notice that any set of the form $W_{m,n} := \{1\} \union \{ \ceiling{b\mu}-b\lambda \mid b\in\ints \myand m\le b\le m+n\}$ for $m\in\ints\intersect\clcl{-n,0}$ works just as well as $X_n$ in covering $\clcl{0,1}$.  This is because $W_{m,n}$ is the image of $X_n$ under the cyclic shift permutation $\map{c}{\clcl{0,1}}{\clcl{0,1}}$ defined for $x\in\clcl{0,1}$ by
\[ c(x) = \left\{ \begin{array}{ll}
1 & \mbox{if $x = 1$,} \\
(x + \{m\mu\}) \bmod 1 & \mbox{if $x < 1$.}
\end{array}\right. \]
which preserves the set of distances between adjacent points in $X_n$ versus those in $W_{m,n}$.  This means that, given our choice of $n$ above, we can set $X := W_{m,n}$ where $m$ minimizes $\max_{x\in X}|x| = \max_{x\in W_{m,n}}|x|$.  If $\lambda > 0$, then one can verify that $m = \ceiling{-n/2}$, and likewise if $\lambda < 0$, then $m = \floor{-n/2}$.  In the former case,
\begin{equation}\label{eqn:max-X-pos-lambda}
\max_{x\in X} |x| = \left\{\begin{array}{ll}
\lambda n/2 - \floor{\mu n/2} & \mbox{if $n$ is even and $n>0$,} \\
\lambda(n+1)/2 - \ceiling{\mu(n+1)/2} & \mbox{if $n$ is odd.}
\end{array}\right.
\end{equation}
In the latter case,
\begin{equation}\label{eqn:max-X-neg-lambda}
\max_{x\in X} |x| = \left\{\begin{array}{ll}
\ceiling{\mu n/2} - \lambda n/2 & \mbox{if $n$ is even and $n>0$,} \\
\floor{\mu(n+1)/2} - \lambda(n+1)/2 & \mbox{if $n$ is odd.}
\end{array}\right.
\end{equation}
(Verifying the above is made easier by noting that $\lambda \notin \clcl{-1,2}$.)  In either case, if $n=0$, then $W_{m,n} = \{0,1\}$, so $\max_{x\in X} |x| = 1$.

We summarize the foregoing in the following proposition.

\begin{proposition}\label{prop:sufficient-X}
In Theorem~\ref{thm:add-points} where $\lambda$ is real quadratic, it suffices to let
\[ X := \{1\} \union \{ \ceiling{b\mu}-b\lambda \mid b\in\ints\myand m \le b \le m+n\}\;, \]
where $n = q_k-1$ for the least $k$ satisfying (\ref{eqn:X-bound}) for $\eta_k$ defined by (\ref{eqn:eta-k}) and $q_k$ defined by (\ref{eqn:rational-approximants}) based on the continued fraction expansion of $\lambda'$, and $m = \ceiling{-n/2}$ if $\lambda > 0$ and $m = \floor{-n/2}$ otherwise.  For this $X$, the value $\max_{x\in X}|x|$ is given by Equation~(\ref{eqn:max-X-pos-lambda}) or (\ref{eqn:max-X-neg-lambda}), provided $n>0$.
\end{proposition}

\subsubsection{Higher Degrees}

Theorem~\ref{thm:add-points} is specific to quadratic sPV's or
non-real sPV's of degree 3. In this section we turn to real sPV's of arbitrary
degree, real or complex.

For any subset $R$ of a topological space, we let $\intr{R}$ denote the interior of $R$.

Fix a strong PV number $\lambda$ with $k$ many conjugates $0 < \mu_0 < \mu_1 < \cdots < \mu_{k-1} < 1$ in $\opop{0,1}$.  Here we assume that $\lambda$ is nontrivial, so that $k\ge 1$.  (We handle the case of trivial sPV $\lambda$ in Proposition~\ref{prop:added-points-trivial-sPV}, below).  For any $z\in\reals^k$ and $0\le i<k$, we let $z_i$ be the $\ordst{(i+1)}$ component of $z$.  Particularly, for any $x\in\ints[\lambda]$ and $0\le i<k$, we let $x'_i$ denote the conjugate of $x$ that is the image of $x$ under the unique ring homomorphism $\ints[\lambda]\rightarrow\ints[\mu_i]$ mapping $\lambda$ to $\mu_i$ (cf.\ Definition~\ref{def:Sigma_P}), and so we let $x' := (x'_0,x'_1,\ldots,x'_{k-1})$, viewed as a column vector.  We also let $\Lambda_x := \diag(x'_0,x'_1,\ldots,x'_{k-1})$, the diagonal matrix with the $x'_i$'s along the diagonal.  Set $P := \clcl{0,1}^{\times k}$ and for any $S\subseteq \reals^k$ let $\Sigma_S$ denote $\Sigma_S^{(\lambda)}$, recalling that this is $\{x\in\ints[\lambda] \mid x'\in S\}$.

The following generalizes Theorem~\ref{thm:add-points},
but with any sub-box $S$ of $P$ that can be arbitrarily close to, but never equals, $P$.

\begin{theorem}\label{thm:added-points}
Suppose $\Sigma_P$ contains an $\alpha \ne 1$ that is a unit of $\ints[\lambda]$.  Fix any $0<\eps<1/2$, and define $S := \clcl{\eps,1-\eps}^{\times k}$.  There exists a finite set $Y\subseteq \Sigma_P$ such that $\Sigma_S \subseteq Q_\alpha(Y)$.  If, in addition, $\alpha\in Q_\lambda$, then $\Sigma_S \subseteq Q_\lambda(Y)$.
\end{theorem}

\begin{proof}
This proof follows the general approach of that of Theorem~\ref{thm:add-points}.  Since $\alpha\in\Sigma_P$, we have $0<\alpha'_i<1$ for all $0\le i<k$.  Since $\alpha$ is a unit, $1/\alpha\in\ints[\lambda]$, and $(1/\alpha)' = (1/\alpha'_0,1/\alpha'_1,\ldots,1/\alpha'_k)$.  Set $\Lambda := \Lambda_{\alpha} = \diag(\alpha'_0,\ldots,\alpha'_{k-1})$.  For any $x\in\reals^k$, define
\[ J_x := x\xL S = \{z\in\reals^k \mid (\exists y\in S)\;z = x\xL y\} = \{z\in\reals^k \mid x\xs{\Lambda^{-1}} z \in S\}\;. \]
For the last equality, one can readily check that for all $x,y,z\in\reals^k$, if $x\xL y = z$ then $y = x \xs{\Lambda^{-1}} z$ and conversely.

\begin{claim}\label{claim:intervals}
For all $x\in\reals^k$, \ $J_x = \clcl{x_0\xs{\alpha'_0}\eps,\;x_0\xs{\alpha'_0}(1-\eps)} \times \cdots \times \clcl{x_{k-1}\xs{\alpha'_{k-1}}\eps,\;x_{k-1}\xs{\alpha'_{k-1}}(1-\eps)}$.
\end{claim}

\begin{proof}
For $z\in\reals^k$, set $y := x\xs{\Lambda^{-1}} z$, so that $z = x\xL y$.  Then we have $z\in J_x$ if and only if $x\xL y \in J_x$, if and only if $y\in S$ (since $\xL$ is one-to-one in its second argument), if and only if $\eps \le y_i \le 1-\eps$ for all $0\le i<k$, if and only if $x_i\xs{\alpha'_i} \eps \le x_i\xs{\alpha'_i} y_i \le x_i\xs{\alpha'_i} (1-\eps)$ for all $i$ (since $\xs{\alpha'_i}$ is monotone increasing in its second argument), if and only if $x_i\xs{\alpha'_i} \eps \le z_i \le x_i\xs{\alpha'_i} (1-\eps)$ for all $i$, if and only if $z$ belongs to the right-hand side.
\end{proof}

Note then that $\intr{J_x} = \opop{x_0\xs{\alpha'_0}\eps,\;x_0\xs{\alpha'_0}(1-\eps)} \times \cdots \times \opop{x_{k-1}\xs{\alpha'_{k-1}}\eps,\;x_{k-1}\xs{\alpha'_{k-1}}(1-\eps)}$.

For $z\in\reals^k$, define
\[ J^{-1}_z := \{ x\in\reals^k \mid z\in \intr{J_x} \}\;. \]
For any $x\in\reals^k$,
we have $x\in J^{-1}_z$ if and only if $z\in \intr{J_x}$, if and only if $x_i\xs{\alpha'_i}\eps < z_i < x_i\xs{\alpha'_i}(1-\eps)$ for all $i$, by the claim.  That is, $x\in J^{-1}_z$ if and only if, for all $0\le i<k$,
\[ (1-\alpha'_i)x_i + \alpha'_i\eps < z_i < (1-\alpha'_i)x_i + \alpha'_i(1-\eps)\;. \]
Solving for $x_i$, this pair of inequalities is seen to be equivalent to
\[ \frac{z_i - \alpha'_i(1-\eps)}{1-\alpha'_i} < x_i < \frac{z_i - \alpha'_i\eps}{1-\alpha'_i} \]
or alternatively,
\[ (1-\eps)\xs{1/(1-\alpha'_i)} z_i < x_i < \eps\xs{1/(1-\alpha'_i)} z_i\;, \]
Thus we obtain
\[ J^{-1}_z = \opop{(1-\eps)\xs{1/(1-\alpha'_0)} z_0,\;\eps\xs{1/(1-\alpha'_0)} z_0}\times \cdots \times\opop{(1-\eps)\xs{1/(1-\alpha'_{k-1})} z_{k-1},\;\eps\xs{1/(1-\alpha'_{k-1})} z_{k-1}}\;. \]

\begin{remark}
Letting $\one = (1,1,\ldots,1)\in\reals^k$, we see that $J^{-1}_z$ is an open box with opposite corners $(1-\eps)\one \xs{(I-\Lambda)^{-1}} z$ and $\eps\one \xs{(I-\Lambda)^{-1}} z$.
\end{remark}

\begin{claim}\label{claim:nonempty}
$J^{-1}_z \intersect \intr{P} \ne \emptyset$ for all $z\in S$.
\end{claim}

\begin{proof}
It suffices to show for every $0\le i<k$ that if $\eps \le z_i \le 1-\eps$, then
\[ (1-\eps)\xs{1/(1-\alpha'_i)} z_i < 1\;\; \mbox{ and }\;\; \eps\xs{1/(1-\alpha'_i)} z_i > 0\;. \]
Fixing $i$ and letting $\beta := \alpha'_i$, we get the following chain of equivalences for the first inequality above:
\begin{align*}
\left(1-\frac{1}{1-\beta}\right)(1-\eps) + \frac{1}{1-\beta} z_i &< 1 \\
z_i - \beta(1-\eps) &< 1-\beta \\
z_i &< 1-\beta\eps\;,
\end{align*}
and the last inequality is certainly true if $z_i \le 1-\eps$ (because $\beta < 1$).  Similarly, for the second inequality:
\begin{align*}
\left(1-\frac{1}{1-\beta}\right)\eps + \frac{1}{1-\beta} z_i &> 0 \\
z_i - \beta\eps &> 0 \\
z_i &> \beta\eps\;,
\end{align*}
which is again true if $z_i\ge\eps$.
\end{proof}

We know that $\Sigma_P'$ is dense in $P$ (and in fact, $Q_\alpha' = Q_\Lambda(\{0,\one\})$ is dense in $P$ by Lemma~\ref{lem:cube}).  For every $z\in S$, choose an element $x_z$ of $J^{-1}_z \intersect \Sigma_P'$ in some standard way.  Such an element exists by density and because $J^{-1}_z \intersect \intr{P}$ is open and nonempty by Claim~\ref{claim:nonempty}.  Then $z\in \intr{J_{x_z}}$ for all $z\in S$, meaning that the family $\{\intr{J_{x_z}} \mid z\in S\}$ is an open cover of $S$.  Since $S$ is compact, we may choose a finite subcover, indexed by a finite set $X\subseteq\Sigma_P$ so that $\{\intr{J_{x'}} \mid x\in X\}$ covers $S$.

\begin{remark}
It is only necessary that $S$ be covered by the $J_{x'}$ themselves, not their interiors.  The interiors were only used to apply the compactness argument.  Thus one may be able to find a smaller set $X$ such that the $J_{x'}$ cover $S$.
\end{remark}

The rest of the proof is very close to that of Theorem~\ref{thm:add-points}, with only some subtle differences.  An argument similar to the one used in the proof of that theorem shows that $|\alpha| > 1$.  Let
\[ Y := X \union\{y\in\Sigma_S : |y| \le M\}\;, \]
where
\[ M := \left(\frac{|\alpha-1|}{|\alpha|-1}\right)\max_{x\in X} |x|\;. \]
$Y$ is certainly finite, because $\Sigma_P$ is uniformly discrete. 
We now apply Lemma~\ref{lem:z-induction}, given below, to obtain the first
part of the theorem.

If $\alpha\in Q_\lambda$, then $Q_\alpha(Y) \subseteq Q_\lambda(Y)$ by Lemma~\ref{lem:extend}, and the rest of the theorem follows.
\end{proof}

\begin{remark}
There is a possible trade-off between the choice of $X$ versus $Y$.  It is not necessary that $X\subseteq Y$, only that $X\subseteq Q_\alpha(Y)$.  So, for example, if we choose $X$ so that $X\subseteq Q_\alpha$ (which is always possible by Lemma~\ref{lem:cube}), then we can just let $Y := \{y\in\Sigma_S : |y| \le M\}$ (albeit with a possibly bigger value of $M$).
\end{remark}

\begin{remark}
Theorems~\ref{thm:cutandproject} and \ref{thm:add-points} use essentially
the same technique to induct on $|z|$. We abstract that process in the following Lemma, which we applied in Theorem~\ref{thm:added-points},
and will also be used in subsequent case studies.
\end{remark}

\begin{lemma}\label{lem:z-induction}
Let $\lambda$, $P$, $S$, $\alpha$, and $k$ be as in Theorem~\ref{thm:added-points}. 
Suppose there exists a finite set $X \subseteq \ints[\lambda]$ such that, setting $\Lambda := \diag(\alpha'_0,\ldots,\alpha'_{k-1})$ and
$J_{x'} := x'\xL S$ for each $x \in X$, it holds that
$S \subseteq \bigcup\limits_{x \in X} J_{x'}$. Let $Y = X \cup \{y \in \Sigma_S : |y| \le M\}$,
where $M = \frac{|\alpha-1|}{|\alpha|-1} \max_{x\in X} |x|$. 
Then $\Sigma_S \subseteq Q_\alpha(Y)$.
\end{lemma}
\begin{proof}
 We show for any $z\in\Sigma_S$ that $z$ is in $Q_\alpha(Y)$.  We do this by induction on $|z|$ as in Theorem~\ref{thm:add-points}. The only difference
 here is that, due to the arbitrary degree of $\lambda$, there are multiple conjugates.  If $|z| \le M$, then $z\in Y \subseteq
 Q_\alpha(Y)$ and we are done, so suppose
 $z \in \Sigma_S$ with $|z|>M$. By definition of $\Sigma_S$, 
 $z' \in S$. Since $\{J_{x'} | x \in X\}$ covers $S$,
for some $x\in X$ we have $z' \in J_{x'}$. Let $y := x \xainv z$, so that $x\xa y = z$.  As before, since $\alpha$ is a unit in $\ints[\lambda]$, we have $y\in\ints[\lambda]$.  We also have that $|y| < |z|$ by an argument identical to the one for Theorem~\ref{thm:add-points}, repeated here in more detail:
\begin{eqnarray*}
   |y| & = & \left| \left(1-\frac{1}{\alpha}\right) x + \frac{1}{\alpha}z\right|\\
       & \le & \left|1-\frac{1}{\alpha}\right| \cdot |x|
                 + \frac{|z|}{|\alpha|}\\
       & = & \frac{1}{|\alpha|}\left(|\alpha-1|\cdot |x| + |z|\right)\\
       & \le & \frac{1}{|\alpha|}\left(|\alpha-1|\cdot\frac{|\alpha|-1}{|\alpha-1|}M + |z|\right)\\
       & = & \frac{1}{|\alpha|}((|\alpha|-1)M + |z|)\\
       & < & \frac{1}{|\alpha|}((|\alpha| - 1)|z| + |z|)\\
       & = & |z|.
\end{eqnarray*}

It remains to show that $y' \in S$, thus putting $y$ in $\Sigma_S$.  This suffices: applying the inductive hypothesis to $y$ to get $y\in Q_\alpha(Y)$,
and given that $x \in Y \subseteq Q_\alpha(Y)$, we get $z = x\xa y \in Q_\alpha(Y)$ as desired.  We have
\[ y' = (x\xainv z)' = x' \xs{\Lambda^{-1}} z'\;. \]
Since $z' \in J_{x'}$, we then have $y' \in x' \xs{\Lambda^{-1}} J_{x'}
= x'\xs{\Lambda^{-1}} (x'\xs{\Lambda} S) = S$ by the definition of $J_{x'}$
and a simple calculation.
We conclude by induction that $\Sigma_S \subseteq Q_\alpha(Y)$.
\end{proof}

In the remainder of this subsection, we apply Theorem~\ref{thm:add-points} to three values of $\lambda$.  We first apply it to the two non-real values of $\lambda$ depicted in Figure~\ref{fig:case2}, showing that these two sets are both model sets and hence relatively dense Meyer sets.  Secondly, we apply the theorem to the one real quadratic case left unresolved by Theorem~\ref{theorem:main-candp}, showing that $Q_\lambda = \Sigma_\clcl{0,1}^{(\lambda)}$ for $\lambda := (-3-\sqrt{17})/2$.  This last case also makes use of Proposition~\ref{prop:sufficient-X}.  Finally, we apply Theorem~\ref{thm:added-points} to one \emph{real} $\lambda$ of degree~$3$. This is more complicated than the non-real complex $\lambda$'s, since there are two conjugates in $\opop{0,1}$.

\subsubsection{Case Study: Two Non-Real Values of $\lambda$}

We first consider the two non-real values of $\lambda$ given in Figure~\ref{fig:case2}.

\begin{proposition}\label{prop:cubicomplexcandp1}
Let $\lambda$ be the root of the polynomial $x^3+x^2-1$ closest to the point $-0.877+0.745i$ (see Figure~\ref{fig:case2}).  Then $R_\lambda = \Sigma_P$.
\end{proposition}

\begin{proof}
We have that $\lambda$ is cubic sPV with conjugate $\mu \approx 0.754878$.  (The exact form of $\mu$ is not important.)  As in other cases, $\lambda$ is itself a unit of $\ints[\lambda]$, so we set $\alpha := \lambda$ whence $\beta = \mu$.  Since $\mu > 1/2$, we may take $X := \{0,1\}$.  Plugging into Equation~(\ref{eqn:seed-upper-bound}), we get $M \approx 13.379361$.  Letting $Y$ be as in Equation~(\ref{eqn:seed-set}), we get $R_\lambda(Y) = \Sigma_P$.  Thus it suffices to show that $Y\subseteq R_\lambda$.

We can rewrite Equation~(\ref{eqn:cubic-Sigma-P}) as
\[ \Sigma_P = \{1\} \union \{ \ceiling{b\mu+c\mu^2} - b\lambda - c\lambda^2 \mid b,c\in\ints \}\;. \]
We check for all $b,c\in\ints$ such that $\left|\ceiling{b\mu+c\mu^2} - b\lambda - c\lambda^2\right| \le M$ that the resulting point is in $R_\lambda$.  Assuming this inequality, we have
\begin{align*}
M &\ge \left|\ceiling{b\mu+c\mu^2} - b\lambda - c\lambda^2\right| \ge \left|b\mu + c\mu^2 - b\lambda-c\lambda^2\right| - 1 = \left|b(\mu-\lambda) + c(\mu^2-\lambda^2)\right| - 1 \\
&= |\mu-\lambda|\left|b+c(\mu+\lambda)\right| - 1 = |\mu-\lambda|\sqrt{(b+xc)^2+(yc)^2} - 1\;,
\end{align*}
where $x := \Re(\mu+\lambda) = \mu + \Re(\lambda)$ and $y := \Im(\mu+\lambda) = \Im(\lambda) > 0$.  Set $r := |\mu+\lambda|$.  Letting $b$ and $c$ be arbitrary real numbers for a moment, a little calculus shows that for fixed $b$ the quantity $(b+xc)^2+(yc)^2$ is minimized by setting $c := -xb/r^2$.  Then we have
\[ \sqrt{(b+xc)^2+(yc)^2} \ge \sqrt{\left(b-\frac{x^2b}{r^2}\right)^2 + \frac{y^2x^2b^2}{r^4}} = \frac{|b|}{r^2}\sqrt{y^4 + y^2x^2} = \frac{|b|y}{r}\;. \]
Thus $|\mu-\lambda|\sqrt{(b+xc)^2+(yc)^2} - 1 \ge |\mu-\lambda||b|y/r - 1$, which means it suffices to consider only those $b\in\ints$ such that $|\mu-\lambda||b|y/r - 1 \le M$, or equivalently,
\[ |b| \le \frac{r(M+1)}{y|\mu-\lambda|} \approx 11.41\;. \]
Similarly, fixing $c$, the quantity $(b+xc)^2+(yc)^2$ is minimized by setting $b := -xc$, which gives $\sqrt{(b+xc)^2+(yc)^2} \ge |c|y$, so it suffices to consider only those $c\in\ints$ such that $|\mu-\lambda||c|y - 1 \le M$, or equivalently,
\[ |c| \le \frac{M+1}{y|\mu-\lambda|} \approx 10.76\;. \]
A check by computer verifies that all values of the form $\ceiling{b\mu+c\mu^2} - b\lambda - c\lambda^2$ are in $R_\lambda$ for integers $-11 \le b \le 11$ and $-10 \le c \le 10$.
\end{proof}

\begin{proposition}\label{prop:cubicomplexcandp2}
Let $\lambda$ be the root of the polynomial $x^3+x-1$ closest to the point $-0.341+1.162i$ (see Figure~\ref{fig:case2}).  Then $R_\lambda = \Sigma_P$.
\end{proposition}

\begin{proof}
This proof proceeds just as the previous proof but with different values.  We have that $\lambda$ is cubic sPV with conjugate $\mu \approx 0.682328$.  Again, we can set $\alpha := \lambda$.  Since $\mu > 1/2$, we may again take $X := \{0,1\}$.  This time we get $M \approx 8.424341$, yielding the upper bound of approximately $7.06$ for $|b|$ and $5.24$ for $|c|$.

As before, check by computer verifies that all values of the form $\ceiling{b\mu+c\mu^2} - b\lambda - c\lambda^2$ are in $R_\lambda$ for integers $-7 \le b \le 7$ and $-5 \le c \le 5$.
\end{proof}

\subsubsection{Case Study: $ -(3+\sqrt{17})/2$}\label{sec:sqrt17}

 Here we prove that
  $R_\lambda = \Sigma_\clcl{0,1}^{(\lambda)}$ for $\lambda := (-3-\sqrt{17})/2$.  
We make use of both Theorem~\ref{thm:add-points} and Proposition~\ref{prop:sufficient-X}.

\begin{proposition}\label{prop:sqrt17}
Let $\lambda := -(3+\sqrt{17})/2$.  Then $Q_\lambda = \Sigma_P$, where $P = \clcl{0,1}$.
\end{proposition}

\begin{proof}
The minimal polynomial of $\lambda$ and its conjugate $\mu := (\sqrt{17}-3)/2$ is $x^2+3x-2$.  In this case, $\lambda$ is not a unit of $\ints[\lambda]$, but $Q_\lambda$ contains the value $\alpha := 9-16\lambda \approx 65.9848$, which \emph{is} a unit of $\ints[\lambda]$ (the fundamental unit, actually).  A derivation showing that $\alpha\in Q_\lambda$ was found by computer and is given in Appendix~\ref{sec:computer-derivations}.  The conjugate of $\alpha$ is then $\beta = 1/\alpha = 9-16\mu \approx 0.015155$.

We will apply Proposition~\ref{prop:sufficient-X} in this case.  We start with the continued fraction expansion of $\mu$, which is $[a_0,a_1,a_2,\ldots] = [0,\overline{1,1,3}]$, where the bar indicates that the pattern $1,1,3$ repeats forever.  The table below gives the values needed to calculate $n$ and $m$ in that proposition.  To get $p_k$ and $q_k$ for $k\ge 0$, we can use the recurrences $p_k = a_kp_{k-1} + p_{k-2}$ and $q_k = a_kq_{k-1} + q_{k-2}$.  The values of $p_k$ and $q_k$ for $k<0$ are given in order to obtain the correct initial values when $k=0$ and $k=1$.  Recall that $\eta_k = (-1)^k(q_k\mu-p_k)$.
\[ \begin{array}{c||r|r|r|r|r|r|r|r|r|r|}
k   & -2 & -1 & 0 & 1 & 2 & 3 & 4 & 5 & 6 & 7 \\\hline\hline
a_k &    &    & 0 & 1 & 1 & 3 & 1 & 1 & 3 & 1 \\ \hline
p_k &  0 &  1 & 0 & 1 & 1 & 4 & 5 & 9 &32 &41 \\ \hline
q_k &  1 &  0 & 1 & 1 & 2 & 7 & 9 &16 &57 &73 \\ \hline
\eta_k & \mu & 1 & \mu & 1-\mu & 2\mu-1 & 4-7\mu & 9\mu-5 & 9-16\mu & 57\mu-32 & 41-73\mu \\ \hline
\end{array} \]
From the table, we get that $\eta_7 + \eta_6 = 9-16\mu = \beta < \beta/(1-\beta)$, so $k=7$ satisfies Equation~(\ref{eqn:X-bound}).  Furthermore, this is the least such $k$-value: $0.023665 \approx\eta_6+\eta_5 = 41\mu-23 > \beta/(1-\beta) = (2\mu-1)/8 \approx 0.015388$.  Thus we set $n := q_7-1 = 72$, whence $m = -36$.  Then, by Equation~(\ref{eqn:max-X-neg-lambda}), we have
\[ \max_{x\in X}|x| = \ceiling{36\mu} - 36\lambda = 20 - 36\lambda \approx 148.215901 \]
Thus we can set $X := \{1\} \union \{ \ceiling{b\mu}-b\lambda \mid b\in\ints \myand -36 \le b \le 36\}$, and this satisfies Equation~(\ref{eqn:X-condition}).\footnote{A computer check reveals that $X$ is minimal, i.e., no proper subset of $X$ satisfies (\ref{eqn:X-condition}).}  Since $\alpha > 1$, we have, from Equation~(\ref{eqn:seed-upper-bound}), $M = \max_{x\in X}|x| = 20-36\lambda$ and thus $Y = X$ in (\ref{eqn:seed-set}).

To summarize, we have $\Sigma_P = Q_\lambda(X)$ by Theorem~\ref{thm:add-points}.  Finally, a computer run shows that in fact, $X\subseteq Q_\lambda$, and thus $\Sigma_P = Q_\lambda$.
\end{proof}


\subsubsection{Case Study: $\lambda_7$}


Recall Proposition~\ref{prop:lambda-n-PV}, in which we characterize
those $\lambda_n$ (Eq.~(\ref{eqn:lambda-n})) that are sPV (thus giving rise to uniformly discrete $Q_{\lambda_n}$), for odd $n$.  We know that $Q_{\lambda_3}$
and $Q_{\lambda_5}$ are also relatively dense. In this section
we use Theorem~\ref{thm:added-points} to prove that $Q_{\lambda_7}$ is also relatively dense.

For the remainder of this
 section let $\lambda$ denote $\lambda_7 \approx 5.0489$, whose minimal polynomial is
$x^3-6x^2+5x-1$.
Denote the conjugates of $\lambda$ by $\mu, \nu \in \clcl{0,1}$, where $\mu \approx 0.30798$ and $\nu \approx 0.64310$.  For $x \in \ints[\lambda]$,
where $x$ is of the form $a_0+a_1\lambda+a_2\lambda^2$ for $a_0, a_1, a_2 \in \ints$, we denote its conjugates
by $x' =  a_0+a_1\mu+a_2\mu^2$ and $x'' = a_0+a_1\nu+a_2\nu^2$.
By Fact~\ref{fact:main}, we have $Q_\lambda \subseteq \Sigma_P$, where
$P = \clcl{0,1}^{\times 2}$. Let $S \subseteq P$ denote the
$(2\nu-1) \times (2\nu-1)$ square with corners at
$(1-\nu,1-\nu)$ and $(\nu, \nu)$. It is convenient to introduce 
the parallelogram $T$ 
with corners at $(0,0), (\mu,\nu), (1-\mu,1-\nu), $ and $(1,1)$.
It is easy to see that $S \subseteq T \subseteq P$
and hence $\Sigma_S \subseteq \Sigma_T \subseteq \Sigma_P$. We have,
\begin{eqnarray*}
   \Sigma_T = \{x \in \ints[\lambda] \mid (x',x'') \in T\}.
\end{eqnarray*}
We can characterize the points of $\Sigma_T$ as follows. Let
$\beta = \mu+\nu-\mu\nu$ and define 
\begin{eqnarray*}
    n(m) & := & \lceil m/\lambda \rceil = \lceil \mu\nu m\rceil\\
    \ell(m) & := & \lfloor \beta m \rfloor + n(m)\\
    p(m) & := & n(m) - \ell(m)\lambda + m \lambda^2.
\end{eqnarray*}

\begin{proposition}
 $ \Sigma_T = \{1,\lambda, 1-\lambda\}\cup\{p(m) \mid m \in \ints\}.$
\end{proposition}
\begin{proof}
  It is easiest to see this by transforming $P$ via the matrix $W$ defined
in Fact~\ref{fact:main}. In this instance, 
\begin{align*}
W &= \left[\begin{array}{cc}
1 &\mu  \\
1 &\nu\\
\end{array}\right]\;,&
W^{-1} &=\frac{1}{\nu-\mu}\left[\begin{array}{cc}
\nu &-\mu  \\
-1 &  1\\
\end{array}\right].
\;
\end{align*}
We interpret the column vector $[x',x'']^{T}$ as the point $(x',x'')$. 
Let $\Omega = W^{-1}T$. It is easy to see that $W^{-1}$ maps
$(0,0) \mapsto (0,0)$, $(\mu, \nu) \mapsto (0,1)$, $(1,1) \mapsto (1,0)$,
and $(1-\mu,1-\nu) \mapsto (1,-1)$. Thus $\Omega$ is
a parallelogram with corners at $(0,0), (0,1), (1,0)$, and $(1,-1)$.
More generally, $W^{-1}$ maps $(x',x'') = (a_0+a_1\mu+a_2\mu^2, a_0+a_1\nu+a_2\nu^2)$
to $(a_0-\mu\nu a_2, a_1+(\mu+\nu)a_2)$. Thus $(x',x'') \in T$
iff the following conditions hold:
\begin{eqnarray}\label{eqn:p(m)-inequalities}
\begin{array}{ccccc}
0    &  \le & a_0-\mu\nu a_2    & \le & 1\\
-a_0+\mu\nu a_2 &  \le & a_1+(\mu+\nu)a_2  & \le & -a_0+\mu\nu a_2+1\;
\end{array}\;
\end{eqnarray}
Note that for $a_2 \not = 0$, all inequalities above are proper because
$\mu$ and $\nu$ are irrational.
The first inequality shows that $a_0 = \lceil \mu\nu a_2\rceil$,
and the second $a_1 = -\lceil \mu\nu a_2\rceil + \lceil \mu\nu-(\mu+\nu)a_2\rceil
= -\lceil \mu\nu a_2\rceil - \lfloor \beta a_2\rfloor$. Letting $m$ denote
$a_2$, this yields $a_0 = n(m)$ and $a_1 = -\ell(m)$. Thus the points in
$\Sigma_T$ are of the form $p(m)$, $m \in \ints$. The exceptions
are the nonzero corners of $T$, namely
 $(1,0)$, $(0,1)$ and $(1,-1)$, which correspond to the values $1$, $\lambda$
and $1-\lambda$.
\end{proof}

\begin{remark}
  The apparent ``outlier" nature of the points $1, \lambda$, and $1-\lambda$ is
 an artifact of the definition of $p(m)$, which has a $\lambda^2$ term
 for any $m \not = 0$.
 In that case, the inequalities in Eq.~(\ref{eqn:p(m)-inequalities}) are proper.
  However, for $m = a_2 = 0$, the equalities can be
 met, which accounts for the points  $1, \lambda$, and $1-\lambda$.
The one value of $p(m)$ that works out for
one of the initial points is $p(0)$, all of whose terms are $0$.
\end{remark}

\vskip 0.1in

We find it is possible to choose
 a finite number of points in $\Sigma_T$  to construct the set 
$Y$ of Theorem~\ref{thm:added-points}. We remark that all the points
we use are in $T\cmpl S$. \\

We now prove the main result of this section. Since
$\lambda$ is a unit, by Theorem~\ref{thm:added-points}, there exists
a finite $Y \subseteq \Sigma_P$ such that $\Sigma_S \subseteq Q_\lambda(Y)$.
Our task here is to explicitly construct such a $Y$ such that
$Y \subseteq Q_\lambda$. That, in turn, requires finding an
explicit finite cover of the set $S$, indexed by a finite set $X$.
Once $X$ is found, the problem reduces to the tedious but routine
verification that $Y \subseteq Q_\lambda$, and finally
appealing to Lemma~\ref{lem:z-induction}.

\begin{lemma}\label{lem:Sigma7-containment}
 $\Sigma_S \subseteq Q_\lambda$.
\end{lemma}
\begin{proof}
   We adopt the notations and conventions of the proof of
Theorem~\ref{thm:added-points}. We construct the set $Y$ and show that $Y
\subseteq Q_\lambda$. Since $Q_\lambda(Y) = Q_\lambda$, the result follows easily
by Theorem~\ref{thm:added-points}.

Let $X = \{\lambda, 1-\lambda\}\cup\{p(m)|m=\pm12, \pm13, \pm16, \pm17\}$. 
We claim that $C=\{J_x | x \in X\}$ forms a cover of $S$. 
Following the notation of Theorem~\ref{thm:added-points}, we set
$\varepsilon = 1-\nu$. Hence the sets in $C$ are of the form,
\begin{eqnarray}\label{eqn:cover-sets}
 J_x = \clcl{x'\xs{\mu}(1-\nu),\;x'\xs{\mu}\nu} \times  \clcl{x''\xs{\nu}(1-\nu),\;x''\xs{\nu}\nu}.
\end{eqnarray}
These are rectangles of width $(2\nu - 1)\mu$ and height $(2\nu-1)\nu$ extrapolated from
$(x',x'')$ with $1-\nu$ and $\nu$. In Figure~\ref{fig:cover},
we show the set $S$ in red and the sets $J_x$ in blue (displaying the index $x$,
which is either $\lambda$, $1-\lambda$ or $m$ for the appropriate points
$p(m)$). We slightly abuse notation and denote $J_{p(m)}$ by $J_m$. Using decimal approximations,
we list the sets $J_x$ in Table~\ref{tab:cover}, each one specified by the coordinates
$(x'_0,x''_0)$ of its lower left-hand corner and $(x'_1,x''_1)$ of its upper right-hand
corner.

\begin{figure}
\centering
\includegraphics[width=0.65\textwidth]{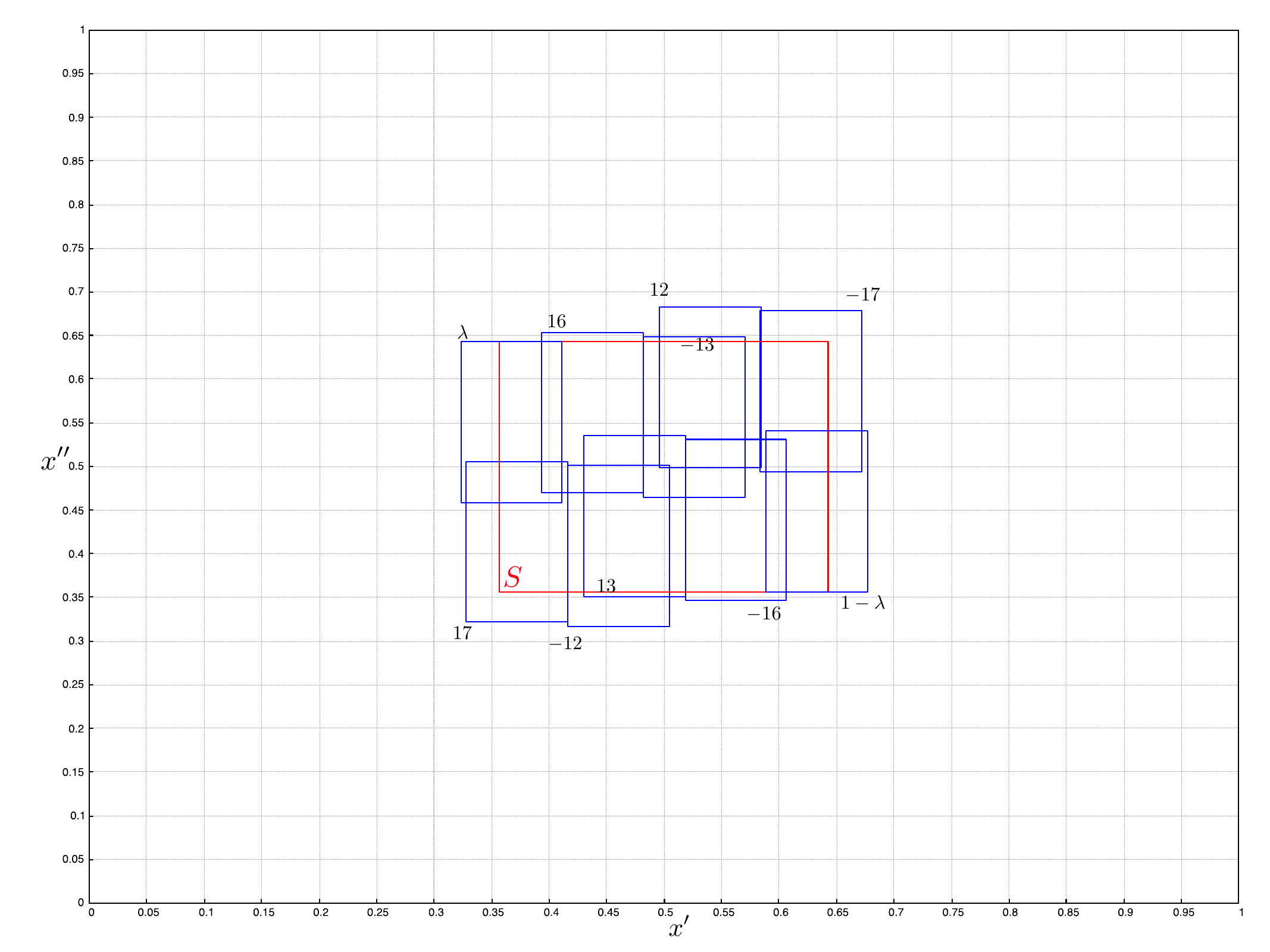}
\caption{A cover of $S$. $S$ is in red, the sets $J_{\lambda}$, 
$J_{1-\lambda}$, and $J_{m}$ for $|m| = 12, 13, 16, 17$
in blue (labelled with the value of $m$ rather than $p(m)$ for the latter sets). The entire graph represents $P$.}
\label{fig:cover}
\end{figure}

\begin{table}[h!]
\centering
\begin{tabular}{|c|c|c|}
\hline
$x$ & $(x'_0,x''_0)$ & $(x'_1,x''_1)$ \\ \hline
$\lambda$  & (0.32304408,0.45904241)  & (0.41119002,0.64310413)\\
$p(16)$ &    (0.39367557,0.46947633)  & (0.48182156,0.65353804)\\
$p(-13)$ &   (0.48175166,0.46483280)  & (0.56989766,0.64889452)\\
$p(12)$  &   (0.49574111,0.49871152)  & (0.58388711,0.68277324)\\
$p(-17)$ &   (0.58381721,0.49406799)  & (0.67196321,0.67812971)\\
$p(17)$ &    (0.32803679,0.32187029)  & (0.41618279,0.50593201)\\
$p(-12)$ &   (0.41611289,0.31722676)  & (0.50425889,0.50128848)\\
$p(13)$  &   (0.43010234,0.35110548)  & (0.51824834,0.53516720)\\
$p(-16)$ &   (0.51817843,0.34646195)  & (0.60632443,0.53052367)\\
$1-\lambda$ & (0.58880998,0.35689587)  & (0.67695598,0.54095759)\\
\hline
\end{tabular}
\caption{Coordinates of the $J_x$ of Figure~\ref{fig:cover}.}\label{tab:cover}
\end{table}

Observe that for each $x$ such that $x''_0 
\le 1-\nu \approx 0.35689587$, the value of $x''_1$ exceeds $1/2$ (these are the points $p(m)$ for $m = 17, -12, 13, -16$ and $1-\lambda$; for that last point, $x''_0 = 1-\nu$ exactly).
Similarly, for each $x$ such that $x''_1 \ge \nu \approx 0.64310413$,
the value of $x''_0$ is less than $1/2$
(these are the reflections $p(m)$ for $m = -17, 12, -13, 16$ and $\lambda$; for that last point, $x''_1 = \nu$ exactly). We also note that those pairs of $J_x$ which appear to touch
along the vertical boundaries actually intersect. For example, for $p(17)$, 
$x'_1 = 0.41618279$, which exceeds $0.41611289$, the value
of $x'_0$ for $p(-12)$. This can be similarly verified for the pairs
$(p(16), p(-13))$, $(p(12),p(-17))$, $(p(13), p(-16))$, so that 
the sets $J_{17}\cup J_{-12}$, $J_{16}\cup J_{-13}$, $J_{12}\cup J_{-17}$, 
and $J_{13}\cup J_{-16}$ are each connected. Thus we
claim that if we divide $S$ into the two subsets $L = \{(x',x'') \in S| 1-\nu \le x' \le \nu ~\mbox{\rm and}~ 1-\nu \le x'' \le 1/2\}$
and $U = \{(x',x'') \in S| 1-\nu \le x' \le \nu ~\mbox{\rm and}~ 1/2 \le x'' \le \nu\}$, we find that,
\begin{eqnarray}
\label{eqn:Lcover}
L & \subseteq & (J_{17}\cup J_{-12}) \cup (J_{13} \cup J_{-16}) \cup J_{1-\lambda},\\
\label{eqn:Ucover}
U & \subseteq & J_{\lambda}\cup ( J_{16} \cup J_{-13})\cup ( J_{12}\cup J_{-17}),
\end{eqnarray}
which shows that $S \subseteq \bigcup C = \bigcup_{x \in X} J_x$. 

To verify Eq.~(\ref{eqn:Lcover}), given that $x''_0 \le 1-\nu$ and $x''_1 > 1/2$ for
each set on the RHS, we only need to verify that each point $(x',x'')$ with
 $1-\nu \le x' \le \nu$ and $1-\nu \le x'' \le 1/2$ is contained in one of $J_{17}\cup J_{-12}$, $J_{13}\cup J_{-16}$, or $J_{1-\lambda}$. Using Table \ref{tab:cover}, we enumerate the cases
 in the following table:

\begin{table}[h!]
\centering
\begin{tabular}{|c|c|}
\hline
$1-\nu \le x'' \le 1/2$ and $x'$ in: & $(x',x'')$ in: \\ \hline
$[1-\nu, 0.50425889]$  &   $J_{17}\cup J_{-12}$\\
$[0.43010234, 0.60632443]$  &   $J_{13}\cup J_{-16} $\\
$[0.58880998, \nu]$  &   $J_{1-\lambda} $\\
\hline
\end{tabular}
\end{table}
\noindent The union of the intervals in the left column is $[1-\nu,\nu]$, so Eq.~(\ref{eqn:Lcover}) follows.

\noindent Eq.~(\ref{eqn:Ucover}) is established similarly, for $x'' \ge 1/2$:
 
 \begin{table}[h!]
\centering
\begin{tabular}{|c|c|}
\hline
$1/2 \le x'' \le \nu$ and $x'$ in: & $(x',x'')$ in: \\ \hline
$[1-\nu, 0.41119002]$  &   $J_{\lambda} $\\
$[0.39367557, 0.56989766]$  &   $ J_{16}\cup J_{-13} $\\
$[0.49574111,\nu]$  &   $J_{12}\cup J_{-17}$\\
\hline
\end{tabular}
\end{table}

\noindent We conclude that $C$ covers $S$.\\

We now determine the set $Y := X \union\{y\in\Sigma_S : |y| \le M\}$.  Let $Y' := X \union\{y\in\Sigma_T : |y| \le M\}$.  In this case $M = \frac{|\lambda-1|}{|\lambda|-1} \max_{x\in X} |x| = \max_{x\in X} |x|$, since $\lambda > 1$.
Observing that the $\lambda^2$ term dominates, it is not hard to see that $|p(m)|$ is monotonically increasing in $|m|$. Hence
$Y' = \{1,\lambda, 1-\lambda\}\cup\{p(m) : |m| \le 17\}$.
A derivation
of each of these points from the starting points $p_0=0$ and $p_1=1$ is given in 
Appendix~\ref{sec:computer-derivations}, Table~\ref{tab:derivation}. This proves that $Y' \subseteq Q_\lambda$, and
hence $Q_\lambda(Y') = Q_\lambda$.  It follows that $Q_\lambda(Y) = Q_\lambda$, since $\{0,1\}\subseteq Y \subseteq Y'$.  Now $Y$ has been constructed to satisfy
the properties of the $Y$ of Lemma~\ref{lem:z-induction} and, furthermore,
$\lambda$ is a unit so we can use it in place of $\alpha$ in that lemma. Hence
 $\Sigma_S \subseteq Q_\lambda$.
\end{proof}

  By Lemma~\ref{lem:Sigma7-containment}, $\Sigma_S \subseteq Q_{\lambda}$,
hence since $\Sigma_S$ is relatively dense, so is $Q_\lambda$.
Moreover, by Fact~\ref{fact:main}, $Q_\lambda \subseteq \Sigma_P$, 
so we have,
  
\begin{theorem}\label{thm:lambda-7-meyer}
$Q_{\lambda_7}$ is a Meyer set.
\end{theorem}

This, and many similar results, can be proved via Corollary~\ref{cor:meyer-if-unit} below, although that does not show explicitly that $Q_\lambda$ contains a model set.



\subsection{Affine Embedding}\label{sec:affine}

\subsubsection{Definitions and General Results}

Recall that for $x,y\in\complexes$, the map $\map{\rho_{x,y}}{\complexes}{\complexes}$ maps $z\in\complexes$ to $x\xz y$ (Definition~\ref{def:rho}).  This definition extends naturally to any module over a commutative ring (Definition~\ref{def:E-rho-general}).

\begin{definition}\label{def:affine-embedding}
Let $R$ be a commutative ring and $A, B$ subsets of $R$. 
If $\rho_{a,b}(A) \subseteq B$ for some $a,b\in R$ such that $x(a-b)\ne 0$ for all $x\in R\cmpl\{0\}$, then
we say $A$ \emph{affinely embeds}
in $B$, and write $A \affemb B$. 
If $A \affemb B$ and $B \affemb A$,
we say $A$ and $B$ are \emph{affine-equivalent}, and write 
$A \affeq B$.
\end{definition}

The condition on $a$ and $b$ in Definition~\ref{def:affine-embedding} is equivalent to $\rho_{a,b}$ being one-to-one, hence the use of the term ``embedding.''  If $R$ has no zero-divisors, then $\rho_{a,b}$ is a one-to-one map just when $a\ne b$.  We will only apply Definition~\ref{def:affine-embedding} when $R$ has no zero-divisors.

The $\affemb$ relation is clearly reflexive and transitive (cf.\ Fact~\ref{fact:rho-basic}(\ref{item:rho-composition}) for $R = \complexes$), making $\affeq$ an equivalence relation.

\begin{fact}\label{fact:affine-embedding}
Let $F$ be either $\reals$ or $\complexes$.  Let $A$ and $B$ be subsets of $F$ such that $A\affemb B$.
\begin{enumerate}
\item
If $B$ is uniformly discrete, then $A$ is uniformly discrete.
\item
If $A$ is relatively dense in $F$, then $B$ is relatively dense in $F$.
\item\label{item:affine-embedding-Q-R}
For any $\lambda\in F$, \ $Q_\lambda(A) \affemb Q_\lambda(B)$ and $R_\lambda(A) \affemb R_\lambda(B)$.  (This follows from Lemma~\ref{lem:translate-Q-R}.)
\end{enumerate}
\end{fact}

\begin{lemma}\label{lem:universal-affine-embedding}
Let $T$ be any finite subset of $\rats[x]$. Then $T \hookrightarrow Q_{[x]}$.  Furthermore, there exist distinct $a,b\in\ints[x]$ such that $\rho_{a,b}(T)\subseteq Q_{[x]}$ and the only irreducible monic polynomials in $\rats[x]$ dividing $b-a$ are $x$ and $x-1$.
\end{lemma}

\begin{proof}
  Suppose $T = \{f_1,\ldots,f_n\}$, where $f_i \in \rats[x]$
  for each $1 \le i \le n$. 
Choose $d,\ell, k \in \ints^+$ sufficiently large that for all $1 \le i \le n$, $df_i \in \ints[x]$ and
$(df_i(\lambda)+\ell)\cdot \lambda^k(1-\lambda)^k \in \opop{0,1}$ for any
$\lambda \in \opop{0,1}$. Then, writing $h(x) := x^k(1-x)^k$ and $g_i := (df_i+\ell)\cdot h$,
we find that $g_i\in\ints[x]$ and $g_i(\lambda) \in \opop{0,1}$ for any $\lambda \in \opop{0,1}$. 
By the characterization given in Theorem~\ref{thm:Q-x}, 
this implies $g_i \in Q_{[x]}$.  Now define the polynomials,
\begin{align*}
   a &:= \ell h\;,\\
   b &:= a+dh = (\ell+d) h\;.
\end{align*} 
Then $a,b\in\ints[x]\subseteq \rats[x]$ and $a\ne b$.  In fact, $b-a = dh$, whose only irreducible monic factors in $\rats[x]$ are $x$ and $x-1$.  Moreover, for $1\le i\le n$,
\[ \rho_{a,b}(f_i) = (1-f_i)a + f_ib = a + f_i (b - a) = \ell h + df_i h = (df_i + \ell) h = g_i\;. \]
Thus $\rho_{a,b}(T) \subseteq Q_{[x]}$, and hence $T \hookrightarrow Q_{[x]}$ via $a$ and $b$ satisfying the additional hypothesis. 
\end{proof}

\begin{remark}
Lemma~\ref{lem:universal-affine-embedding} holds just as well for infinite $T$, provided $d$, $\ell$, and $k$ exist with the requisite properties, i.e., $\bigcup_{f\in T}\{ |f(\lambda)| : \lambda\in\opop{0,1}\}$ is bounded and all coefficients in elements of $T$ share a common denominator.
\end{remark}

\begin{lemma}\label{lem:lambda-embedding}
Let $\lambda \in \complexes\cmpl\{0,1\}$ and $S$ be any finite subset of $\rats[\lambda]$. 
Then $S \hookrightarrow Q_\lambda$.
\end{lemma}

\begin{proof}
Let $f_1,\ldots,f_n\in\rats[x]$ be such that $S = \{f_1(\lambda),\ldots,f_n(\lambda)\}$.  Set $T := \{f_1,\ldots,f_n\}$.  Then let $a,b\in\ints[x]$ be as in Lemma~\ref{lem:universal-affine-embedding}, and set
\begin{align*}
c &:= a(\lambda)\;, & d &:= b(\lambda)\;.
\end{align*}
We claim first that $c\ne d$:  If $\lambda$ is transcendental over $\rats$, then $c \ne d$ follows immediately from the fact that $a \ne b$.  If $\lambda$ is algebraic over $\rats$, then let $p\in\rats[x]$ be the minimal (monic) polynomial of $\lambda$.  Since $\lambda\notin\{0,1\}$, \ $p$ is neither $x$ nor $x-1$, and thus $p \notdiv b-a$.  It follows that $a(\lambda) \ne b(\lambda)$, i.e., $c \ne d$.  This proves the claim.

Finally, we have $\rho_{a,b}(T)\subseteq Q_{[x]}$ by Lemma~\ref{lem:universal-affine-embedding}.  Evaluating both sides at $\lambda$, we then get $\rho_{c,d}(S) \subseteq Q_\lambda$ by Lemma~\ref{lem:poly-eval}. 
Thus $S \hookrightarrow Q_\lambda$.
\end{proof}

\begin{theorem}
For any nontrivial sPV $\lambda$ such that $Q_\lambda$ contains a unit of $\ints[\lambda]$ other than $1$, and for any $S$ defined as in Theorem~\ref{thm:added-points}, $\Sigma_S \hookrightarrow Q_\lambda$. 
\end{theorem}

\begin{proof}
  By Theorem~\ref{thm:added-points}, there exists a finite set $Y \subseteq \Sigma_P 
  \subseteq \ints[\lambda]$ such that $\Sigma_S \subseteq Q_\lambda(Y)$. 
By Lemma~\ref{lem:lambda-embedding}, $Y \hookrightarrow Q_\lambda$. 
By Fact~\ref{fact:affine-embedding}(3), 
$Q_\lambda(Y) \hookrightarrow Q_\lambda(Q_\lambda) = Q_\lambda$. 
Hence $\Sigma_S \hookrightarrow Q_\lambda$.
\end{proof}

\noindent Using Fact~\ref{fact:affine-embedding}(2), 

\begin{corollary}\label{cor:unit-gives-relative-density}
For any nontrivial sPV $\lambda$ such that $Q_\lambda$ contains a unit of $\ints[\lambda]$ other than $1$, \ $Q_\lambda$ is relatively dense (in either $\reals$ or $\complexes$, as $\lambda\in\reals$ or $\lambda\in\complexes\cmpl\reals$, respectively).
\end{corollary}

Since $\{0,1\} \subseteq \Sigma_P$, it follows trivially that
$\{0,1\} \hookrightarrow \Sigma_P$. Therefore $Q_\lambda(\{0,1\}) = Q_\lambda \hookrightarrow
Q_\lambda(\Sigma_P)$. However, $\Sigma_P$ is closed under $\lambda$-extrapolation, i.e.,
$Q_\lambda(\Sigma_P) = \Sigma_P$. Hence $Q_\lambda \hookrightarrow \Sigma_P$.
Since $\Sigma_P$ is uniformly discrete, by Fact~\ref{fact:affine-embedding}(1) 
we can use this as an alternative proof
that $Q_\lambda$ is uniformly discrete.

\begin{corollary}
For any sPV $\lambda$ such that $Q_\lambda$ contains a unit of $\ints[\lambda]$ other than $1$, \ $Q_\lambda$ is a Delone set.
\end{corollary}

Because $Q_\lambda \subseteq \Sigma_P$, and $\Sigma_P$ is a Meyer set, we have,

\begin{corollary}\label{cor:meyer-if-unit}
For any nontrivial sPV $\lambda$ such that $Q_\lambda$ contains a unit of $\ints[\lambda]$ other than $1$, \ $Q_\lambda$ is a Meyer set.
\end{corollary}

\begin{remark}
   Alternatively, we
   note that if $A \affemb B \affemb C$ and $A, C$ are Meyer sets, then
 $B$ is also Meyer.
In this case we have 
$\Sigma_S \hookrightarrow Q_\lambda \hookrightarrow \Sigma_P$,
and hence $Q_\lambda$ is a Meyer set.
\end{remark}

The next lemma (actually its corollary) has the same conclusion as Corollary~\ref{cor:meyer-if-unit} but with different hypotheses; it does not assume that $\lambda$ is sPV.

\begin{lemma}\label{lem:discrete-minkowski-difference}
For any $\lambda\in\complexes$, if $R_\lambda$ is discrete, then $R_\lambda - R_\lambda$ is uniformly discrete.
\end{lemma}

\begin{proof}
The lemma is trivial if $\lambda\in\{0,1\}$, so assume otherwise.  If $R_\lambda$ is discrete, then $R_\lambda$ is uniformly discrete by Corollary~\ref{cor:general-equivalences} (and furthermore, $R_\lambda = Q_\lambda$).  We now have
\[ R_\lambda - R_\lambda = Q_\lambda - Q_\lambda = Q_\lambda(\{0,1\}-\{0,1\}) = Q_\lambda(\{-1,0,1\})\;, \]
the last equality due to Lemma~\ref{lem:distribute}.  $R_\lambda$ contains the three distinct values $0$, $\lambda(1-\lambda)$, and $2\lambda(1-\lambda)$, which form a $3$-arithmetic progression (see Corollary~\ref{cor:arith-progressions}), and thus it is clear that $\{-1,0,1\} \affemb \{0,\lambda(1-\lambda),2\lambda(1-\lambda)\}$.  We therefore get $R_\lambda - R_\lambda \affemb R_\lambda(\{0,\lambda(1-\lambda),2\lambda(1-\lambda)\}) = R_\lambda$ by Fact~\ref{fact:affine-embedding}.  It then follows from the same Fact that $R_\lambda - R_\lambda$ is uniformly discrete.
\end{proof}

\begin{corollary}\label{cor:Meyer-set}
For $F$ being either $\reals$ or $\complexes$, and for any $\lambda\in F$, if $R_\lambda$ is discrete and relatively dense in $F$, then $R_\lambda$ is a Meyer set.
\end{corollary}

\begin{proof}
By Lemma~\ref{lem:discrete-minkowski-difference} and Definition~\ref{def:Meyer-set}.
\end{proof}

\subsubsection{Case Study: $\lambda := -(3+\sqrt{13})/2$}\label{sec:case-study-root13}

In this subsection we use Theorems~\ref{theorem:main-candp} and \ref{thm:add-points} to identify points that are in $\Sigma_P$ but not in $Q_\lambda$, for $\lambda := -(3+\sqrt{13})/2$ and $P := \clcl{0,1}$.
Again making use of
Theorem~\ref{thm:add-points}, as well as affine embedding,
we also show that $Q_\lambda$ is relatively dense in $\reals$ and hence a Meyer set. 

\begin{proposition}\label{prop:root-13}
Letting $\lambda := -(3+\sqrt{13})/2$, we have
\[ Q_\lambda(\{0,1,2\lambda\}) = Q_\lambda(\{0,1,1-2\lambda\}) = \Sigma_P\;, \]
where $\Sigma_P = \Sigma_P^{(\lambda)}$ is the cut-and-project (model) set $\{x\in\ints[\lambda] : x' \in \clcl{0,1}\}$ of Definition~\ref{def:Sigma_P}.
\end{proposition}

\begin{proof}
We let $\mu := \lambda' = (\sqrt{13}-3)/2 \approx 0.302776$ be the conjugate of $\lambda$, and we note that $\mu^2 = 1-3\mu$.  The fact that $\lambda$ is already a unit of $\ints[\lambda]$ means we can set $\alpha := \lambda$ (and thus $\beta = \mu$) in Theorem~\ref{thm:add-points} (and obviously, $\alpha\in Q_\lambda$ in this case).  We have $(2\lambda)' = 2\mu \in \clcl{0,1}$, so $2\lambda\in\Sigma_P$.  In order to apply Theorem~\ref{thm:add-points}, we let $X := \{0,1,\lambda,2\lambda\}$.  One can readily check using decimal approximations that
\[ \clcl{0,1} = \clcl{0,\mu} \union \clcl{1-\mu,1} \union \clcl{4\mu-1,5\mu-1} \union \clcl{8\mu-2,9\mu-2} = \bigcup_{x\in X} \clcl{(1-\mu)x',(1-\mu)x' + \mu}\;, \]
and thus $X$ satisfies Eq.~(\ref{eqn:X-condition}).  Letting $M$ be as in Eq.~(\ref{eqn:seed-upper-bound}), we have $M = \frac{\lambda-1}{\lambda+1}\max_{x\in X}|x| = \frac{2+\sqrt{13}}{3}|2\lambda| = \frac{2+\sqrt{13}}{3}(-2\lambda) = (19+5\sqrt{13})/3 \approx 12.342585$.  Then by Eqs.~(\ref{eqn:degree2}) and (\ref{eqn:seed-set}),
\[ Y = \Sigma_P \intersect \clcl{-M,M} = \{3\lambda,\;2\lambda,\;\lambda,\;0,\;1,\;1-\lambda,\;1-2\lambda,\;1-3\lambda\}\;. \]
From Theorem~\ref{thm:add-points}, we know that $Q_\lambda(Y) = \Sigma_P$, and thus we are done if we can show that $Y \subseteq Q_\lambda(\{0,1,2\lambda\})$.  Clearly, $\{2\lambda,\lambda,0,1,1-\lambda\}\subseteq Q_\lambda(\{0,1,2\lambda\})$.  For the other three elements of $Y$, we have, noting that $\lambda^2 = 1-3\lambda$,
\begin{align*}
3\lambda &= 1-(1-3\lambda) = 1-\lambda^2 = (1-\lambda) + \lambda(1-\lambda) = 1\xl (1-\lambda)\;, \\
1-2\lambda &= \lambda + (1-3\lambda) = \lambda+\lambda^2 = (1-\lambda)\lambda + 2\lambda^2 = \lambda \xl (2\lambda)\;, \\
1-3\lambda &= \lambda^2 = 0\xl \lambda\;.
\end{align*}
Thus $\Sigma_P = Q_\lambda(Y) \subseteq Q_\lambda(Q_\lambda(\{0,1,2\lambda\})) = Q_\lambda(\{0,1,2\lambda\}) \subseteq \Sigma_P$, since $\{0,1,2\lambda\}\subseteq \Sigma_P$ and $\Sigma_P$ is $\lambda$-convex.  Thus $Q_\lambda(\{0,1,2\lambda\}) = \Sigma_P$.  To see that $Q_\lambda(\{0,1,1-2\lambda\}) = \Sigma_P$, observe that $1-2\lambda\in\Sigma_P$ and that $2\lambda \in Q_\lambda(\{0,1,1-2\lambda\})$; indeed,
\[ 2\lambda = 1-\lambda-(1-3\lambda) = 1-\lambda-\lambda^2 = (1-\lambda)(1-\lambda) + \lambda(1-2\lambda) = (1-\lambda)\xl(1-2\lambda)\;. \]
Thus we have $Q_\lambda(\{0,1,1-2\lambda\}) = Q_\lambda(\{0,1,2\lambda\}) = \Sigma_P$.
\end{proof}

We know from Theorem~\ref{theorem:main-candp} that $Q_{-(3+\sqrt{13})/2} \ne \Sigma_P$.  Proposition~\ref{prop:root-13} gives us two specific points that are in $\Sigma_P$ but not in $Q_\lambda$:

\begin{corollary}\label{cor:non-membership}
Let $\lambda := -(3+\sqrt{13})/2$.  Then $\{2\lambda,1-2\lambda\} \subseteq \Sigma_P \cmpl Q_\lambda$.
\end{corollary}

\begin{proof}
One readily checks that $2\lambda$ and $1-2\lambda$ are both in $\Sigma_P$.  But if either one (say, $2\lambda$) is in $Q_\lambda$, then we would have $Q_\lambda = Q_\lambda(\{0,1,2\lambda\}) = \Sigma_P$, and we know by Theorem~\ref{theorem:main-candp} that $Q_\lambda \ne \Sigma_P$.
\end{proof}

\begin{remark}
Another way to see that $2\lambda \notin Q_\lambda$ without using Proposition~\ref{prop:root-13} is to notice that $2\lambda \in (1-\lambda)\ints[\lambda] + 2$.  This implies $2\lambda \notin (1-\lambda)\ints[\lambda] + \{0,1\}$ because $1-\lambda \notdiv 2$.  Then apply Corollary~\ref{cor:ideal} with $D = \ints[\lambda]$.
\end{remark}

Using the idea of affine embedding, it follows easily from Proposition~\ref{prop:root-13} that $Q_{-(3+\sqrt{13})/2}$ is relatively dense in $\reals$. 


\begin{corollary}[to Proposition~\ref{prop:root-13}]\label{cor:13-relatively-dense}
$Q_\lambda$ is relatively dense in $\reals$, where $\lambda := -(3+\sqrt{13})/2$.
\end{corollary}

\begin{proof}
Observe that $\{0,1,2\lambda\} \affemb Q_\lambda$.  In fact,
\[ \rho_{0,1-\lambda}(\{0,1,2\lambda\}) = \{0,1-\lambda,2\lambda(1-\lambda)\} = \{0,1-\lambda,\lambda \xl (1-\lambda)\} \subseteq Q_\lambda\;. \]
Thus by Proposition~\ref{prop:root-13} and Fact~\ref{fact:affine-embedding} (with $F = \reals$), we have $\Sigma_P = Q_\lambda(\{0,1,2\lambda\}) \affemb Q_\lambda(Q_\lambda) = Q_\lambda$, and since $\Sigma_P$ is relatively dense in $\reals$, it follows that $Q_\lambda$ is relatively dense in $\reals$.
\end{proof}

\begin{remark}
The last proof shows that $\Sigma_P \affeq Q_\lambda$ for $\lambda := -(3+\sqrt{13})/2$, because clearly, $Q_\lambda \subseteq \Sigma_P$ and thus $Q_\lambda\affemb\Sigma_P$.
\end{remark}

\subsubsection{A Concise Proof of Theorem~\ref{thm:main}}
\label{sec:concise-proof}

In this subsection, we use affine embedding to give an alternate, shorter proof of Theorem~\ref{thm:main}.  The proof we give here is less informative than the original one, as it completely obscures any connection with cut-and-project sets.

\begin{definition}\label{def:sub-Meyer}
For $n\ge 1$, we will say that a point set $S\subseteq\reals^n$ is \emph{sub-Meyer} if $S-S$ is uniformly discrete.
\end{definition}

Fact~\ref{fact:meyer-set-characterization} implies that if $S\subseteq\reals^n$ is sub-Meyer and relatively dense in $\reals^n$, then $S$ is a Meyer set.  We can identify $\complexes$ with $\reals^2$ in the usual way.  

Recall that $\discrete = \complexes\cmpl\convex = \{ \lambda\in\complexes\mid \mbox{$R_\lambda$ is not convex} \} =  \{ \lambda\in\complexes\mid \mbox{$R_\lambda$ is discrete} \}$ (Definition~\ref{def:convex} and Corollary~\ref{cor:general-equivalences}).

The next theorem is a strengthened version of Lemma~\ref{lem:discrete-minkowski-difference}, and some of the proof is cribbed from there.

\begin{theorem}\label{thm:sub-Meyer}
$Q_\lambda(S)$ is sub-Meyer for any $\lambda\in\discrete$ and finite $S\subseteq\rats(\lambda)$.
\end{theorem}

\begin{proof}
Fix $\lambda\in\discrete$ and finite $S\subseteq\rats(\lambda)$.  We have that $Q_\lambda$ is uniformly discrete by Corollary~\ref{cor:general-equivalences}.  Letting $T := S-S$, we see that $T$ is a finite subset of $\rats(\lambda)$.  We have that $\lambda$ is algebraic by Theorem~\ref{thm:transcendental-lambda} (which is subsumed by Theorem~\ref{thm:general-case}), and thus $\rats(\lambda) = \rats[\lambda]$.  It follows by Lemma~\ref{lem:lambda-embedding} that $T\affemb Q_\lambda$.  Applying $Q_\lambda(\cdot)$ to both sides and using Fact~\ref{fact:affine-embedding}, we get $Q_\lambda(T) \affemb Q_\lambda(Q_\lambda) = Q_\lambda$, and thus $Q_\lambda(T)$ is uniformly discrete (again by Fact~\ref{fact:affine-embedding}).  Finally, using Lemma~\ref{lem:distribute}, we have $Q_\lambda(S)-Q_\lambda(S) = Q_\lambda(S-S) = Q_\lambda(T)$, which proves the theorem.
\end{proof}

\begin{theorem}\label{thm:sub-meyer-spv}
$Q_\lambda(S)$ is sub-Meyer for any sPV $\lambda$ and finite $S\subseteq\rats(\lambda)$.
\end{theorem}

\begin{proof}
Let $\lambda$ be any sPV number.  By Theorem~\ref{thm:sub-Meyer} it suffices to show that $\lambda\in\discrete$.  The case where $\lambda\in\reals$ was proven by Pinch and is Proposition~\ref{prop:pinch-discrete}.  The case where $\lambda\notin\reals$ is by Proposition~\ref{prop:complex-Pinch}.
\end{proof}



\subsection{Higher Dimensions}\label{sec:higher-dimensions}

For certain $\lambda$, we can use relative density results in $\reals$ to prove
relative density in higher dimensions.  We have already shown (Proposition~\ref{prop:1-plus-phi}, and again in Theorem~\ref{theorem:main-candp}) that $Q_{1+\p}$ ($= R_{1+\p}$) is relatively dense, where $\p$ is the golden ratio, and similarly for 
$Q_{2+\sqrt{2}}$ and 
$Q_{2+\sqrt{3}}$ in Theorem~\ref{theorem:main-candp}.


The next two lemmas allow us in some cases to prove relative density in higher dimensions based on relative density in $\reals$.

\begin{lemma}\label{lem:parallel}
For any $\lambda\in\reals$ and any $S\subseteq\reals^n$ (for $n\ge 1$), if $Q_\lambda$ is relatively dense in $\reals$ and $Q_\lambda(S)$ contains the vertices of a parallelepiped with nonzero $n$-dimensional volume, then $Q_\lambda(S)$ is relatively dense in $\reals^n$.
\end{lemma}

\begin{proof}
We first prove the special case where $S = C_n := \{0,1\}^n \subseteq \reals^n$, the corners of the unit hypercube in $\reals^n$.  In this case, we show that $Q_\lambda(C_n) \supseteq Q_\lambda \times \cdots \times Q_\lambda$ ($n$-fold Cartesian product) by induction on $n$.\footnote{In fact, equality holds, but we will not need this.}  Assuming $Q_\lambda$ is relatively dense in $\reals$, the right-hand side is clearly relatively dense in $\reals^n$, which proves the lemma in this special case.  Afterwards, we show how the general case follows easily.

In the special case, the statement above is trivial for $n=1$, where $C_1 = \{0,1\}$ and hence $Q_\lambda(C_1) = Q_\lambda$.  Now assume $n\ge 1$ and that the statement above holds for $n$, i.e., $Q_\lambda(C_n) \supseteq (Q_\lambda)^{\times n}$.  In $\reals^{n+1}$, it is easily checked that $Q_\lambda(C_n\times \{b\}) = Q_\lambda(C_n)\times \{b\}$ for $b\in\{0,1\}$.  Whence
\[ Q_\lambda(C_{n+1}) = Q_\lambda(C_n\times\{0,1\}) \supseteq Q_\lambda(C_n)\times\{0,1\}\;. \]
That is, for every $v \in Q_\lambda(C_n)$, we have $v\times\{0,1\}\subseteq Q_\lambda(C_{n+1})$, and from this we get
\[ Q_\lambda(\{v\}\times\{0,1\}) = \{v\}\times Q_\lambda \subseteq Q_\lambda(C_{n+1})\;. \]
This holds for all $v\in Q_\lambda(C_n)$, and so by the inductive hypothesis, it holds for all $v\in (Q_\lambda)^{\times n}$.  Thus,
\[ (Q_\lambda)^{\times (n+1)} = (Q_\lambda)^{\times n}\times Q_\lambda \subseteq Q_\lambda(C_{n+1})\;, \]
as we wished to show.

Now in the general case, we can assume without loss of generality that $S = \{x_b \mid b\in\{0,1\}^n\} \subseteq \reals^n$, where the $x_b$ form the corners of a nondegenerate parallelepiped.  We can index the points so that there exists an invertible $\reals$-linear map $\map{\ell}{\reals^n}{\reals^n}$ that maps each $b\in\{0,1\}^n$ to $x_b - x_0$.  Letting $S' := \ell^{-1}(S - x_0)$, we see that $S' = \{0,1\}^n$, and thus $Q_\lambda(S')$ is relatively dense in $\reals^n$ by the special case proved above.  By Lemma~\ref{lem:translate-general}\footnote{where $R = \reals$, \ $M = N = \reals^n$, and $f = \ell$} we have
\[ x_0 + \ell(Q_\lambda(S')) = x_0 + Q_\lambda(\ell(S')) = x_0 + Q_\lambda(S-x_0) = Q_\lambda(S)\;. \]
The left-hand side is seen to be relatively dense in $\reals^n$ because this property is preserved under invertible $\reals$-linear maps and under translation.
\end{proof}

\begin{lemma}\label{lem:isosceles}
For any $\lambda \in \reals$ and $S\subseteq\complexes$, if $Q_\lambda$ is relatively dense in $\reals$ and $S$ includes three non-colinear points, then $Q_\lambda(S)$ is relatively dense in $\complexes$.
\end{lemma}

\begin{proof}
For any $z\in\complexes$ such that $\Im(z) > 0$, define $\Delta_z := \{0,1,z\}$.  To say that $S$ contains three non-colinear points is equivalent to saying that $\Delta_z \hookrightarrow S$ for some such $z$.  Thus it suffices by Fact~\ref{fact:affine-embedding} to show that $Q_\lambda(\Delta_z)$ is relatively dense in $\complexes$ for all such $z$.

Fix a $z := x+yi$ for some arbitrary $x,y\in\reals$ with $y>0$.  Since $Q_\lambda$ is relatively dense in $\reals$, we must have $\lambda\notin\clcl{0,1}$.  Let $\sigma := \lambda(1-\lambda)$.  Then $\sigma < 0$, and both $\sigma$ and $2\sigma$ are in $Q_\lambda$:
\begin{align*}
\sigma &= \lambda \xl 0\;, & 2\sigma &= \lambda \xl (1-\lambda)\;.
\end{align*}
Consider the four points
\begin{align*}
a &:= 0\xsig 1 & b &:= 1\xsig 0 \\
c &:= 0\xssig z & d &:= 1\xssig z
\end{align*}
All four points are in $Q_\lambda(\Delta_z)$ by Lemma~\ref{lem:extend}(1): $a$ and $b$ are in $Q_\sigma \subseteq Q_\lambda \subseteq Q_\lambda(\Delta_z)$, and $c$ and $d$ are in $Q_{2\sigma}(\Delta_z) \subseteq Q_\lambda(\Delta_z)$.  A little calculation shows that
\begin{align*}
a &= \sigma & b &= 1 - \sigma \\
c &= 2\sigma z = 2\sigma x + 2\sigma yi & d &= (1-2\sigma) + 2\sigma z = 1-2\sigma + 2\sigma x + 2\sigma yi
\end{align*}
Thus $a$, $b$, $c$, and $d$ form the vertices of a parallelogram in $\complexes$, and this parallelogram is nondegenerate, since $\sigma \ne 1-\sigma$ and $2\sigma y \ne 0$.  It follows that $Q_\lambda(\Delta_z)$ is relatively dense in $\complexes$ by Lemma~\ref{lem:parallel}.
\end{proof}

Recall from Section~\ref{sec:reg-polygons} that for integer $n\ge 3$, the point set $P_n\subseteq\complexes$ is the set of vertices of the regular $n$-gon, located so that $\{0,1\}\subseteq P_n$ and $\Im(z) \ge 0$ for all $z\in P_n$ (and so the unit interval $\clcl{0,1}$ forms the bottom base of the $n$-gon).

\begin{corollary}\label{cor:polygon-rel-dense}
For any $\lambda \in \reals$, if $Q_\lambda$ is relatively dense in $\reals$, then $Q_\lambda(P_n)$ is relatively dense in $\complexes$ for all $n\ge 3$.
\end{corollary}

The following corollary covers all the sets shown in Figures~\ref{fig:pentagon}--\ref{fig:hex-oct}, showing that they are all relatively dense in $\complexes$.

\begin{corollary}\label{cor:P-n-relative-density}
The following sets are all relatively dense in $\complexes$:
\begin{enumerate}
\item
$Q_{\lambda_5}(P_5)$ and $Q_{\lambda_5}(P_{10})$ (recall $\lambda_5 = 1+\p = (3+\sqrt 5)/2$)
\item
$Q_{\lambda_7}(P_7)$ and $Q_{\lambda_7}(P_{14})$ (cf.\ Equation~(\ref{eqn:lambda-n}))
\item
$Q_{2+\sqrt 2}(P_8)$
\item
$Q_{2+\sqrt 3}(P_{12})$
\item
$Q_{\lambda_9}(P_9)$ and $Q_{\lambda_9}(P_{18})$ (cf.\ Equation~(\ref{eqn:lambda-n}))
\item
$Q_{\lambda_{15}}(P_{15})$ and $Q_{\lambda_{15}}(P_{30})$ (cf.\ Equation~(\ref{eqn:lambda-n}))
\end{enumerate}
\end{corollary}

\begin{proof}
Consulting the table in the proof of Proposition~\ref{prop:lambda-n-PV}, we see that $\lambda_n$ has field norm $\pm 1$ for $n\in\{5,7,9,15\}$ and is thus a unit of $\ints[\lambda_n]$.  Applying Corollary~\ref{cor:meyer-if-unit}, we get that $Q_{\lambda_n}$ is relatively dense in $\reals$ for these $n$.\footnote{Isolated cases of these results are proved elsewhere in the paper.  We have that $Q_{1+\p}$ is relatively dense in $\reals$, either by Proposition~\ref{prop:1-plus-phi} or by Theorem~\ref{thm:cutandproject}.  By Theorem~\ref{thm:lambda-7-meyer}, $Q_{\lambda_7}$ is relatively dense in $\reals$.}
Theorem~\ref{thm:cutandproject} also implies $Q_{2+\sqrt 2}$ and $Q_{2+\sqrt 3}$ are relatively dense in $\reals$.  We apply Corollary~\ref{cor:polygon-rel-dense} to all these sets.
%
\end{proof}

\begin{corollary}\label{cor:P-n-meyer-set}
All the sets mentioned in Corollary~\ref{cor:P-n-relative-density} are Meyer sets.
\end{corollary}

\begin{proof}
For $S$ be any of these sets, we have that $S-S$ is uniformly discrete by Propositions~\ref{prop:discrete-reg-n-gon}, \ref{prop:lambda-n-PV}, and \ref{prop:octagon}.  Now we apply Fact~\ref{fact:meyer-set-characterization}(2) to $S$.
\end{proof}

\begin{proposition}
The $(1+\p)$-convex closure of a regular dodecahedron in $\reals^3$ is relatively dense in $\reals^3$.
\end{proposition}

\begin{proof}
Let $\lambda := 1+\p$ and let $D$ be a regular dodecahedron in $\reals^3$.  Choose some pair of opposite pentagonal faces $F$ and $F'$ of $D$.  The set of corners of each of $F$ and $F'$ is congruent to $P_5$, and so $Q_\lambda(D)$ includes a congruent copy of $Q_\lambda(P_5)$ in each of the (parallel) planes containing $F$ and $F'$.  By the previous proposition, each copy contains the corners of a nondegenerate parallelogram (rhombus, actually), and these rhombi can be chosen so that the one in $F'$ is a translation of the one in $F$.  See Figure~\ref{fig:parallelograms}.
\begin{figure}
\centering
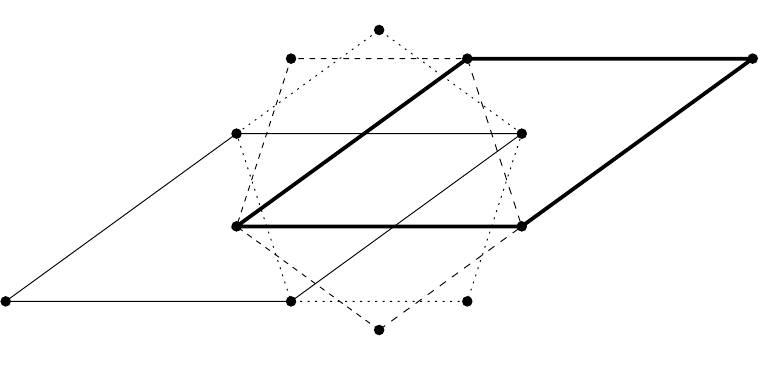
\caption{Two opposite pentagonal faces $F = \{a,b,c,d,e\}$ (dashed lines) and $F'=\{a',b',c',d',e'\}$ (dotted lines) of a regular dodecahedron $D$ are shown together with two rhombi extending out from each in opposite directions.  The rhombus on the right (thick solid lines) lies in the plane of $F$ and is formed from points $a$, $c$, $e$, and $f$, the latter of which is seen to be in $Q_\lambda(F)$.  Similarly, the rhombus on the left (thin solid lines) lies in the (parallel) plane of $F'$ and is formed from the points $a'$, $c'$, $e'$, and $f'\in Q_\lambda(F')$.  (These points are antipodal with respect to $D$ to the points $a$, $c$, $e$, and $f$, respectively.)  Thus the points $\{a,c,e,f,a',c',e',f'\}$ are all in $Q_\lambda(D)$ and form the corners of a nondegenerate parallelepiped in $\reals^3$.  Here, $\lambda = 1+\p$, and $\x$ means $\xl$.}
\label{fig:parallelograms}
\end{figure}
\end{proof}

\subsection{Trivial sPV $\lambda$}
\label{sec:trivial-sPV}

In this subsection, we show relative density of $Q_\lambda$ in some cases for trivial sPV $\lambda$.

\begin{proposition}\label{prop:trivial-sPV}
Let $\lambda\in\complexes$ be a trivial strong PV number.  If $Q_\lambda$ contains an integer other than $0$ or $1$, then $Q_\lambda$ is relatively dense in either $\reals$ or $\complexes$ as $\lambda\in\reals$ or $\lambda\in\complexes\cmpl\reals$, respectively.
\end{proposition}

\begin{proof}
Suppose $\alpha$ is in $Q_\lambda \intersect (\ints\cmpl\{0,1\})$.  Then $Q_\alpha$ is relatively dense in $\reals$ because it is periodic (Theorem~\ref{thm:period}).  Further, $Q_\alpha \subseteq Q_\lambda$ (Corollary~\ref{cor:extend}).  By Fact~\ref{fact:trivial-spv}, $\lambda$ is either a $\ints$-integer or a nonreal quadratic integer.  In the former case, $Q_\lambda$ is clearly relatively dense in $\reals$.  In the latter case, $Q_\lambda$ contains three non-colinear points, and thus $Q_\alpha(Q_\lambda)$ is relatively dense in $\complexes$ by Lemma~\ref{lem:isosceles}.  Since $Q_\alpha(Q_\lambda) \subseteq Q_\lambda(Q_\lambda) = Q_\lambda$ (Lemma~\ref{lem:extend}), the result follows.
\end{proof}

\begin{proposition}\label{prop:small-real-part}
If $\lambda\in\complexes\cmpl\reals$ is trivial sPV and $Q_\lambda$ contains a point $z\ne 0$ such that $\Re(z)\in\{0,1/2\}$, then $Q_\lambda$ is relatively dense in $\complexes$.
\end{proposition}

\begin{proof}
By Fact~\ref{fact:trivial-spv} (q.v.\ Lemma~\ref{lem:discrete-subring}), $z$ is a non-real quadratic integer.  If $\Re(z) = 0$, then $z = i\sqrt n$ for some positive $n\in\ints$.  But then, $0\xz z = z^2 = -n \in Q_z \subseteq Q_\lambda$ (Corollary~\ref{cor:extend}), and hence $Q_\lambda$ is relatively dense in $\complexes$ by Proposition~\ref{prop:trivial-sPV}.  Similarly, if $\Re(z) = 1/2$, then $z = (1 + ib\sqrt{4n-1})/2$ for some integers $b$ and $n$ with $b$ odd and $n>0$.  It follows that $z\xz 0 = (1-z)z = (1+b^2(4n-1))/4 \in Q_z\intersect\ints \subseteq Q_\lambda\intersect\ints$.  If $b\ne\pm 1$ or $n>1$, then $z\xz 0 \notin\{0,1\}$, which implies $Q_\lambda$ is relatively dense in $\complexes$ by Proposition~\ref{prop:trivial-sPV}.  If $n = 1 = \pm b$, then $z = e^{\pm i\tau/6}$, whence $-1 = z^3 = 0 \xz (0\xz z) \in Q_z\subseteq Q_\lambda$, and again by Proposition~\ref{prop:trivial-sPV}, $Q_\lambda$ is relatively dense in $\complexes$.
\end{proof}

The next proposition mirrors Corollary~\ref{cor:unit-gives-relative-density} but for trivial sPV $\lambda$.  (In Theorem~\ref{thm:added-points}, if $\lambda$ is trivial, then $k=0$, and thus $\Sigma_P^{(\lambda)} = \ints[\lambda]$ by Definition~\ref{def:Sigma_P}.)

\begin{proposition}\label{prop:added-points-trivial-sPV}
Let $\lambda\in\complexes$ be a trivial sPV number.  If $\ints[\lambda]$ contains an $\alpha\ne 1$ that is a unit of $\ints[\lambda]$, then $Q_\lambda$ is relatively dense in either $\reals$ or $\complexes$ as $\lambda\in\reals$ or $\lambda\in\complexes\cmpl\reals$, respectively.
\end{proposition}

\begin{proof}
It follows from Fact~\ref{fact:trivial-spv} that $\ints[\lambda]$ contains no points $z\in\complexes$ such that $|z|<1$.  Thus if $\alpha$ is a unit of $\ints[\lambda]$, it must be that $|\alpha| = 1$, and since $\alpha$ is either a $\ints$-integer or non-real quadratic integer, we have $\Re(\alpha)\in \{-1,-1/2,0,1/2\}$.  Proposition~\ref{prop:small-real-part} applies if $\Re(\alpha)\in\{0,1\}$.  Proposition~\ref{prop:trivial-sPV} applies if $\Re(\alpha) = -1$ (so $\alpha = -1$).  If $\Re(\alpha) = -1/2$, then $\alpha = (-1\pm i\sqrt 3)/2$, whence $\alpha \xa 0 = (1-\alpha)\alpha = 4i\sqrt 3 \in Q_\alpha\subseteq Q_\lambda$ (Corollary~\ref{cor:extend}), so Proposition~\ref{prop:small-real-part} applies.
\end{proof}

\pagebreak

\newpage


\begin{center}
\noindent{\Large\bf Part III: Bent Paths}\addcontentsline{toc}{part}{Part III: Bent Paths}
\end{center}

\section{The $\lambda$-convex closure of a bent path}
\label{sec:bent-path}

We continue to use $\x$ without subscript to denote $\xl$.

This part of the paper is dedicated to proving Theorem~\ref{thm:bent-path-first}.  We sequester the proof in this way because it uses some concepts and techniques that are not used anywhere else in the paper, particularly, winding number and some basic homology and homotopy theory.  We only need a few facts about these:
\begin{itemize}
\item
For every loop $\ell$ in $\complexes$ (see the start of Section~\ref{sec:equivalences} for definitions) and every point $x$ not on $\ell$, \ $\ell$ has a well defined \emph{winding number} about $x$, which is an integer indicating the number of times $\ell$ ``wraps around'' $x$---positive for counterclockwise, negative for clockwise.
\item
If two loops are homologous in $\complexes \cmpl \{x\}$ (that is, their difference can be expressed as the sum of boundaries of continuous images of disks in $\complexes \cmpl \{x\}$), then they have the same winding number about $x$.
\item
Any two loops that are homotopic in $\complexes \cmpl \{x\}$ (that is, one can continuously deform one loop into the other, staying within $\complexes \cmpl \{x\}$ and keeping the endpoint fixed) are also homologous in $\complexes \cmpl \{x\}$, and thus have the same winding number about $x$.
\item
The winding number of a finite sum of loops---about some $x$ not on any of the loops---is the sum of the winding numbers of the individual loops about $x$.
\item
If $x,y\in\complexes$ and $\ell$ is a loop in $\complexes \cmpl \{x,y\}$ such that there is a path from $x$ to $y$ that does not intersect $\ell$, then $\ell$ has the same winding number about $y$ as it has about $x$.
\end{itemize}

\begin{definition}
A path in $\complexes$ is \emph{bent} if it does not lie within any single straight line.
\end{definition}

Theorem~\ref{thm:bent-path-first} can then be restated as follows:

\begin{theorem}\label{thm:bent-path}
$Q_\lambda(c) = \complexes$ for any $\lambda\in\complexes\cmpl\clcl{0,1}$ and any bent path $c$.
\end{theorem}

\begin{definition}
Let $\map{c}{\clcl{0,1}}{\complexes}$ be a path.  A \emph{subpath} of $c$ is any path $\map{d}{\clcl{0,1}}{\complexes}$ that starts at some point $c(a)$, follows $c$, and ends at some point $c(b)$, where $0\le a \le b \le 1$.  That is, there exist $0\le a\le b\le 1$ such that $d(x) = c(\rho_{a,b}(x))$ for all $x\in\clcl{0,1}$.

The \emph{loop closure} of $c$, denoted $\loopcl(c)$, is the loop obtained by first following $c$ (double speed) from $c(0)$ to $c(1)$, then along a straight line from $c(1)$ back to $c(0)$.  It can be parameterized thus:
\[ \loopcl(c)(x) = \left\{ \begin{array}{ll}
c(2x) & \mbox{if $0\le x \le 1/2$,} \\
\rho_{c(1),c(0)}(2x-1) & \mbox{if $1/2 \le x\le 1$.}
\end{array} \right. \]
\end{definition}

The proof of Theorem~\ref{thm:bent-path} uses the following lemma:

\begin{lemma}\label{lem:bent-path}
If $\map{c}{\clcl{0,1}}{\complexes}$ is a bent path that does not include any nonempty open subset of $\complexes$, then $c$ includes a subpath $\map{d}{\clcl{0,1}}{\complexes}$ with the following properties:
\begin{enumerate}
\item
$d$ lies entirely in a closed half-plane whose boundary passes through its endpoints $d(0)$ and $d(1)$, and
\item
there exists a point $x\in\complexes$ such that $\loopcl(d)$ has nonzero winding number about $x$ (which is not on $\loopcl(d)$).
\end{enumerate}
\end{lemma}

\begin{proof}
Let $A$, $B$, and $C$ be three noncolinear points along $c$.  We can assume without loss of generality that $A=c(0)$, $B=c(1/2)$, and $C=c(1)$ (otherwise, take an appropriately reparameterized subpath of $c$).  By our assumption about $c$ not filling any space, we can choose a point $x$ in the interior of the triangle $\triangle ABC$ that does not lie on $c$.  We first show that there is a subpath $e$ of $c$ that satisfies the second condition of the lemma with respect to $x$.  Let $t$ be the (oriented) loop formed by tracing the perimeter of $\triangle ABC$, starting at $A$, going to $B$, then to $C$, then back to $A$.  Clearly, $t$ has winding number $\pm 1$ about $x$.  (It is not necessary, but we can assume that $t$ goes counterclockwise, so its winding number about $x$ is $+1$.)  Now, $t$ is evidently homologous to the sum of the following three loops (in fact, the difference is empty):
\begin{itemize}
\item
the loop closure $\ell_1$ of $c$,
\item
the loop closure $\ell_2$ of the path obtained by following $c$ backwards from $B$ to $A$, and
\item
the loop closure $\ell_3$ of the path obtained by following $c$ backwards from $C$ to $B$.
\end{itemize}
Since the winding number around $x$ is invariant under homology of loops in $\complexes \cmpl \{x\}$, and the winding number of a sum is the sum of the winding numbers, it follows that the winding numbers (around $x$) of $\ell_1$, $\ell_2$, and $\ell_3$ sum to $1$.  Thus at least one of these three loops has nonzero winding number around $x$.  If it is $\ell_1$, then we take $e := c$; otherwise, if it is $\ell_2$, then we let $e$ be $c$ restricted to $\clcl{0,1/2}$ (reparameterized); and otherwise (if it is $\ell_3$) we take $e$ to be $c$ restricted to $\clcl{1/2,1}$ (reparameterized).  Then $e$ and $x$ satisfy the second condition of the lemma with $d:=e$ (but not necessarily the first).  See Figure~\ref{fig:path1}.  (Note that the orientation of $e$ does not matter here, because the winding number will be nonzero regardless of orientation.)

\begin{figure}[htbp]
\begin{center}
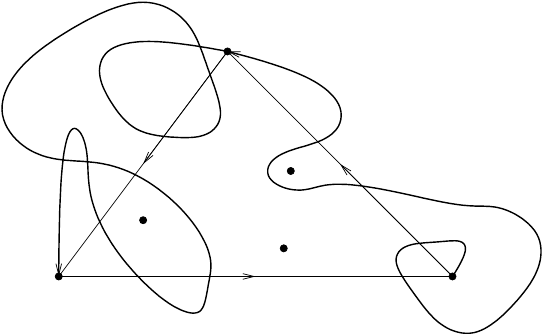
\caption{The curve $c$ and the triangle $\triangle ABC$.  For each $i\in\{1,2,3\}$, the loop $\ell_i$ has nonzero winding number around the point $x_i$.}\label{fig:path1}
\end{center}
\end{figure}

We now find a subpath $d$ of $e$ that satisfies both conditions.  Let $x$ be as above (so that $\loopcl(e)$ has nonzero winding number about $x$), let $L$ be the line through $P:=e(0)$ and $Q:=e(1)$, and let $L'$ be the line through $x$ parallel to $L$.  The situation might look like Figure~\ref{fig:path2}.

\begin{figure}[htbp]
\begin{center}
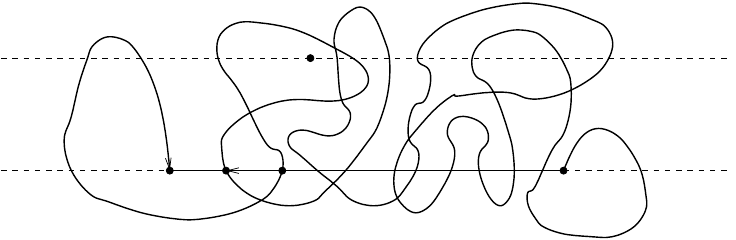
\caption{The curve $e$ and the point $x$.  (The small loop that follows the curve $e$ from $Z$ to $W$ then straight back to $Z$ has nonzero winding number around $x$.)}\label{fig:path2}
\end{center}
\end{figure}

Let $H$ be the open halfplane of $\complexes$ with boundary $L$ and containing $x$.  Then by continuity, $e^{-1}(H)$ is an open subset of $\opop{0,1}$, and hence is the disjoint union of at most countably many open intervals $I_0,I_1,I_2,\ldots\subseteq\opop{0,1}$.  For $i=0,1,2,\ldots\,$, let $e_i$ be the subpath of $e$ restricted to $I_i$.  Now we claim that there can be only finitely many $i$ such that $e_i$ intersects $L'$.  Indeed, suppose there were infinitely many such $i$, say $i_0,i_1,i_2,\ldots\,$.  For all $j\in\{0,1,2\ldots\}$ let $a_j$ be the left boundary of $I_{i_j}$, and pick some $z_j\in I_{i_j}$ such that $e(z_j)\in L'$.  By the continuity of $e$, we must have $e(a_j)\in L$.  The sequence $z_0,z_1,z_2,\ldots$ has some accumulation point $z\in\clcl{0,1}$.  Take some monotone subsequence $z_{j_0},z_{j_1},z_{j_2},\ldots$ converging to $z$, where $j_0<j_1<j_2<\cdots$.  If this sequence is increasing, then we have $z_{j_k} < a_{j_{k+1}} < z_{j_{k+1}}$ for all $k$, and if it is decreasing, then we have $z_{j_{k+1}} < a_{j_k} < z_{j_k}$ for all $k$.  In either case, the sequence $a_{j_0},a_{j_1},a_{j_2},\ldots$ also converges to $z$, but then since $a_j\in e^{-1}(L)$ and $z_j\in e^{-1}(L')$ for all $j$, and both $e^{-1}(L)$ and $e^{-1}(L')$ are closed, we have $z\in e^{-1}(L) \cap e^{-1}(L') = \emptyset$.  Contradiction.  This establishes the claim.

Now by the above claim, we have $e_{i_1},\ldots,e_{i_n}$ intersect $L'$ for some natural number $n$ and some indices $i_1,\ldots,i_n$, and no other $e_i$ intersect $L'$.  Set $\ell := \loopcl(e)$.  For $1\le j\le n$, let $r_j < s_j$ be the boundary points of $I_{i_j}$, and let $\ell_j$ be the loop closure of the subpath $e_{i_j}$ of $e$ that goes from $e(r_j)$ to $e(s_j)$ (note that $e(r_j)$ and $e(s_j)$ both lie on $L$).  We now can express $\ell$ as the \emph{finite} sum $\ell_1 + \ell_2 + \cdots + \ell_n + p$, where $p := \ell - (\ell_1 + \ell_2 + \cdots + \ell_n)$.  This decomposition (for the curve shown in Figure~\ref{fig:path2}) is illustrated in Figure~\ref{fig:path-decomp}.

\begin{figure}[htbp]
\begin{center}
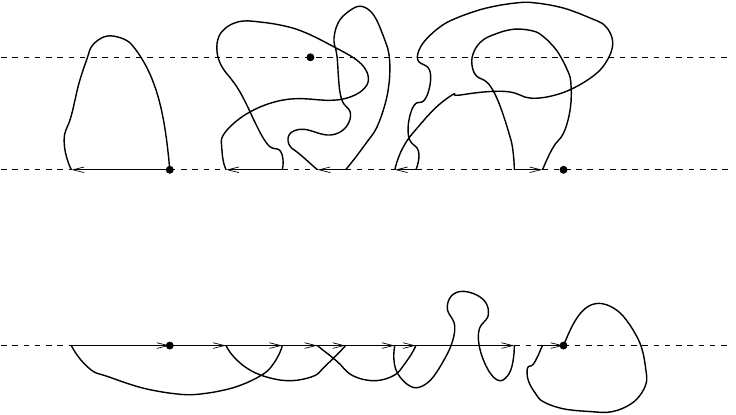
\caption{The decomposition of $\ell$ in the previous figure as a sum $\ell_1+\cdots+\ell_5$ of five loops that intersect $L'$ (top) and those that don't ($p$, bottom).  The ``$2$'' above a segment indicates that it counts double.}\label{fig:path-decomp}
\end{center}
\end{figure}

The winding number of $\ell$ about $x$ is nonzero, because $e$ satisfies the second condition of the lemma with respect to $x$.  This winding number is the sum of the winding numbers (about $x$) of $\ell_1,\ldots,\ell_n$, and $p$.  There is no contribution from $p$ to the winding number, because $p$ stays entirely in the open halfplane bounded by $L'$ and containing $L$ (thus it cannot wrap around $x$).  It follows that there must be some $\ell_j$ that has a nonzero winding number around $x$ (e.g., the small loop through $Z$ and $W$ in Figure~\ref{fig:path2}).  Since $e_{i_j}$ lies entirely in $H$ (except for its endpoints), we can take $d := e_{i_j}$, which then satisfies both conditions of the lemma.
\end{proof}

We first prove the special case of Theorem~\ref{thm:bent-path} where $\lambda$ is real.  This is Proposition~\ref{prop:bent-path-real}, below.  Afterwards, we will explain how to modify the proof for nonreal $\lambda$.

\begin{proposition}\label{prop:bent-path-real}
$Q_\lambda(c) = \complexes$ for any $\lambda\in\reals\cmpl\clcl{0,1}$ and any bent path $c$.
\end{proposition}

\begin{proof}
We can assume that $\lambda > 1$ by Fact~\ref{fact:inner-dual}.  If $c$ includes a nonempty open subset of $\complexes$, then we are done by Proposition~\ref{prop:open-set}, and so from now on we assume that this is not the case.  Then Lemma~\ref{lem:bent-path} implies that we can take a subpath $d$ of $c$ satisfying the two properties of the lemma with respect to some point $x$, and then it is enough to show that $Q_\lambda(d) = \complexes$.  And for \emph{this} it is enough to show that $Q_\lambda(d)$ contains a nonempty open subset of $\complexes$, thanks again to Proposition~\ref{prop:open-set}.

The first property of Lemma~\ref{lem:bent-path} says that $d$ lies entirely to one side of some line $L$ through $d(0)$ and $d(1)$.  (If $d(0) = d(1)$, then $L$ may not be unique.)  Letting $\ell := \loopcl(d)$, the second property of the lemma says that $\ell$ has nonzero winding number about $x$.

Since $d$ is compact and hence closed, there is some ball $B$ of radius $\eps > 0$ about $x$ that is disjoint from $d \union L$.  Furthermore, $\ell$ has the same nonzero winding number about every point $y\in B$ as it has about $x$.  Now notice that the set $B' := \{y\x d(0) \mid y\in B\}$ is an open neighborhood of the point $x' := x\x d(0)$ (in fact, a ball centered at $x'$ with radius $\eps(\lambda-1)$).  Note that $B'$ lies entirely on the side of $L$ opposite $B$ and $d$.  Figure~\ref{fig:path3} shows a typical situation when $\lambda = 2$.

\begin{figure}[htbp]
\begin{center}
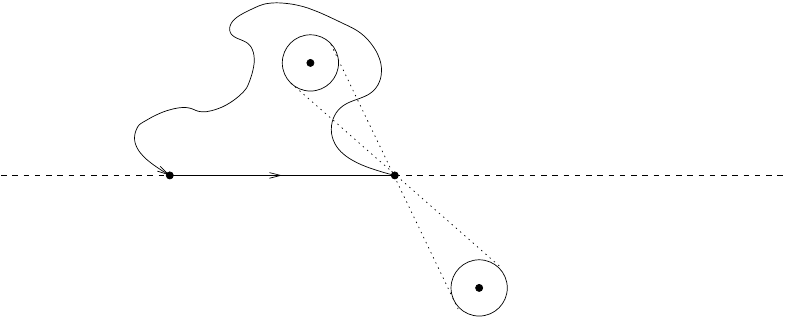
\caption{A typical path $d$ satisfying the two properties of Lemma~\ref{lem:bent-path} when $\lambda = 2$.  The loop $\ell$ has nonzero winding number about every point in the ball $B$ centered at $x$.  The ball $B'$ centered at $x'$ is also shown, and it lies entirely on the opposite side of $L$ from $B$ and $d$.}\label{fig:path3}
\end{center}
\end{figure}

We finish the proof by showing that $B'\subseteq Q_\lambda(d)$, whence $Q_\lambda(d) = \complexes$ by Proposition~\ref{prop:open-set}.  We do this in two steps: (i) we define a loop $\ell'$ entirely included in $Q_\lambda(d)$ that has nonzero winding number about every point $y'\in B'$; and (ii) we exhibit a homotopy $h$ (that stays entirely within $Q_\lambda(d)$) from $\ell'$ to the constant loop $d(0)$.  Assuming for the moment that we can do this, suppose there exists some $y' \in B' \cmpl Q_\lambda(d)$.  Then since $h$ avoids $y'$, it must keep the winding number about $y'$ invariant throughout the deformation of the loop, but this is impossible, because the winding number of $\ell'$ about $y'$ is nonzero whereas the winding number of the constant loop $d(0)$ about $y'$ is zero.  Thus no such $y'$ can exist, and so $B' \subseteq Q_\lambda(d)$ as desired.

The loop $\ell'$ is made up of three segments: the first two are similar to $d$, and the third is $d$ itself in reverse.  We define $\map{\ell'}{\clcl{0,1}}{\complexes}$ formally as follows: for all $s\in\clcl{0,1}$,
\[ \ell'(s) := \left\{ \begin{array}{ll}
d(3s)\x d(0) & \mbox{if $0\le s\le 1/3$,} \\
d(1)\x d(3s-1) & \mbox{if $1/3 \le s \le 2/3$,} \\
d(3-3s) & \mbox{if $2/3 \le s \le 1$.}
\end{array} \right. \]

One readily checks that $\ell'\subseteq Q_\lambda(d)$, since it contains only $\lambda$-extrapolations of points on $d$.  For convenience, we let $a := \ell'(1/3) = d(1)\x d(0)$.  Note that $a$ is colinear with $d(0)$ and $d(1)$, since $\lambda$ is real.  Figure~\ref{fig:path4} shows the $\ell'$ constructed from the path $d$ of Figure~\ref{fig:path3}.

\begin{figure}[htbp]
\begin{center}
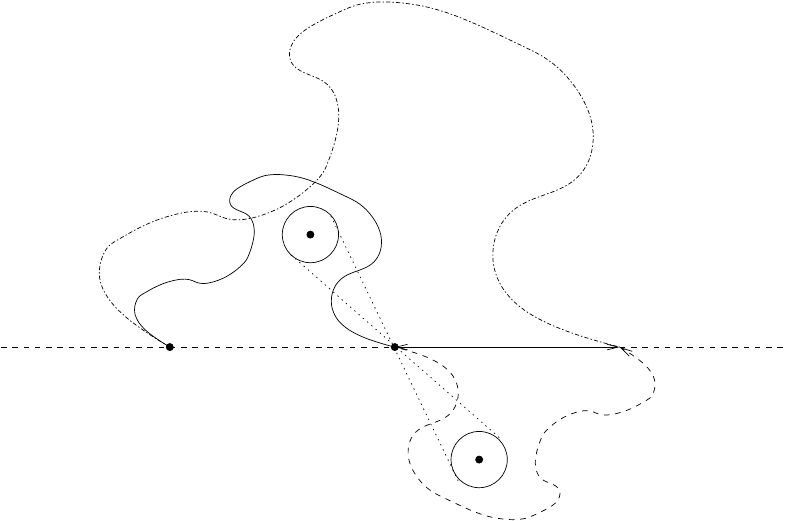
\caption{The loop $\ell'$ constructed from the path $d$ of Figure~\ref{fig:path3}.  The loop starts at $d(0)$, follows the dashed curve to the point $a$, then the dots and dashes to $d(1)$, then the curve $d$ (solid) backwards to $d(0)$.  The sets $B$ and $B'$ are also shown.  The line segment between $d(0)$ and $a$ is used to split $\ell'$ into the sum of two loops: $\ell_1$ lying below $L$ and $\ell_2$ lying above $L$.}\label{fig:path4}
\end{center}
\end{figure}

The loop $\ell'$ is homologous to the sum of two separate loops: $\ell_1$ is the loop closure of the first third of $\ell'$ ($\ell_1$ follows $\ell'$ from $d(0)$ to $a$ (dashed curve in Figure~\ref{fig:path4}) then a straight line segment along $L$ from $a$ back to $d(0)$); $\ell_2$ is the loop closure of the last two thirds of $\ell'$ ($\ell_2$ first follows $\ell'$ from $a$ around through $d(1)$ to $d(0)$, then goes straight from $d(0)$ to $a$ along $L$).  Notice that $B'$ and $\ell_1$ are the images of $B$ and $\ell = \loopcl(d)$, respectively, under the map $z\mapsto z\x d(0)$.  They are both obtained by first rotating by $\pi$ about the point $d(0)$ followed by dilating about $d(0)$ by a factor of $\lambda - 1$.  Thus by similarity, $\ell_1$ has the same nonzero winding number about every point in $B'$ as $\ell$ does about every point in $B$.  Also notice that $\ell_2$ lies entirely to the other side of $L$ from $B'$, because $d$ does, and the middle third of $\ell'$ is just a dilation of $d$ about $d(1)$ by a factor of $\lambda$.  It follows that $\ell_2$ has zero winding number about every point in $B'$, and thus we conclude that $\ell'$ has the same nonzero winding number around $B'$ as $\ell_1$ does.

Finally, we exhibit the promised homotopy $\map{h}{\clcl{0,1}\times\clcl{0,1}}{\complexes}$ from $\ell'$ to the constant loop $d(0)$ and staying inside $Q_\lambda(d)$: for all $s,t\in\clcl{0,1}$, define
\[ h(s,t) := \left\{ \begin{array}{ll}
d(3s(1-t))\x d(0) & \mbox{if $0\le s\le 1/3$,} \\
d(1-t)\x d((3s-1)(1-t)) & \mbox{if $1/3 \le s \le 2/3$,} \\
d((3-3s)(1-t)) & \mbox{if $2/3 \le s \le 1$.}
\end{array} \right. \]

One checks that $h$ is the desired homotopy, that is, $h$ is continuous, and for all $s,t\in\clcl{0,1}$, we have: $h(s,0) = \ell'(s)$; $h(0,t) = h(1,t) = d(0)$; $h(s,1) = d(0)$; and $h(s,t)\in Q_\lambda(d)$.  (Geometrically, for fixed $t$, the loop $h(\cdot,t)$ is defined analogously to the curve $\ell'$, but using only the first $(1-t)$-fraction of $d$.)
\end{proof}

We need one more lemma before we prove Theorem~\ref{thm:bent-path}.  The preceding proof does not quite work as is when $\lambda\notin\reals$, because the point $x' = x\x d(0)$ shown in Figure~\ref{fig:path3} may not lie below the line $L$, which means it may be tangled up with $d$ in such a way that the winding numbers of the two loops $\ell_1$ and $\ell_2$ (see Figure~\ref{fig:path4}) may cancel, leaving a zero winding number of $\ell'$ about $x'$ when we need it to be nonzero.  To fix this, we do not use $\lambda$ but instead use another point $\mu \in Q_\lambda$ that is close enough to being real that the point $x\xm d(0)$ does lie below $L$.  Then the whole proof of Proposition~\ref{prop:bent-path-real} goes through with $\mu$ replacing $\lambda$.

\begin{remark}
It is interesting and a bit frustrating to note that Theorem~\ref{thm:accum-implies-convex} almost suffices to prove Theorem~\ref{thm:bent-path} in the case where $\lambda\notin\reals$, because $Q_\lambda(c)$ clearly has an accumulation point and $Q_\lambda(c) \subseteq F_\lambda = \complexes$.  Unfortunately, Theorem~\ref{thm:accum-implies-convex} only applies to $\lambda$-clonvex sets, and so it only asserts that $R_\lambda(c) = \complexes$.  This is enough to get $Q_\lambda(c)$ dense in $\complexes$ but does not quite show equality.
\end{remark}
  
For any $z\ne 0$, we define $\arg z$ to be the unique $\theta\in\clop{0,\tau}$ such that $z = |z|e^{i\theta}$.

\begin{lemma}\label{lem:good-mu}
For any $\lambda\in\complexes \cmpl \clcl{0,1}$ and any $\eps > 0$, there exists $\mu \in Q_\lambda \cmpl \{1\}$ such that $\arg(\mu-1) < \eps$.
\end{lemma}

\begin{proof}
Assume, without loss of generality, that $\eps < \pi/2$.  It suffices to find a point $\mu\in Q_\lambda$ such that $\Re(\mu) > 2$ and $\arg \mu < \tan^{-1}((\tan\eps)/2)$.  That is, $\mu$ is somewhere in the shaded region in Figure~\ref{fig:good-mu}.

\begin{figure}[htbp]
\begin{center}
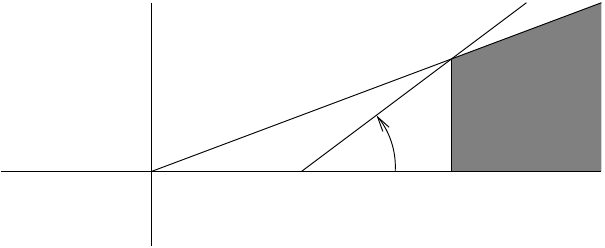
\caption{The shaded region is bounded by the real axis, the vertical line connecting $2$ and $z = 2 + i\tan\eps$, and the line from the origin through $z$.  If $\mu$ is in the closure of this region and $\mu \ne z$, then $\arg (\mu - 1) < \eps$.}\label{fig:good-mu}
\end{center}
\end{figure}

We know that $Q_\lambda$ is unbounded by Corollary~\ref{cor:unbounded}.  Fix some $\nu\in Q_\lambda$ with $|\nu|>1$, and note that all positive powers of $\nu$ are in $Q_\lambda$.  If $(\arg\nu)/\tau$ is rational, then there exists $n_0\in\posints$ such that $\arg(\nu^{kn_0}) = 0$ for all $k\in\posints$; then pick $k$ large enough so that $\mu := \nu^{kn_0}$ has real part $> 2$.  If $(\arg\nu)/\tau$ is irrational, then a standard pigeonhole argument shows that the set $\{(n\arg\nu) \bmod \tau \mid n\in\posints \}$ is dense in $\clop{0,\tau}$, and so contains infinitely many points in $\clop{0,\eps}$.  Thus we can find an $n\in\posints$ such that $\arg(\nu^n) = ((n\arg\nu)\bmod\tau) < \eps$ and $|\nu^n| = |\nu|^n$ is large enough to put $\nu^n$ in the interior of the shaded region.  Set $\mu := \nu^n$.
\end{proof}

Now we prove Theorem~\ref{thm:bent-path} by modifying the proof of Proposition~\ref{prop:bent-path-real} for nonreal $\lambda$.

\begin{proof}[Proof of Theorem~\ref{thm:bent-path}]
Let $c$ be a bent path, and let $\lambda$ be a point in $\complexes \cmpl \reals$.  As in the proof of Proposition~\ref{prop:bent-path-real}, we can assume $c$ includes no nonempty open subset of $\complexes$ and replace $c$ by a subpath $d$ satisfying Lemma~\ref{lem:bent-path}.  As before, let $x$ be given by that Lemma, let $L$ be a straight line through $d(0)$ and $d(1)$ not containing $x$, and let $\ell := \loopcl(d)$.  By extending $d$ a little bit along $L$ if necessary, we can assume that $d(0) \ne d(1)$.  By reversing $d$ if necessary, we can also assume that
\[ \Im\left(\frac{x - d(1)}{d(0) - d(1)}\right) > 0\;, \]
that is, the three points $d(1),d(0),x$ are oriented counterclockwise as they are in Figure~\ref{fig:path3}.

Now let $\eps$ be the angle $\angle x,d(0),d(1)$ formed by rays from $d(0)$ through $x$ and $d(1)$, respectively.  That is,
\[ \eps = \arg\left(\frac{d(1) - d(0)}{x - d(0)}\right)\;. \]
By our choice of orientation, we know that $0 < \eps < \pi$.  By Lemma~\ref{lem:good-mu}, there exists a $\mu \in Q_\lambda$ such that $\arg(\mu - 1) < \eps$.  By part~(1\@.) of Lemma~\ref{lem:extend}, we have $Q_\mu(d) \subseteq Q_\lambda(d)$, and so it suffices to show that $Q_\mu(d)$ contains a nonempty open subset of $\complexes$.  This will be just as we did in the proof of Proposition~\ref{prop:bent-path-real} but with the ``almost real'' $\mu$ replacing the real $\lambda$.  If $\mu$ is in fact real, then $\mu > 1$, and we are done by Proposition~\ref{prop:bent-path-real}, and so we assume $\mu\notin\reals$, whence $\Im(\mu) > 0$.  It follows that for any $z,w\in\complexes$, the three points $z,w,z\xm w$ are oriented counterclockwise.

Set $x' := x\xm d(0)$.  Then---and this is the crucial point---$x'$ lies opposite the line $L$ from $x$, as shown in Figure~\ref{fig:path5}.

\begin{figure}[htbp]
\begin{center}
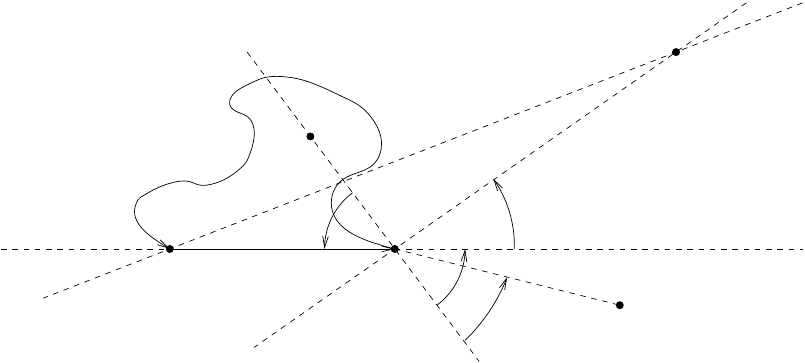
\caption{The point $x' = x\xm d(0)$ (lower right) lies below the line $L$, whereas $x$ lies above $L$.  The point $a = d(1)\xm d(0)$ is also shown, as well as various lines and angles, including the line $L'$ through $d(0)$ and $a$ and the line $L''$ through $d(1)$ and $a$.  The point $x'$ must also lie on the opposite side of $L'$ as $d(1)$, and hence on the same side of $L''$ as $d(0)$.}\label{fig:path5}
\end{center}
\end{figure}

Figure~\ref{fig:path5} is analogous to Figures~\ref{fig:path3} and \ref{fig:path4}.  As before, let $a := d(1)\xm d(0)$.  Let $L'$ be the line through $d(0)$ and $a$, and let $L''$ be the line through $d(1)$ and $a$ (see Figure~\ref{fig:path5}).  Notice that $x'$ must be on the opposite side of $L'$ from $d(1)$, and this together with the position of $x'$ with respect to $L$ implies that $x'$ must be on the same side of $L''$ as $d(0)$.  As before, we can find an open ball $B$ surrounding $x$ such that: (i) $\ell$ has the same nonzero winding number about every $y\in B$ as it has about $x$; and (ii) the open ball $B' = \{y\xm d(0) \mid y\in B\}$ surrounding $x'$ intersects none of the three lines $L$, $L'$, or $L''$.

Now we define the loop $\ell'\subseteq Q_\mu(d)$ similar to the proof of Proposition~\ref{prop:bent-path-real}.
\[ \ell'(s) := \left\{ \begin{array}{ll}
d(3s)\xm d(0) & \mbox{if $0\le s\le 1/3$,} \\
d(1)\xm d(3s-1) & \mbox{if $1/3 \le s \le 2/3$,} \\
d(3-3s) & \mbox{if $2/3 \le s \le 1$.}
\end{array} \right. \]
The first third of $\ell'$ runs from $d(0)$ to $a$ and lies opposite $L'$ from $d(1)$.  As before, its loop closure $\ell_1$ is a rotated, dilated copy of $\ell$ and so has the same nonzero winding number about every point in $B'$ as $\ell$ does about $x$.  The middle third of $\ell'$ runs from $a$ to $d(1)$ and stays on the other side of $L''$ from $d(0)$, and hence also from $B'$.  Finally, our choice of $\mu$ ensures that the last third of $\ell'$, which coincides with $d$, stays on the side of $L$ opposite $B'$.  Thus the loop closure $\ell_2$ of the final two thirds of $\ell'$ cannot contribute to the winding number of $\ell'$ about any point in $B'$.  It follows that $\ell'$ has the same nonzero winding number about every point in $B'$ as $\ell_1$ has.

We define the homotopy $h$ just as before, but with $\xm$ instead of $\xl$:
\[ h(s,t) := \left\{ \begin{array}{ll}
d(3s(1-t))\xm d(0) & \mbox{if $0\le s\le 1/3$,} \\
d(1-t)\xm d((3s-1)(1-t)) & \mbox{if $1/3 \le s \le 2/3$,} \\
d((3-3s)(1-t)) & \mbox{if $2/3 \le s \le 1$.}
\end{array} \right. \]
This homotopy stays within $Q_\mu(d)$ and contracts $\ell'$ to the constant point $d(0)$, whose winding number about any point in $B'$ is zero.  Thus the curve must pass through each point in $B'$ sometime during the deformation, and this puts $B'\subseteq Q_\mu(d)$ as before.  Hence $\complexes = Q_\mu(d) = Q_\lambda(c)$.
\end{proof}

\newpage

\begin{center}
\noindent{\Large\bf Part IV: Concluding Remarks}\addcontentsline{toc}{part}{Part IV: Concluding Remarks}
\end{center}

\section{Conjectures, open problems, and future research}
\label{sec:open}

We have many more questions than we can investigate in any reasonable length of time.  We only give a sampling in this section.  Some may be easy, but we have just not looked at them in depth.

Recall the set $\convex$ (Definition~\ref{def:convex}) and its complement $\discrete := \complexes\cmpl\convex$.  We know that $\discrete$ is closed, discrete, and contains only algebraic integers (Theorem~\ref{thm:general-case}).  We also know that $\discrete$ contains all strong PV numbers (Theorem~\ref{thm:main}), but we know of no other elements of $\discrete$ than these.

\begin{conjecture}\label{conj:CandD}
$Q_\lambda$ is discrete if and only if $\lambda$ is a strong PV number.
\end{conjecture}

To make progress towards this conjecture, we can use various constructions 
we have developed to carve out more territory for $\convex$ in the complex plane.  This approach was started in our paper.  Pinch~\cite[Theorem~15]{Pinch} showed that every real, non-integer element of $\discrete$ must have at least one conjugate in $\opop{0,1}$, but beyond these facts, we know little about $\convex$ and $\discrete$.

There are a number of open questions about $Q_\lambda$ when $\lambda$ belongs to a discrete subring $D$ of $\complexes$.  For example, we conjecture that equality holds in Lemma~\ref{lem:ideal} for every such $D$.  We know that it holds for $D := \ints$ (Theorem~\ref{thm:period}).

\begin{conjecture}\label{conj:ints}
For any discrete subring $D$ of $\complexes$ and for any $\lambda\in D$,
\begin{equation}\label{eqn:mod}
Q_\lambda = \lambda(1-\lambda)D + \{0,1,\lambda,1-\lambda\}\;.
\end{equation}
\end{conjecture}




The next open question (and its generalizations) is one of the most interesting.

\begin{open}\label{open:Qequality}
For which sPV $\lambda$ does set equality hold in Eq.~(\ref{eqn:main3})?
\end{open}

We have settled this question for real quadratic sPV $\lambda$ (Theorem~\ref{theorem:main-candp}), strengthening previous results of Mas{\'a}kov{\'a} et al.~\cite{MPP:sconvexjournal} obtained via different means.  When equality holds, $Q_{\lambda}$ is a model set and relative density of $Q_{\lambda}$ follows immediately.  We thus know $Q_\lambda$ is not a model set of the form $\Sigma_P$
for most real quadratic sPV $\lambda$.  Indeed we have now proved that $4+\sqrt{13}\notin Q_{(5+\sqrt{13})/2}$ (Corollary~\ref{cor:non-membership}), a result hinted at by computer.  Nonetheless, we conjecture that relative density holds in general; see Open Question~\ref{open:reldens} below.  The resulting sets would thus be Meyer sets but not model sets of the form $\Sigma_P$.

\begin{research}\label{research:4plussqrt13}
Generalize the proof that $4+\sqrt{13} \notin Q_{(5+\sqrt{13})/2}$ to prove nonmembership in $Q_\lambda$ of as many specific points as possible for as many $\lambda$ as possible.  Use computation to suggest such points.
\end{research}

All this leads to a general open question:

\begin{open}\label{open:reldens}
Is $Q_\lambda$ relatively dense in $\reals$ for all $\lambda\in \reals\cmpl\clcl{0,1}$?  Is $Q_\lambda$ relatively dense in $\complexes$ for all $\lambda\in \complexes\cmpl\reals$?  (If not, then for which $\lambda$?)
\end{open}

If true, then the discrete $Q_\lambda$ are all Meyer sets by Corollary~\ref{cor:Meyer-set}.  We conjecture that this is indeed true, but so far we can only prove relative density in the restricted cases addressed in Section~\ref{sec:relative-density}, particularly Corollary~\ref{cor:unit-gives-relative-density}.  Proving relative density more generally will apparently require new techniques.
Failing a general proof of relative density, we could at least apply the technique of Corollary~\ref{cor:13-relatively-dense} to show relative density of $Q_\lambda$ for other individual unitary quadratic sPV $\lambda$.



The technique we used to prove Theorem~\ref{thm:add-points} only works when applied to quadratic sPV $\lambda$ such that $\Sigma^{(\lambda)}_P$ contains a non-trivial unit. 
We believe it holds for any quadratic sPV number $\lambda$, either via an alternative proof, or because every $\Sigma^{(\lambda)}_P$ contains a non-trivial unit,

\begin{conjecture}
If $\lambda$ is quadratic sPV, the corresponding cut-and-project set $\Sigma_P$ is finitely generated, that is, there exists a finite $Y\subseteq\Sigma_P$ such that $Q_\lambda(Y) = \Sigma_P$.
\end{conjecture}

Of far greater consequence is the case of arbitrary degree, which is partially answered in Theorem~\ref{thm:added-points}, in that case under the stronger hypothesis that $Q_{\lambda}$ contains a non-trivial unit. Extending the above, we conjecture the following, further generalizing
 Theorem~\ref{thm:added-points} so that it holds for containment
of the entire set $\Sigma_P$ in $Q_\lambda(Y)$.

\begin{conjecture}
  For any sPV $\lambda$, there exists a finite set $Y\subseteq \Sigma_P$ such that $\Sigma_P \subseteq Q_\lambda(Y)$. 
\end{conjecture}

\subsection{$L$-convex, auto-convex, and $\lambda$-semiconvex sets}

\begin{definition}
Let $L$ and $S$ be any subsets of $\complexes$.
\begin{itemize}
\item
$S$ is \emph{$L$-convex} iff $S$ is $\lambda$-convex for all $\lambda\in L$.
\item
Let $Q_L(S)$ be the least $L$-convex superset of $S$.
\item
$S$ is \emph{auto-convex} iff $S$ is $S$-convex.
\end{itemize}
\end{definition}

All the $Q_\lambda$ sets are auto-convex by
the following result:\footnote{Pinch proved a more general result for real $\lambda$: If $S\subseteq\reals$ is auto-convex, then $Q_\lambda(S)$ is auto-convex~\cite{Pinch}.  The proof generalizes trivially to the complex numbers.}
\begin{proposition}\label{prop:extend}
For any $\lambda,\mu\in\complexes$,
if $\mu\in Q_\lambda$, then $Q_\lambda$ is $\mu$-convex, and consequently, $Q_\mu \subseteq Q_\lambda$.
\end{proposition}
The same goes for all $R_\lambda$.  The converse does not hold, however: every subring of $\complexes$ is clearly autoconvex, but there is no $\lambda\in\complexes$ such that $Q_\lambda = \ints[2i]$, for example.  Instead, we hazard a weaker conjecture.

\begin{conjecture}
For any $\lambda\in\complexes$, if $S\subseteq R_\lambda$ is auto-convex, then $S = R_\mu$ for some $\mu \in R_\lambda$.
\end{conjecture}


This conjecture implies that the set $\{R_\lambda \mid \lambda \in \complexes \}$ is closed under arbitrary intersections, because the intersection of any family of closed, auto-convex sets is clearly auto-convex (and closed).

\begin{definition}
Fixing $\lambda\in\complexes$, we will say that a set $S\subseteq\complexes$ is \emph{$\lambda$-semiconvex} iff, for every $a,b\in S$, at least one of the points $a\xl b$ and $b\xl a$ is in $S$.
\end{definition}

Clearly, if $S$ is $\lambda$-semiconvex, then $S$ is also $(1-\lambda)$-semiconvex.  Note that the intersection of two $\lambda$-semiconvex sets need not be $\lambda$-semiconvex.

\begin{research}
For which $\lambda\in\complexes$ do there exist nontrivial, bounded $\lambda$-semiconvex sets?
\end{research}

It is not too hard to see that if $|\lambda - 1/2| \le 1/2$, then any disk (closed or open) is $\lambda$-semiconvex.  There are at least two values of $\lambda$ outside this range where bounded $\lambda$-semiconvex sets exist: if $\lambda = (1\pm i\sqrt 3)/2$, then the vertices of any equilateral triangle form a $\lambda$-semiconvex set.  What other such $\lambda$ are there?

\subsection{Computational Questions}

Questions such as Research Plan~\ref{research:4plussqrt13} lead us to
the following line of inquiry:
If $Q_\lambda$ is discrete, then  $\lambda$ is an algebraic integer, and thus each element of 
$Q_{\lambda}$ can be expressed (encoded) as an integer polynomial in $\lambda$, of degree bounded by the degree of $\lambda$.  $Q_{\lambda}$ is countably infinite.
Furthermore, it is computably enumerable: If
$x \in Q_{\lambda}$, by enumerating all extrapolations
beginning with $\{0,1\}$, we will eventually obtain $x$.  It is
conceivable by similar reasoning that $Q_{\lambda}$ is in $\NP$; however
it is not at all obvious that a succinct proof that $x \in Q_{\lambda}$
(to say nothing of $x \not \in Q_{\lambda}$)
can always
be found, and its computational complexity is
wide open\footnote{See, e.g., \cite{Sipser:theory} for information on the relevant computability
and complexity notions discussed here.}.
For those $\lambda$ that affirmatively answer Open Question~\ref{open:Qequality}, we have an effective procedure to determine if any
$x$ is in $Q_{\lambda}$.
For other $\lambda$, however, especially if $\lambda$ is not sPV but $Q_{\lambda}$ is nevertheless discrete,
we do not know if $Q_{\lambda}$ is decidable (although we do know that $Q_{[x]}$ is decidable).

\begin{open}
 If $Q_{\lambda}$ is discrete, is it decidable?  Uniformly in $\lambda$?
If it is decidable, what is its computational complexity? When is it
true that $Q_{\lambda} \in \NP$,
and how does this depend on $\lambda$?
\end{open}

Note that if Conjecture~\ref{conj:CandD} is false, it is possible
that the answer depends on whether or not $\lambda$ is a strong PV
number. We conjecture that $Q_{\lambda}$ is always decidable,
but have no intuition
regarding its containment in $\NP$, to say nothing of $\Pe$.

\subsection{Miscellaneous Open Problems}

It would be interesting to pin down $\convex \intersect \reals$ and $\convex \intersect \{ z\in\complexes \mid \Re(z) = 1/2 \}$.  These two cases may be easier than the general case, as they present symmetries not shared by all $\lambda$.

\begin{research}\label{research:cintR}
Determine which $\lambda>3$ yield $R_\lambda = \reals$.  Determine which $\lambda$ with $\Re(\lambda) = 1/2$ yield $R_\lambda = \complexes$.
\end{research}

\begin{research}\label{research:Dgraphic}
Get a reasonably good graphical picture of $\discrete$.
\end{research}

Recall that by $Q_{[x]}$ we denote the set of polynomials in $\ints[x]$ generated by the constant polynomials $0$ and $1$, and by repeated applications of $\star_x$; that is, $Q_{[x]}$ is the smallest set of polynomials containing $0,1$ and closed under the binary operation $(p,q)\mapsto (1-x)p+xq$.

$Q_{[x]}$ has some interesting properties.  Recall that any element of $Q_{\lambda}$ can be written as $p(\lambda)$ where $p\in Q_{[x]}$, and conversely.  Pinch showed~\cite[Corollary~4.1]{Pinch} that an integer polynomial $p$ is in $Q_{[x]}$ if and only if there exist $n\ge 0$ and integers $b_0,\ldots,b_n$ such that $p(x) = \sum_{i=0}^n b_i x^i(1-x)^{n-i}$ and $0\le b_i \le \binom{n}{i}$ for all $0\le i\le n$ (see Lemma~\ref{lem:level-characterization}).  We have an alternate characterization of $Q_{[x]}$: an integer polynomial $p$ is in $Q_{[x]}$ if and only if either $p\in\{0,1\}$ or $0<p(\mu)<1$ for all $0<\mu<1$ (see Theorem~\ref{thm:Q-x}).  This latter characterization can be used to computably enumerate the integer polynomials \emph{not} in $Q_{[x]}$, leading to a decision procedure for $Q_{[x]}$.  What, then, is the complexity of deciding $Q_{[x]}$?

There are other interesting (though noncomputational) questions regarding $Q_{[x]}$.  We can list all 14 polynomials in $Q_{[x]}$ of degree $\le 2$, and get a finite upper bound on the number of polynomials in $Q_{[x]}$ of any given degree bound.  However, we don't even know
how many polynomials there are in $Q_{[x]}$ of degree~3.

\begin{open}\label{open:eltsofQ}
How many elements of $Q_{[x]}$ are there of degree $3$?
\end{open}

Our techniques give an upper bound of $717$, and an extensive computer search finds only $90$.  Perhaps Fact~\ref{fact:critical-points} can reduce the upper bound.


Other open problems include the following:\\

Call the triangle with vertices $(0,1,\lambda)$ the \emph{fundamental triangle}.  George McNulty offers the following conjecture~\cite{McNulty:conjecture}:

\begin{conjecture}[McNulty]\label{conj:mcnulty}
If\/ $Q_\lambda$ contains a point in the interior of the fundamental triangle, then $R_\lambda$ is convex.
\end{conjecture}

%

\begin{definition}\rm
We will call a discrete set $Q_\lambda$ \emph{maximal} if it is not a proper subset of any other discrete $Q_\mu$.
\end{definition}

\begin{open}\label{open:maximal}
Do maximal $Q_\lambda$ exist?  Is there an easy way to characterize the $\lambda$ such that $Q_\lambda$ is maximal?  Is there an interesting notion of minimal $Q_\lambda$?
\end{open}

\begin{open} Is there an easy way to characterize the minimum polynomials of strong PV numbers?  Short of that, find such polynomials of higher and higher degree.  (We currently can characterize all such polynomials of degree $\le 4$.)
\end{open}
\begin{open}
We say that a $\lambda$-convex set $A$ is \emph{cohesive} if $A = Q_\lambda(A\cmpl T)$ for any finite set $T$.  We can show that $\Sigma_P$ (the right-hand side of Eq.~(\ref{eqn:main3})) is cohesive for all unitary quadratic sPV $\lambda$---if you remove $0$ and $1$.  This implies that
there are infinitely many points ``missing'' from $Q_{(5+\sqrt{13})/2}$, that is, $\Sigma_{\clcl{0,1}}^{(\lambda)}\cmpl Q_\lambda$ is infinite for $\lambda = (5+\sqrt{13})/2$.  For which $\lambda$ is $Q_\lambda\cmpl\{0,1\}$ cohesive?  All sPV numbers, perhaps?
\end{open}

\begin{open}
Let $A$ be $\lambda$-convex as in the previous question.  An \emph{essential point} of $A$ is some $x\in A$ such that $A\cmpl\{x\}$ is $\lambda$-convex\footnote{And hence $x \not \in Q_{\lambda}(A\cmpl\{x\})$}.  For example, $0$ and $1$ are both essential points of $Q_\lambda$ for all sPV $\lambda$.  The question is: if $Q_\lambda\cmpl\{0,1\}$ is not cohesive, must it have an essential point?  More generally, what is the smallest size of a set you can remove from a noncohesive set that leaves a $\lambda$-convex set?
\end{open}

\newpage

\section*{Acknowledgments}\addcontentsline{toc}{section}{Acknowledgements}

We wish to thank Rohit Gurjar for proving part of Theorem~\ref{thm:main}---the case of non-real $\lambda$.  We would also like to thank George McNulty for supplying an interesting conjecture related to this work and for help simplifying the proof of Theorem~\ref{thm:convex-is-easy}, as well as other useful discussions and pointers to the literature.  We are grateful to Stuart Kurtz for suggesting Lemma~\ref{lem:kurtz}, which is key to proving that $R_\lambda$ is convex for all transcendental $\lambda$.  We would also like to thank Joshua Cooper, Arpita Korwar, Jochen Me{\ss}ner, Danny Rorabaugh, Heather Smith, and Thomas Thierauf for interesting and helpful discussions.

\newpage

\appendix

\section{Appendix: Supporting Facts}

In this appendix we cover a few propositions that support our results but are not central to them.

The next proposition backs up an assertion made in the remark on page~\pageref{rem:no-new-lambda} following the proof of the main theorem of Section~\ref{sec:alg-int}---Theorem~\ref{thm:main}.

\begin{proposition}\label{prop:no-new-lambda}
Let $D$ be any discrete subring of $\complexes$.  Suppose that $p\in D[x]$ is a monic polynomial of degree $d>0$ such that all but one of its roots lie in the open unit interval $\opop{0,1}$.  Then its remaining root $\lambda \notin \opop{0,1}$ is a strong PV number.
\end{proposition}

\begin{proof}
Let $\mu_0,\ldots,\mu_{d-2} \in \opop{0,1}$ be the roots of $p$ other than $\lambda$.  We have two cases:
\begin{description}
\item[Case~1: $\lambda \in \reals$.] In this case, all the roots of $p$ are real, which implies all the coefficients of $p$ are real.  These coefficients are thus integers by Lemma~\ref{lem:discrete-ring}, and so $p$ is in $\ints[x]$, making $\lambda$ an algebraic integer.  The conjugates of $\lambda$ are among $\mu_0,\ldots,\mu_{d-2}$ (not necessarily all; we are not assuming $p$ is irreducible).  Thus $\lambda$ is a strong PV number in this case.
\item[Case~2: $\lambda \notin\reals$.] Let $p^*$ be the polynomial obtained by complex-conjugating all the coefficients of $p$.  We have that $p^*$ is also in $D[x]$ (this follows from Lemma~\ref{lem:discrete-subring}, for example), and the roots of $p^*$ are $\mu_0,\ldots,\mu_{d-2},\lambda^*$.  Letting $q := pp^*$, we see that $q$ is monic, that $q\in D[x]$, and that
\[ q(x) = (x-\lambda)(x-\lambda^*)(x-\mu_0)^2\cdots(x-\mu_{d-2})^2\;. \]
The product of the first two factors is in $\reals[x]$, as are each of the other factors, and so $q\in\reals[x]$.  But then $q\in\ints[x]$ by Lemma~\ref{lem:discrete-ring}, whence $\lambda$ is an algebraic integer.  The conjugates of $\lambda$ are evidently among $\mu_0,\ldots,\mu_{d-2}$ and $\lambda^*$, and so this again makes $\lambda$ a strong PV number.
\end{description}
\end{proof}

The next theorem is a standard result in algebra, and its corollary justifies an assertion made in the proof of Proposition~\ref{prop:Sigma_P-sPV}.  Background concepts are taken from Chapter~4 of Jacobson~\cite{Jacobson:BasicAlgebraI}.

\begin{theorem}\label{thm:Galois-stuff}
Let $F$ be a field and let polynomial $f\in F[x]$ of degree $d > 0$ be separable\footnote{A polynomial $f\in F[x]$ is \emph{separable} if each of its irreducible factors has distinct roots in any splitting field of $f$ over $F$; equivalently, $f$ is coprime with its formal derivative.  If $\characteristic\, F = 0$, then every nonzero $f\in F[x]$ is separable.} and irreducible over $F$.  Let $E$ be a splitting field of $f$ over $F$, and let $R := \{\mu_0,\ldots,\mu_{d-1}\}$ be the set of roots of $f$ in $E$.  Let $p\in F[x]$ be a polynomial over $F$, and let $z := p(\mu_0) \in E$.  Then $z$ is algebraic over $F$, and its conjugates are $p(\mu_0),\ldots,p(\mu_{d-1})\in E$ (not necessarily distinct).  Furthermore, the mapping $p$ restricted to $R$ is $m$-to-one for some positive integer $m$.
\end{theorem}

\begin{proof}
Note that $|R| = d$ (i.e., the $\mu_i$ are pairwise distinct), because $f$ is separable and irreducible.

$E$ is finite dimensional (as a vector space) over $F$, and so $E$ is an algebraic extension of $F$, making $z$ algebraic over $F$.  (Indeed, $E$ is a Galois extension of $F$.)  Let $G := \gal(E/F)$ be the Galois group of $E/F$.\footnote{That is, the group of field automorphisms of $E$ that leave $F$ pointwise fixed.}  Then $G$ acts on $R$, i.e., each element of $G$ permutes the elements of $R$.  In fact, $G$ is isomorphic to the group of permutations on $R$ induced by $G$.  We also know that $G$ acts transitively on $R$, because $f$ is irreducible~\cite[Theorem~4.14, p.~259]{Jacobson:BasicAlgebraI}.  In particular, $R = \{ \eta(\mu_0) : \eta \in G\}$.  Let $S := \{p(\mu) : \mu \in R\} = \{ p(\mu_i) : 0\le i < d\}$.  Then clearly, $G$ also acts on $S$.  Furthermore, this action is also transitive, which can be seen as follows: We have
\[ S = \{ p(\eta(\mu_0)) : \eta\in G\} = \{\eta(z) : \eta\in G\}\;, \]
and so for any $\mu,\nu\in S$, there exist $\eta,\theta\in G$ such that $\mu = \eta(z)$ and $\nu = \theta(z)$.  Then $\nu = \theta(z) = (\theta\eta^{-1})(\mu)$.

Let $g\in F[x]$ be the minimum (monic) polynomial of $z$.  For every $\eta\in G$ we have
\[ 0 = \eta(0) = \eta(g(z)) = g(\eta(z)) = g(\eta(p(\mu_0))) = g(p(\eta(\mu_0))) = g(p(\mu_i)) \]
for the unique $0\le i < d$ such that $\mu_i = \eta(\mu_0)$.  Thus $p(\mu_i)$ is a conjugate of $z$, and because $G$ acts transitively on $S$, all elements of $S$ are conjugates of $z$.  This shows one direction of the theorem; it remains to show that $z$ has no other conjugates but these.

Let $c>0$ be the degree of $g$ (and of $z$).  We have $F \subseteq F(z) \subseteq F(\mu_0) \subseteq E$, and $[F(z):F] = c$.\footnote{For any field extension $K$ of $F$, \ $[K:F]$ is the \emph{index} of $K$ over $F$, i.e., the dimension of $K$ viewed as a vector space over $F$.}  Set $H := \gal(E/F(z))$.  Then $H$ is a subgroup of $G$ (the automorphisms of $G$ that fix $F(z)$ pointwise), and by the fundamental Galois pairing~\cite[p.~239]{Jacobson:BasicAlgebraI}, we have
\[ \frac{|G|}{|H|} = \frac{[E:F]}{[E:F(z)]} = \frac{[E:F]}{[E:F]/[F(z):F]} = [F(z):F] = c\;. \]
Now consider how $G$ and $H$ act on $S$.  The group $H$ contains exactly those elements of $G$ that fix $z$: since $z$ generates $F(z)$, any automorphism of $E/F$ that fixes $z$ also fixes every element of $F(z)$.  Thus $H = \stab~z$, the stabilizer of $z$ in $G$.  A standard result of group theory (see \cite[Theorem~1.10, p.~75]{Jacobson:BasicAlgebraI} and the text that follows the proof) is that $|S| = |G|/|\stab~z|$, because $G$ acts transitively on $S$.  Thus we have
\[ |S| = \frac{|G|}{|\stab~z|} = \frac{|G|}{|H|} = c\;, \]
and this implies that all the conjugates of $z$ ($c$ many of them) lie in $S$.

To prove the last statement, we observe that
for all $z\in S$ and $\eta\in G$ we have
\[ \eta(\{ \mu\in R : p(\mu) = z \}) = \{\mu \in R : p(\mu) = \eta(z) \}\;, \]
whence the statement follows by the transitivity of $G$ acting on $S$, and for each $z\in S$ the set $\{ \mu\in R : p(\mu) = z \}$ has size $m := |R|/|S| = d/c$.
\end{proof}

Define the \emph{field norm} $N(z)$ of $z\in E$ to be the product of its conjugates.  Up to change of sign, $N(z)$ is the constant term in the minimal (monic) polynomial of $z$.

\begin{corollary}\label{cor:norm}
Let $F$ be a field with characteristic $0$, let $\mu$ be algebraic over $F$, and let $p\in F[x]$ be a polynomial.  Then $\prod_\nu p(\nu) = N(p(\mu))^m$ for some positive integer $m$, where $\nu$ runs through the conjugates of $\mu$.
\end{corollary}

\begin{proof}
Let $f\in F\in F[x]$ be the minimal polynomial of $\mu$, and let $E$ be a splitting field of $f$ over $F$.  Let $d$ be the degree of $f$ (i.e., of $\mu$), let $c$ be the degree of $p(\mu)$, and let $m := d/c$.  Then by Theorem~\ref{thm:Galois-stuff}, each conjugate of $p(\mu)$ has multiplicity $m$ in the multiset $\{ p(\nu) : \mbox{$\nu$ conjugate to $\mu$} \}$ of the conjugates of $p(\mu)$.
\end{proof}

\begin{remark}
The theorem and corollary above both go through for any $F$-rational function $p\in F(x)$, not necessarily a polynomial.
\end{remark}

The corollary below follows by setting $F := \rats$ and recalling that all polynomials are separable in characteristic~$0$.

\begin{corollary}\label{cor:Q-conjugates}
If $\lambda$ is algebraic (over $\rats$) with conjugates $\mu_0,\ldots,\mu_{d-1}$ and $p\in\rats[x]$ is a polynomial, then $p(\lambda)$ is algebraic with conjugates $p(\mu_0),\ldots,p(\mu_{d-1})$.
\end{corollary}

\begin{corollary}\label{cor:Z-algebraic}
If $\lambda$ is an algebraic integer and $z\in\ints[\lambda]$, then $z$ is an algebraic integer, and for every conjugate $c$ of $z$, there exists a conjugate $\mu$ of $\lambda$ such that $c = h_\mu(z)$, where $\map{h_\mu}{\ints[\lambda]}{\ints[\mu]}$ is the unique ring isomorphism mapping $\lambda$ to $\mu$.
\end{corollary}

\begin{proof}
This follows from Corollary~\ref{cor:Q-conjugates} and the fact (see Jacobson~\cite[Theorem~4.23]{Jacobson:BasicAlgebraI} for example) that the algebraic integers form a subring of $\complexes$.
\end{proof}

\newpage

\section{Appendix: Computer-Aided Derivations}
\label{sec:computer-derivations}

The following table gives a derivation in $Q_\lambda$ of $\alpha = 9-16\lambda$, the fundamental unit of $\ints[\lambda]$, where $\lambda = -(3+\sqrt{17})/2$ (minimal polynomial $x^2+3x-2$), used in the proof of Proposition~\ref{prop:sqrt17}.  This may not be the shortest derivation possible; it was found by a program that favors combining points with small absolute value.  Here, $\x$ means $\xl$.
\[ \begin{array}{r|l|l}
\textup{point} & \textup{value} & \textup{equals} \\ \hline
p_1    & =1-\lambda       & =1\x 0\\
p_2    & =3-5\lambda      & =p_1\x 0\\
p_3    & =13-23\lambda    & =p_2\x 0\\
p_4    & =\lambda         & =0\x 1\\
p_5    & =2-3\lambda      & =0\x p_4\\
p_6    & =-3+6\lambda     & =p_1\x p_5\\
p_7    & =12-21\lambda    & =0\x p_6\\
p_8    & =17-30\lambda    & =p_3\x p_7\\
p_9    & =-1+2\lambda     & =p_3\x p_8\\
p_{10} & =-5+9\lambda     & =p_9\x 0\\
p_{11} & =-23+41\lambda   & =p_{10}\x 0\\
p_{12} & =-105+187\lambda & =p_{11}\x 0\\
p_{13} & =4-7\lambda      & =0\x p_9\\
p_{14} & =-14+25\lambda   & =0\x p_{13}\\
p_{15} & =27-48\lambda    & =p_{10}\x p_{14}\\
p_{16} & =-96+171\lambda  & =0\x p_{15}\\
p_{17} & =-137+244\lambda & =p_{12}\x p_{16}\\
\alpha & =9-16\lambda& =p_{12}\x p_{17}
\end{array} \]

\pagebreak
The following table gives the derivation of $Y$
(from $\{0,1\}$, so that $Y \subseteq R_\lambda$) of the set
$Y$ used in Lemma~\ref{lem:Sigma7-containment}.

\begin{table}[h!]
\[ \begin{tabular}{|l|l|l|l|}
\hline
\textup{point} & \textup{element of $Y$} & \textup{value} & \textup{equals} \\ \hline
$p_0$    & $= p(0)$ &  = $0$       		  & ~~~- - \\
$p_1$    & $= 1$  &  = $1$       		  & ~~~- - \\
$p_2$    & = $\lambda$  & $= \lambda$      & = $p_0\x p_1$\\
$p_3$    & = $1-\lambda$  &$= 1-\lambda$    & = $p_1\x p_0$ \\
$p_4$    & = $p(1)$      & $=  1-\lambda+\lambda^2$    & = $p_1\x  p_2$\\
$p_5$    & = $p(-1)$     & $ = \lambda-\lambda^2 $     & = $p_0\x  p_3$\\
$p_6$    & = $p(2)$      & $= 1-2\lambda+2\lambda^2$   & = $p_3\x  p_2$\\
$p_7$    & = $p(-2)$     & $= 2\lambda-2\lambda^2$     & = $p_2\x  p_3$\\
$p_8$    & = $p(3)$      & $= 1-3\lambda+3\lambda^2$   & = $p_5\x  p_3$\\
$p_9$    & = $p(-3)$     & $= 3\lambda-3\lambda^2$     & = $p_4\x  p_2$\\
$p_{10}$   & = $p(4)$    & $= 1-4\lambda+4\lambda^2$   & = $p_5\x  p_0$\\
$p_{11}$   & = $p(-4)$   & $= 4\lambda - 4\lambda^2$   & = $p_4\x  p_1$\\
$p_{12}$   & = $p(5)$    & $= 1-4\lambda+5\lambda^2$   & = $p_0\x  p_4$\\
$p_{13}$   & = $p(-5)$   & $= 4\lambda-5\lambda^2$     & = $p_1\x  p_5$\\
$p_{14}$   & = $p(6)$    & $= 2-6\lambda+6\lambda^2$   & = $p_3\x  p_4$\\
$p_{15}$   & = $p(-6)$   & $= -1+6\lambda-6\lambda^2$  & = $p_2\x  p_5$\\
$p_{16}$   & = $p(7)$    & $= 2-7\lambda+7\lambda^2$   & = $p_7\x  p_3$\\
$p_{17}$   & = $p(-7)$   & $= -1+7\lambda-7\lambda^2$  & = $p_6\x  p_2$\\
$p_{18}$   & = $p(8)$    & $= 2-8\lambda+8\lambda^2$   & = $p_7\x  p_0$\\
$p_{19}$   & = $p(-8)$   & $= -1+8\lambda-8\lambda^2$  & = $p_6\x  p_1$\\
$p_{20}$   & = $p(9)$    & $= 2-8\lambda+9\lambda^2$   & = $p_7\x  p_2$\\
$p_{21}$   & = $p(-9)$   & $= -1+8\lambda-9\lambda^2$  & = $p_6\x  p_3$\\
$p_{22}$   & = $p(10)$   & $= 2-9\lambda+10\lambda^2$  & = $p_0\x  p_6$\\
$p_{23}$   & = $p(-10)$  & $= -1+9\lambda-10\lambda^2$ & = $p_1\x  p_5$\\
$p_{24}$   & = $p(11)$   & $= 3-11\lambda+11\lambda^2$  & = $p_3\x  p_6$\\
$p_{25}$   & = $p(-11)$  & $= -2+11\lambda-11\lambda^2$ & = $p_2\x  p_7$\\
$p_{26}$   & = $p(12)$   & $= 3-12\lambda+12\lambda^2$   & = $p_9\x  p_0$\\
$p_{27}$   & = $p(-12)$  & $= -2+12\lambda-12\lambda^2$  & = $p_8\x  p_1$\\
$p_{28}$   & = $p(13)$   & $= 3-12\lambda+13\lambda^2$   & = $p_9\x  p_2$\\
$p_{29}$   & = $p(-13)$  & $= -2+12\lambda-13\lambda^2$  & = $p_8\x  p_3$\\
$p_{30}$   & = $p(14)$   & $= 3-13\lambda+14\lambda^2$   & = $p_5\x  p_6$\\
$p_{31}$   & = $p(-14)$  & $= -2+13\lambda-14\lambda^2$  & = $p_4\x  p_7$\\
$p_{32}$   & = $p(15)$   & $=  3-14\lambda+15\lambda^2$  & = $p_0\x  p_8$\\
$p_{33}$   & = $p(-15)$  & $= -2+14\lambda-15\lambda^2$  & = $p_1\x  p_9$\\
$p_{34}$   & = $p(16)$   & $=  4-16\lambda+16\lambda^2$  & = $p_3\x  p_8$\\
$p_{35}$   & = $p(-16)$  & $=  -3+16\lambda-16\lambda^2$ & = $p_2\x  p_9$\\
$p_{36}$   & = $p(17)$   & $=  4-16\lambda+17\lambda^2$  & = $p_{11}\x p_2$\\
$p_{37}$   & = $p(-17)$  & $= -3+16\lambda-17\lambda^2$  & = $p_{10}\x p_3$\\
\hline
\end{tabular} \]
\caption{Derivation of $Y$}\label{tab:derivation}
\end{table}

\bibliography{master}\addcontentsline{toc}{part}{References}

\end{document}